\documentclass[11pt]{article}
\usepackage{mathtools}
\usepackage{authblk}
\usepackage{placeins}
\usepackage{tikz}
\usepackage{etoolbox}
\usepackage{url}
\usepackage{cite}
\usepackage{colortbl}
\usepackage{caption}
\usepackage{amsmath}
\usepackage{amssymb}
\usepackage{ifthen}
\usepackage{bm}
\usepackage{amsfonts}
\usepackage{multirow}
\usepackage{latexsym}
\usepackage{amsthm}
\usepackage{mathrsfs}
\usepackage{float}
\usepackage{varioref}
\usepackage[top=1in, bottom=1in, left=1in, right=1in]{geometry}
\usepackage{array}
\usepackage{ifpdf}
\usepackage{placeins}
\ifpdf
\usepackage{epstopdf}
\fi
\newcommand{\Dchaintwo}[4]{
\rule[-3\unitlength]{0pt}{8\unitlength}
\begin{picture}(14,5)(0,3)
\put(1,2){\ifthenelse{\equal{#1}{l}}{\circle*{2}}{\circle{2}}}
\put(2,2){\line(1,0){10}}
\put(13,2){\ifthenelse{\equal{#1}{r}}{\circle*{2}}{\circle{2}}}
\put(1,5){\makebox[0pt]{\scriptsize #2}}
\put(7,4){\makebox[0pt]{\scriptsize #3}}
\put(13,5){\makebox[0pt]{\scriptsize #4}}
\end{picture}}
\relax

\newcommand{\Dchainfive}[9]{
\rule[-4\unitlength]{0pt}{5\unitlength}
\begin{picture}(50,4)(0,3)
\put(1,1){\circle{2}}
\put(2,1){\line(1,0){10}}
\put(13,1){\circle{2}}
\put(14,1){\line(1,0){10}}
\put(25,1){\circle{2}}
\put(26,1){\line(1,0){10}}
\put(37,1){\circle{2}}
\put(38,1){\line(1,0){10}}
\put(49,1){\circle{2}}
\put(1,4){\makebox[0pt]{\scriptsize #1}}
\put(7,3){\makebox[0pt]{\scriptsize #2}}
\put(13,4){\makebox[0pt]{\scriptsize #3}}
\put(19,3){\makebox[0pt]{\scriptsize #4}}
\put(26,4){\makebox[0pt]{\scriptsize #5}}
\put(31,3){\makebox[0pt]{\scriptsize #6}}
\put(37,4){\makebox[0pt]{\scriptsize #7}}
\put(43,3){\makebox[0pt]{\scriptsize #8}}
\put(49,4){\makebox[0pt]{\scriptsize #9}}
\end{picture}}

\newcommand{\Dchainfivec}[9]{
\rule[-4\unitlength]{0pt}{5\unitlength}
\begin{picture}(50,4)(0,3)
\put(1,1){\circle{2}}
\put(2,1){\line(1,0){10}}
\put(13,1){\circle{2}}
\put(14,1){\line(1,0){10}}
\put(25,1){\circle*{2}}
\put(26,1){\line(1,0){10}}
\put(37,1){\circle{2}}
\put(38,1){\line(1,0){10}}
\put(49,1){\circle{2}}
\put(1,4){\makebox[0pt]{\scriptsize #1}}
\put(7,3){\makebox[0pt]{\scriptsize #2}}
\put(13,4){\makebox[0pt]{\scriptsize #3}}
\put(19,3){\makebox[0pt]{\scriptsize #4}}
\put(25,4){\makebox[0pt]{\scriptsize #5}}
\put(31,3){\makebox[0pt]{\scriptsize #6}}
\put(37,4){\makebox[0pt]{\scriptsize #7}}
\put(43,3){\makebox[0pt]{\scriptsize #8}}
\put(49,4){\makebox[0pt]{\scriptsize #9}}
\end{picture}}

\newcommand{\Dchainfivee}[9]{
\rule[-4\unitlength]{0pt}{5\unitlength}
\begin{picture}(50,4)(0,3)
\put(1,1){\circle{2}}
\put(2,1){\line(1,0){10}}
\put(13,1){\circle{2}}
\put(14,1){\line(1,0){10}}
\put(25,1){\circle{2}}
\put(26,1){\line(1,0){10}}
\put(37,1){\circle{2}}
\put(38,1){\line(1,0){10}}
\put(49,1){\circle*{2}}
\put(1,4){\makebox[0pt]{\scriptsize #1}}
\put(7,3){\makebox[0pt]{\scriptsize #2}}
\put(13,4){\makebox[0pt]{\scriptsize #3}}
\put(19,3){\makebox[0pt]{\scriptsize #4}}
\put(25,4){\makebox[0pt]{\scriptsize #5}}
\put(31,3){\makebox[0pt]{\scriptsize #6}}
\put(37,4){\makebox[0pt]{\scriptsize #7}}
\put(43,3){\makebox[0pt]{\scriptsize #8}}
\put(49,4){\makebox[0pt]{\scriptsize #9}}
\end{picture}}

\newcommand{\Drightofway}[9]{
\rule[-9\unitlength]{0pt}{12\unitlength}
\begin{picture}(28,12)(0,9)
\put(2,10){\ifthenelse{\equal{#1}{l}}{\circle*{2}}{\circle{2}}}
\put(3,10){\line(1,0){10}}
\put(14,10){\ifthenelse{\equal{#1}{m}}{\circle*{2}}{\circle{2}}}
\put(15,10){\line(1,1){7}}
\put(15,10){\line(1,-1){7}}
\put(22,18){\ifthenelse{\equal{#1}{t}}{\circle*{2}}{\circle{2}}}
\put(22,2){\ifthenelse{\equal{#1}{b}}{\circle*{2}}{\circle{2}}}
\put(22,3){\line(0,1){14}}
\put(2,12){\makebox[0pt]{\scriptsize #2}}
\put(8,11){\makebox[0pt]{\scriptsize #3}}
\put(14,12){\makebox[0pt]{\scriptsize #4}}
\put(19,16){\makebox[0pt][r]{\scriptsize #5}}
\put(19,4){\makebox[0pt][r]{\scriptsize #6}}
\put(24,18){\makebox[0pt][l]{\scriptsize #7}}
\put(23,10){\makebox[0pt][l]{\scriptsize #8}}
\put(24,1){\makebox[0pt][l]{\scriptsize #9}}
\end{picture}
}
\newcommand{\neworrenewcommand}[1]{\providecommand{#1}{}\renewcommand{#1}}
\newcommand{\Dchainsix}[9]{
    \neworrenewcommand{\ffoo}[2]{
\rule[-4\unitlength]{0pt}{5\unitlength}
\begin{picture}(62,4)(0,3)
\put(1,1){\circle{2}}
\put(2,1){\line(1,0){10}}
\put(13,1){\circle{2}}
\put(14,1){\line(1,0){10}}
\put(25,1){\circle{2}}
\put(26,1){\line(1,0){10}}
\put(37,1){\circle{2}}
\put(38,1){\line(1,0){10}}
\put(49,1){\circle{2}}
\put(50,1){\line(1,0){10}}
\put(61,1){\circle{2}}
\put(1,4){\makebox[0pt]{\scriptsize #1}}
\put(7,3){\makebox[0pt]{\scriptsize #2}}
\put(13,4){\makebox[0pt]{\scriptsize #3}}
\put(19,3){\makebox[0pt]{\scriptsize #4}}
\put(25,4){\makebox[0pt]{\scriptsize #5}}
\put(31,3){\makebox[0pt]{\scriptsize #6}}
\put(37,4){\makebox[0pt]{\scriptsize #7}}
\put(43,3){\makebox[0pt]{\scriptsize #8}}
\put(49,4){\makebox[0pt]{\scriptsize #9}}
\put(55,3){\makebox[0pt]{\scriptsize ##1}}
\put(61,4){\makebox[0pt]{\scriptsize ##2}}
\end{picture}
    }
    \ffoo
}

\newcommand{\Dchainfiveu}[9]{
     \neworrenewcommand{\ffooy}[1]{
\rule[-10\unitlength]{0pt}{12\unitlength}
\begin{picture}(36,10)(0,9)
\put(1,9){\circle{2}}
\put(2,9){\line(1,0){10}}
\put(13,9){\circle{2}}
\put(14,9){\line(1,0){10}}
\put(25,9){\circle{2}}
\put(26,9){\line(1,1){7}}
\put(26,9){\line(1,-1){7}}
\put(33,17){\circle{2}}
\put(33,1){\circle*{2}}
\put(33,2){\line(0,1){14}}
\put(1,12){\makebox[0pt]{\scriptsize #1}}
\put(7,11){\makebox[0pt]{\scriptsize #2}}
\put(13,12){\makebox[0pt]{\scriptsize #3}}
\put(19,11){\makebox[0pt]{\scriptsize #4}}
\put(24,12){\makebox[0pt]{\scriptsize #5}}
\put(30,14){\makebox[0pt][r]{\scriptsize #6}}
\put(29,4){\makebox[0pt][r]{\scriptsize #7}}
\put(35,16){\makebox[0pt][l]{\scriptsize #8}}
\put(34,9){\makebox[0pt][l]{\scriptsize #9}}
\put(35,1){\makebox[0pt][l]{\scriptsize ##1}}
\end{picture}
}
    \ffooy
}

\newcommand{\Dchainfivex}[9]{
     \neworrenewcommand{\ffool}[1]{
\rule[-4\unitlength]{0pt}{5\unitlength}
\begin{picture}(38,11)(0,3)
\put(1,1){\circle{2}}
\put(2,1){\line(1,0){10}}
\put(13,1){\circle{2}}
\put(13,2){\line(2,3){6}}
\put(14,1){\line(1,0){10}}
\put(25,1){\circle{2}}
\put(25,2){\line(-2,3){6}}
\put(19,12){\circle{2}}
\put(26,1){\line(1,0){10}}
\put(37,1){\circle{2}}
\put(1,4){\makebox[0pt]{\scriptsize #1}}
\put(7,3){\makebox[0pt]{\scriptsize #2}}
\put(12,4){\makebox[0pt]{\scriptsize #3}}
\put(19,3){\makebox[0pt]{\scriptsize #4}}
\put(26,4){\makebox[0pt]{\scriptsize #5}}
\put(31,3){\makebox[0pt]{\scriptsize #6}}
\put(37,4){\makebox[0pt]{\scriptsize #7}}
\put(13.3,8){\makebox[0pt]{\scriptsize #8}}
\put(23,8){\makebox[0pt][l]{\scriptsize #9}}
\put(24,14){\makebox[0pt]{\scriptsize ##1}}
\end{picture}
}
    \ffool
}

\newcommand{\Dchainfivew}[9]{
     \neworrenewcommand{\ffoor}[1]{
\rule[-10\unitlength]{0pt}{12\unitlength}
\begin{picture}(36,10)(0,9)
\put(1,9){\circle{2}}
\put(2,9){\line(1,0){10}}
\put(13,9){\circle{2}}
\put(14,9){\line(1,0){10}}
\put(25,9){\circle*{2}}
\put(26,9){\line(1,1){7}}
\put(26,9){\line(1,-1){7}}
\put(33,17){\circle{2}}
\put(33,1){\circle{2}}
\put(33,2){\line(0,1){14}}
\put(1,12){\makebox[0pt]{\scriptsize #1}}
\put(7,11){\makebox[0pt]{\scriptsize #2}}
\put(13,12){\makebox[0pt]{\scriptsize #3}}
\put(19,11){\makebox[0pt]{\scriptsize #4}}
\put(23.5,12){\makebox[0pt]{\scriptsize #5}}
\put(30,14){\makebox[0pt][r]{\scriptsize #6}}
\put(29,4){\makebox[0pt][r]{\scriptsize #7}}
\put(35,16){\makebox[0pt][l]{\scriptsize #8}}
\put(34,9){\makebox[0pt][l]{\scriptsize #9}}
\put(35,1){\makebox[0pt][l]{\scriptsize ##1}}
\end{picture}
}
    \ffoor
}

\newcommand{\Dchainfivet}[9]{
     \neworrenewcommand{\ffoou}[1]{
\rule[-10\unitlength]{0pt}{12\unitlength}
\begin{picture}(36,10)(0,9)
\put(1,9){\circle{2}}
\put(2,9){\line(1,0){10}}
\put(13,9){\circle{2}}
\put(14,9){\line(1,0){10}}
\put(25,9){\circle{2}}
\put(26,9){\line(1,1){7}}
\put(26,9){\line(1,-1){7}}
\put(33,17){\circle{2}}
\put(33,1){\circle{2}}
\put(33,2){\line(0,1){14}}
\put(1,12){\makebox[0pt]{\scriptsize #1}}
\put(7,11){\makebox[0pt]{\scriptsize #2}}
\put(13,12){\makebox[0pt]{\scriptsize #3}}
\put(19,11){\makebox[0pt]{\scriptsize #4}}
\put(24,12){\makebox[0pt]{\scriptsize #5}}
\put(30,14){\makebox[0pt][r]{\scriptsize #6}}
\put(29,4){\makebox[0pt][r]{\scriptsize #7}}
\put(35,16){\makebox[0pt][l]{\scriptsize #8}}
\put(34,9){\makebox[0pt][l]{\scriptsize #9}}
\put(35,1){\makebox[0pt][l]{\scriptsize ##1}}
\end{picture}
}
    \ffoou
}

\newcommand{\Dchainfivev}[9]{
     \neworrenewcommand{\ffoon}[1]{
\rule[-10\unitlength]{0pt}{12\unitlength}
\begin{picture}(36,10)(0,9)
\put(36,9){\circle{2}}
\put(35,9){\line(-1,0){10}}
\put(24,9){\circle{2}}
\put(23,9){\line(-1,0){10}}
\put(12,9){\circle{2}}
\put(11,9){\line(-1,1){7}}
\put(11,9){\line(-1,-1){7}}
\put(4,17){\circle{2}}
\put(4,1){\circle{2}}
\put(4,2){\line(0,1){14}}
\put(36,12){\makebox[0pt]{\scriptsize #1}}
\put(30,11){\makebox[0pt]{\scriptsize #2}}
\put(24,12){\makebox[0pt]{\scriptsize #3}}
\put(18,11){\makebox[0pt]{\scriptsize #4}}
\put(12,12){\makebox[0pt]{\scriptsize #5}}
\put(7,14){\makebox[0pt][l]{\scriptsize #6}}
\put(8,4){\makebox[0pt][l]{\scriptsize #7}}
\put(2,16){\makebox[0pt][r]{\scriptsize #8}}
\put(3,9){\makebox[0pt][r]{\scriptsize #9}}
\put(2,1){\makebox[0pt][r]{\scriptsize ##1}}
\end{picture}
}
    \ffoon
}

\newcommand{\Dchainsixa}[9]{
    \neworrenewcommand{\ffoop}[2]{
\rule[-4\unitlength]{0pt}{5\unitlength}
\begin{picture}(62,4)(0,3)
\put(1,1){\circle*{2}}
\put(2,1){\line(1,0){10}}
\put(13,1){\circle{2}}
\put(14,1){\line(1,0){10}}
\put(25,1){\circle{2}}
\put(26,1){\line(1,0){10}}
\put(37,1){\circle{2}}
\put(38,1){\line(1,0){10}}
\put(49,1){\circle{2}}
\put(50,1){\line(1,0){10}}
\put(61,1){\circle{2}}
\put(1,4){\makebox[0pt]{\scriptsize #1}}
\put(7,3){\makebox[0pt]{\scriptsize #2}}
\put(13,4){\makebox[0pt]{\scriptsize #3}}
\put(19,3){\makebox[0pt]{\scriptsize #4}}
\put(25,4){\makebox[0pt]{\scriptsize #5}}
\put(31,3){\makebox[0pt]{\scriptsize #6}}
\put(37,4){\makebox[0pt]{\scriptsize #7}}
\put(43,3){\makebox[0pt]{\scriptsize #8}}
\put(49,4){\makebox[0pt]{\scriptsize #9}}
\put(55,3){\makebox[0pt]{\scriptsize ##1}}
\put(61,4){\makebox[0pt]{\scriptsize ##2}}
\end{picture}
    }
    \ffoop
}

\newcommand{\Dchainsixb}[9]{
    \neworrenewcommand{\ffoob}[2]{
\rule[-4\unitlength]{0pt}{5\unitlength}
\begin{picture}(62,4)(0,3)
\put(1,1){\circle{2}}
\put(2,1){\line(1,0){10}}
\put(13,1){\circle*{2}}
\put(14,1){\line(1,0){10}}
\put(25,1){\circle{2}}
\put(26,1){\line(1,0){10}}
\put(37,1){\circle{2}}
\put(38,1){\line(1,0){10}}
\put(49,1){\circle{2}}
\put(50,1){\line(1,0){10}}
\put(61,1){\circle{2}}
\put(1,4){\makebox[0pt]{\scriptsize #1}}
\put(7,3){\makebox[0pt]{\scriptsize #2}}
\put(13,4){\makebox[0pt]{\scriptsize #3}}
\put(19,3){\makebox[0pt]{\scriptsize #4}}
\put(25,4){\makebox[0pt]{\scriptsize #5}}
\put(31,3){\makebox[0pt]{\scriptsize #6}}
\put(37,4){\makebox[0pt]{\scriptsize #7}}
\put(43,3){\makebox[0pt]{\scriptsize #8}}
\put(49,4){\makebox[0pt]{\scriptsize #9}}
\put(55,3){\makebox[0pt]{\scriptsize ##1}}
\put(61,4){\makebox[0pt]{\scriptsize ##2}}
\end{picture}
    }
    \ffoob
}

\newcommand{\Dchainsixc}[9]{
    \neworrenewcommand{\ffooc}[2]{
\rule[-4\unitlength]{0pt}{5\unitlength}
\begin{picture}(62,4)(0,3)
\put(1,1){\circle{2}}
\put(2,1){\line(1,0){10}}
\put(13,1){\circle{2}}
\put(14,1){\line(1,0){10}}
\put(25,1){\circle*{2}}
\put(26,1){\line(1,0){10}}
\put(37,1){\circle{2}}
\put(38,1){\line(1,0){10}}
\put(49,1){\circle{2}}
\put(50,1){\line(1,0){10}}
\put(61,1){\circle{2}}
\put(1,4){\makebox[0pt]{\scriptsize #1}}
\put(7,3){\makebox[0pt]{\scriptsize #2}}
\put(13,4){\makebox[0pt]{\scriptsize #3}}
\put(19,3){\makebox[0pt]{\scriptsize #4}}
\put(25,4){\makebox[0pt]{\scriptsize #5}}
\put(31,3){\makebox[0pt]{\scriptsize #6}}
\put(37,4){\makebox[0pt]{\scriptsize #7}}
\put(43,3){\makebox[0pt]{\scriptsize #8}}
\put(49,4){\makebox[0pt]{\scriptsize #9}}
\put(55,3){\makebox[0pt]{\scriptsize ##1}}
\put(61,4){\makebox[0pt]{\scriptsize ##2}}
\end{picture}
    }
    \ffooc
}

\newcommand{\Dchainsixd}[9]{
    \neworrenewcommand{\ffood}[2]{
\rule[-4\unitlength]{0pt}{5\unitlength}
\begin{picture}(62,4)(0,3)
\put(1,1){\circle{2}}
\put(2,1){\line(1,0){10}}
\put(13,1){\circle{2}}
\put(14,1){\line(1,0){10}}
\put(25,1){\circle{2}}
\put(26,1){\line(1,0){10}}
\put(37,1){\circle*{2}}
\put(38,1){\line(1,0){10}}
\put(49,1){\circle{2}}
\put(50,1){\line(1,0){10}}
\put(61,1){\circle{2}}
\put(1,4){\makebox[0pt]{\scriptsize #1}}
\put(7,3){\makebox[0pt]{\scriptsize #2}}
\put(13,4){\makebox[0pt]{\scriptsize #3}}
\put(19,3){\makebox[0pt]{\scriptsize #4}}
\put(25,4){\makebox[0pt]{\scriptsize #5}}
\put(31,3){\makebox[0pt]{\scriptsize #6}}
\put(37,4){\makebox[0pt]{\scriptsize #7}}
\put(43,3){\makebox[0pt]{\scriptsize #8}}
\put(49,4){\makebox[0pt]{\scriptsize #9}}
\put(55,3){\makebox[0pt]{\scriptsize ##1}}
\put(61,4){\makebox[0pt]{\scriptsize ##2}}
\end{picture}
    }
    \ffood
}

\newcommand{\Dchainsixf}[9]{
    \neworrenewcommand{\ffoof}[2]{
\rule[-4\unitlength]{0pt}{5\unitlength}
\begin{picture}(62,4)(0,3)
\put(1,1){\circle{2}}
\put(2,1){\line(1,0){10}}
\put(13,1){\circle{2}}
\put(14,1){\line(1,0){10}}
\put(25,1){\circle{2}}
\put(26,1){\line(1,0){10}}
\put(37,1){\circle{2}}
\put(38,1){\line(1,0){10}}
\put(49,1){\circle{2}}
\put(50,1){\line(1,0){10}}
\put(61,1){\circle*{2}}
\put(1,4){\makebox[0pt]{\scriptsize #1}}
\put(7,3){\makebox[0pt]{\scriptsize #2}}
\put(13,4){\makebox[0pt]{\scriptsize #3}}
\put(19,3){\makebox[0pt]{\scriptsize #4}}
\put(25,4){\makebox[0pt]{\scriptsize #5}}
\put(31,3){\makebox[0pt]{\scriptsize #6}}
\put(37,4){\makebox[0pt]{\scriptsize #7}}
\put(43,3){\makebox[0pt]{\scriptsize #8}}
\put(49,4){\makebox[0pt]{\scriptsize #9}}
\put(55,3){\makebox[0pt]{\scriptsize ##1}}
\put(61,4){\makebox[0pt]{\scriptsize ##2}}
\end{picture}
    }
    \ffoof
}

\newcommand{\Dchainsixe}[9]{
    \neworrenewcommand{\ffooe}[2]{
\rule[-4\unitlength]{0pt}{5\unitlength}
\begin{picture}(62,4)(0,3)
\put(1,1){\circle{2}}
\put(2,1){\line(1,0){10}}
\put(13,1){\circle{2}}
\put(14,1){\line(1,0){10}}
\put(25,1){\circle{2}}
\put(26,1){\line(1,0){10}}
\put(37,1){\circle{2}}
\put(38,1){\line(1,0){10}}
\put(49,1){\circle*{2}}
\put(50,1){\line(1,0){10}}
\put(61,1){\circle{2}}
\put(1,4){\makebox[0pt]{\scriptsize #1}}
\put(7,3){\makebox[0pt]{\scriptsize #2}}
\put(13,4){\makebox[0pt]{\scriptsize #3}}
\put(19,3){\makebox[0pt]{\scriptsize #4}}
\put(25,4){\makebox[0pt]{\scriptsize #5}}
\put(31,3){\makebox[0pt]{\scriptsize #6}}
\put(37,4){\makebox[0pt]{\scriptsize #7}}
\put(43,3){\makebox[0pt]{\scriptsize #8}}
\put(49,4){\makebox[0pt]{\scriptsize #9}}
\put(55,3){\makebox[0pt]{\scriptsize ##1}}
\put(61,4){\makebox[0pt]{\scriptsize ##2}}
\end{picture}
    }
    \ffooe
}

\newcommand{\DchainsixM}[9]{
    \neworrenewcommand{\fffoo}[3]{
\rule[-4\unitlength]{0pt}{5\unitlength}
\begin{picture}(50,11)(0,3)
\put(1,1){\circle{2}}
\put(2,1){\line(1,0){10}}
\put(13,1){\circle{2}}
\put(14,1){\line(1,0){10}}
\put(25,1){\circle{2}}
\put(25,2){\line(2,3){6}}
\put(26,1){\line(1,0){10}}
\put(37,1){\circle{2}}
\put(37,2){\line(-2,3){6}}
\put(31,12){\circle{2}}
\put(38,1){\line(1,0){10}}
\put(49,1){\circle{2}}
\put(1,4){\makebox[0pt]{\scriptsize #1}}
\put(7,3){\makebox[0pt]{\scriptsize #2}}
\put(13,4){\makebox[0pt]{\scriptsize #3}}
\put(19,3){\makebox[0pt]{\scriptsize #4}}
\put(24,4){\makebox[0pt]{\scriptsize #5}}
\put(31,3){\makebox[0pt]{\scriptsize #6}}
\put(38,4){\makebox[0pt]{\scriptsize #7}}
\put(43,3){\makebox[0pt]{\scriptsize #8}}
\put(49,4){\makebox[0pt]{\scriptsize #9}}
\put(26,8){\makebox[0pt]{\scriptsize ##1}}
\put(36,8){\makebox[0pt]{\scriptsize ##2}}
\put(34,13){\makebox[0pt]{\scriptsize ##3}}
\end{picture}
    }
    \fffoo
}

\newcommand{\Dchainsixze}[9]{
    \neworrenewcommand{\fffooq}[3]{
\rule[-4\unitlength]{0pt}{5\unitlength}
\begin{picture}(50,11)(0,3)
\put(1,1){\circle{2}}
\put(2,1){\line(1,0){10}}
\put(13,1){\circle{2}}
\put(14,1){\line(1,0){10}}
\put(25,1){\circle{2}}
\put(26,1){\line(1,0){10}}
\put(37,1){\circle{2}}
\put(37,2){\line(2,3){6}}
\put(43,12){\circle{2}}
\put(38,1){\line(1,0){10}}
\put(49,1){\circle{2}}
\put(49,2){\line(-2,3){6}}
\put(1,4){\makebox[0pt]{\scriptsize #1}}
\put(7,3){\makebox[0pt]{\scriptsize #2}}
\put(13,4){\makebox[0pt]{\scriptsize #3}}
\put(19,3){\makebox[0pt]{\scriptsize #4}}
\put(24,4){\makebox[0pt]{\scriptsize #5}}
\put(31,3){\makebox[0pt]{\scriptsize #6}}
\put(37,4){\makebox[0pt]{\scriptsize #7}}
\put(43,3){\makebox[0pt]{\scriptsize #8}}
\put(50,4){\makebox[0pt]{\scriptsize #9}}
\put(38,8){\makebox[0pt]{\scriptsize ##1}}
\put(48,8){\makebox[0pt]{\scriptsize ##2}}
\put(46,13){\makebox[0pt]{\scriptsize ##3}}
\end{picture}
    }
    \fffooq
}

\newcommand{\DchainsixMc}[9]{
    \neworrenewcommand{\fffooc}[3]{
\rule[-4\unitlength]{0pt}{5\unitlength}
\begin{picture}(50,11)(0,3)
\put(1,1){\circle{2}}
\put(2,1){\line(1,0){10}}
\put(13,1){\circle{2}}
\put(14,1){\line(1,0){10}}
\put(25,1){\circle*{2}}
\put(25,2){\line(2,3){6}}
\put(26,1){\line(1,0){10}}
\put(37,1){\circle{2}}
\put(37,2){\line(-2,3){6}}
\put(31,12){\circle{2}}
\put(38,1){\line(1,0){10}}
\put(49,1){\circle{2}}
\put(1,4){\makebox[0pt]{\scriptsize #1}}
\put(7,3){\makebox[0pt]{\scriptsize #2}}
\put(13,4){\makebox[0pt]{\scriptsize #3}}
\put(19,3){\makebox[0pt]{\scriptsize #4}}
\put(24,4){\makebox[0pt]{\scriptsize #5}}
\put(31,3){\makebox[0pt]{\scriptsize #6}}
\put(38,4){\makebox[0pt]{\scriptsize #7}}
\put(43,3){\makebox[0pt]{\scriptsize #8}}
\put(49,4){\makebox[0pt]{\scriptsize #9}}
\put(26,8){\makebox[0pt]{\scriptsize ##1}}
\put(36,8){\makebox[0pt]{\scriptsize ##2}}
\put(34,13){\makebox[0pt]{\scriptsize ##3}}
\end{picture}
    }
    \fffooc
}

\newcommand{\rneworrenewcommand}[1]{\providecommand{#1}{}\renewcommand{#1}}
\newcommand{\Dchainseven}[9]{
    \rneworrenewcommand{\ffoos}[4]{
\rule[-4\unitlength]{0pt}{5\unitlength}
\begin{picture}(62,4)(0,3)
\put(1,1){\circle{2}}
\put(2,1){\line(1,0){8}}
\put(11,1){\circle{2}}
\put(12,1){\line(1,0){8}}
\put(21,1){\circle{2}}
\put(22,1){\line(1,0){8}}
\put(31,1){\circle{2}}
\put(32,1){\line(1,0){8}}
\put(41,1){\circle{2}}
\put(42,1){\line(1,0){8}}
\put(51,1){\circle{2}}
\put(52,1){\line(1,0){8}}
\put(61,1){\circle{2}}
\put(1,4){\makebox[0pt]{\scriptsize #1}}
\put(6,3){\makebox[0pt]{\scriptsize #2}}
\put(11,4){\makebox[0pt]{\scriptsize #3}}
\put(16,3){\makebox[0pt]{\scriptsize #4}}
\put(21,4){\makebox[0pt]{\scriptsize #5}}
\put(26,3){\makebox[0pt]{\scriptsize #6}}
\put(31,4){\makebox[0pt]{\scriptsize #7}}
\put(36,3){\makebox[0pt]{\scriptsize #8}}
\put(41,4){\makebox[0pt]{\scriptsize #9}}
\put(46,3){\makebox[0pt]{\scriptsize ##1}}
\put(51,4){\makebox[0pt]{\scriptsize ##2}}
\put(56,3){\makebox[0pt]{\scriptsize ##3}}
\put(61,4){\makebox[0pt]{\scriptsize ##4}}
\end{picture}
 }
    \ffoos
}

\newcommand{\Dchainsevena}[9]{
    \rneworrenewcommand{\ffooz}[4]{
\rule[-4\unitlength]{0pt}{5\unitlength}
\begin{picture}(62,4)(0,3)
\put(1,1){\circle{2}}
\put(2,1){\line(1,0){8}}
\put(11,1){\circle*{2}}
\put(12,1){\line(1,0){8}}
\put(21,1){\circle{2}}
\put(22,1){\line(1,0){8}}
\put(31,1){\circle{2}}
\put(32,1){\line(1,0){8}}
\put(41,1){\circle{2}}
\put(42,1){\line(1,0){8}}
\put(51,1){\circle{2}}
\put(52,1){\line(1,0){8}}
\put(61,1){\circle{2}}
\put(1,4){\makebox[0pt]{\scriptsize #1}}
\put(6,3){\makebox[0pt]{\scriptsize #2}}
\put(11,4){\makebox[0pt]{\scriptsize #3}}
\put(16,3){\makebox[0pt]{\scriptsize #4}}
\put(21,4){\makebox[0pt]{\scriptsize #5}}
\put(26,3){\makebox[0pt]{\scriptsize #6}}
\put(31,4){\makebox[0pt]{\scriptsize #7}}
\put(36,3){\makebox[0pt]{\scriptsize #8}}
\put(41,4){\makebox[0pt]{\scriptsize #9}}
\put(46,3){\makebox[0pt]{\scriptsize ##1}}
\put(51,4){\makebox[0pt]{\scriptsize ##2}}
\put(56,3){\makebox[0pt]{\scriptsize ##3}}
\put(61,4){\makebox[0pt]{\scriptsize ##4}}
\end{picture}
 }
    \ffooz
}

\newcommand{\Dchainsevenb}[9]{
    \rneworrenewcommand{\ffooh}[4]{
\rule[-4\unitlength]{0pt}{5\unitlength}
\begin{picture}(62,4)(0,3)
\put(1,1){\circle*{2}}
\put(2,1){\line(1,0){8}}
\put(11,1){\circle{2}}
\put(12,1){\line(1,0){8}}
\put(21,1){\circle{2}}
\put(22,1){\line(1,0){8}}
\put(31,1){\circle{2}}
\put(32,1){\line(1,0){8}}
\put(41,1){\circle{2}}
\put(42,1){\line(1,0){8}}
\put(51,1){\circle{2}}
\put(52,1){\line(1,0){8}}
\put(61,1){\circle{2}}
\put(1,4){\makebox[0pt]{\scriptsize #1}}
\put(6,3){\makebox[0pt]{\scriptsize #2}}
\put(11,4){\makebox[0pt]{\scriptsize #3}}
\put(16,3){\makebox[0pt]{\scriptsize #4}}
\put(21,4){\makebox[0pt]{\scriptsize #5}}
\put(26,3){\makebox[0pt]{\scriptsize #6}}
\put(31,4){\makebox[0pt]{\scriptsize #7}}
\put(36,3){\makebox[0pt]{\scriptsize #8}}
\put(41,4){\makebox[0pt]{\scriptsize #9}}
\put(46,3){\makebox[0pt]{\scriptsize ##1}}
\put(51,4){\makebox[0pt]{\scriptsize ##2}}
\put(56,3){\makebox[0pt]{\scriptsize ##3}}
\put(61,4){\makebox[0pt]{\scriptsize ##4}}
\end{picture}
 }
    \ffooh
}

\newcommand{\Dchainsevenf}[9]{
    \rneworrenewcommand{\ffoom}[4]{
\rule[-4\unitlength]{0pt}{5\unitlength}
\begin{picture}(62,4)(0,3)
\put(1,1){\circle{2}}
\put(2,1){\line(1,0){8}}
\put(11,1){\circle{2}}
\put(12,1){\line(1,0){8}}
\put(21,1){\circle{2}}
\put(22,1){\line(1,0){8}}
\put(31,1){\circle*{2}}
\put(32,1){\line(1,0){8}}
\put(41,1){\circle{2}}
\put(42,1){\line(1,0){8}}
\put(51,1){\circle{2}}
\put(52,1){\line(1,0){8}}
\put(61,1){\circle{2}}
\put(1,4){\makebox[0pt]{\scriptsize #1}}
\put(6,3){\makebox[0pt]{\scriptsize #2}}
\put(11,4){\makebox[0pt]{\scriptsize #3}}
\put(16,3){\makebox[0pt]{\scriptsize #4}}
\put(21,4){\makebox[0pt]{\scriptsize #5}}
\put(26,3){\makebox[0pt]{\scriptsize #6}}
\put(31,4){\makebox[0pt]{\scriptsize #7}}
\put(36,3){\makebox[0pt]{\scriptsize #8}}
\put(41,4){\makebox[0pt]{\scriptsize #9}}
\put(46,3){\makebox[0pt]{\scriptsize ##1}}
\put(51,4){\makebox[0pt]{\scriptsize ##2}}
\put(56,3){\makebox[0pt]{\scriptsize ##3}}
\put(61,4){\makebox[0pt]{\scriptsize ##4}}
\end{picture}
 }
    \ffoom
}

\newcommand{\Dchainseveng}[9]{
    \rneworrenewcommand{\ffoog}[4]{
\rule[-4\unitlength]{0pt}{5\unitlength}
\begin{picture}(62,4)(0,3)
\put(1,1){\circle{2}}
\put(2,1){\line(1,0){8}}
\put(11,1){\circle{2}}
\put(12,1){\line(1,0){8}}
\put(21,1){\circle*{2}}
\put(22,1){\line(1,0){8}}
\put(31,1){\circle{2}}
\put(32,1){\line(1,0){8}}
\put(41,1){\circle{2}}
\put(42,1){\line(1,0){8}}
\put(51,1){\circle{2}}
\put(52,1){\line(1,0){8}}
\put(61,1){\circle{2}}
\put(1,4){\makebox[0pt]{\scriptsize #1}}
\put(6,3){\makebox[0pt]{\scriptsize #2}}
\put(11,4){\makebox[0pt]{\scriptsize #3}}
\put(16,3){\makebox[0pt]{\scriptsize #4}}
\put(21,4){\makebox[0pt]{\scriptsize #5}}
\put(26,3){\makebox[0pt]{\scriptsize #6}}
\put(31,4){\makebox[0pt]{\scriptsize #7}}
\put(36,3){\makebox[0pt]{\scriptsize #8}}
\put(41,4){\makebox[0pt]{\scriptsize #9}}
\put(46,3){\makebox[0pt]{\scriptsize ##1}}
\put(51,4){\makebox[0pt]{\scriptsize ##2}}
\put(56,3){\makebox[0pt]{\scriptsize ##3}}
\put(61,4){\makebox[0pt]{\scriptsize ##4}}
\end{picture}
 }
    \ffoog
}

\newcounter{cthm}%

\newcommand{\Aut }{\mathrm{Aut}}
\newcommand{\Hom }{\mathrm{Hom}}
\newcommand{\End }{\mathrm{End}}

\newcommand{\cB }{\mathcal{B}}
\newcommand{\cC }{\mathcal{C}}
\newcommand{\cR }{\mathcal{R}}

\newcommand{\cD }{\mathcal{D}}

\newcommand{\cI }{\mathcal{I}}

\newcommand{\cX }{\mathcal{X}}
\newcommand{\cY }{\mathcal{Y}}
\newcommand{\cW }{\mathcal{W}}

\newcommand{\gr }{\mathrm{gr}}

\newcommand{\ndN }{\mathbb{N}}

\newcommand{\ndZ }{\mathbb{Z}}

\newcommand{\roots }{\boldsymbol{\Delta }}

\newcommand{\YD }{Yetter--Drinfel'd }

\newcommand{\ydH }{ {}^H_H\mathcal{YD}}
\newcommand{\ydD }{ {}^{G}_{G}\mathcal{YD}}

\newcommand{\id}{\mathrm{id}}
\newcommand{\al }{\alpha }

\newcommand{\btxandshort}[1]{and}%
\newcommand{\btxpagesshort}[1]{pp.}%
\newcommand{\Btxinshort}[1]{In}%
\newcommand{\btxphdthesis}[1]{phd-thesis}%
\newcommand{\btxeditorshort}[1]{Ed.}%
\newcommand{\btxeditorsshort}[1]{Eds.}%
\newcommand{\btxvolumeshort}[1]{vol.}%
\newcommand{\btxofseriesshort}[1]{ser.}%
\newcommand{\ffg }{\mathcal{F}_r^{G} }   % irreducible objects
\newcommand{\fiso }{\mathcal{X}_r }      % family of isoclasses
\newcommand{\rersys }[1]{\boldsymbol{\Delta }{}^{#1\,\mathrm{re}}}
\newcommand{\rsys }{\boldsymbol{\Delta }}
\newcommand{\ad }{\mathrm{ad}}

\newtheorem{theorem}{\bf Theorem}[section]
\newtheorem{lemma}[theorem]{\bf Lemma}
\newtheorem{prop}[theorem]{\bf Proposition}
\newtheorem{cor}[theorem]{\bf Corollary}

\newtheorem{question}[theorem]{\bf Question}
\newtheorem{definition}[theorem]{\bf Definition}

\newtheorem{remark}[theorem]{\bf Remark}

\newcommand{\gDd }{\mathcal{D}}

\begin{document}

\setlength{\unitlength}{1mm}

\title{Higher rank Nichols algebras of diagonal type with finite arithmetic root system in positive characteristic}

\author{Lijie Lei, Chen Yuan, Chen Qian, Jing Wang\affil{Beijing Forestry University, 100083, Beijing, China.}
\thanks{Corresponding author: wang\_jing619@163.com}
\thanks{supported by National Natural Science Foundation of China(No.12471038).}
\date{}
}

\maketitle
\begin{abstract}
The classification of finite dimensional Nichols algebras is a key part of the lifting method by N. Andruskiewitsch and H.-J. Schneider for finite dimensional pointed Hopf algebras. In this paper, we classify all rank $r\geq 5$ Nichols algebras of diagonal type with finite irreducible root systems over fields of positive characteristic. Using Weyl groupoids, arithmetic root systems, and Cartan graphs, we show that such Nichols algebras correspond exactly to the generalized Dynkin diagrams listed in Table \ref{tab.1}.

\end{abstract}

\section{Introduction}
The theory of Nichols algebras has its origins in the theory of Hopf algebras and plays a significant role in that field. Furthermore, it has applications in other areas of research, such as Katz–Moody Lie superalgebras and conformal field theory~\cite{inp-Andr14,Lentner2021,Lentner2022,Semi-2011,Semi-2012,Semi-2013}.

Classifying all finite dimensional Hopf algebras remains a long-standing challenging open problem in Hopf algebra theory~\cite{inp-Andr02}. The classification of finite dimensional Nichols algebras is a core ingredient of the lifting method proposed by N. Andruskiewitsch and H.-J. Schneider for the classification of finite dimensional pointed Hopf algebras~\cite{AnS1998,AnS2000,AnS2001,AnS2010}.
Nichols algebras arise canonically in the structure analysis of finite dimensional pointed Hopf algebras via coradical filtrations. For a Hopf algebra $H$, consider its coradical filtration $$H_0\subset H_1\subset \ldots $$ such that $H_0$ is a Hopf subalgebra of $H$ and the associated graded coalgebra $$\gr H=\bigoplus _{n\geq0} H_n/H_{n-1},$$ where $H_{-1}=0$. Then $\gr H$ is a graded Hopf algebra, since the coradical $H_0$ of $H$ is a Hopf
subalgebra. Consider the projection $\pi\colon \gr H\to H_{0}$ and let $R$ be the algebra of coinvariants of $\pi$. A classical result due to Radford and Majid guarantees that $R$ is a braided Hopf algebra and $\gr H$ is the bosonization (or
biproduct) of $R$ and $H_{0}\colon \gr H \simeq  R\# H_{0}$.

The lifting method proceeds in three successive steps: first analyze the braided Hopf algebra $R$, then pull structural data back to $\gr H$ via bosonization, and finally lift information from $\gr H$ to $H$ \cite{AnS1998,AnS2001}. 
The braided Hopf algebra $R$ is generated by the vector space $V$ of $H_0$-coinvariants of $H_1/H_0$, namely Nichols algebra $\cB(V)$ which is generated by $V$. in commemoration of W.~Nichols who started to study these objects as bialgebras of type one~\cite{AnS1998,bialgebras}. Notably, the degree-one generation conjecture put forward by N. Andruskiewitsch and H.-J. Schneider fails over fields of positive characteristic, leaving an important direction for subsequent research. There exist multiple equivalent formulations for describing Nichols algebras; see for instance \cite{Lusztig01,Lusztig10,Rosso98,w1987,w1989}.

Accordingly, determining all finite dimensional Nichols algebras $\cB(V)$ becomes an indispensable prerequisite for classifying pointed Hopf algebras. This motivates the following open question originally raised by N.~Andruskiewitsch.

\begin{question} \textbf{(N.~Andruskiewitsch~\cite{inp-Andr02})}\label{quse:classification} \it ~
 Given a braiding matrix $(q_{ij})_{1\leq i,j\leq \theta} $ whose entries are roots of $1$, when $\cB(V)$ is finite dimensional, where $V$ is a vector space with basis $x_1,\dots ,x_{\theta}$ and braiding $c(x_i\otimes x_j)=q_{ij}(x_j\otimes x_i)$? If so,
compute $\dim_\Bbbk \cB(V)$, and give a ``nice'' presentation by generators and relations.
\end{question}

Over fields of characteristic zero, extensive classifications of finite dimensional and infinite dimensional Nichols algebras of Cartan type have been developed by numerous authors, with representative results presented in \cite{AnS2000, zbMATH05027328, Rosso98}. I. Heckenberger further accomplished the complete classification of all finite dimensional Nichols algebras of diagonal type in \cite{HeckcAS3, a-Heck04d, a-Heck04bb, HeckCAS}, while the generator-relation structures of such Nichols algebras are systematically investigated in \cite{Ang1,Ang2}.

Based on these foundational classification results, N. Andruskiewitsch and H.-J. Schneider \cite{AnS2010} established a classification theorem for finite dimensional pointed Hopf algebras under suitable technical assumptions. In subsequent series of works \cite{Lifting1,Lifting2,Lifting3,pointlifting}, the full classification of finite dimensional pointed Hopf algebras with abelian coradicals was finally completed. Motivated by these fruitful achievements in characteristic zero, the study of Nichols algebraic structures over fields of positive characteristic has emerged as a research topic of great significance, offering promising prospects for further applications.

To date, finite dimensional Nichols algebras of diagonal type of rank $2$, $3$, and $4$ over fields of positive characteristic have been classified in \cite{HW,W3,W4}. On the one hand, these low rank classification results demonstrate that the classical framework for characteristic zero cannot be directly extended to positive characteristic, necessitating innovative concepts and tools to characterize the distinctive structural features in positive characteristic settings. On the other hand, these classifications yield vital applications to the structure theory of odd-dimensional pointed Hopf algebras. Specifically, they provide an explicit characterization for the finite dimensionality of Nichols algebras associated with prime-dimensional \YD modules. Both research directions fundamentally depend on the aforementioned classification theory as an indispensable core tool \cite{HMV,AHV2024}.

The classification of finite dimensional Nichols algebras relies principally on two core theoretical tools: the combinatorics of root systems and the Weyl groupoid associated to such algebras. First, I. Heckenberger \cite{zbMATH05027328} introduced arithmetic root systems and Weyl groupoids of Nichols algebra of diagonal type based on the ideas of V. Kharchenko on Poincar\'{e}–Birkhoff–Witt basis of such algebras \cite{Khar1999}. Subsequent work extended root systems and Weyl groupoids to broader families of Nichols algebras \cite{a-AHS08,right13,HS10}. 
To quantify the scope of this generalization, M. Cuntz and I. Heckenberger completed a full classification of all finite Weyl groupoids in \cite{caH12,caH15}. In particular, the main theorem from \cite{caH15} reads as follows and furnishes the foundational framework for our present investigation:
\begin{theorem}\label{rootsystems}
There are exactly three families of connected simply connected Cartan graphs for which the real roots form a finite irreducible root system$\colon$
\begin{itemize}
\item[$(1)$] The family of Cartan graphs of rank two parametrized by triangulations of a convex \( n \)-gon by non-intersecting diagonals.

\item[$(2)$] For each rank \( r > 2 \), the standard Cartan graphs of type \( A_r, B_r, C_r \) and \( D_r \), and a series of \( r - 1 \) further Cartan graphs described explicitly in Theorem \ref{rootsd}.

\item[$(3)$] A family consisting of 74 further ``sporadic" Cartan graphs (including those of type $F_4, E_6, E_7$ and $E_8$).
\end{itemize}
\end{theorem}

This work concerns Nichols algebras of diagonal type with finite irreducible root systems over fields of positive characteristic. Our main contribution is the following theorem, which provides a complete classification of all such Nichols algebras of rank $r\geq 5$.
\begin{theorem}
Let $\Bbbk$ be a field of characteristic $p>0$. Let $r\geq5$ and $I=\{1,2,\ldots,r\}$. Let $(V,c)$ be a braided vector space of diagonal type over $\Bbbk$ with basis $\{x_k|k\in I\}$ satisfying
\begin{equation*}
c(x_i \otimes x_j) = q_{ij}x_j \otimes x_i , \quad q_{ij} \in \Bbbk^*.
\end{equation*}
Suppose that the braiding matrix $(q_{ij})_{i,j\in I}$ is indecomposable and set $M\coloneqq (\Bbbk x_i)_{i\in I}$. Then the Nichols algebra $\cB(V)$ generated by $(V,c)$ has a finite set of roots ${\roots}^{[M]}$ if and only if
the generalized Dynkin diagram $\cD$ of $V$ is listed in Table~\ref{tab.1}. 
\end{theorem}

The rest of the paper is organized as follows. Section $\ref{two}$ reviews the fundamental notations and preliminary results required throughout this work. One key preliminary result allows us to associate a connected finite Cartan graph of rank \(r\) with a tuple \(M=(M_1, \dots, M_r)\) consisting of one-dimensional \YD modules over an abelian group. For \(r=1\), the corresponding Cartan graph carries no valid structural information of \(\mathcal{B}(V)\). We therefore restrict our consideration to the case \(r\geq 2\). A finite Cartan graph is said to be non-sporadic if it is not sporadic.
Section $\ref{three}$ provides an explicit characterization of finite Cartan graphs of rank \(r\geq5\). We divide such finite Cartan graphs into two categories$\colon$ non-sporadic and sporadic ones. For non-sporadic cases, we first recall the construction of root systems for classical finite Cartan graphs of types \(A_r\), \(B_r\), \(C_r\) and \(D_r\), as well as a family of additional Cartan graphs of type \(D'(r,s)\) with \(s\in \{1,\dots ,r-1\}\). To enumerate all \(D'(r,s)\)-type Cartan graphs for admissible values of \(s\), we introduce the notion of a good \(D'_r\) neighborhood and employ GAP software to implement its construction and computation.
We prove in Theorem \ref{Dynkindiagrams} that every non-sporadic finite Cartan graph of rank $r\geq5$ is either standard or contains a point admitting a good $D'_r$ neighborhood. In the second subsection we further introduce the concept of good $A_r$ neighborhoods to explicitly characterize sporadic finite Cartan graphs. The relevant structural properties are formally summarized in Theorem \ref{thm:goodnei}. Finally, Section~\ref{sec:classification} establishes our main classification result in Theorem~\ref{theoremking}. As a direct consequence,
all finite dimensional rank $r \geq 5$ Nichols algebras of diagonal type over fields of positive characteristic are completely classified,
as recorded in Corollary~\ref{coro-cla}. In addition, Table \ref{tab.1} consists of all associated generalized Dynkin diagrams and in Table \ref{tab.2} we present the corresponding exchange graphs for the sporadic case of $\cC(M)$.

Throughout this paper, $\Bbbk$ denotes a field of characteristic $p> 0$.
Let $\Bbbk^*=\Bbbk\setminus \{0\}$ and $\ndN_0=\ndN\bigcup \{0\}$. For any $n \in \ndN$, define $G'_n=\{q\in \Bbbk^*|\,\, q^n=1, q^k\not=1~\text{for all}~ 1\leq k < n\}$.
Let $r \in \ndN_0$ with $ r\geq 2$, and let $I = \{1,2,3, \ldots, r\}$. We denote by $ E=(\alpha_i)_{i \in I} $ the standard basis of $ \mathbb{Z}^I $ , and let $\chi$ be a bicharacter on $\mathbb{Z}^I$. Set $q_{ij}\coloneqq \chi(\alpha_i,\alpha_j)$.
%and  $(\alpha_i)_{i \in I}$ be the standard basis of $\ndZ^I$.
%\newpage
\section{Preliminaries}\label{two}
%\section{Preliminaries}

We recall the standard definitions from\cite{CaH,HW,HaY,W3,W4,HeckCAS}. Further details can be found in \cite{hopfandroot}.

\subsection{\YD modules and braided vector space}
In this section, we recall the basic definitions of braided vector spaces, braiding matrices, and Yetter–Drinfeld modules, with particular emphasis on those of diagonal type, and establish their equivalence relation.
\begin{definition}
Let $V$ be an $r$-dimensional vector space over $\Bbbk$. A pair $(V, c)$ is called a \textit{braided vector space}, if
$c\in \Aut(V\otimes V)$ satisfies the braid equation$\colon$
\begin{equation*}
 (c \otimes \id)(
 \id \otimes c)(c \otimes \id) = (\id \otimes c)(c \otimes \id)(\id\otimes c).
\end{equation*}
A braided vector space $(V, c)$ is said to be \textit{of diagonal type}
if $V$ admits a basis $\{x_i | i\in I\}$ such that 
\begin{equation*}
  c(x_i \otimes x_j) = q_{ij}x_j \otimes x_i \quad \textit{for some} \quad q_{ij} \in \Bbbk^*.
\end{equation*}
for all $i, j \in I$.
\end{definition}
\begin{definition}
The matrix $(q_{ij})_{i,j\in I}$ is called the \textit{braiding matrix} of $V$. The braiding matrix $(q_{ij})_{i,j\in I}$ is \textit{indecomposable} if for any $i\not=j$ there exists a sequence $i_1 = i, i_2, \dots, i_t = j$ of elements of $I$ such that $q_{i_si_{s+1}}q_{i_{s+1}i_s}\not= 1$ for all $1\leq s\leq t-1$.
\end{definition}
This paper studies braided vector spaces of diagonal type with indecomposable braiding matrices.
\begin{definition}
Let $H$ be a Hopf algebra. A \textit{\YD module} $V$ over $H$ is a left $H$-module with action $\lambda_L \colon 
H\otimes  V \longrightarrow V$ and a left $H$-comodule with coaction $\delta_L\colon V \longrightarrow H \otimes V$ satisfying the compatibility condition
\begin{equation*}
\delta_L(h.v) = h_{(1)}v_{(-1)} \kappa (h_{(3)}) \otimes h_{(2)}.v_{(0)}, h \in H, v\in V.
\end{equation*}
 A \YD module $V$ is \textit{of diagonal type} if $H=\Bbbk G$ for some abelian group $G$ and $V$ decomposes as a direct sum of one-dimensional \YD modules over the group algebra $\Bbbk G$.
\end{definition}

We write $\ydH$ for the category of \YD modules over $H$, whose morphisms preserve both the action and coaction of $H$.
The category $\ydH$ is braided with braiding
\begin{equation}\label{def.brading}
  c_{V,W}\colon V\otimes W \longrightarrow W \otimes V , \quad v\otimes w \longrightarrow  v_{(-1)}.w \otimes v_{(0)}
\end{equation}
for all $V, W\in \ydH$, $v\in V$, and $w\in W$. In fact, the category $\ydH$ is a braided monoidal category, whose monoidal structure is given by the tensor product over $\Bbbk$. In particular, every \YD module $V\in \ydH$ over $H$ admits a braiding $c_{V,V}$, making $(V, c_{V,V})$ a braided vector space. Conversely, a braided vector space can be realized as a \YD module over some Hopf algebras if and only if the braiding is rigid \cite[Section 2]{TakM}. Note that different Hopf algebras may yield isomorphic braidings on $V$
via distinct Yetter–Drinfeld module structures,
and not every braiding on $V$ is of the form \eqref{def.brading}.

Now let $ H = \Bbbk G $. We write $\ydD$ for the category of \YD  module over $\Bbbk G$, and we say that an object $ V \in \ydD $ is a  \YD  module over $G$. Notice that if $ V $ is of diagonal type, then $(V, c_{V,V})$ is a braided vector space of diagonal type. Conversely, any braided vector space of diagonal type can be regarded as a \YD module of diagonal type. Indeed, let $(V, c)$ be a braided vector space of diagonal type with an indecomposable braiding matrix $(q_{ij})_{i,j\in I}$ of a basis $\{x_i|i \in I\}$. Let $ G_0 $ be the free abelian group generated by elements $\{g_i|i \in I\}$. We define the left coaction and left action as follows
\[
\delta_L\colon V \longrightarrow \Bbbk G_0 \otimes V, \quad
x_i \longrightarrow g_i \otimes x_i, \quad g_i.x_j = q_{ij}x_j \in V.
\]
Then $ V = \oplus_{i \in I}  \Bbbk x_i $ and each $ \Bbbk x_i $ is one-dimensional \YD over $ G_0 $. Consequently, $V$ is a Yetter-Drinfel’d module of diagonal type over $ G_0 $.

\subsection{Nichols Algebras of Diagonal Type}
In this section, we recall the definition of the Nichols algebra $\cB(V)$ generated by $(V, c)$.
\begin{definition}
Let $H$ be a Hopf algebra. A \textit{braided Hopf algebra} in $\ydH$ is a $6$-tuple $\cB=(\cB, \mu, 1, \delta, \epsilon, \mathcal{S})$, where $(B, \mu, 1)$ is an algebra in $\ydH$ and a coalgebra $(B, \delta, \epsilon)$ in $\ydH$, and $\mathcal{S}\colon B\rightarrow B$ is a morphism in $\ydH$ such that $\delta$, $\epsilon$ and $\mathcal{S}$ satisfy $ \mathcal{S}(b^{(1)})b^{(2)}=b^{(1)}\mathcal{S}(b^{(2)})=\epsilon(b) 1$, where we define $\delta(b)=b^{(1)}\otimes b^{(2)}$ as the coproduct of $\cB$ to avoid the confusion.
\end{definition}
Let $(V, c)$ be an $r$-dimensional braided vector space of diagonal type.
The Nichols algebras generated by $(V, c)$ is defined as follows.
%can be defined in the following way. %There are equivalent descriptions of the definition of Nichols algebras in \cite{w1987,w1989,Rosso98,Lusztig01,Lusztig10}.

\begin{definition}
  The \textit{Nichols algebra} generated by $V\in \ydH$ is defined as the quotient
  \[\cB(V)=T(V)/\cI(V)=(\oplus_{n=0}^{\infty} V^{\otimes n})/\cI(V)\]
  where $\cI(V)$ is the unique maximal coideal of $T(V)$ contained in $\oplus_{n\geq 2}V^{\otimes n}$.
 A Nichols algebra $\cB(V)$ is said to be~\textit{of diagonal type} if $V$ is a \YD module of diagonal type.
  The dimension of $V$ is called the \textit{rank} of Nichols algebra $\cB(V)$.
\end{definition}

Note that if $B\in \ydH$ and $B$ is an algebra in $\ydH$, then $B\otimes B$ also admits an algebra structure in $\ydH$, whose  multiplication is given by
\begin{equation}\label{eq-alg}
  (a\otimes b)(c\otimes d)=a\bigl(b_{(-1)}. c\bigr) \otimes b_{(0)}d,
\end{equation}
for all $a, b, c, d\in B$, where $.$ stands for the left action of $H$ on $B$.

The \textit{tensor algebra} $T(V)$ carries a natural structure of a \YD module and admits an algebra structure in $\ydH$. Moreover, $T(V)$ is an $\ndN_0$-graded braided Hopf algebra in $\ydH$. Its comultiplication $\delta(v)$ and counit $\epsilon$ are defined on generators by $\delta(v)=1\otimes v+ v\otimes 1$ and $\epsilon(v)=0$ for all $v\in V$, where $\delta$ and $\epsilon$ are algebra morphisms. The antipode of $T(V)$ is well-defined \cite[Section 2.1]{AnS2001}. Note that the multiplication on $T(V)$ given by Equation (\ref{eq-alg}) is the unique algebra structure compatible with the condition $\delta(v)=1\otimes v+ v\otimes 1\in T(V)\otimes T(V)$ for all $v\in V$. The comultiplication can be uniquely extended from $V$ to $T(V)$. For instance, for any $v, w\in V$ (we omit the tensor product symbol for elements of $T(V)$ for simplicity), one computes
 \begin{align*}
 \begin{split}
\delta(vw)=&\delta(v)\delta(w)\\
           =&(1\otimes v+ v\otimes 1)(1\otimes w+ w\otimes 1)\\
           =& 1\otimes vw+ v_{(-1)}.w\otimes v_{(0)}+v\otimes w+vw\otimes 1.
\end{split}
\end{align*}
Let $(I_i)_{i\in I}$ denote the family of all coideals of $T(V)$ contained in $\oplus_{n\geq 2}V^{\otimes n}$, which satisfy
\[\delta(I_i)\subset I_i\otimes T(V)+ T(V)\otimes I_i.\]
Then $\cI(V)\coloneqq \sum_{i\in I}I_i$ is again a coideal. Consequently, the quotient $\cB(V)$ is a braided Hopf algebra in $\ydH$. By \cite[Proposition 3.2.12]{AnGr1998}, the Nichols algebra $\cB(V)$ is the unique $\ndN_0$-graded braided Hopf algebra generated by $V$ in $\ydH$ with graded components satisfying $\cB(V)(0)=\Bbbk$, $\cB(V)(1)=V$, and $P(\cB(V))=V$. Here $P(\cB(V))$ stands for the space of primitive elements of $\cB(V)$.

\subsection{Cartan graphs, root systems, and Weyl groupoids}
\begin{definition} \it 
A \textit{generalized Cartan matrix} is a matrix $A =(a_{ij})_{i,j\in I}$ with integer entries such that
\begin{itemize}
   \itemsep=0pt
     \item[$(1)$] $a_{ii}=2$ and $a_{jk}\le 0$ for any $i,j,k\in I$ with
    $j\not=k$,
     \item[$(2)$]  if $i,j\in I$ and $a _{ij}=0$, then $a_{ji}=0$.
\end{itemize}
\end{definition}

A generalized Cartan matrix $A\in \ndZ ^{I \times I}$ is \textit{decomposable} if there exists a nonempty proper subset $I_1\subset I$ such that $a_{ij}=0$ for any $i \in I_1$ and $j \in I\setminus I_1$. And $A$ is called \textit{indecomposable} if it is not decomposable.

\begin{definition}
  Let $\cX$ be a non-empty set and $A^X =(a_{ij}^X)_{i,j \in I}$ be a generalized Cartan matrix for all $X\in \cX$.
  For any $i\in I$ let $r_i \colon \cX \to \cX$, $X \mapsto r(i,X)$, where $r\colon I\times \cX \to \cX$ is a map.
  The quadruple
  \[\cC = \cC (I, \cX, r, (A^X)_{X \in \cX})
  \]
  is called a \textit{semi-Cartan graph} if
  $r_i^2 = \id_{\cX}$ for all $i \in I$,  and $a^X_{ij} = a^{r_i(X)}_{ij}$ for all $X\in \cX$ and $i,j\in I$.
  We say that a semi-Cartan graph $\cC$ is~\textit{indecomposable}
  if $A^X$ is indecomposable for all $X\in \cX$.
\end{definition}

The elements of the set $\{r_i(X), i\in I\}$ are called the \textit{neighbors} of $X$ for all $X\in \cX$.
The cardinality of $I$ is called the \textit{rank} of $\cC$ and the elements of $\cX$ are called the \textit{points} of $\cC$.

\begin{definition}
The \textit{exchange graph} of $\cC (I, \cX, r, (A^X)_{X \in \cX})$ is a labeled non-oriented graph with vertices corresponding to the elements of $\cX$, and edges marked by elements
$i \in I$, where two vertices $X, Y$ are connected by an edge $i$ if and only if $X\not=Y$ and $r_i(X)=Y$ (and $r_i(Y)=X$).
\end{definition}

Note that the exchange graph of $\cC$ may have multiple edges with different labels. For simplicity, we display only one edge with several labels instead of several edges.  

\begin{definition}\label{def: standard}
Let $\cC$ be a semi-Cartan graph. $\cC$ is said to be \textit{connected} if there is no proper non-empty subset $\cY \subseteq \cX$ such that $r_i(Y) \in \cY$ for all $i\in I$, $Y\in \cY$.
\end{definition}
Let $\cC = \cC (I, \cX, r, (A^X)_{X \in \cX})$ be a connected semi-Cartan graph. For any $X\in \cX$ and any $i\in I$ let
\begin{equation}\label{eq-si}
s_i^X\in \Aut(\ndZ^I), ~~
s_i^X \alpha _j=\alpha_j-a_{ij}^X \alpha_i,
\end{equation}
for all $j\in I$.
Let $\mathcal{D}(\mathcal{X}, I)$ be the category such that Ob$\mathcal{D}(\mathcal{X}, I)=\mathcal{X}$ and morphisms $\Hom(X,Y)=\{(Y,f,X)|f\in \End(\mathbb{Z} ^I)\}$ for $X, Y\in \mathcal{X}$,
where the composition of morphisms is defined by $(Z,g,Y)\circ (Y,f,X)=(Z, gf, X)$ for all $ X, Y, Z\in \mathcal{X}$, $f,g\in \End(\mathbb{Z} ^I)$. 
Let \textit{$\cW(\cC)$} be the smallest subcategory of $\cD(\cX, I)$ which contains all morphisms $(r_i(X), s_i^X, X)$ with $i\in I$, $X\in \cX$. For the sake of simplicity, we write $s_i^X$ instead of $(r_i(X), s_i^X, X)$ if no confusion is possible. Notice that $\cW(\cC)$ is a groupoid since all generators are invertible.

\begin{definition}
For all $X\in \cX$, the set
\begin{equation}\label{eq-realroots}
  \rersys{X}=\{\omega\alpha_i \in \ndZ^I|\omega \in \Hom(\cW(\cC),X)\}
\end{equation}
is called the set of real roots of $\cC$ at $X$. The real roots $\alpha_i$, $i\in I$, are called simple. The elements of $\rersys{X}_{\boldsymbol{+}}=\rersys{X}\cap \ndN_0^I$ and $\rersys{X}_{\boldsymbol{-}}=\rersys{X}\cap -\ndN_0^I$ are called positive roots and negative roots, respectively. If the set $\rersys{X}$ is finite for all $X\in \cX$, then we say that $\cC$ is finite.
\end{definition}

\begin{definition}\label{def.Weylgroupoid}
The semi-Cartan graph $\cC$ is called a Cartan graph if the following hold $\colon$
\begin{itemize}
  \item[$(1)$] For all $X\in \cX$, the set $\rersys{X}$ consists of positive and negative roots.
   \item[$(2)$] If $m_{ij}^X\coloneqq |\rersys{X}\cap (\ndN_0 \alpha_i+\ndN_0 \alpha_j)|$ is finite, then $(r_i r_j)^{m_{ij}^X}(X)=X$, where $i, j\in I$ and $X\in \cX$.
\end{itemize}
In this case, $\cW(\cC)$ is called the Weyl groupoid of $\cC$.
\end{definition}

\begin{definition}%\label{def.rootsystem}
$\cR=\cR(\cC,(\rsys ^{X})_{X\in \cX})$ is a \textit{root system of type $\cC$} if for all $X \in \cX $, the sets $\rsys ^X $ are the subsets of $\ndZ ^I$ such that
\begin{itemize}
\itemsep=0pt
\item[$(1)$] $\rsys ^X=(\rsys^ X\cap \ndN _0^I)\cup -(\rsys ^X\cap \ndN _0^I)$.
\item[$(2)$] $\rsys ^X\cap \ndZ \al _i=\{\al _i,-\al _i\}$ for all $i\in I$.
\item[$(3)$] $s_i^X(\rsys ^X)= \rsys ^{r_i(X)}$ for all $i \in I$.
\item[$(4)$] $(r_i r_j)^{m_{ij}^X}(X) = X$ for all $i\not=j \in I$ where $m_{ij}^X= |\rsys^X\cap (\ndN _0\al _i + \ndN _0 \al_j)|$ is finite.
\end{itemize}
For all $X\in \cX$, $\cR$ is~\textit{finite} if $\rsys ^{X}$ is finite.
\end{definition}

We call $\cW(R)\coloneqq \cW(\cC)$ is the groupoid of $\cR$. As in \cite[Definition 4.3]{caHW} we say that $\cR$ is \textit{reducible} if there exist non-empty disjoint subsets of $I',I''\subset I$ such that $I=I'\cup I''$ and $a_{i j}=0$ for all $i\in I'$, $j\in I''$ and
  \[\rsys ^{X}=\Big(\rsys ^{X}\cap \sum _{i\in I'}\ndZ \al _i\Big)\cup
  \Big(\rsys ^{X}\cap \sum _{j\in I''}\ndZ \al _j\Big)\qquad
  \text{for all}\quad X\in \cX.\]
In this case, we write $\cR =\cR |_{I_1}\oplus \cR |_{I_2}$. Conversely, we say that $\cR$ is irreducible if it is not reducible.

Let $\cR=\cR(\cC,(\rsys ^{X})_{X\in \cX})$ be a root system of type $\cC$. Next we recall the following properties of $\cR$ from \cite{caHW} and \cite{caH12}.

\begin{lemma}\label{lem:jik}
 Let $X\in \cX$, $k \in \ndZ$, and $i, j\in I$ such that $i\not= j$. Then $\alpha _j + k\alpha_i\in \rersys{X}$ if and only if $0\leq k\leq -a_{ij}^X$.
\end{lemma}

 \begin{lemma}\label{lem:finite}
   Let $\cC$ be a connected semi-Cartan graph and $\cR=\cR(\cC,(\rsys ^{X})_{X\in \cX})$ be a root system of type $\cC$. Then the following are equivalent.
   \begin{itemize}
   \itemsep=0pt
     \item[$(1)$] $\cR$ is finite.
     \item[$(2)$] $\rsys^{X}$ is finite for some $X\in \cX$.
     \item[$(3)$] $\cC$ is finite.
     \item[$(4)$] $\cW(\cR)$ is finite.
   \end{itemize}
\end{lemma}

\begin{prop}\label{prop.indecom}
  Let $\cC$ be a connected semi-Cartan graph and $\cR=\cR(\cC,(\rsys ^{X})_{X\in \cX})$ be a root system of type $\cC$. Then
  \begin{itemize}
    \item[$(1)$] there exists $X\in \cX$ such that $A^X$ is indecomposable if and only if the semi-Cartan graph $\cC$ is indecomposable.
    \item[$(2)$] if $\cR$ is finite then the semi-Cartan graph $\cC$ is indecomposable if and only if the root system $\cR$ is irreducible.
   \end{itemize}
\end{prop}

The following proposition implies that if $\cR$ is a finite root system of type $\cC$ then all roots are real and $\cR$ is uniquely determined by $\cC$.

\begin{prop} \label{prop.allposroots}
  Let $\cR=\cR(\cC,(\rsys ^{X})_{X\in \cX})$ be a root system of type $\cC$. Let $X\in \cX$, $m\in \ndN _0$, and $i_1,\ldots ,i_m\in I$ such that
  \[\omega =\id_X s _{i_1} s_{i_2}\cdots s_{i_m}\in \Hom(\cW(\cC), X)\] and $\ell (\omega )=m$.
  Then the elements
  \[\beta _n=\id_X s_{i_1} s_{i_2} \cdots s_{i_{n-1}}(\alpha_{i_n})\in \rsys^ X\cap \ndN _0^I,\]
  are pairwise different, where $n\in \{1,2,\ldots ,m\}$ $($and $\beta _1=\alpha _{i_1}$$)$. Here,
  \[\ell(\omega)=\mathrm{min}\{m\in \ndN_0|\omega=\id_X s_{i_1}s_{i_2}\cdots s_{i_m}, i_1, i_2, \ldots, i_m\in I\}\]
   is the length of $\omega\in \Hom(\cW(\cC), X)$.
  In particular, if $\cR$ is finite and
  $\omega \in \Hom (\cW (\cC ))$ is the longest element, 
  then
  \[\{\beta _n\,|\,1\le n\le \ell (\omega )=|\rsys ^{X}|/2\}=\rsys^ X\cap \ndN _0^I.\]
\end{prop}

\begin{remark}%\label{re:irre}
 If $\cC$ is a finite Cartan graph then $\cR$ is finite and hence $$\cR^{re}=\cR^{re}(\cC,(\rersys{X})_{X\in \cX})$$ is the unique root system of type $\cC$ by Proposition~\ref{prop.allposroots}, that is, $\cR$ is uniquely determined by $\cC$. And we say that $(\cR^{re})^X_{+}$ is the set of positive real roots at $X\in \mathcal{X}$.
\end{remark}

\subsection{The semi-Cartan Graph of a Nichols Algebra of Diagonal Type}
In this section, we attach a semi-Cartan graph to a tuple of finite dimensional \YD modules under some finiteness conditions.
Let $G$ be an abelian group and $\Bbbk G$ its group algebra. Let $\ydD$ denote the category of \YD modules over $\Bbbk G$. Let $\ffg$ be the collection of $r$-tuples of finite dimensional irreducible objects in $\ydD$.
Let $\fiso^G$ be the set of $r $-tuples of isomorphism classes of finite dimensional irreducible objects in $\ydD$. For any $(M_1, \dots, M_{r})\in \ffg$, we write $[M]\coloneqq ([M_1], \dots, [M_{r}])\in \fiso^G$ to represent its corresponding isomorphism class.

Fix a tuple $M=(\Bbbk x_1, \Bbbk x_2, \dots, \Bbbk x_{r})\in \ffg$ consisting of one-dimensional \YD over $G$ and $[M]\in \fiso^G$. Let $V_M=\oplus_{i\in I}\Bbbk x_i\in \ydD$ be a \YD module of diagonal type over $G$, where $\{x_i|i\in I\}$ forms a basis of $V$. Then there exist a matrix $(q_{ij})_{i,j\in I}$ and a set of elements $\{g_i|i\in I\} \subset G$ such that $\delta(x_i)=g_i\otimes x_i$ and $g_i. x_j=q_{ij}x_j$ for all $i,j\in I$. We refer to $(q_{ij})_{i,j\in I}$ as the braiding matrix of $M$. This matrix is invariant up to permutation of the index set $I$ and does not rely on the choice of basis $\{x_i|i\in I\}$. We define the Nichols algebra of the tuple $M$ as $\cB(V_M)=\cB(\oplus_{i=1}^{n}\Bbbk x_i)$, which is simply denoted by $\cB(M)$.

For a Nichols algebra $\cB(M)$ with antipode $S$, the adjoint action $\ad$ is given in \cite{hopfandroot} as the linear map 
\[\ad \colon \cB(M) \otimes \cB(M) \rightarrow \cB(M) \] satisfying
\[
x \otimes y \rightarrow  \ (ad x)(y) \coloneqq x_{(1)}yS(x_{(2)}) 
\] 
for all $x, y \in \cB(M)$.

In order to construct a semi-Cartan graph attached to $M$, we recall several finiteness conditions from \cite{AnS2001} and \cite{HS10}.
%\cite[defnition~6.4]{HS10}

\begin{definition}
Let $i\in I$. $M$ is \textit{$i$-finite} if for any $j\in I\setminus \{i\}$, $(\ad_{c} x_i)^m (x_j)=0$ for some $m\in \ndN$.
\end{definition}
%~\cite[Lemma~3.7]{AnS2001}
\begin{lemma}\label{le:aijM}
  For any $i,j\in I$ with $i\not=j$,
  the following are equivalent.
  \begin{itemize}
    \item[$(1)$] $(m+1)_{q_{ii}}(q_{ii}^mq_{ij}q_{ji}-1)=0$ and $(k+1)_{q_{ii}}(q_{ii}^kq_{ij}q_{ji}-1)\not=0$ for all $0\leq k<m$.
    \item[$(2)$] $(ad_{c}x_i)^{m+1}(x_j)=0$ and $(ad_{c}x_i)^m(x_j)\not=0$ in $\cB (V)$.
  \end{itemize}
  Here $(n)_q\coloneqq 1+q+\cdots+q^{n-1}$,
  which is $0$ if and only if $q^n=1$ for $q\not=1$ or $p|n$ for $q=1$. Notice that $(1)_q\not=0$ for any $q\in \ndN$.
\end{lemma}

The lemma below follows directly from Lemma \ref{le:aijM}.
%We get the following lemma from Lemma~\ref{le:aijM}.

\begin{lemma}\label{lem:aij}
  Let $i\in I$.
  Then $M=(\Bbbk x_j)_{j\in I}$ is $i$-finite if and only if for any $j\in I\setminus\{i\}$ there is a non-negative integer $m$ satisfying $(m+1)_{q_{ii}}(q_{ii}^mq_{ij}q_{ji}-1)=0$.
\end{lemma}

Let $i\in I$. Assume that $M$ is $i$-finite. Let $(a_{ij}^{M})_{j\in I}\in \ndZ^I$ and $R_i(M)=({R_i(M)}_j)_{j\in I}$, where
\begin{align*}
  a_{ij}^M=&
  \begin{cases}
    2& \text{if $j=i$,}\\
    -\mathrm{max}\{m\in \ndN_0 \,|\,(\ad_c  x_i)^m(x_j)\not=0 \}& \text{if $j\not=i$.}
  \end{cases}
\end{align*}
\begin{equation}\label{eq-ri}
{R_i(M)}_i= \Bbbk y_i,\qquad
{R_i(M)}_j= \Bbbk(\ad_{c}x_i)^{-a_{ij}^M}(x_j),
\end{equation}
where $y_i\in (\Bbbk x_i)^*\setminus \{0\}$. If $M$ is not $i$-finite, then let $R_i(M)=M$.
Then $R_i(M)$ is an $r$-tuple of one-dimensional \YD modules over $G$.

Let
\[
\ffg(M)=\{R_{i_1} \cdots R_{i_n}(M)\in \ffg|\, n\in \ndN_0, i_1,\dots, i_n\in I\}
\]
and
\[
\fiso^G(M)=\{[R_{i_1} \cdots R_{i_n}(M)]\in \fiso^G |\,n\in \ndN_0, i_1,\dots, i_n\in I\}.
\]

\begin{definition}\label{defn-admitsallref}
 If $N$ is $i$-finite for all $N\in \ffg(M)$, then $M$ \textit{admits all reflections}.
\end{definition}
Notice that these reflections depend solely on the braiding matrix $(q_{ij})_{i,j\in I}$. We next recall the definition of generalized Dynkin diagram\cite{Heck2008}. %for a braided vector spaces of diagonal type \cite{Heck2008}.

\begin{definition} % % \label{def.Dyndia}
Let $V$ be an $r$-dimensional braided vector space of diagonal type with the braiding matrix $(q_{ij})_{i,j\in I}$. The $\textit{generalized Dynkin diagram}$ of $V$ is a non-directed graph $\cD$ with the following properties$\colon$
\begin{itemize}
\itemsep=0pt
\item[$(1)$] there is a bijective map $\phi$ from $I$ to the vertices of $\cD$,
\item[$(2)$] for all $i\in I$ the vertex $\phi (i)$ is labeled by $q_{ii}$,
\item[$(3)$]  for all $i,j\in I$ with $i\not=j$,
 the number $n_{ij}$ of edges between $\phi (i)$ and $\phi (j)$ is either $0$ or $1$. If $q_{ij}q_{ji}=1$ then $n_{ij}=0$, otherwise $n_{ij}=1$ and the edge is labeled by $q_{ij}q_{ji}.$
\end{itemize}
\end{definition}
The \textit{generalized Dynkin diagram of $M$} is defined as the generalized Dynkin diagram associated with the braided vector space $\oplus_{i\in I}M_i$.
In particular, the generalized Dynkin diagram of $M$ is connected if its braiding matrix is indecomposable.

To characterize the labels of the generalized Dynkin diagram of $R_i(M)=(R_i(M)_j)_{j\in I}$, we establish the following results.

%We can get the labels of the generalized Dynkin diagram of $R_i(M)=(R_i(M)_j)_{j\in I}$ by the following lemma.
\begin{lemma}\label{jslemma}
Let $i\in I$.
Assume that $M$ is $i$-finite and let $a_{ij}\coloneqq a_{ij}^M$ for all $j\in I$.
Let $(q'_{jk})_{j,k\in I}$ be the braiding matrix of $R_i(M)$ with respect to $(y_j)_{j\in I}$.
Then
    \begin{gather*}
q_{jj}'=
 \begin{cases}
 q_{ii} & \text{if $j=i$},\\
 q_{jj}  & \text{if $j\not=i$, $q_{ij}q_{ji}=q_{ii}^{a_{ij}}$},\\
 q_{ii}q_{jj}{(q_{ij}q_{ji})^{-a_{ij}}} & \text{if $j\not=i$, $q_{ii}\in G'_{1-a_{ij}}$},\\
 q_{jj}{(q_{ij}q_{ji})^{-a_{ij}}} & \text{if $j\not=i$, $q_{ii}=1$},
 \end{cases}
 \end{gather*}
 \begin{gather*}
 q_{ij}'q_{ji}'=
 \begin{cases}
 q_{ij}q_{ji} & \text{if $j\not=i$, $q_{ij}q_{ji}=q_{ii}^{a_{ij}}$},\\
 q_{ii}^2(q_{ij}q_{ji})^{-1} & \text{if $j\not=i$, $q_{ii}\in G'_{1-a_{ij}}$},\\
 (q_{ij}q_{ji})^{-1} & \text{if $j\not=i$, $q_{ii}=1$},
 \end{cases}
 \end{gather*}
 and
 \begin{gather*}
q_{jk}'q_{kj}'=
 \begin{cases}
 q_{jk}q_{kj}  & \text{if $q_{ir}q_{ri}=q_{ii}^{a_{ir}}$, $r\in \{j, k\}$},\\
 q_{jk}q_{kj}(q_{ik}q_{ki}q_{ii}^{-1})^{-a_{ij}}& \text{if $q_{ij}q_{ji}=q_{ii}^{a_{ij}}$, $q_{ii}\in G'_{1-a_{ik}}$},\\
 q_{jk}q_{kj}(q_{ij}q_{ji})^{-a_{ik}}(q_{ik}q_{ki})^{-a_{ij}} &  \text{if $q_{ii}=1$,}\\
%                                                         &  \text{$p=1-a_{ik}$, $q_{ii}=1$, $q_{ik}q_{ki}\not=1$},\\
 q_{jk}q_{kj}q_{ii}^{2}(q_{ij}q_{ji}q_{ik}q_{ki})^{-a_{ij}} & \text{if $q_{ii}\in G'_{1-a_{ik}}$, $q_{ii}\in G'_{1-a_{ij}}$},
  \end{cases}
 \end{gather*}
 for $j\not=k$, $|j-i|=1$, $|k-i|=1$. For $|j-i|> 1$, $|k-i|> 1$, $q'_{jk}q'_{kj}=q_{jk}q_{kj}$.
Here, $G_n'$ denotes the set of primitive $n$-th roots of $1$
in $\Bbbk$, that is, $G'_n=\{q\in \Bbbk^*|\,\, q^n=1, q^k\not=1~\text{for all}~ 1\leq k < n\}$ for $n\in \ndN$.
 \end{lemma}

If $M$ admits all reflections, then we are able to construct a semi-Cartan graph $\cC(M)$ of $M$ by \cite[Proposition 1.5]{W3}.
\begin{theorem}\label{theo.regualrcar}
  Assume that $M$ admits all reflections.
  For all $X\in \fiso^G(M)$
  let \[[X]_{r}=\{Y\in \fiso^G(M)| \,\text{Y and X have the same generalized Dynkin diagram}\}.\]
  Let $\cY_{r}(M)=\{[X]_{r} |\, X\in \fiso^G(M)\}$
  and $A^{[X]_{r}}=A^X$ for all $X\in \fiso^G(M)$.
  Let $t\colon I\times \cY_{r}(M)\rightarrow \cY_{r}(M)$,
  $(i, [X]_{r})\mapsto [R_i(X)]_{r}$.
  Then the tuple
  \[
  \cC(M)=\{I, \cY_{r}(M), t, (A^Y)_{Y\in \cY_{r}(M)}\}
  \]
  is a connected semi-Cartan graph. We say that $\cC(M)$ is the semi-Cartan graph attached to $M$.
  \end{theorem}

Furthermore, we can attach a groupoid $\cW(M)\coloneqq \cW(\cC(M))$ to $M$ if $M$ admits all reflections.\par
Notice that
Nichols algebra $\cB(M)$ is $\ndN_0^{r}$-graded with
$\deg M_i=\al_i $ for all $i\in I$.
Following the terminology in~\cite{HS10},
we say that the
Nichols algebra $\cB(M)$ is \textit{decomposable}
if there exists a totally ordered index set $(L,\le)$ and a sequence
$(W_l)_{l\in L}$ of finite dimensional irreducible $\ndN _0^r $-graded objects in
$\ydD $
such that
\begin{equation}\label{eq-decom}
  \cB (M)\simeq
  \bigotimes _{l\in L}\cB (W_l).
\end{equation}
For each decomposition~(\ref{eq-decom}),
we define the set of~ \textit{positive roots} $\rsys^{[M]}_{+}\subset \ndZ^I$ and the set of~ \textit{roots} $\rsys^{[M]}\subset \ndZ^I$ of $[M]$ by
$$\rsys^{[M]}_{+}=\{\deg(W_l)|\, l\in L\}, \quad
\rsys^{[M]}=\rsys^{[M]}_{+}\cup-\rsys^{[M]}_{+}.$$
By~\cite[Theorem~4.5]{HS10} we obtain that
the set of roots $\rsys^{[M]}$ of $[M]$ does not depend on the choice of the decomposition.
\begin{remark}\label{rem-decom}
 If $\dim M_i=1$ for all $i\in I$,
 then the Nichols algebra $\cB(M)$ is decomposable by a theorem of V.~Kharchenko~\cite[Theorem~2]{Khar1999}. The set of roots of Nichols algebra $\cB(M)$ is thus well-defined and denoted by $\rsys^{[M]}$. Whenever the set of roots $\rsys^{[M]}$ is finite, $M$ admits all reflections by \cite[Corollary 6.12]{HS10}.
\end{remark}

If $M$ admits all reflections and $\rsys^{[M]}$ is finite, then we can define a finite root system \\ $\cR(M)(\cC(M),(\rsys^{[N]})_{N\in \ffg(M)})$ of type $\cC(M)$ \cite[Theorem 1.25]{W4}.
\begin{theorem} \label{thm:rootofR_M}
 Assume that $M$ admits all reflections. Then the following are equivalent.
\begin{itemize}
\itemsep=0pt
\item[$(1)$] $\rsys^{[M]}$ is finite.
\item[$(2)$] $\cC(M)$ is a finite Cartan graph.
\item[$(3)$] $\cW(M)$ is finite.
\item[$(4)$] $\cR(M)\coloneqq \cR(M)(\cC(M),(\rsys^{[N]})_{N\in \ffg(M)})$ is finite.
\end{itemize}
In all cases, $\cR(M)$ is the unique root system of type $\cC(M)$.
\end{theorem}

To simplify the presentation of the generalized Dynkin diagrams, we recall the notion of simple chains\cite{HeckCAS}. 

\begin{definition}\label{simplechains}
Assume that $(\Delta, \chi, E)$ is an arithmetic root system of rank $r > 2$ and $\mathcal{D}_{\chi, E}$ is a labeled path graph. Call this graph a \textit{simple chain} (of length $r$) if 
\[
(q_{11} q_{12} q_{21} - 1)(q_{11} + 1) = 0, \quad (q_{rr} q_{r,r-1} q_{r-1,r} - 1)(q_{rr} + 1) = 0
\]
\[
q_{ii}^2 q_{i-1,i} q_{i,i-1} q_{i,i+1} q_{i+1,i} = 1 \quad \text{for } 1 < i < r
\]
The latter equations hold if and only if 
\begin{equation*}\label{chains}
q_{ii} + 1 = q_{i-1,i} q_{i,i-1} q_{i,i+1} q_{i+1,i} - 1 = 0 \quad \text{or} \quad q_{ii} q_{i-1,i} q_{i,i-1} = q_{ii} q_{i,i+1} q_{i+1,i} = 1.
\end{equation*}
\end{definition}

Let $ \mathcal{D}_{\chi,E} $ be a simple chain and set $ q \coloneqq q^2_{rr}q_{r-1,r}q_{r,r-1} $. By Definition \ref{simplechains}, we obtain 
\[
q = 
\begin{cases}
q_{rr} & \text{if } q_{rr} \neq -1, \\
q_{r-1,r}q_{r,r-1} & \text{if } q_{rr}=-1.
\end{cases}
\]
\noindent
Then $q^2_{11}q_{12}q_{21} \in \{q,q^{-1}\}$ and $q_{i-1,i}q_{i-1,i} \in \{q,q^{-1}\}$ for all $i \in \{2,\ldots r\}$. Moreover, we observe that the simple chain $\mathcal{D}_{\chi,E}$ is uniquely determined by the value of $q$ and the indices $i$ $( 1 \leq i \leq r)$ satisfying $q_{i-1,i}q_{i,i-1} = q$, with $q_{01}q_{10}$ defined as $1/(q^2_{11}q_{12}q_{21})$. In this case, we use the symbol
\begin{align*}\label{eq-simplechaingraph}
\begin{picture}(26,6)
\put(13,3){\oval(26,6)}
\put(0,0){\makebox(26,6){\scriptsize $C(r,q;i_1,\ldots ,i_j)$}}
\end{picture}
\end{align*}

\noindent to denote $\mathcal{D}_{\chi,E}$ , where $q_{i-1,i}q_{i,i-1} = q$ if and only if $i \in \{i_1,i_2,\ldots,i_j\}$. 

\section{Finite Cartan graphs}\label{three}

In this section, we establish the key properties of finite connected indecomposable Cartan graphs of rank $r\geq5$ in Theorem~\ref{Dynkindiagrams} and Theorem~\ref{thm:goodnei}, which provide the foundation for our classification result in the subsequent section.

Let $M=(\Bbbk x_1,\dots,\Bbbk x_r)$ be a tuple of one-dimensional Yetter–Drinfeld modules over $\Bbbk G$. Suppose that $M$ admits all reflections, and let $\mathcal{C}(M)=\mathcal{C}(I,\mathcal{X},r,(A^X)_{X\in\mathcal{X}})$ be the associated indecomposable semi-Cartan graph of $M$. Recall that a semi-Cartan graph $\mathcal{C}$ is called standard if $A^X=A^Y$ for all $X,Y\in\mathcal{X}$.

Set $ \textit{A}_r \coloneqq
\begin{pmatrix}
    2 & -1 & 0 & 0 & 0 & \cdots & 0 & 0 & 0 \\
    -1 & 2 & -1 & 0 & 0 & \cdots & 0 & 0 & 0 \\
    0 & -1 & 2 & -1 & 0 & \cdots & 0 & 0 & 0 \\
    0 & 0 & -1 & 2 & -1 & \cdots & 0 & 0 & 0 \\
    \vdots & \vdots & \vdots & \vdots & \vdots & \ddots & \vdots & \vdots & \vdots \\
    0 & 0 & 0 & 0 & 0 & \cdots & -1 & 2 & -1\\
    0 & 0 & 0 & 0 & 0 & \cdots & 0 & -1 & 2
    
\end{pmatrix}$,

$\textit{B}_r \coloneqq
\begin{pmatrix}
    2 & -2 & 0 & 0 & 0 & \cdots & 0 & 0 & 0 \\
    -1 & 2 & -1 & 0 & 0 & \cdots & 0 & 0 & 0 \\
    0 & -1 & 2 & -1 & 0 & \cdots & 0 & 0 & 0 \\
    0 & 0 & -1 & 2 & -1 & \cdots & 0 & 0 & 0 \\
    \vdots & \vdots & \vdots & \vdots & \vdots & \ddots & \vdots & \vdots & \vdots \\
    0 & 0 & 0 & 0 & 0 & \cdots & -1 & 2 & -1\\
    0 & 0 & 0 & 0 & 0 & \cdots & 0 & -1 & 2
\end{pmatrix}$,

%\textit{B}_r^{\prime}
$\textit{C}_r \coloneqq B_r^{\mathrm{T}} =
\begin{pmatrix}
    2 & -1 & 0 & 0 & 0 & \cdots & 0 & 0 & 0 \\
    -2 & 2 & -1 & 0 & 0 & \cdots & 0 & 0 & 0 \\
    0 & -1 & 2 & -1 & 0 & \cdots & 0 & 0 & 0 \\
    0 & 0 & -1 & 2 & -1 & \cdots & 0 & 0 & 0 \\
    \vdots & \vdots & \vdots & \vdots & \vdots & \ddots & \vdots & \vdots & \vdots \\
    0 & 0 & 0 & 0 & 0 & \cdots & -1 & 2 & -1\\
    0 & 0 & 0 & 0 & 0 & \cdots & 0 & -1 & 2
\end{pmatrix}$,

$\textit{D}_r \coloneqq
\begin{pmatrix}
    2 & 0 & -1 & 0 &0 & \cdots & 0 & 0 & 0 \\
    0 & 2 & -1 & 0 & 0& \cdots & 0 & 0 & 0 \\
    -1 & -1 & 2 & -1 & 0 & \cdots & 0 & 0 & 0 \\
    0 & 0 & -1 & 2 & -1 & \cdots & 0 & 0 & 0 \\
    \vdots & \vdots & \vdots & \vdots & \vdots & \ddots & \vdots & \vdots & \vdots \\
    0 & 0 & 0 & 0 & 0 & \cdots & -1 & 2 & -1\\
    0 & 0 & 0 & 0 & 0 & \cdots & 0 & -1 & 2 
\end{pmatrix}$,

and $\textit{E}_r \coloneqq
\begin{pmatrix}
    2 & -1 & 0 & 0 & 0 & 0 & \cdots & 0 & 0 & 0 \\
    -1 & 2 & 0 & -1 & 0 & 0 &\cdots & 0 & 0 & 0 \\
    0 & 0 & 2 & -1 & 0 & 0 & \cdots & 0 & 0 & 0 \\
    0 & -1 & -1 & 2 & -1 & 0 & \cdots & 0 & 0 & 0 \\
    0 & 0 & 0 & -1 & 2 & -1 & \cdots & 0 & 0 & 0 \\
    \vdots & \vdots & \vdots & \vdots & \vdots & \vdots & \ddots & \vdots & \vdots & \vdots \\
    0 & 0 & 0 & 0 & 0 & 0 & \cdots & -1 & 2 & -1\\
    0 & 0 & 0 & 0 & 0 & 0 & \cdots & 0 & -1 & 2 
\end{pmatrix}$.

\subsection{Non-sporadic finite Cartan graphs}
In this section, we present the properties of non-sporadic finite Cartan graphs of rank $r\geq 5$ in Theorem~\ref{Dynkindiagrams}, which are crucial for our classification in Theorem \ref{Dydiagrams}. We first recall the arithmetic root systems of non-sporadic finite Cartan graphs of rank $r\geq5$. Motivated by the structure of these root systems, we introduce the notion of good $D'_r$ neighborhood characterize all such Cartan graphs.

\subsubsection{The root systems}
We mainly focus on the construction of root systems of type $D'(r,s)$ for $s\in \{1,\ldots , r-1\}$ from \cite{caH15}. 
For the details on the construction of root systems of standard type $A_r,B_r,C_r,D_r$, we refer to the book\cite[Chapter~III, \S~12.1]{humphreys1972lie}, where Humphreys’ classical theory of Lie algebra root systems serves as a standard reference for their construction. 

For \( 1 \leq i, j \leq r \), set
\[
\eta_{i,j} \coloneqq
\begin{cases}
\sum_{k=i}^{j} \alpha_k & i \leq j, \\
0 & i > j.
\end{cases}
\]

\begin{definition}
Let \( Z \subseteq \{1, \ldots, r-1\} \). Let \(\Phi_{r,Z}\) denote the set of roots
\begin{align*}
&\eta_{i,j-1}, \quad 1 \leq i < j \leq r, \\
&\eta_{i,r-2} + \alpha_r, \quad 1 \leq i < r, \\
&\eta_{i,r} + \eta_{j,r-2}, \quad 1 \leq i < j < r, \\
&\eta_{j,r} + \eta_{j,r-2}, \quad j \in Z.
\end{align*}
\noindent
Let \( Y \subseteq \{1, \ldots, r-1\} \). Let \(\Psi_{r,Y}\) denote the set of roots
\begin{align*}
&\eta_{i,j}, \quad 1 \leq i \leq j \leq r, \\
&\eta_{i,r} + \eta_{j,r-1}, \quad 1 \leq i < j < r, \\
&\eta_{j,r} + \eta_{j,r-1}, \quad j \in Y.
\end{align*}
Further, let $\Psi'_{r,Y}$ denote the set obtained from \(\Psi_{r,Y}\) by swapping \(\alpha_{r-1}\) and \(\alpha_r\).
\end{definition}

\begin{remark}
The sets \(\Phi_{r,\emptyset}\) and \(\Psi_{r,\{1, \ldots, r-1\}}\) are  precisely the sets of positive real roots of the Weyl groups of type \(D_r\) and \(C_r\), respectively, as constructed in \cite{humphreys1972lie}.
\end{remark}

\begin{prop}\label{root4} 
Let \( Z_1, Z_2, Y_1, Y_2 \subseteq \{1, \ldots, r-1\} \) with
\[
|Z_1| = |Z_2| = |Y_1| + 1 = |Y_2| + 1.
\]
Then there exists a finite Cartan graph \(\mathcal{C}\) with objects \(a, b, c, d\) such that
\[
(R^{re})^a_+ = \Phi_{r, Z_1},\quad (R^{re})^b_+ = \Phi_{r, Z_2},\quad (R^{re})^c_+ = \Psi_{r, Y_1}, \quad \text{and} \quad (R^{re})^d_+ = \Psi'_{r, Y_2}.
\]
\end{prop}

\begin{remark}
Let \( Z, Y\subseteq \{1, \ldots, r-1\} \) with $|Z| = |Y| + 1 $. Then $\Phi_{r, Z}$, $\Psi_{r, Y}$ and $\Psi'_{r, Y}$ are isomorphic by Proposition \ref{root4}.
\end{remark}

\begin{definition}
Let \(\mathcal{C}\) be a finite Cartan graph of rank \(r\geq3\). If there exists a subset \(Z \subseteq \{1, \ldots, r-1\}\) such that \(\Phi_{r, Z} = (R^{\mathrm{re}})^a_+\) for some object \(a\), then \(\mathcal{C}\) is said to be of type \(D'(r, |Z|)\). If there exists a subset \(Y \subseteq \{1, \ldots, r-1\}\) such that \(\Psi_{r, Y} = (R^{\mathrm{re}})^a_+\) for some object \(a\), then \(\mathcal{C}\) is said to be of type \(D'(r, |Y| + 1)\).

Note that if \(\mathcal{C}\) is of type \(D'(r, 0)\), then it is standard of type \(D\), and if \(\mathcal{C}\) is of type \(D'(r, r)\), then it is standard of type \(C\).
\end{definition}

\begin{prop}
Let $\mathcal{C}$ be a non-sporadic finite Cartan graph of rank $r\geq 3$. Let $a$ be an object of $\mathcal{C}$ whose generalized Dynkin diagram is of type $D'_r$. Then $\Phi_{r,\{r-1\}} \subseteq (R^{\mathrm{re}})_+^a$.
\end{prop}

\begin{remark}\cite[Proposition 3.8]{caH15}
Let $\mathcal{C}$ be a non-sporadic finite Cartan graph of rank $r\geq 3$. Then either $\mathcal{C}$ is a standard Cartan graph of type $A_r, B_r, C_r, D_r$, or it contains an object with generalized Dynkin diagram of type $D'_r$.
\end{remark}

\begin{theorem} \label{rootsd}
Let \(\mathcal{C}\) be a non-sporadic finite Cartan graph of rank \(r\geq3\), and set  
\[
\mathcal{R}_+ \coloneqq \{ (R^{\mathrm{re}})_+^a \mid a \in \mathcal{C} \}.
\]
\noindent
Then exactly one of the following holds:  

\begin{enumerate}
\item[$(1)$] The Cartan graph \(\mathcal{C}\) is standard (\(|\mathcal{R}_+| = 1\)) of type $A_r, B_r, C_r, D_r$.  
\item[$(2)$] Up to equivalence, the root sets of \(\mathcal{C}\) are given by  
\[
\mathcal{R}_+ = \{\Phi_{r,Z}, \Psi_{r,Y}, \Psi'_{r,Y} \mid Z, Y \subseteq \{1, \ldots, r-1\}, |Z| = s, |Y| = s-1\}
\]
for some \(s \in \{1, \ldots, r-1\}\).
\end{enumerate}
\end{theorem}

\begin{remark}
Let \(\mathcal{C}\) be a finite Cartan graph of rank \(r\geq3\). All its root systems can be constructed explicitly via Theorem \ref{rootsd}. In particular, we list all such root systems for ranks $3\leq r\leq 9$ in Appendix \ref{app:rslist}, where each entry records the positive real roots of a representative object. The rank $3$ cases are taken from \cite{caH12} and there are exactly five non-sporadic Cartan graphs, since $A_3=D_3$. 
\end{remark}

\subsubsection{Good $D'_r$ neighborhood} \label{non-sporadic finite cartan graphs}

Each point of a finite Cartan graph possesses its individual neighborhood structure. We compute all positive real root sets $\Delta^{[M]}$ for non-sporadic finite Cartan graphs via GAP computations. Motivated by these computational results, we introduce the definition of a good $D'_r$ neighborhood, which enables us to cover all finite connected indecomposable Cartan graphs: every such graph necessarily contains at least one point admitting a good $D'_r$ neighborhood.

\begin{definition}\label{goodneighborhood}
We say that $X$ has a \textbf{good $D'_r$ neighborhood} if, after a suitable permutation of the index set $I$, there exists $a \in \mathbb{N}$ such that
$\textit{A}^X$=$\textit{A}^{r_i(X)}$ for all $i\in\{1,4,\cdots,r\}$,

$\textit{A}^{r_2(X)}=
\begin{pmatrix}
    2 & -1 & -1 & 0 &0& \cdots & 0 & 0 & 0 \\
    -1 & 2 & -1 & 0 & 0 &\cdots & 0 & 0 & 0 \\
    -1 & -1 & 2 & -1 & 0 & \cdots & 0 & 0 & 0 \\
    0 & 0 & -1 & 2 & -1 & \cdots & 0 & 0 & 0 \\
    \vdots & \vdots & \vdots & \vdots & \vdots & \ddots & \vdots & \vdots & \vdots \\
    0 & 0 & 0 & 0 & 0 &\cdots & -1 & 2 & -1\\
    0 & 0 & 0 & 0 & 0 & \cdots & 0 & -1 & 2    
\end{pmatrix}$, 

$\textit{A}^{r_3(X)}=
\begin{pmatrix}
    2 & -1 & 0 & 0 & 0 & \cdots & 0 & 0 & 0 \\
    -a & 2 & -1 & 0 & 0 & \cdots & 0 & 0 & 0 \\
    0 & -1 & 2 & -1 & 0 &\cdots & 0 & 0 & 0 \\
    0 & 0 & -1 & 2 & -1 &\cdots & 0 & 0 & 0 \\
    \vdots & \vdots & \vdots & \vdots & \vdots & \ddots & \vdots & \vdots & \vdots \\
    0 & 0 & 0 & 0 & 0 &\cdots & -1 & 2 & -1\\
    0 & 0 & 0 & 0 & 0 &\cdots & 0 & -1 & 2
\end{pmatrix}$, \\where $a\in\{1,2\}$ and $a^{r_3r_2(X)}_{14}=0$.
\end{definition}

The following theorem establishes a fundamental property for non-sporadic finite Cartan graphs.
\begin{theorem}\label{Dynkindiagrams}
Let $\mathcal{C}$ be a non-sporadic finite Cartan graph of rank $r\geq 5$.  Then exactly one of the two alternatives holds:

\begin{enumerate}
    \item[$(1)$] The Cartan graph $ \mathcal{C}(M) $ is standard of type $ A_r, B_r, C_r$, or $D_r $.
    \item[$(2)$] Up to equivalence, there exists a point $ Y \in \mathcal{X} $ such that $ Y $ has a good $D'_r$ neighborhood.
\end{enumerate}
\end{theorem}

\begin{proof}
Let $\cR=\cR(\cC,(\rsys ^{X})_{X\in \cX})$ be a root system of $\mathcal{C}(M)$.
For any $X \in \mathcal{X} $, let $\rersys{X}_{\boldsymbol{+}}$ be the positive roots of $X$. By Theorem~\ref{rootsd}, upon permuting $I$, there exists a point $X \in \mathcal{X} $ satisfying that the set $\rersys{X}_{\boldsymbol{+}}$ is of the type \( A_r, B_r, C_r, D_r,\) or \(D'(r, s)\) for $s\in \{1,\ldots , r-1\}$. There are precisely $r-1+4$ distinct such possible sets of real roots. We analyze all $r+3$ sets %by GAP algorithms
in the following way. Fix an arbitrary point $Y$ and its positive roots set $\rersys{Y}_{\boldsymbol{+}}$. 
For any index $i\in I$, the reflection $s_i^Y$ restricts to a bijection $\rersys{Y}_{\boldsymbol{+}}\setminus \{\alpha_i\}\to \rersys{r_i(Y)}_{\boldsymbol{+}}\setminus \{\alpha_i\}$, together with Lemma \ref{lem:jik}, this determines the Cartan matrices of all neighbors of $Y$. If the Cartan graph $\cC(M)$ is standard or $Y$ has a good $D'_r$ neighborhood, the assertion is already satisfied. Otherwise repeat the the procedure for neighboring points. Since $\cX$ is finite, the iterative process terminates. Detailed computational checks via GAP are omitted here. 
\end{proof}

\subsection{Sporadic finite Cartan graphs}

In this section, we establish in Theorem \ref{thm:goodnei} the structural properties of sporadic finite Cartan graphs of rank \(r\geq 5\). We follow the strategy developed in Section \ref{non-sporadic finite cartan graphs} to investigate the positive real roots of sporadic finite Cartan graphs. All such roots are listed in \cite[Appendix B.2]{caH15}. Accordingly, we introduce the notions of good \(A_5\) neighborhood, good \(A_6\) neighborhood, and good \(A_7\) neighborhood for Cartan graphs of rank \(5\), \(6\), and \(7\), respectively.

\begin{definition}\label{defA5}
  We say that $X$ has a \textbf{good $A_5$ neighborhood} if there exists a permutation of $I$ and integers $a,b,c,d,e\in \ndN$ such that\\
$A^X=A^{r_1(X)}=A^{r_2(X)}=A_5$,
$A^{r_3(X)}=\begin{pmatrix}2&-1&0&0&0\\-1&2&-1&\bm{-a}&0\\0&-1&2&-1&0\\0&\bm{-b}&-1&2&-1\\0&0&0&-1&2\end{pmatrix}$,\\ 
$A^{r_4(X)}=\begin{pmatrix}2&-1&0&0&0\\-1&2&-1&0&0\\0&-1&2&-1&\bm{-c}\\0&0&-1&2&-1\\0&0&\bm{-d}&-1&2\end{pmatrix}$,  $A^{r_5(X)}=\begin{pmatrix}2&-1&0&0&0\\-1&2&-1&0&0\\0&-1&2&-1&0\\0&0&\bm{-e}&2&-1\\0&0&0&-1&2\end{pmatrix}$,\\  $a_{15}^{r_2r_3r_4(X)}=0$, 
  $a_{45}^{r_3r_2r_1(X)}\not=-3$, and one of the three conditions below holds:
  \begin{enumerate}
  	\item[$(A_{5_1})$] 
  	$(a, b, c, d, e)=(1,1,0,0,1)$,   	$(a_{41}^{r_2r_3(X)},a_{45}^{r_2r_3(X)})=(0,-1)$,
  	$(a_{32}^{r_4r_5(X)}, a_{34}^{r_2r_1(X)})\not=(-2,-2)$;
  	\item[$(A_{5_2})$] 
  	$(a, b, c, d, e)=(0,0,1,1,2)$, 
  	$a_{25}^{r_3r_4(X)}\in \{0, -1\}$, $(a_{45}^{r_3r_2r_1(X)},a_{45}^{r_3r_2(X)})\not=(-2,-2)$;
  	\item[$(A_{5_3})$] 
  	$(a, b, c, d, e)=(0,0,1,1,1)$,
  	$(a_{52}^{r_3r_4(X)}, a_{32}^{r_5r_4(X)})=(-1,-1)$.
  \end{enumerate}
\end{definition}

\begin{definition}\label{defA6}
  We say that $X$ has a \textbf{good $A_6$ neighborhood} if there exists a permutation of $I$ and $a\in \ndN$ satisfying $A^X=A^{r_1(X)}=A^{r_2(X)}=A^{r_3(X)}=A^{r_6(X)}=A_6$,
  \\  $A^{r_4(X)}=\begin{pmatrix}2&-1&0&0&0&0\\-1&2&-1&0&0&0\\0&-1&2&-1&-1&0\\0&0&-1&2&-1&0\\0&0&-1&-1&2&-1\\0&0&0&0&-1&2\end{pmatrix}$,
  $A^{r_5(X)}=\begin{pmatrix}2&-1&0&0&0&0\\-1&2&-1&0&0&0\\0&-1&2&-1&0&0\\0&0&-1&2&-1&\bm{-a}\\0&0&0&-1&2&-1\\0&0&0&\bm{-a}&-1&2\end{pmatrix}$, \\
$a_{36}^{r_5r_4(X)}=0$, and $a_{25}^{r_3r_4(X)}=0$,
  and either
  \begin{enumerate}
  	\item[$(A_{6_1})$] $a=0$,   	$a_{32}^{r_5r_4(X)}=-1$,  $a_{43}^{r_5r_6(X)}=-1$,
$(a_{45}^{r_3r_2r_1(X)},a_{45}^{r_3r_2(X)})\ne (-2,-2)$;
  	\item[$(A_{6_2})$] $a=1$, 
  	$a_{45}^{r_3r_2(x)}=-1$,
    $a_{45}^{r_3r_2r_1(X)}=-1$.
  	%\item[$(2)'$] $a=-1$, $a_{36}^{r_5r_4(X)}=0$, $a_{56}^{r_3r_4(X)}=-1$
  \end{enumerate}
\end{definition}
\begin{definition}\label{defA7}
We say that $X$ has a \textbf{good $A_7$ neighborhood} if there exists a permutation of $I$  such that $A^{X}=A^{r_1(X)}=A^{r_2(X)}=A^{r_4(X)}=A^{r_5(X)}=A^{r_6(X)}=A^{r_7(X)}=A_7$,
  \\
  $A^{r_3(X)}=\begin{pmatrix}
  2&-1&0&0&0&0&0\\
  -1&2&-1&-1&0&0&0\\
  0&-1&2&-1&0&0&0\\
  0&-1&-1&2&-1&0&0\\
  0&0&0&-1&2&-1&0\\
  0&0&0&0&-1&2&-1\\
  0&0&0&0&0&-1&2\\
\end{pmatrix}$, 
together with $(a_{25}^{r_4r_3(X)},a_{21}^{r_4r_3(X)})=(0,-1)$, 
\\ $a_{45}^{r_2r_3(x)}=a_{34}^{r_2r_1(X)}=a_{32}^{r_4r_5(X)}=a_{32}^{r_{4}r_{5}r_{6}(x)}=a_{32}^{r_4r_5r_6r_7(X)}=-1$. 
\end{definition}

With these preparations, we establish the structural characterization for sporadic finite Cartan graphs in the next theorem.

\begin{theorem}\label{thm:goodnei}
Let $\mathcal{C}$ be a sporadic finite Cartan graph of rank $r\geq5$. Then exactly one of the following alternatives holds:
\begin{enumerate}
    \item[$(1)$] The Cartan graph $ \mathcal{C}(M) $ is standard of type $ E_6, E_7,$ or $E_8 $.
    \item[$(2)$] Up to equivalence, there exists a point $ Y \in \mathcal{X} $ such that $ Y $ has one of the good $A_{\theta}$ neighborhood, where $\theta\in\{5,6,7\}$.
\end{enumerate}
\end{theorem}
\begin{proof}
For any $X\in \cX$, let $\rersys{X}_{\boldsymbol{+}}$ denote its set of positive real roots. By~\cite[Theorem 4.1]{caH15} there exists some point $X\in \cX$ satisfying that $\rersys{X}_{\boldsymbol{+}}$ coincides with one of the root configurations listed in~\cite[Appendix B.2.]{caH15} up to a permutation of $I$. In particular, there are exactly $6$, $4$, $2$, and $1$  distinct root sets for rank $5$, $6$, $7$, and $8$, respectively. We analyze all these root sets case by case. For any $i\in I$, the reflection $s_i^X$ bijectively maps $\rersys{X}_{\boldsymbol{+}}\setminus \{\alpha_i\}$ to $\rersys{r_i(X)}_{\boldsymbol{+}}\setminus \{\alpha_i\}$. Combining this property with Lemma~\ref{lem:jik}, we explicitly determine the Cartan matrices of all neighboring vertices of $X$. If $\cC(M)$ is standard or $Y$ has a good $A_\theta$ neighborhood, the desired conclusion holds. Otherwise, we recursively implement the above analysis on the neighbors of $Y$. The finiteness of the set of points $\mathcal{X}$ ensures that this recursive procedure terminates in finite steps.
\end{proof}

\section{The classification theorem}\label{sec:classification}
In this section, we classify all rank $r\geq5$ Nichols algebras of diagonal type with a finite irreducible root system.
Our main classification result is stated in Theorem \ref{theoremking}, with all corresponding generalized Dynkin diagrams summarized in Table \ref{tab.1}.
\begin{theorem}\label{theoremking}
Let $\Bbbk$ be a field of characteristic $p>0$. Let $r\geq5$ and $I=\{1,2,\ldots,r\}$. Let $(V,c)$ be a braided vector space of diagonal type over $\Bbbk$ with basis $\{x_k|k\in I\}$ satisfying
\begin{equation*}
c(x_i \otimes x_j) = q_{ij}x_j \otimes x_i , \quad q_{ij} \in \Bbbk^*.
\end{equation*}
Suppose that $(q_{ij})_{i,j\in I}$ is indecomposable and set $M \colon =(\Bbbk x_i)_{i\in I}$. The Nichols algebra $\cB(V)$ generated by $(V,c)$ has a finite set of roots ${\roots}^{[M]}$ if and only if
the generalized Dynkin diagram $\cD$ of $V$ belongs to the list in Table~\ref{tab.1}. each row of Table~\ref{tab.1} exhausts all generalized Dynkin diagrams associated with the vertices of the Cartan graph $\cC(M)$. For sporadic Cartan graphs, the corresponding exchange graphs are recorded in Table~\ref{tab.2}.
\end{theorem}
%The corresponding row of Table~\ref{tab.2} contains the exchange graph of $\cC(M)$.

By virtue of \cite[Corollary~6]{HW}, Theorem~\ref{theoremking} yields a complete classification of finite-dimensional diagonal-type Nichols algebras.

\begin{cor}\label{coro-cla}
Let $(V,c)$ be a braided vector space of diagonal type. The Nichols algebra $\cB(V)$ is finite dimensional if and only if the generalized Dynkin diagram $\cD$ of $V$ appears in Table~\ref{tab.1} and the labels of all vertices of $\cD$ are roots of unity (including $1$).
\end{cor}

\begin{remark}
We fix standard labeling conventions below to eliminate ambiguities in later arguments.
\begin{figure}[h!]
\begin{center}
\begin{picture}(58,20)(0,9)
\put(1,23){\circle{2}}
\put(1,8){\line(0,1){14}}
\put(1,7){\circle{2}}
\put(10,15){\circle{2}}
\put(9,15){\line(-1,1){7}}
\put(9,15){\line(-1,-1){7}}
\put(11,15){\line(1,0){10}}
\put(22,15){\circle{2}}
\put(23,15){\line(1,0){10}}
\put(34,15){\circle{2}}
\put(35,15){\line(1,0){2}}
\put(40,15){\makebox[0pt]{$\ldots $}}
\put(43,15){\line(1,0){2}}
\put(46,15){\circle{2}}
\put(47,15){\line(1,0){10}}
\put(58,15){\circle{2}}
\put(-2,23){\makebox[0pt]{\scriptsize $1$}}
\put(-2,7){\makebox[0pt]{\scriptsize $2$}}
\put(6,21){\makebox[0pt]{\scriptsize $ $}}
\put(6,9){\makebox[0pt]{\scriptsize $ $}}
\put(10,18){\makebox[0pt]{\scriptsize $3$}}
\put(16,17){\makebox[0pt]{\scriptsize $ $}}
\put(22,18){\makebox[0pt]{\scriptsize $4$}}
\put(28,17){\makebox[0pt]{\scriptsize $ $}}
\put(34,18){\makebox[0pt]{\scriptsize $5$}}
\put(46,18){\makebox[0pt]{\scriptsize $r-1$}}
\put(52,17){\makebox[0pt]{\scriptsize $ $}}
\put(58,18){\makebox[0pt]{\scriptsize $r$}}
\end{picture}
 \end{center}
\caption{ Dynkin diagram of type $D'_r$}
\label{fig_D_r}
\end{figure}

\begin{figure}[h!]
\begin{center}
% p1 p2 p3 p4 p5 p6 p7 p8 p9 p10 p11 p12 p13 p14 p15
% x-----x-----x--...--x-----x-----x---...---x-------x
\rule[-4\unitlength]{0pt}{5\unitlength}
\begin{picture}(63,4)(0,3)
\put(1,1){\circle{2}}
\put(2,1){\line(1,0){10}}
\put(13,1){\circle{2}}
\put(14,1){\line(1,0){10}}
\put(25,1){\circle{2}}
\put(26,1){\line(1,0){2}}
\put(31,1){\makebox[0pt]{$\ldots $}}
\put(34,1){\line(1,0){2}}
\put(37,1){\circle{2}}
\put(38,1){\line(1,0){10}}
\put(49,1){\circle{2}}
\put(50,1){\line(1,0){10}}
\put(61,1){\circle{2}}
\put(62,1){\line(1,0){2}}
\put(67,1){\makebox[0pt]{$\ldots $}}
\put(70,1){\line(1,0){2}}
\put(73,1){\circle{2}}
\put(74,1){\line(1,0){10}}
\put(85,1){\circle{2}}
\put(1,4){\makebox[0pt]{\scriptsize $q_{11}$}}
\put(7,3){\makebox[0pt]{\scriptsize $ $}}
\put(13,4){\makebox[0pt]{\scriptsize $q_{22}$}}
\put(19,3){\makebox[0pt]{\scriptsize $ $}}
\put(25,4){\makebox[0pt]{\scriptsize $q_{33}$}}
\put(37,4){\makebox[0pt]{\scriptsize $q_{k-1,k-1}$}}
\put(43,3){\makebox[0pt]{\scriptsize $ $}}
\put(49,4){\makebox[0pt]{\scriptsize $q_{kk}$}}
\put(49,-4){\makebox[0pt]{\scriptsize $k$}}
\put(55,3){\makebox[0pt]{\scriptsize $ $}}
\put(61,4){\makebox[0pt]{\scriptsize $q_{k+1,k+1}$}}
\put(75,4){\makebox[0pt]{\scriptsize $q_{r-1,r-1}$}}
\put(79,3){\makebox[0pt]{\scriptsize $ $}}
\put(85,4){\makebox[0pt]{\scriptsize $q_{rr}$}}
\end{picture}
\end{center}
\caption{generalized Dynkin diagram of type $A_r$}
\label{fig_A_r}
\end{figure}

\begin{itemize}
\item[$(1)$]
Vertices of generalized Dynkin diagrams are labeled consecutively $1,\dots,r$ from left to right and top to bottom. For example, the $D'_r$ case follows the labeling in Figure~\ref{fig_D_r}.

\item[$(2)$] 
Given a generalized Dynkin diagram $\mathcal{D}$ and indices $i_1,i_2,\dots,i_r\in I$, let $\tau_{i_1i_2\cdots i_r}\mathcal{D}$ denote the graph obtained by relabeling the vertices of $\mathcal{D}$ as $i_1,i_2,\dots,i_r$.

\item[$(3)$] For a generalized Dynkin diagram $\cD$ and $k\in I$, the integer $k$ denotes the label of the $k-$th vertex of the generalized Dynkin diagrams. For instance, label $A_r$ as in Fig.\ \ref{fig_A_r}).

\item[$(4)$] For a generalized Dynkin diagram $\cD$ and $i, j\in \ndN$, write $\cD^r_{i,j}$ for the diagram $\cD$ of rank $r$ located in the $j$-th position of row $i$.
\end{itemize}
\end{remark}

We verify the main theorem via preliminary analysis of several representative special cases.

Set $X=[M]$. Let $A^X \colon=(a_{ij})_{i,j\in I}$ be the Cartan matrix of $X$,  $(q_{i,j})_{i,j\in I}$ be the braiding matrix of $X$ and $(q_{i,j}^{r_i(X)})_{i,j\in I}$ be the braiding matrix of $r_i(X)$.  To shorten subsequent notation, write $\widetilde{q_{ij}}\colon=q_{ij}q_{ji}$ for $1\leq i,j\leq r$.

The next lemma simplifies the upcoming casework.

\begin{lemma}\label{one-classificationtheorem}
Suppose $\mathcal{D}_{\chi,E}$ admits the diagram
\begin{align*}
\rule[-4\unitlength]{0pt}{5\unitlength}
\begin{picture}(63,4)(0,3)
\put(1,1){\circle{2}}
\put(2,1){\line(1,0){10}}
\put(13,1){\circle{2}}
\put(14,1){\line(1,0){10}}
\put(25,1){\circle{2}}
\put(26,1){\line(1,0){10}}
\put(37,1){\circle{2}}
\put(38,1){\line(1,0){10}}
\put(49,1){\circle{2}}
\put(50,1){\line(1,0){2}}
\put(55,1){\makebox[0pt]{$\ldots $}}
\put(58,1){\line(1,0){2}}
\put(61,1){\circle{2}}
\put(62,1){\line(1,0){10}}
\put(73,1){\circle{2}}
\put(1,4){\makebox[0pt]{\scriptsize $q_{11}$}}
\put(7,3){\makebox[0pt]{\scriptsize $r$}}
\put(13,4){\makebox[0pt]{\scriptsize $q_{22}$}}
\put(19,3){\makebox[0pt]{\scriptsize $q$}}
\put(25,4){\makebox[0pt]{\scriptsize $q_{33}$}}
\put(31,3){\makebox[0pt]{\scriptsize $s$}}
\put(37,4){\makebox[0pt]{\scriptsize $q_{44}$}}
\put(43,3){\makebox[0pt]{\scriptsize $t$}}
\put(49,4){\makebox[0pt]{\scriptsize $q_{55}$}}
\put(60,4){\makebox[0pt]{\scriptsize $q_{r-1,r-1}$}}
\put(67,3){\makebox[0pt]{\scriptsize $u$}}
\put(73,4){\makebox[0pt]{\scriptsize $q_{rr}$}}
\end{picture}
\end{align*}
and has a good $D_r'$-neighborhood. Then $q_{22}=-1$ and $qr\neq 1$.
\end{lemma}

\begin{proof}
If $q_{22}\neq -1$, Lemma \ref{jslemma} yields $A^{r_2(X)}\cong A_r$, contradicting Definition \ref{goodneighborhood}. The condition $a_{13}^{r_2(X)}=-1$ further implies $qr\neq 1$.
\end{proof}

\begin{theorem}\label{Dydiagrams}
Suppose ${\roots}^{[M]}$ is a finite set of roots of non-sporadic finite Cartan graphs of rank $\geq 5$. Then every admissable generalized Dynkin diagrams of $V$ occurs in rows $1$-$10$ of Table \ref{tab.1}.
\end{theorem}

\begin{proof}
 By Theorem \ref{Dynkindiagrams}, either $\mathcal{C}(M)$ is standard or there exists a point $X$ such that $A^X$ has a good $D'_r$ neighborhood.

Case $a$. Suppose $\mathcal{C}(M)$ is standard of type $A_r$. Then $A^X = A_r$ yields $(2)_{q_{ii}}(q_{ii}\widetilde{q_{i,i+1}}-1) = (2)_{q_{jj}}(q_{jj}\widetilde{q_{j-1j}}-1) = 0$, for all $i \in \{1,2,3,\ldots,r-1\}$ and  $j \in \{2,3,4,\ldots, r\}$. Applying Lemma \ref{jslemma},  we separate the analysis into three subcases $\colon (a_1), (a_2), (a_3)$.

Subcase $a_1$. Let $q_{ii}\widetilde{q_{i,i+1}} = q_{jj}\widetilde{q_{j-1,j}} = 1$ for all $i \in \{1,2,3,\ldots, r-1\}$ and $j \in \{2,3,4,\ldots, r\}$. To shorten notation, set $q\colon=q_{11} \neq-1$. Then $\mathcal{D}\colon$
\begin{align*}
% p1 p2 p3 p4 p5 p6 p7 p8 p9 p10 p11 p12 p13 p14
% x-----x-----x-----x-----x---...---x-------x
\rule[-4\unitlength]{0pt}{5\unitlength}
\begin{picture}(63,4)(0,3)
\put(1,1){\circle{2}}
\put(2,1){\line(1,0){10}}
\put(13,1){\circle{2}}
\put(14,1){\line(1,0){10}}
\put(25,1){\circle{2}}
\put(26,1){\line(1,0){10}}
\put(37,1){\circle{2}}
\put(38,1){\line(1,0){10}}
\put(49,1){\circle{2}}
\put(50,1){\line(1,0){2}}
\put(55,1){\makebox[0pt]{$\ldots $}}
\put(58,1){\line(1,0){2}}
\put(61,1){\circle{2}}
\put(62,1){\line(1,0){10}}
\put(73,1){\circle{2}}
\put(1,4){\makebox[0pt]{\scriptsize $q$}}
\put(7,3){\makebox[0pt]{\scriptsize $q^{-1}$}}
\put(13,4){\makebox[0pt]{\scriptsize $q$}}
\put(19,3){\makebox[0pt]{\scriptsize $q^{-1}$}}
\put(25,4){\makebox[0pt]{\scriptsize $q$}}
\put(31,3){\makebox[0pt]{\scriptsize $q^{-1}$}}
\put(37,4){\makebox[0pt]{\scriptsize $q$}}
\put(43,3){\makebox[0pt]{\scriptsize $q^{-1}$}}
\put(49,4){\makebox[0pt]{\scriptsize $q$}}
\put(60,4){\makebox[0pt]{\scriptsize $q$}}
\put(67,3){\makebox[0pt]{\scriptsize $q^{-1}$}}
\put(73,4){\makebox[0pt]{\scriptsize $q$}}
\end{picture}
\end{align*}
\noindent
is abbreviated by the simple chain
\begin{align*}
\rule[-4\unitlength]{0pt}{5\unitlength}
\begin{picture}(26,4)(0,3)
\put(13,3){\oval(26,6)}
\put(0,0){\makebox(26,6){\scriptsize $C(r,q;)$}}
\end{picture}
\end{align*}
This family appears in the first row of Tabl \ref{tab.1}.

Subcase $a_{2}$. Let $(2)_{q_{kk}}=0$ for all integers $1\le k\le (r+1)/2$, where this restriction eliminates redundant symmetric configurations. Assume
\[
q_{ii}\widetilde q_{i,i+1}=q_{jj}\widetilde q_{j-1,j}=1,\qquad
\forall\,i\in\{1,2,\dots,r-1\}\setminus\{k\},\ j\in\{2,3,\dots,r\}\setminus\{k\}.
\]
To simplify notation, set $r:=\widetilde q_{k-1,k}\neq -1$ and $q:=\widetilde q_{k,k+1}\neq -1$. When $k\neq 1$, the condition $a_{k-1,k+1}^{r_k(X)}=0$ forces $qr=1$. Hence $\mathcal{D}\colon$
\begin{align*}
% p1 p2 p3 p4 p5 p6 p7 p8 p9 p10 p11 p12 p13 p14 p15
% x-----x-----x--...--x-----x-----x---...---x-------x
\rule[-4\unitlength]{0pt}{5\unitlength}
\begin{picture}(63,4)(0,3)
\put(1,1){\circle{2}}
\put(2,1){\line(1,0){8}}
\put(11,1){\circle{2}}
\put(12,1){\line(1,0){8}}
\put(21,1){\circle{2}}
\put(22,1){\line(1,0){2}}
\put(27,1){\makebox[0pt]{$\ldots $}}
\put(30,1){\line(1,0){2}}
\put(33,1){\circle{2}}
\put(34,1){\line(1,0){8}}
\put(43,1){\circle{2}}
\put(44,1){\line(1,0){8}}
\put(53,1){\circle{2}}
\put(54,1){\line(1,0){2}}
\put(59,1){\makebox[0pt]{$\ldots $}}
\put(62,1){\line(1,0){2}}
\put(65,1){\circle{2}}
\put(66,1){\line(1,0){8}}
\put(75,1){\circle{2}}
\put(1,4){\makebox[0pt]{\scriptsize $q$}}
\put(6,3){\makebox[0pt]{\scriptsize $q^{-1}$}}
\put(11,4){\makebox[0pt]{\scriptsize $q$}}
\put(16,3){\makebox[0pt]{\scriptsize $q^{-1}$}}
\put(21,4){\makebox[0pt]{\scriptsize $q$}}
\put(33,4){\makebox[0pt]{\scriptsize $q$}}
\put(38,3){\makebox[0pt]{\scriptsize $q^{-1}$}}
\put(43,4){\makebox[0pt]{\scriptsize $-1$}}
\put(43,-4){\makebox[0pt]{\scriptsize $k$}}
\put(48,3){\makebox[0pt]{\scriptsize $q$}}
\put(53,4){\makebox[0pt]{\scriptsize $q^{-1}$}}
\put(65,4){\makebox[0pt]{\scriptsize $q^{-1}$}}
\put(70,3){\makebox[0pt]{\scriptsize $q$}}
\put(75,4){\makebox[0pt]{\scriptsize $q^{-1}$}}
\end{picture}
\end{align*}
is abbreviated by the simple chain
\begin{align*}
\begin{picture}(26,8)
\put(13,3){\oval(26,6)}
\put(0,0){\makebox(26,6){\scriptsize $C(r,q^{-1};1,\ldots ,k)$}}
\end{picture}
\end{align*}
This family appears in the second row of Table \ref{tab.1}.

When $k = 1$, $\mathcal{D}$ takes the form
\begin{align*}
% p1 p2 p3 p4 p5 p6 p7 p8 p9 p10 p11 p12 p13 p14
% x-----x-----x-----x-----x---...---x-------x
\rule[-4\unitlength]{0pt}{5\unitlength}
\begin{picture}(63,4)(0,3)
\put(1,1){\circle{2}}
\put(2,1){\line(1,0){10}}
\put(13,1){\circle{2}}
\put(14,1){\line(1,0){10}}
\put(25,1){\circle{2}}
\put(26,1){\line(1,0){10}}
\put(37,1){\circle{2}}
\put(38,1){\line(1,0){10}}
\put(49,1){\circle{2}}
\put(50,1){\line(1,0){2}}
\put(55,1){\makebox[0pt]{$\ldots $}}
\put(58,1){\line(1,0){2}}
\put(61,1){\circle{2}}
\put(62,1){\line(1,0){10}}
\put(73,1){\circle{2}}
\put(1,4){\makebox[0pt]{\scriptsize $-1$}}
\put(7,3){\makebox[0pt]{\scriptsize $q$}}
\put(13,4){\makebox[0pt]{\scriptsize $q^{-1}$}}
\put(19,3){\makebox[0pt]{\scriptsize $q$}}
\put(25,4){\makebox[0pt]{\scriptsize $q^{-1}$}}
\put(31,3){\makebox[0pt]{\scriptsize $q$}}
\put(37,4){\makebox[0pt]{\scriptsize $q^{-1}$}}
\put(43,3){\makebox[0pt]{\scriptsize $q$}}
\put(49,4){\makebox[0pt]{\scriptsize $q^{-1}$}}
\put(60,4){\makebox[0pt]{\scriptsize $q^{-1}$}}
\put(67,3){\makebox[0pt]{\scriptsize $q$}}
\put(73,4){\makebox[0pt]{\scriptsize $q^{-1}$}}
\end{picture}
\end{align*}
\noindent
and is abbreviated by the simple chain
\begin{align*}
\begin{picture}(26,8)
\put(13,3){\oval(26,6)}
\put(0,0){\makebox(26,6){\scriptsize $C(r,q^{-1};1)$}}
\end{picture}
\end{align*}
which is listed in the second row of Table \ref{tab.1}.
The same reflection conditions hold for configurations with
\[
q_{i_1i_1}=q_{i_2i_2}=\cdots=q_{i_ki_k}=-1,\qquad i_1,i_2,\dots,i_k\in\{1,2,\dots,r\},\;2\le k<r,
\]
together with
\[
q_{ii}\widetilde q_{i,i+1}=q_{jj}\widetilde q_{j-1,j}=1,\qquad
\forall\,i\in\{1,2,\dots,r-1\}\setminus\{i_1,\dots,i_k\},\ j\in\{2,\dots,r\}\setminus\{i_1,\dots,i_k\}.
\]
All such generalized Dynkin diagrams $\mathcal{D}$ are collected in the second row of Table \ref{tab.1}.

Subcase $a_3$. Let $(2)_{q_{ii}}=0$ for all $i \in \{1,2,\ldots, r\}$. Then $\mathcal{D}$ is the case of $q = -1$ and $p=3$ in the first row of Table \ref{tab.1}.

Case $b$. Suppose $\mathcal{C}(M)$ is standard of type $B_r$. Then $A^{X}= B_r$ implies that $(3)_{q_{11}}(q_{11}^2\widetilde{q_{12}}-1)=(2)_{q_{ii}}(q_{ii}\widetilde{q_{i,i+1}}-1)=(2)_{q_{jj}}(q_{jj}\widetilde{q_{j-1,j}}-1)=0$ for all $i \in \{2,3,\ldots, r-1\}$ and $j \in \{2,3,\ldots, r\}$. Hence we divide the analysis into four subcases$\colon (b_1), (b_2), (b_3), (b_4)$.

Subcase $b_1$. Let $q_{11}^2\widetilde{q_{12}}=q_{ii}\widetilde{q_{i,i+1}}=q_{jj}\widetilde{q_{j-1,j}}=1$ for all $i \in \{2,3,\ldots, r-1\}$ and $j \in \{2,3,\ldots, r\}$. Set $q\colon=q_{11}$. Then $q\neq-1$ and $\mathcal{D}\colon$
\begin{align*}
% p1 p2 p3 p4 p5 p6 p7 p8 p9 p10 p11 p12 p13 p14
% x-----x-----x-----x-----x---...---x-------x
\rule[-4\unitlength]{0pt}{5\unitlength}
\begin{picture}(63,4)(0,3)
\put(1,1){\circle{2}}
\put(2,1){\line(1,0){10}}
\put(13,1){\circle{2}}
\put(14,1){\line(1,0){10}}
\put(25,1){\circle{2}}
\put(26,1){\line(1,0){10}}
\put(37,1){\circle{2}}
\put(38,1){\line(1,0){10}}
\put(49,1){\circle{2}}
\put(50,1){\line(1,0){2}}
\put(55,1){\makebox[0pt]{$\ldots $}}
\put(58,1){\line(1,0){2}}
\put(61,1){\circle{2}}
\put(62,1){\line(1,0){10}}
\put(73,1){\circle{2}}
\put(1,4){\makebox[0pt]{\scriptsize $q$}}
\put(7,3){\makebox[0pt]{\scriptsize $q^{-2}$}}
\put(13,4){\makebox[0pt]{\scriptsize $q^2$}}
\put(19,3){\makebox[0pt]{\scriptsize $q^{-2}$}}
\put(25,4){\makebox[0pt]{\scriptsize $q^2$}}
\put(31,3){\makebox[0pt]{\scriptsize $q^{-2}$}}
\put(37,4){\makebox[0pt]{\scriptsize $q^2$}}
\put(43,3){\makebox[0pt]{\scriptsize $q^{-2}$}}
\put(49,4){\makebox[0pt]{\scriptsize $q^2$}}
\put(60,4){\makebox[0pt]{\scriptsize $q^2$}}
\put(67,3){\makebox[0pt]{\scriptsize $q^{-2}$}}
\put(73,4){\makebox[0pt]{\scriptsize $q^2$}}
\end{picture}
\end{align*}
\noindent
is abbreviated by the simple chain
\begin{align*}
\tau_{r\cdots321}\
\begin{picture}(48,8)
\put(17,3){\oval(34,6)}
\put(0,0){\makebox(34,6){\scriptsize $C(r-1,q^2;)$}}
\put(34,3){\line(1,0){10}}
\put(45,3){\circle{2}}
\put(40,4){\makebox[0pt]{\scriptsize $q^{-2}$}}
\put(45,5){\makebox[0pt]{\scriptsize $q$}}
\end{picture}
\end{align*}
which is listed in the third row of Table \ref{tab.1}.
\par
Subcase $b_2$. Let $q_{22}=-1$ and $q_{11}^2\widetilde{q_{12}}=q_{ii}\widetilde{q_{i,i+1}}=q_{jj}\widetilde{q_{j-1,j}}=1$ for all $i \in \{3,4,\ldots, r-1\}$ and $j \in \{3,4,\ldots, r\}$. In this case, we assume  $\widetilde{q_{23}}\neq -1$ to avoid redundant arguments. Set $q_{11}\colon=q$, $\widetilde{q_{23}}\colon=r$. Then $q\neq-1$, $\widetilde{q_{12}}=q^{-2}$. Hence $a^{r_2(X)}_{13}=0$ implies $\widetilde{q_{12}}\widetilde{q_{23}}=q^{-2}r=1$. Then $\mathcal{D}\colon$
\begin{align*}
% p1 p2 p3 p4 p5 p6 p7 p8 p9 p10 p11 p12 p13 p14
% x-----x-----x-----x-----x---...---x-------x
\rule[-4\unitlength]{0pt}{5\unitlength}
\begin{picture}(63,4)(0,3)
\put(1,1){\circle{2}}
\put(2,1){\line(1,0){10}}
\put(13,1){\circle{2}}
\put(14,1){\line(1,0){10}}
\put(25,1){\circle{2}}
\put(26,1){\line(1,0){10}}
\put(37,1){\circle{2}}
\put(38,1){\line(1,0){10}}
\put(49,1){\circle{2}}
\put(50,1){\line(1,0){2}}
\put(55,1){\makebox[0pt]{$\ldots $}}
\put(58,1){\line(1,0){2}}
\put(61,1){\circle{2}}
\put(62,1){\line(1,0){10}}
\put(73,1){\circle{2}}
\put(1,4){\makebox[0pt]{\scriptsize $q$}}
\put(7,3){\makebox[0pt]{\scriptsize $q^{-2}$}}
\put(13,4){\makebox[0pt]{\scriptsize $-1$}}
\put(19,3){\makebox[0pt]{\scriptsize $q^2$}}
\put(25,4){\makebox[0pt]{\scriptsize $q^{-2}$}}
\put(31,3){\makebox[0pt]{\scriptsize $q^2$}}
\put(37,4){\makebox[0pt]{\scriptsize $q^{-2}$}}
\put(43,3){\makebox[0pt]{\scriptsize $q^2$}}
\put(49,4){\makebox[0pt]{\scriptsize $q^{-2}$}}
\put(60,4){\makebox[0pt]{\scriptsize $q^{-2}$}}
\put(67,3){\makebox[0pt]{\scriptsize $q^2$}}
\put(73,4){\makebox[0pt]{\scriptsize $q^{-2}$}}
\end{picture}
\end{align*}
\noindent
is abbreviated by the simple chain
\begin{align*}
\tau_{r\cdots321}\
\begin{picture}(44,8)
\put(17,3){\oval(34,6)}
\put(0,0){\makebox(34,6){\scriptsize $C(r-1,q^2;1,\ldots ,r-1)$}}
\put(34,3){\line(1,0){10}}
\put(45,3){\circle{2}}
\put(39,4){\makebox[0pt]{\scriptsize $q^{-2}$}}
\put(45,5){\makebox[0pt]{\scriptsize $q$}}
\end{picture}
\end{align*}
which is listed in the fourth row of Table \ref{tab.1}.

Identical reflection constraints hold for $q_{kk}=-1$ for $k \in \{3,\ldots, r\}$ and $q_{11}^2\widetilde{q_{12}}=q_{ii}\widetilde{q_{i,i+1}}=q_{jj}\widetilde{q_{j-1,j}}=1$ for all $i \in \{3,4,\ldots, r-1\}\verb|\|\{k\}$ and $j \in \{3,4,\ldots, r\}\verb|\|\{k\}$. The same reasoning extends to subcases with multiple distinguished vertices satisfying $q_{i_1,i_1}=q_{i_2,i_2}=\ldots=q_{i_k,i_k}=-1$ for $i_1,i_2,\ldots,i_k\in \{2,3,\ldots,r\}$ and $1<k<r$. All associated generalized Dynkin diagrams $\mathcal{D}$ are listed in the fourth row of Table \ref{tab.1}.

Subcase $b_3$. Let $(3)_{q_{11}}=0$ and $q_{ii}\widetilde{q_{i,i+1}}=q_{jj}\widetilde{q_{j-1,j}}=1$ for all $i \in \{2,3,\ldots, r-1\}$ and $j \in \{2,3,\ldots, r\}$. We assume  $q_{11}^{2}\widetilde{q_{12}}-1\neq 0$ and $q_{22}\neq-1$ to avoid redundant cases. Set $q_{11} \colon= \zeta$ and $q_{22}\colon=q$. The reflection of X is 
\begin{align*}
X\colon ~
% p1 p2 p3 p4 p5 p6 p7 p8 p9 p10 p11 p12 p13 p14
% x-----x-----x-----x-----x---...---x-------x
\rule[-4\unitlength]{0pt}{5\unitlength}
\begin{picture}(63,4)(0,3)
\put(1,1){\circle*{2}}
\put(2,1){\line(1,0){8}}
\put(11,1){\circle{2}}
\put(12,1){\line(1,0){8}}
\put(21,1){\circle{2}}
\put(22,1){\line(1,0){8}}
\put(31,1){\circle{2}}
\put(32,1){\line(1,0){8}}
\put(41,1){\circle{2}}
\put(42,1){\line(1,0){2}}
\put(47,1){\makebox[0pt]{$\ldots $}}
\put(50,1){\line(1,0){2}}
\put(53,1){\circle{2}}
\put(54,1){\line(1,0){8}}
\put(63,1){\circle{2}}
\put(1,4){\makebox[0pt]{\scriptsize $\zeta$}}
\put(6,3){\makebox[0pt]{\scriptsize $q^{-1}$}}
\put(11,4){\makebox[0pt]{\scriptsize $q$}}
\put(16,3){\makebox[0pt]{\scriptsize $q^{-1}$}}
\put(21,4){\makebox[0pt]{\scriptsize $q$}}
\put(26,3){\makebox[0pt]{\scriptsize $q^{-1}$}}
\put(31,4){\makebox[0pt]{\scriptsize $q$}}
\put(36,3){\makebox[0pt]{\scriptsize $q^{-1}$}}
\put(41,4){\makebox[0pt]{\scriptsize $q$}}
\put(54,4){\makebox[0pt]{\scriptsize $q$}}
\put(59,3){\makebox[0pt]{\scriptsize $q^{-1}$}}
\put(63,4){\makebox[0pt]{\scriptsize $q$}}
\end{picture}
\end{align*}
\begin{align*}
\quad \Rightarrow
r_1(X)\colon ~
% p1 p2 p3 p4 p5 p6 p7 p8 p9 p10 p11 p12 p13 p14
% x-----x-----x-----x-----x---...---x-------x
\rule[-4\unitlength]{0pt}{5\unitlength}
\begin{picture}(63,4)(0,3)
\put(1,1){\circle{2}}
\put(2,1){\line(1,0){10}}
\put(13,1){\circle{2}}
\put(14,1){\line(1,0){10}}
\put(25,1){\circle{2}}
\put(26,1){\line(1,0){10}}
\put(37,1){\circle{2}}
\put(38,1){\line(1,0){10}}
\put(49,1){\circle{2}}
\put(50,1){\line(1,0){2}}
\put(55,1){\makebox[0pt]{$\ldots $}}
\put(58,1){\line(1,0){2}}
\put(61,1){\circle{2}}
\put(62,1){\line(1,0){10}}
\put(73,1){\circle{2}}
\put(1,4){\makebox[0pt]{\scriptsize $\zeta$}}
\put(6,3){\makebox[0pt]{\scriptsize $\zeta^{-1}q$}}
\put(14,4){\makebox[0pt]{\scriptsize $\zeta q^{-1}$}}
\put(20,3){\makebox[0pt]{\scriptsize $q^{-1}$}}
\put(25,4){\makebox[0pt]{\scriptsize $q$}}
\put(31,3){\makebox[0pt]{\scriptsize $q^{-1}$}}
\put(37,4){\makebox[0pt]{\scriptsize $q$}}
\put(43,3){\makebox[0pt]{\scriptsize $q^{-1}$}}
\put(49,4){\makebox[0pt]{\scriptsize $q$}}
\put(61,4){\makebox[0pt]{\scriptsize $q$}}
\put(67,3){\makebox[0pt]{\scriptsize $q^{-1}$}}
\put(73,4){\makebox[0pt]{\scriptsize $q$}}
\end{picture}
\end{align*}
with $(3)_{\zeta}=0$ and $q\notin \{-1, \zeta,\zeta^{-1}\}$ and applying Lemma \ref{jslemma} to $A^{r_1(X)} = B_r$, we obtain $\zeta q^{-2}=1$ or $\zeta q^{-1}=-1$. Hence $q = -\zeta^{-1}$, $p\neq 2$ or $q=-\zeta$, $p\neq 2$. If $p\neq 2,3$, then one has $q=-\zeta^{-1}=-\zeta$. Hence $\mathcal{D}\colon$
\begin{align*}
% p1 p2 p3 p4 p5 p6 p7 p8 p9 p10 p11 p12 p13 p14
% x-----x-----x-----x-----x---...---x-------x
\rule[-4\unitlength]{0pt}{5\unitlength}
\begin{picture}(63,4)(0,3)
\put(1,1){\circle{2}}
\put(2,1){\line(1,0){10}}
\put(13,1){\circle{2}}
\put(14,1){\line(1,0){10}}
\put(25,1){\circle{2}}
\put(26,1){\line(1,0){10}}
\put(37,1){\circle{2}}
\put(38,1){\line(1,0){10}}
\put(49,1){\circle{2}}
\put(50,1){\line(1,0){2}}
\put(55,1){\makebox[0pt]{$\ldots $}}
\put(58,1){\line(1,0){2}}
\put(61,1){\circle{2}}
\put(62,1){\line(1,0){10}}
\put(73,1){\circle{2}}
\put(1,4){\makebox[0pt]{\scriptsize $\zeta$}}
\put(7,3){\makebox[0pt]{\scriptsize $-\zeta$}}
\put(13,4){\makebox[0pt]{\scriptsize $-\zeta^{-1}$}}
\put(19,3){\makebox[0pt]{\scriptsize $-\zeta$}}
\put(25,4){\makebox[0pt]{\scriptsize $-\zeta^{-1}$}}
\put(31,3){\makebox[0pt]{\scriptsize $-\zeta$}}
\put(37,4){\makebox[0pt]{\scriptsize $-\zeta^{-1}$}}
\put(43,3){\makebox[0pt]{\scriptsize $-\zeta$}}
\put(49,4){\makebox[0pt]{\scriptsize $-\zeta^{-1}$}}
\put(60,4){\makebox[0pt]{\scriptsize $-\zeta^{-1}$}}
\put(67,3){\makebox[0pt]{\scriptsize $-\zeta$}}
\put(73,4){\makebox[0pt]{\scriptsize $-\zeta^{-1}$}}
\end{picture}
\end{align*}
\noindent 
is abbreviated by the simple chain
\begin{align*}
\tau_{r\cdots321}\
\begin{picture}(48,8)
\put(17,3){\oval(34,6)}
\put(0,0){\makebox(34,6){\scriptsize $C(r-1,-\zeta ^{-1};)$}}
\put(34,3){\line(1,0){10}}
\put(45,3){\circle{2}}
\put(40,4){\makebox[0pt]{\scriptsize $-\zeta $}}
\put(45,5){\makebox[0pt]{\scriptsize $\zeta $}}
\end{picture}
\end{align*}
which is listed in the fifth row of Table \ref{tab.1}.
If $p=3$, then one has $q=-1$ for $q=-\zeta^{-1}$ and $q=-\zeta$. Hence $\mathcal{D}$ appears in row $5'$ of Table \ref{tab.1}. 

Subcase $b_4$. Let $q_{22}=-1$, $(3)_{q_{11}}=0$ and $q_{ii}\widetilde{q_{i,i+1}}=q_{jj}\widetilde{q_{j-1,j}}=1$ for all $i \in \{3,4,\ldots, r-1\}$ and $j \in \{3,4,\ldots, r\}$. We exclude the cases $q_{11}^2\widetilde{q_{12}}= 1$ and $\widetilde{q_{23}}=-1$ to avoid redundancy. Set $q_{11}\colon=\zeta$, $q\colon=\widetilde{q_{12}}$ and $r\colon=\widetilde{q_{23}}$. Then $r\neq -1$, $qr=1$ by $A^{r_2(X)}=B_r$. And one has $(3)_\zeta =0$ and $q \notin \{-1,\zeta, \zeta^{-1}\}$. The reflection of $X$
\begin{align*}
X\colon~
% p1 p2 p3 p4 p5 p6 p7 p8 p9 p10 p11 p12 p13 p14
% x-----x-----x-----x-----x---...---x-------x
\rule[-4\unitlength]{0pt}{5\unitlength}
\begin{picture}(63,4)(0,3)
\put(1,1){\circle*{2}}
\put(2,1){\line(1,0){8}}
\put(11,1){\circle{2}}
\put(12,1){\line(1,0){8}}
\put(21,1){\circle{2}}
\put(22,1){\line(1,0){8}}
\put(31,1){\circle{2}}
\put(32,1){\line(1,0){8}}
\put(41,1){\circle{2}}
\put(42,1){\line(1,0){2}}
\put(47,1){\makebox[0pt]{$\ldots $}}
\put(50,1){\line(1,0){2}}
\put(53,1){\circle{2}}
\put(54,1){\line(1,0){8}}
\put(63,1){\circle{2}}
\put(1,4){\makebox[0pt]{\scriptsize $\zeta$}}
\put(6,3){\makebox[0pt]{\scriptsize $q$}}
\put(11,4){\makebox[0pt]{\scriptsize $-1$}}
\put(16,3){\makebox[0pt]{\scriptsize $q^{-1}$}}
\put(21,4){\makebox[0pt]{\scriptsize $q$}}
\put(26,3){\makebox[0pt]{\scriptsize $q^{-1}$}}
\put(31,4){\makebox[0pt]{\scriptsize $q$}}
\put(36,3){\makebox[0pt]{\scriptsize $q^{-1}$}}
\put(41,4){\makebox[0pt]{\scriptsize $q$}}
\put(54,4){\makebox[0pt]{\scriptsize $q$}}
\put(59,3){\makebox[0pt]{\scriptsize $q^{-1}$}}
\put(63,4){\makebox[0pt]{\scriptsize $q$}}
\end{picture}
\end{align*}
\begin{align*}
\quad \Rightarrow
r_1(X)\colon ~
% p1 p2 p3 p4 p5 p6 p7 p8 p9 p10 p11 p12 p13 p14
% x-----x-----x-----x-----x---...---x-------x
\rule[-4\unitlength]{0pt}{5\unitlength}
\begin{picture}(63,4)(0,3)
\put(1,1){\circle{2}}
\put(2,1){\line(1,0){12}}
\put(15,1){\circle{2}}
\put(16,1){\line(1,0){12}}
\put(29,1){\circle{2}}
\put(30,1){\line(1,0){12}}
\put(43,1){\circle{2}}
\put(44,1){\line(1,0){12}}
\put(57,1){\circle{2}}
\put(58,1){\line(1,0){2}}
\put(63,1){\makebox[0pt]{$\ldots $}}
\put(66,1){\line(1,0){2}}
\put(69,1){\circle{2}}
\put(70,1){\line(1,0){12}}
\put(83,1){\circle{2}}
\put(1,4){\makebox[0pt]{\scriptsize $\zeta$}}
\put(7,3){\makebox[0pt]{\scriptsize $(\zeta q)^{-1}$}}
\put(16,4){\makebox[0pt]{\scriptsize $-\zeta q^{2}$}}
\put(23,3){\makebox[0pt]{\scriptsize $q^{-1}$}}
\put(29,4){\makebox[0pt]{\scriptsize $q$}}
\put(36,3){\makebox[0pt]{\scriptsize $q^{-1}$}}
\put(43,4){\makebox[0pt]{\scriptsize $q$}}
\put(50,3){\makebox[0pt]{\scriptsize $q^{-1}$}}
\put(57,4){\makebox[0pt]{\scriptsize $q$}}
\put(69,4){\makebox[0pt]{\scriptsize $q$}}
\put(76,3){\makebox[0pt]{\scriptsize $q^{-1}$}}
\put(83,4){\makebox[0pt]{\scriptsize $q$}}
\end{picture}
\end{align*}
imply that $-\zeta q^{2}=-1$ or $-\zeta q^2(\zeta q)^{-1}= -\zeta q^2q^{-1}=1$ by $A^{r_1(X)}=B_r$. Then $q=-\zeta$, $p\neq 2,3$. Hence $\mathcal{D}\colon$
\begin{align*}
% p1 p2 p3 p4 p5 p6 p7 p8 p9 p10 p11 p12 p13 p14
% x-----x-----x-----x-----x---...---x-------x
\rule[-4\unitlength]{0pt}{5\unitlength}
\begin{picture}(63,4)(0,3)
\put(1,1){\circle{2}}
\put(2,1){\line(1,0){10}}
\put(13,1){\circle{2}}
\put(14,1){\line(1,0){10}}
\put(25,1){\circle{2}}
\put(26,1){\line(1,0){10}}
\put(37,1){\circle{2}}
\put(38,1){\line(1,0){10}}
\put(49,1){\circle{2}}
\put(50,1){\line(1,0){2}}
\put(55,1){\makebox[0pt]{$\ldots $}}
\put(58,1){\line(1,0){2}}
\put(61,1){\circle{2}}
\put(62,1){\line(1,0){10}}
\put(73,1){\circle{2}}
\put(1,4){\makebox[0pt]{\scriptsize $\zeta$}}
\put(7,3){\makebox[0pt]{\scriptsize $-\zeta$}}
\put(13,4){\makebox[0pt]{\scriptsize $-1$}}
\put(19,3){\makebox[0pt]{\scriptsize $-\zeta^{-1}$}}
\put(25,4){\makebox[0pt]{\scriptsize $-\zeta$}}
\put(31,3){\makebox[0pt]{\scriptsize $-\zeta^{-1}$}}
\put(37,4){\makebox[0pt]{\scriptsize $-\zeta$}}
\put(43,3){\makebox[0pt]{\scriptsize $-\zeta^{-1}$}}
\put(49,4){\makebox[0pt]{\scriptsize $-\zeta$}}
\put(60,4){\makebox[0pt]{\scriptsize $-\zeta$}}
\put(67,3){\makebox[0pt]{\scriptsize $-\zeta^{-1}$}}
\put(73,4){\makebox[0pt]{\scriptsize $-\zeta$}}
\end{picture}
\end{align*}
\noindent 
is abbreviated by the simple chain
\begin{align*}
\tau_{r\cdots321}\
\begin{picture}(48,8)
\put(20,3){\oval(40,6)}
\put(0,0){\makebox(40,6){\scriptsize $C(r-1,-\zeta ^{-1};1,\ldots ,r-1)$}}
\put(40,3){\line(1,0){10}}
\put(51,3){\circle{2}}
\put(45,4){\makebox[0pt]{\scriptsize $-\zeta $}}
\put(51,5){\makebox[0pt]{\scriptsize $\zeta $}}
\end{picture}
\end{align*}
which is listed in the sixth row of Table \ref{tab.1}.
\par
Similar reflections and conditions occur in cases $q_{kk}=-1$ for $k \in \{3,4,\ldots, r\}$, $(3)_{q_{11}}=0$ and $q_{ii}\widetilde{q_{i,i+1}}=q_{jj}\widetilde{q_{j-1,j}}=1$ for all $i \in \{3,4,\ldots, r-1\}\verb|\|\{k\}$ and $j \in \{3,4,\ldots, r\}\verb|\|\{k\}$. It happens similarly for subcases with more vertices satisfying $q_{i_1,i_1}=q_{i_2,i_2}=\ldots=q_{i_k,i_k}=-1$ for $i_1,i_2,\ldots,i_k\in \{2,3,\ldots,r\}$ and $1<k<r$. In these cases, all the generalized Dynkin diagrams $\mathcal{D}$ are listed in the sixth row of Table \ref{tab.1}.

Case $c$. Suppose $\mathcal{C}(M)$ is standard of type $C_r$. Since $A^X=C_r$, we obtain $(3)_{q_{22}}(q_{22}^2\widetilde{q_{12}}-1)=(2)_{q_{ii}}(q_{ii}\widetilde{q_{i,i+1}}-1)=(2)_{q_{jj}}(q_{jj}\widetilde{q_{j-1,j}}-1)=0$ for all $i \in \{1,3,\ldots, r-1\}$ and $j \in \{3,4,\ldots, r\}$. We split the analysis into two subcases $\colon (c_1), (c_2)$.

Subcase $c_1$.
Let $q^2_{22}\widetilde{q_{12}}-1=q_{22}\widetilde{q_{23}}-1=(2)_{q_{ii}}(q_{ii}\widetilde{q_{i,i+1}}-1)=(2)_{q_{jj}}(q_{jj}\widetilde{q_{j-1,j}}-1)=0$ for all $i\in \{1,3,\ldots,r-1\}$ and $j\in \{3,4,\ldots,r\}$. Set $q\colon=q_{22}$. Then $q\neq-1$, $\widetilde{q_{12}}=q^{-2}$ and $\widetilde{q_{23}}=q^{-1}$. Hence $q_{33}\widetilde{q_{23}}=1$. Otherwise, if $q_{33}\widetilde{q_{23}}\neq1$ then $q_{33}=-1$ and hence $a^{r_3}_{21}=-1$, which is a contradiction. Then $q_{33}\widetilde{q_{34}}=1$ by $a_{34}=-1$. Hence $q_{44}\widetilde{q_{34}}=1$ and $q_{44}\neq -1$. Otherwise, if $q_{44}=-1$ then $a^{r_3r_4(X)}_{21}=-1$ by Lemma \ref{jslemma}, which is again a contradiction. In the same way, we obtain that $q_{ii}\widetilde{q_{i,i+1}}=q_{jj}\widetilde{q_{j-1,j}}=1$ for all $i \in \{3,4,\ldots, r-1\}$ and $j \in \{3,4,\ldots,r\}$. If $q_{11}=-1$ and $\widetilde{q_{12}} \neq -1$, then the reflection of X 
\begin{align*}
X\colon~
% p1 p2 p3 p4 p5 p6 p7 p8 p9 p10 p11 p12 p13 p14
% x-----x-----x-----x-----x---...---x-------x
\rule[-4\unitlength]{0pt}{5\unitlength}
\begin{picture}(63,4)(0,3)
\put(1,1){\circle*{2}}
\put(2,1){\line(1,0){8}}
\put(11,1){\circle{2}}
\put(12,1){\line(1,0){8}}
\put(21,1){\circle{2}}
\put(22,1){\line(1,0){8}}
\put(31,1){\circle{2}}
\put(32,1){\line(1,0){8}}
\put(41,1){\circle{2}}
\put(42,1){\line(1,0){2}}
\put(47,1){\makebox[0pt]{$\ldots $}}
\put(50,1){\line(1,0){2}}
\put(53,1){\circle{2}}
\put(54,1){\line(1,0){8}}
\put(63,1){\circle{2}}
\put(1,4){\makebox[0pt]{\scriptsize $-1$}}
\put(6,3){\makebox[0pt]{\scriptsize $q^{-2}$}}
\put(11,4){\makebox[0pt]{\scriptsize $q$}}
\put(16,3){\makebox[0pt]{\scriptsize $q^{-1}$}}
\put(21,4){\makebox[0pt]{\scriptsize $q$}}
\put(26,3){\makebox[0pt]{\scriptsize $q^{-1}$}}
\put(31,4){\makebox[0pt]{\scriptsize $q$}}
\put(36,3){\makebox[0pt]{\scriptsize $q^{-1}$}}
\put(41,4){\makebox[0pt]{\scriptsize $q$}}
\put(54,4){\makebox[0pt]{\scriptsize $q$}}
\put(59,3){\makebox[0pt]{\scriptsize $q^{-1}$}}
\put(63,4){\makebox[0pt]{\scriptsize $q$}}
\end{picture}
\end{align*}
\begin{align*}
\quad \Rightarrow
r_1(X)\colon ~
% p1 p2 p3 p4 p5 p6 p7 p8 p9 p10 p11 p12 p13 p14
% x-----x-----x-----x-----x---...---x-------x
\rule[-4\unitlength]{0pt}{5\unitlength}
\begin{picture}(63,4)(0,3)
\put(1,1){\circle{2}}
\put(2,1){\line(1,0){12}}
\put(15,1){\circle{2}}
\put(16,1){\line(1,0){12}}
\put(29,1){\circle{2}}
\put(30,1){\line(1,0){12}}
\put(43,1){\circle{2}}
\put(44,1){\line(1,0){12}}
\put(57,1){\circle{2}}
\put(58,1){\line(1,0){2}}
\put(63,1){\makebox[0pt]{$\ldots $}}
\put(66,1){\line(1,0){2}}
\put(69,1){\circle{2}}
\put(70,1){\line(1,0){12}}
\put(83,1){\circle{2}}
\put(1,4){\makebox[0pt]{\scriptsize $-1$}}
\put(8,3){\makebox[0pt]{\scriptsize $q^2$}}
\put(16,4){\makebox[0pt]{\scriptsize $-q^{-1}$}}
\put(23,3){\makebox[0pt]{\scriptsize $q^{-1}$}}
\put(29,4){\makebox[0pt]{\scriptsize $q$}}
\put(36,3){\makebox[0pt]{\scriptsize $q^{-1}$}}
\put(43,4){\makebox[0pt]{\scriptsize $q$}}
\put(50,3){\makebox[0pt]{\scriptsize $q^{-1}$}}
\put(57,4){\makebox[0pt]{\scriptsize $q$}}
\put(69,4){\makebox[0pt]{\scriptsize $q$}}
\put(76,3){\makebox[0pt]{\scriptsize $q^{-1}$}}
\put(83,4){\makebox[0pt]{\scriptsize $q$}}
\end{picture}
\end{align*}
gives that $-q^{-1}q^{-1}=1$ by applying Lemma \ref{jslemma} on $a^{r_1(X)}_{23}=-1$. Then $\widetilde{q_{12}}=q^{-2}=-1$, which is a contradiction. If $q_{11}\widetilde{q_{12}}=1$, then $\mathcal{D}\colon$
\begin{align*}
% p1 p2 p3 p4 p5 p6 p7 p8 p9 p10 p11 p12 p13 p14
% x-----x-----x-----x-----x---...---x-------x
\rule[-4\unitlength]{0pt}{5\unitlength}
\begin{picture}(63,4)(0,3)
\put(1,1){\circle{2}}
\put(2,1){\line(1,0){10}}
\put(13,1){\circle{2}}
\put(14,1){\line(1,0){10}}
\put(25,1){\circle{2}}
\put(26,1){\line(1,0){10}}
\put(37,1){\circle{2}}
\put(38,1){\line(1,0){10}}
\put(49,1){\circle{2}}
\put(50,1){\line(1,0){2}}
\put(55,1){\makebox[0pt]{$\ldots $}}
\put(58,1){\line(1,0){2}}
\put(61,1){\circle{2}}
\put(62,1){\line(1,0){10}}
\put(73,1){\circle{2}}
\put(1,4){\makebox[0pt]{\scriptsize $q^2$}}
\put(7,3){\makebox[0pt]{\scriptsize $q^{-2}$}}
\put(13,4){\makebox[0pt]{\scriptsize $q$}}
\put(19,3){\makebox[0pt]{\scriptsize $q^{-1}$}}
\put(25,4){\makebox[0pt]{\scriptsize $q$}}
\put(31,3){\makebox[0pt]{\scriptsize $q^{-1}$}}
\put(37,4){\makebox[0pt]{\scriptsize $q$}}
\put(43,3){\makebox[0pt]{\scriptsize $q^{-1}$}}
\put(49,4){\makebox[0pt]{\scriptsize $q$}}
\put(60,4){\makebox[0pt]{\scriptsize $q$}}
\put(67,3){\makebox[0pt]{\scriptsize $q^{-1}$}}
\put(73,4){\makebox[0pt]{\scriptsize $q$}}
\end{picture}
\end{align*}
\noindent
is abbreviated by the simple chain
\begin{align*}
\tau_{r\cdots321}\
\rule[-4\unitlength]{0pt}{5\unitlength}
\begin{picture}(50,5)(0,3)
\put(13,3){\oval(26,6)}
\put(0,0){\makebox(26,6){\scriptsize $C(r-2,q;)$}}
\put(26,3){\line(1,0){10}}
\put(37,3){\circle{2}}
\put(38,3){\line(1,0){10}}
\put(49,3){\circle{2}}
\put(31,4){\makebox[0pt]{\scriptsize $q^{-1}$}}
\put(37,5){\makebox[0pt]{\scriptsize $q$}}
\put(43,4){\makebox[0pt]{\scriptsize $q^{-2}$}}
\put(49,5){\makebox[0pt]{\scriptsize $q^2$}}
\end{picture}
\end{align*}
which is listed in the seventh row of Table \ref{tab.1}.

Subcase $c_2$. Let $(3)_{q_{22}}=q_{22}\widetilde{q_{23}}-1=(2)_{ii}(q_{ii}\widetilde{q_{i,i+1}}-1)=(2)_{jj}(q_{jj}\widetilde{q_{j-1,j}}-1)=0$ for all $i\in \{1,3,\ldots, r-1\}$ and $j\in \{3,4,\ldots, r\}$. We may assume  $q^2_{22}\widetilde{q_{12}}-1 \neq 0$ avoid redundant cases. Set $q_{22}\colon=\zeta$. Then $\widetilde{q_{12}} \notin \{\zeta,\zeta^{-1}\}$. Since $q_{22}\widetilde{q_{23}}=1$, we get $\widetilde{q_{23}}=\zeta^{-1}$. Meanwhile, $a^{r_2(X)}_{13}=0$ yields $\widetilde{q_{12}}(q_{22})^{-1}=1$, so $\widetilde{q_{12}}=\zeta$. This is a contradiction.

Case $d$. Suppose $\mathcal{C}(M)$ is standard of type $D_r$. Applying Lemma \ref{jslemma} to $A^X=D_r$ shows that $(2)_{q_{33}}(q_{33}\widetilde{q_{3i}}-1)=(2)_{q_{ii}}(q_{ii}\widetilde{q_{3i}}-1)=0$
 for all $i \in \{1,2,4\}$ and $(2)_{q_{jj}}(q_{jj}\widetilde{q_{j,j+1}}-1)=(2)_{q_{kk}}(q_{kk}\widetilde{q_{k-1,k}}-1)=0$ for all $j\in \{4,5,\ldots, r-1\}$ and $k \in \{5,6,\ldots,r\}$. We divide the analysis into three subcases $\colon (d_1), (d_2), (d_3)$.

Subcase $d_1$. Let $q_{33}\widetilde{q_{3i}}=q_{ii}\widetilde{q_{3i}}=q_{jj}\widetilde{q_{j,j+1}}=q_{kk}\widetilde{q_{k-1,k}}=1$ for all $i \in \{1,2,4\}$, $j\in \{4,5,\ldots,r-1\}$ and $k \in \{5,6,\ldots,r\}$. To avoid redundant cases, set $q\colon=q_{33}\neq-1$. Then $\mathcal{D}\colon$
\begin{align*}
%         x p1
%          \p3
%           \p5 p6 p7 p8 p9  p10 p11 p12
%            x------x-----x-...-x-----x
%           /
%          /p4
%         x p2
\rule[-9\unitlength]{0pt}{12\unitlength}
\begin{picture}(58,10)(0,9)
\put(1,18){\circle{2}}
\put(1,2){\circle{2}}
\put(10,10){\circle{2}}
\put(9,10){\line(-1,1){7}}
\put(9,10){\line(-1,-1){7}}
\put(11,10){\line(1,0){10}}
\put(22,10){\circle{2}}
\put(23,10){\line(1,0){10}}
\put(34,10){\circle{2}}
\put(35,10){\line(1,0){2}}
\put(40,10){\makebox[0pt]{$\ldots $}}
\put(43,10){\line(1,0){2}}
\put(46,10){\circle{2}}
\put(47,10){\line(1,0){10}}
\put(58,10){\circle{2}}
\put(1,20){\makebox[0pt]{\scriptsize $q$}}
\put(1,-1){\makebox[0pt]{\scriptsize $q$}}
\put(6,16){\makebox[0pt]{\scriptsize $q^{-1}$}}
\put(6,3){\makebox[0pt]{\scriptsize $q^{-1}$}}
\put(10,13){\makebox[0pt]{\scriptsize $q$}}
\put(16,12){\makebox[0pt]{\scriptsize $q^{-1}$}}
\put(22,13){\makebox[0pt]{\scriptsize $q$}}
\put(28,12){\makebox[0pt]{\scriptsize $q^{-1}$}}
\put(34,13){\makebox[0pt]{\scriptsize $q$}}
\put(46,13){\makebox[0pt]{\scriptsize $q$}}
\put(52,12){\makebox[0pt]{\scriptsize $q^{-1}$}}
\put(58,13){\makebox[0pt]{\scriptsize $q$}}
\end{picture}
\end{align*}
\noindent
is abbreviated by the simple chain
\begin{align*}
\tau_{r\cdots321}\
\rule[-10\unitlength]{0pt}{10\unitlength}
\begin{picture}(45,10)(0,10)
\put(17,10){\oval(34,6)}
\put(0,7){\makebox(34,6){\scriptsize $C(r-2,q;)$}}
\put(34,10){\line(1,1){7}}
\put(34,10){\line(1,-1){7}}
\put(41,18){\circle{2}}
\put(41,2){\circle{2}}
\put(39,15){\makebox[0pt][r]{\scriptsize $q^{-1}$}}
\put(39,3){\makebox[0pt][r]{\scriptsize $q^{-1}$}}
\put(43,17){\makebox[0pt][l]{\scriptsize $q$}}
\put(43,2){\makebox[0pt][l]{\scriptsize $q$}}
\end{picture}    
\end{align*}
which is listed in the eighth row of Table \ref{tab.1}.

\par
 Subcase $d_2$. Let $q_{33}\widetilde{q_{3i}}-1=0$ for all $i\in \{1,2,4\}$ and there exists $l\in \{1,2,4,5,\ldots,r\}$ such that $(2)_{q_{ll}}=(2)_{q_{mm}}(q_{mm}\widetilde{q_{3m}}-1)=(2)_{q_{jj}}(q_{jj}\widetilde{q_{j,j+1}}-1)=(2)_{q_{kk}}(q_{kk}\widetilde{q_{k-1,k}}-1)=0$ for all $m \in \{1,2,4\}\verb|\|\{l\}, j \in \{4,5,\ldots,r-1\}\verb|\|\{l\}, k \in \{5,6,\ldots,r\}\verb|\|\{l\}$. To avoid redundant arguments, suppose $q_{33}\neq -1$ here, and set $q\colon=q_{33}$. So $q \neq \pm1$. Applying Lemma \ref{jslemma} to $A^{r_3\cdots r_l(X)}=D_r$ gives $q^{-2}=1$, which is a contradiction. For other subcases where more  more vertices satisfy $(2)_{q_{i_1,i_1}}=(2)_{q_{i_2,i_2}}=\ldots=(2)_{q_{i_k,i_k}}=0$ with $i_1,i_2,\ldots,i_k\in \{1,2,4,5,\ldots,r\}$ and $1<k<r$, we reach a contradictionby the same reasoning.
\par
Subcase $d_3$. Consider $(2)_{q_{33}}=(2)_{q_{ii}}(q_{ii}\widetilde{q_{3i}}-1)=(2)_{q_{jj}}(q_{jj}\widetilde{q_{j,j+1}}-1)=(2)_{q_{kk}}(q_{kk}\widetilde{q_{k-1,k}}-1)=0$ for all $i \in \{1,2,4\}$, $j \in \{4,5,\ldots,r-1\}$, $ k \in \{5,\ldots, r\}$. Set $q\colon=\widetilde{q_{31}}, r\colon=\widetilde{q_{32}}, s\colon=\widetilde{q_{34}}$. To avoid repetition, let $q,r,s\neq -1$. Since $A^{r_3(X)}=D_r$, we have $qr=rs=qs=1$. Then $q=r=s=-1$, which is a contradiction. 

Case $d'$. Suppose $A^X$ has a good $D'_r$ neighborhood. Set $a \colon= a_{21}^{r_3(X)}$. Since $A^X = A_r$, we have $(2)_{q_{ii}}(q_{ii}\widetilde{q_{i,i+1}}-1) = (2)_{q_{jj}}(q_{jj}\widetilde{q_{j-1j}}-1) = 0$ for all $i \in \{1,2,3,\ldots r-1\}$ and  $j \in \{2,3,4,\ldots, r\}$. By Lemma \ref{one-classificationtheorem}, the conditions $q_{22} = -1$ and $\widetilde{q_{12}}\widetilde{q_{23}}\neq 1$ hold throughout. We now consider the following subcases $\colon (d'_1), (d'_2), (d'_3), (d'_4)$.
\par
Subcase $d'_1$. Let $q_{ii}\widetilde{q_{i,i+1}} = q_{jj}\widetilde{q_{j-1,j}}= 1$ for all $i \in \{1,3,4,\ldots, r-1\}$ and $j \in \{3,4,\ldots, r\}$. We assume that $\widetilde{q_{23}} \neq -1$ to avoid redundant cases. From Lemma \ref{jslemma} one has $a = 1$. Set $r \colon= \widetilde{q_{12}}$ and $q \colon= \widetilde{q_{23}}$. Then $qr \neq 1$ and $q \neq -1$. The reflections of $X$
\begin{align*}
X\colon~
% p1 p2 p3 p4 p5 p6 p7 p8 p9 p10 p11 p12 p13 p14
% x-----x-----x-----x-----x---...---x-------x
\rule[-4\unitlength]{0pt}{5\unitlength}
\begin{picture}(63,4)(0,3)
\put(1,1){\circle{2}}
\put(2,1){\line(1,0){8}}
\put(11,1){\circle*{2}}
\put(12,1){\line(1,0){8}}
\put(21,1){\circle{2}}
\put(22,1){\line(1,0){8}}
\put(31,1){\circle{2}}
\put(32,1){\line(1,0){8}}
\put(41,1){\circle{2}}
\put(42,1){\line(1,0){2}}
\put(47,1){\makebox[0pt]{$\ldots $}}
\put(50,1){\line(1,0){2}}
\put(53,1){\circle{2}}
\put(54,1){\line(1,0){8}}
\put(63,1){\circle{2}}
\put(1,4){\makebox[0pt]{\scriptsize $r^{-1}$}}
\put(6,3){\makebox[0pt]{\scriptsize $r$}}
\put(11,4){\makebox[0pt]{\scriptsize $-1$}}
\put(16,3){\makebox[0pt]{\scriptsize $q$}}
\put(21,4){\makebox[0pt]{\scriptsize $q^{-1}$}}
\put(26,3){\makebox[0pt]{\scriptsize $q$}}
\put(31,4){\makebox[0pt]{\scriptsize $q^{-1}$}}
\put(36,3){\makebox[0pt]{\scriptsize $q$}}
\put(41,4){\makebox[0pt]{\scriptsize $q^{-1}$}}
\put(54,4){\makebox[0pt]{\scriptsize $q^{-1}$}}
\put(57,3){\makebox[0pt]{\scriptsize $q$}}
\put(63,4){\makebox[0pt]{\scriptsize $q^{-1}$}}
\end{picture}
\quad \Rightarrow
r_2(X)\colon ~
%         x p1
%         |\p4
%         | \p6 p7 p8 p9 p10  p11 p12 p13
%      p2 |  x------x-----x-...-x-----x
%         | /
%         |/p5
%         x p3
\rule[-9\unitlength]{0pt}{12\unitlength}
\begin{picture}(58,10)(0,9)
\put(2,18){\circle{2}}
\put(2,2){\circle{2}}
\put(2,3){\line(0,1){14}}
\put(10,10){\circle*{2}}
\put(9,10){\line(-1,1){7}}
\put(9,10){\line(-1,-1){7}}
\put(11,10){\line(1,0){10}}
\put(22,10){\circle{2}}
\put(23,10){\line(1,0){10}}
\put(34,10){\circle{2}}
\put(35,10){\line(1,0){2}}
\put(40,10){\makebox[0pt]{$\ldots $}}
\put(43,10){\line(1,0){2}}
\put(46,10){\circle{2}}
\put(47,10){\line(1,0){10}}
\put(58,10){\circle{2}}
\put(2,20){\makebox[0pt]{\scriptsize $-1$}}
\put(0,10){\makebox[0pt]{\scriptsize $r^{-1}$}}
\put(2,-1){\makebox[0pt]{\scriptsize $-1$}}
\put(6,15){\makebox[0pt]{\scriptsize $qr$}}
\put(7,4){\makebox[0pt]{\scriptsize $q^{-1}$}}
\put(10,13){\makebox[0pt]{\scriptsize $-1$}}
\put(16,12){\makebox[0pt]{\scriptsize $q$}}
\put(22,13){\makebox[0pt]{\scriptsize $q^{-1}$}}
\put(28,12){\makebox[0pt]{\scriptsize $q$}}
\put(34,13){\makebox[0pt]{\scriptsize $q^{-1}$}}
\put(46,13){\makebox[0pt]{\scriptsize $q^{-1}$}}
\put(52,12){\makebox[0pt]{\scriptsize $q$}}
\put(58,13){\makebox[0pt]{\scriptsize $q^{-1}$}}
\end{picture}
\end{align*}
\begin{align*}
\quad \Rightarrow
r_3r_2(X)\colon~
\tau_{2314\cdots r} \ \
%                 x p4
%             p5/  \p7
%    p1  p2  p3/ p6 \p8 p9 p10 p11 p12 p13
%    x------- x------x-----x-...-x-----x
%           
%          
%        
\rule[-9\unitlength]{0pt}{12\unitlength}
\begin{picture}(58,10)(0,9)
\put(0,10){\circle{2}}
\put(1,10){\line(1,0){10}}
\put(12,10){\circle{2}}
\put(13,10){\line(1,2){5}}
\put(18,21){\circle{2}}
\put(13,10){\line(1,0){10}}
\put(23,10){\line(-1,2){5}}
\put(24,10){\circle{2}}
\put(25,10){\line(1,0){10}}
\put(36,10){\circle{2}}
\put(37,10){\line(1,0){2}}
\put(42,10){\makebox[0pt]{$\ldots $}}
\put(45,10){\line(1,0){2}}
\put(48,10){\circle{2}}
\put(49,10){\line(1,0){10}}
\put(60,10){\circle{2}}
\put(0,13){\makebox[0pt]{\scriptsize $q^{-1}$}}
\put(6,12){\makebox[0pt]{\scriptsize $q$}}
\put(12,13){\makebox[0pt]{\scriptsize $-1$}}
\put(18,24){\makebox[0pt]{\scriptsize $qr$}}
\put(12,16){\makebox[0pt]{\scriptsize $(qr)^{-1}$}}
\put(18,12){\makebox[0pt]{\scriptsize $q^{-1}$}}
\put(23,16){\makebox[0pt]{\scriptsize $q^2r$}}
\put(24,13){\makebox[0pt]{\scriptsize $-1$}}
\put(30,12){\makebox[0pt]{\scriptsize $q$}}
\put(36,13){\makebox[0pt]{\scriptsize $q^{-1}$}}
\put(48,13){\makebox[0pt]{\scriptsize $q^{-1}$}}
\put(54,12){\makebox[0pt]{\scriptsize $q$}}
\put(60,13){\makebox[0pt]{\scriptsize $q^{-1}$}}
\end{picture}
\end{align*}
imply that $q^2r = 1$ by $a^{r_3r_2(X)}_{14} = 0$, and hence $\mathcal{D}\colon$
\begin{align*}
% p1 p2 p3 p4 p5 p6 p7 p8 p9 p10 p11 p12 p13 p14
% x-----x-----x-----x-----x---...---x-------x
\rule[-4\unitlength]{0pt}{5\unitlength}
\begin{picture}(63,4)(0,3)
\put(1,1){\circle{2}}
\put(2,1){\line(1,0){10}}
\put(13,1){\circle{2}}
\put(14,1){\line(1,0){10}}
\put(25,1){\circle{2}}
\put(26,1){\line(1,0){10}}
\put(37,1){\circle{2}}
\put(38,1){\line(1,0){10}}
\put(49,1){\circle{2}}
\put(50,1){\line(1,0){2}}
\put(55,1){\makebox[0pt]{$\ldots $}}
\put(58,1){\line(1,0){2}}
\put(61,1){\circle{2}}
\put(62,1){\line(1,0){10}}
\put(73,1){\circle{2}}
\put(1,4){\makebox[0pt]{\scriptsize $q^2$}}
\put(7,3){\makebox[0pt]{\scriptsize $q^{-2}$}}
\put(13,4){\makebox[0pt]{\scriptsize $-1$}}
\put(19,3){\makebox[0pt]{\scriptsize $q$}}
\put(25,4){\makebox[0pt]{\scriptsize $q^{-1}$}}
\put(31,3){\makebox[0pt]{\scriptsize $q$}}
\put(37,4){\makebox[0pt]{\scriptsize $q^{-1}$}}
\put(43,3){\makebox[0pt]{\scriptsize $q$}}
\put(49,4){\makebox[0pt]{\scriptsize $q^{-1}$}}
\put(60,4){\makebox[0pt]{\scriptsize $q^{-1}$}}
\put(67,3){\makebox[0pt]{\scriptsize $q$}}
\put(73,4){\makebox[0pt]{\scriptsize $q^{-1}$}}
\end{picture}
\end{align*}
\noindent
is abbreviated by the simple chain
\begin{align*}
\tau_{r\cdots321}\
\rule[-4\unitlength]{0pt}{5\unitlength}
\begin{picture}(42,5)(0,3)
\put(18,3){\oval(36,6)}
\put(0,0){\makebox(36,6){\scriptsize $C(r-1,q;1,\ldots ,r-1)$}}
\put(36,3){\line(1,0){10}}
\put(47,3){\circle{2}}
\put(40,4){\makebox[0pt]{\scriptsize $q^{-2}$}}
\put(47,5){\makebox[0pt]{\scriptsize $q^2$}}
\end{picture}
\end{align*}
which is listed in the tenth row of Table \ref{tab.1}.  

Similar reflections and conditions also hold in cases $q_{i_1,i_1}=q_{i_2,i_2}=\ldots=q_{i_k,i_k}=-1$ for $i_1,i_2,\ldots,i_k\in \{4,5,\ldots,r\}$, $1\leq k<r-2$, and $q_{ii}\widetilde{q_{i,i+1}} = q_{jj}\widetilde{q_{j-1,j}}= 1$ for all $i \in \{1,3,4,\ldots,r-1\}\verb|\|\{i_1,i_2,\ldots,i_k\}$ and $j \in \{3,4,\ldots,r\}\verb|\|\{i_1,i_2,\ldots,i_k\}$. In these cases, all the generalized Dynkin diagrams $\mathcal{D}$ are listed in the tenth row of Table \ref{tab.1}.
%The similar reflections and conditions occur in cases $q_{kk}=-1$ for $k \in \{4,5,\ldots, r\}$ and $q_{ii}\widetilde{q_{i,i+1}}-1 = q_{jj}\widetilde{q_{j-1,j}}-1 = 0$ for all $i \in \{1,3,4,\ldots,r-1\}\verb|\|\{k\}$ and $j \in \{3,4,\ldots,r\}\verb|\|\{k\}$. %Please check Example \ref{example1}. It happens similarly for subcases with more vertices satisfying $q_{i_1,i_1}=q_{i_2,i_2}=\ldots=q_{i_k,i_k}=-1$ for $i_1,i_2,\ldots,i_k\in \{4,5,\ldots,r\}$ and $1<k<r-2$. In these cases, all the generalized Dynkin diagrams $\mathcal{D}$ appear in row 10 of Table \ref{tab.1}.

Subcase $d'_2$. Let $q_{11}=-1$ , $q_{ii}\widetilde{q_{i,i+1}}= q_{jj}\widetilde{q_{j-1,j}}= 1$ for all $i \in \{3,4,\ldots, r-1\}$ and $j \in \{3,4,\ldots, r\}$. From Lemma \ref{jslemma}, it follows that $a = 1$. Set $r \colon= \widetilde{q_{12}}$ and $q \colon= \widetilde{q_{23}}$. Then $qr \neq 1$. By applying Lemma \ref{jslemma} to $a^{r_1(X)}_{23}=-1$, we have $r=-1, q\neq -1$. The reflections of $X$
\begin{align*}
X\colon~
% p1 p2 p3 p4 p5 p6 p7 p8 p9 p10 p11 p12 p13 p14
% x-----x-----x-----x-----x---...---x-------x
\rule[-4\unitlength]{0pt}{5\unitlength}
\begin{picture}(63,4)(0,3)
\put(1,1){\circle{2}}
\put(2,1){\line(1,0){8}}
\put(11,1){\circle*{2}}
\put(12,1){\line(1,0){8}}
\put(21,1){\circle{2}}
\put(22,1){\line(1,0){8}}
\put(31,1){\circle{2}}
\put(32,1){\line(1,0){8}}
\put(41,1){\circle{2}}
\put(42,1){\line(1,0){2}}
\put(47,1){\makebox[0pt]{$\ldots $}}
\put(50,1){\line(1,0){2}}
\put(53,1){\circle{2}}
\put(54,1){\line(1,0){8}}
\put(63,1){\circle{2}}
\put(1,4){\makebox[0pt]{\scriptsize $-1$}}
\put(6,3){\makebox[0pt]{\scriptsize $-1$}}
\put(11,4){\makebox[0pt]{\scriptsize $-1$}}
\put(16,3){\makebox[0pt]{\scriptsize $q$}}
\put(21,4){\makebox[0pt]{\scriptsize $q^{-1}$}}
\put(26,3){\makebox[0pt]{\scriptsize $q$}}
\put(31,4){\makebox[0pt]{\scriptsize $q^{-1}$}}
\put(36,3){\makebox[0pt]{\scriptsize $q$}}
\put(41,4){\makebox[0pt]{\scriptsize $q^{-1}$}}
\put(54,4){\makebox[0pt]{\scriptsize $q^{-1}$}}
\put(57,3){\makebox[0pt]{\scriptsize $q$}}
\put(63,4){\makebox[0pt]{\scriptsize $q^{-1}$}}
\end{picture}
\quad \Rightarrow
r_2(X)\colon ~
%         x p1
%         |\p4
%         | \p6 p7 p8 p9 p10  p11 p12 p13
%      p2 |  x------x-----x-...-x-----x
%         | /
%         |/p5
%         x p3
\rule[-9\unitlength]{0pt}{12\unitlength}
\begin{picture}(58,10)(0,9)
\put(2,18){\circle{2}}
\put(2,2){\circle{2}}
\put(2,3){\line(0,1){14}}
\put(10,10){\circle*{2}}
\put(9,10){\line(-1,1){7}}
\put(9,10){\line(-1,-1){7}}
\put(11,10){\line(1,0){10}}
\put(22,10){\circle{2}}
\put(23,10){\line(1,0){10}}
\put(34,10){\circle{2}}
\put(35,10){\line(1,0){2}}
\put(40,10){\makebox[0pt]{$\ldots $}}
\put(43,10){\line(1,0){2}}
\put(46,10){\circle{2}}
\put(47,10){\line(1,0){10}}
\put(58,10){\circle{2}}
\put(2,20){\makebox[0pt]{\scriptsize $-1$}}
\put(0,10){\makebox[0pt]{\scriptsize $-1$}}
\put(2,-1){\makebox[0pt]{\scriptsize $-1$}}
\put(6,15){\makebox[0pt]{\scriptsize $-q$}}
\put(7,4){\makebox[0pt]{\scriptsize $q^{-1}$}}
\put(10,13){\makebox[0pt]{\scriptsize $-1$}}
\put(16,12){\makebox[0pt]{\scriptsize $q$}}
\put(22,13){\makebox[0pt]{\scriptsize $q^{-1}$}}
\put(28,12){\makebox[0pt]{\scriptsize $q$}}
\put(34,13){\makebox[0pt]{\scriptsize $q^{-1}$}}
\put(46,13){\makebox[0pt]{\scriptsize $q^{-1}$}}
\put(52,12){\makebox[0pt]{\scriptsize $q$}}
\put(58,13){\makebox[0pt]{\scriptsize $q^{-1}$}}
\end{picture}
\end{align*}

\begin{align*}
\quad \Rightarrow
r_3r_2(X)\colon~
\tau_{2314\cdots r} \ \
%                 x p4
%             p5/  \p7
%    p1  p2  p3/ p6 \p8 p9 p10 p11 p12 p13
%    x------- x------x-----x-...-x-----x
%           
%          
%        
\rule[-9\unitlength]{0pt}{12\unitlength}
\begin{picture}(58,10)(0,9)
\put(0,10){\circle{2}}
\put(1,10){\line(1,0){10}}
\put(12,10){\circle{2}}
\put(13,10){\line(1,2){5}}
\put(18,21){\circle{2}}
\put(13,10){\line(1,0){10}}
\put(23,10){\line(-1,2){5}}
\put(24,10){\circle{2}}
\put(25,10){\line(1,0){10}}
\put(36,10){\circle{2}}
\put(37,10){\line(1,0){2}}
\put(42,10){\makebox[0pt]{$\ldots $}}
\put(45,10){\line(1,0){2}}
\put(48,10){\circle{2}}
\put(49,10){\line(1,0){10}}
\put(60,10){\circle{2}}
\put(0,13){\makebox[0pt]{\scriptsize $q^{-1}$}}
\put(6,12){\makebox[0pt]{\scriptsize $q$}}
\put(12,13){\makebox[0pt]{\scriptsize $-1$}}
\put(18,24){\makebox[0pt]{\scriptsize $-q$}}
\put(12,16){\makebox[0pt]{\scriptsize $(-q)^{-1}$}}
\put(18,12){\makebox[0pt]{\scriptsize $q^{-1}$}}
\put(23,16){\makebox[0pt]{\scriptsize $-q^2$}}
\put(24,13){\makebox[0pt]{\scriptsize $-1$}}
\put(30,12){\makebox[0pt]{\scriptsize $q$}}
\put(36,13){\makebox[0pt]{\scriptsize $q^{-1}$}}
\put(48,13){\makebox[0pt]{\scriptsize $q^{-1}$}}
\put(54,12){\makebox[0pt]{\scriptsize $q$}}
\put(60,13){\makebox[0pt]{\scriptsize $q^{-1}$}}
\end{picture}
\end{align*}
   
\noindent imply that $q^2=-1$ by $a^{r_3r_2(X)}_{14}=0$. Then $q\in G'_4 $ and $p\neq 2$. Hence $\mathcal{D}\colon$ 
\begin{align*}
% p1 p2 p3 p4 p5 p6 p7 p8 p9 p10 p11 p12 p13 p14 p15
% x-----x-----x-----x-----x-----x---...---x-------x
\rule[-4\unitlength]{0pt}{5\unitlength}
\begin{picture}(63,4)(0,3)
\put(1,1){\circle{2}}
\put(2,1){\line(1,0){8}}
\put(11,1){\circle{2}}
\put(12,1){\line(1,0){8}}
\put(21,1){\circle{2}}
\put(22,1){\line(1,0){8}}
\put(31,1){\circle{2}}
\put(32,1){\line(1,0){8}}
\put(41,1){\circle{2}}
\put(42,1){\line(1,0){8}}
\put(51,1){\circle{2}}
\put(52,1){\line(1,0){2}}
\put(57,1){\makebox[0pt]{$\ldots $}}
\put(60,1){\line(1,0){2}}
\put(63,1){\circle{2}}
\put(64,1){\line(1,0){8}}
\put(73,1){\circle{2}}
\put(1,4){\makebox[0pt]{\scriptsize $-1$}}
\put(6,3){\makebox[0pt]{\scriptsize $-1$}}
\put(11,4){\makebox[0pt]{\scriptsize $-1$}}
\put(16,3){\makebox[0pt]{\scriptsize $q$}}
\put(21,4){\makebox[0pt]{\scriptsize $q^{-1}$}}
\put(26,3){\makebox[0pt]{\scriptsize $q$}}
\put(31,4){\makebox[0pt]{\scriptsize $q^{-1}$}}
\put(36,3){\makebox[0pt]{\scriptsize $q$}}
\put(41,4){\makebox[0pt]{\scriptsize $q^{-1}$}}
\put(46,3){\makebox[0pt]{\scriptsize $q$}}
\put(51,4){\makebox[0pt]{\scriptsize $q^{-1}$}}
\put(63,4){\makebox[0pt]{\scriptsize $q^{-1}$}}
\put(67,3){\makebox[0pt]{\scriptsize $q$}}
\put(73,4){\makebox[0pt]{\scriptsize $q^{-1}$}}
\end{picture}
\end{align*}
is abbreviated by the simple chain
\begin{align*}
\tau_{r\cdots321}\
\rule[-4\unitlength]{0pt}{5\unitlength}
\begin{picture}(42,5)(0,3)
\put(18,3){\oval(36,6)}
\put(0,0){\makebox(36,6){\scriptsize $C(r-1,q;1,\ldots ,r-1)$}}
\put(36,3){\line(1,0){10}}
\put(47,3){\circle{2}}
\put(40,4){\makebox[0pt]{\scriptsize $-1$}}
\put(47,5){\makebox[0pt]{\scriptsize $-1$}}
\end{picture}
\end{align*}
which is the case of $q^2=-1$ and $p\neq 2$ in the tenth row of Table \ref{tab.1}.

We apply a similar proof as above in cases $q_{11}=q_{i_1,i_1}=q_{i_2,i_2}=\ldots=q_{i_k,i_k}=-1$ for $i_1,i_2,\ldots,i_k\in \{4,5,\ldots,r\}$, $1\leq k<r-2$, and $q_{ii}\widetilde{q_{i,i+1}} = q_{jj}\widetilde{q_{j-1,j}} = 1$ for all $i \in \{3,4,\ldots,r-1\}\verb|\|\{i_1,i_2,\ldots,i_k\}$ and $j \in \{3,4,\ldots,r\}\verb|\|\{i_1,i_2,\ldots,i_k\}$. All associated generalized Dynkin diagrams $\mathcal{D}$ are the cases of $q^2=-1$ and $p\neq 2$ in the tenth row of Table \ref{tab.1}. 

%Similar reflections and conditions occur in cases $q_{11}=q_{i_1,i_1}=q_{i_2,i_2}=\ldots=q_{i_k,i_k}=-1$ for $i_1,i_2,\ldots,i_k\in \{4,5,\ldots,r\}$, $1\leq k<r-2$, and $q_{ii}\widetilde{q_{i,i+1}}-1 = q_{jj}\widetilde{q_{j-1,j}}-1 = 0$ for all $i \in \{3,4,\ldots,r-1\}\verb|\|\{i_1,i_2,\ldots,i_k\}$ and $j \in \{3,4,\ldots,r\}\verb|\|\{i_1,i_2,\ldots,i_k\}$. In these cases, all the generalized Dynkin diagrams $\mathcal{D}$ are the case of $q^2=-1$ and $p\neq 2$ in row 10 of Table \ref{tab.1}. 

%The similar reflections and conditions occur in cases $q_{11}=q_{kk}=-1$ for $k \in \{4,5,\ldots, r\}$ and $q_{ii}\widetilde{q_{i,i+1}}-1 = q_{jj}\widetilde{q_{j-1,j}}-1 = 0$ for all $i \in \{3,4,\ldots,r-1\}\verb|\|\{k\}$ and $j \in \{3,4,\ldots,r\}\verb|\|\{k\}$. %Please check Example \ref{example2}. It happens similarly for subcases with more vertices satisfying $q_{11}=q_{i_1,i_1}=q_{i_2,i_2}=\ldots=q_{i_k,i_k}=-1$ for $i_1,i_2,\ldots,i_k\in \{4,5,\ldots,r\}$ and $1<k<r-2$. In these cases, all the generalized Dynkin diagrams $\mathcal{D}$ are the case of $q^2=-1$ and $p\neq 2$ in row 10 of Table \ref{tab.1}.

Subcase $d'_3$. Let $q_{33} = -1$, $q_{ii}\widetilde{q_{i,i+1}}-1 = q_{jj}\widetilde{q_{j-1,j}}-1 = 0$ for all $i \in \{1,4,\ldots,r-1\}$ and $j \in \{4,5,\ldots,r\}$. We may assume $\widetilde{q_{34}}\neq-1$ to avoid redundant cases. Set $r \colon= \widetilde{q_{12}}$, $q \colon= \widetilde{q_{23}}$ and $s \colon= \widetilde{q_{34}}$. Then $qr\neq 1$, $s\neq -1$. Since $a^{r_3(X)}_{24}=0$, we have $qs = 1$. The reflections of $X$
\begin{align*}
r_2(X)\colon ~
%         x p1
%         |\p4
%         | \p6 p7 p8 p9 p10  p11 p12 p13
%      p2 |  x------x-----x-...-x-----x
%         | /
%         |/p5
%         x p3
\rule[-9\unitlength]{0pt}{12\unitlength}
\begin{picture}(58,18)(0,9)
\put(2,18){\circle{2}}
\put(2,2){\circle{2}}
\put(2,3){\line(0,1){14}}
\put(10,10){\circle{2}}
\put(9,10){\line(-1,1){7}}
\put(9,10){\line(-1,-1){7}}
\put(11,10){\line(1,0){10}}
\put(22,10){\circle{2}}
\put(23,10){\line(1,0){10}}
\put(34,10){\circle{2}}
\put(35,10){\line(1,0){2}}
\put(40,10){\makebox[0pt]{$\ldots $}}
\put(43,10){\line(1,0){2}}
\put(46,10){\circle{2}}
\put(47,10){\line(1,0){10}}
\put(58,10){\circle{2}}
\put(2,20){\makebox[0pt]{\scriptsize $-1$}}
\put(1,10){\makebox[0pt]{\scriptsize $r^{-1}$}}
\put(2,-1){\makebox[0pt]{\scriptsize $-1$}}
\put(6,15){\makebox[0pt]{\scriptsize $qr$}}
\put(7,4){\makebox[0pt]{\scriptsize $q^{-1}$}}
\put(10,13){\makebox[0pt]{\scriptsize $q$}}
\put(16,12){\makebox[0pt]{\scriptsize $q^{-1}$}}
\put(22,13){\makebox[0pt]{\scriptsize $q$}}
\put(28,12){\makebox[0pt]{\scriptsize $q^{-1}$}}
\put(34,13){\makebox[0pt]{\scriptsize $q$}}
\put(46,13){\makebox[0pt]{\scriptsize $q$}}
\put(52,12){\makebox[0pt]{\scriptsize $q^{-1}$}}
\put(58,13){\makebox[0pt]{\scriptsize $q$}}
\end{picture}
\quad \Leftarrow
X\colon~
% p1 p2 p3 p4 p5 p6 p7 p8 p9 p10 p11 p12 p13 p14
% x-----x-----x-----x-----x---...---x-------x
\rule[-4\unitlength]{0pt}{5\unitlength}
\begin{picture}(63,4)(0,3)
\put(1,1){\circle{2}}
\put(2,1){\line(1,0){8}}
\put(11,1){\circle{2}}
\put(12,1){\line(1,0){8}}
\put(21,1){\circle{2}}
\put(22,1){\line(1,0){8}}
\put(31,1){\circle{2}}
\put(32,1){\line(1,0){8}}
\put(41,1){\circle{2}}
\put(42,1){\line(1,0){2}}
\put(47,1){\makebox[0pt]{$\ldots $}}
\put(50,1){\line(1,0){2}}
\put(53,1){\circle{2}}
\put(54,1){\line(1,0){8}}
\put(63,1){\circle{2}}
\put(1,4){\makebox[0pt]{\scriptsize $r^{-1}$}}
\put(6,3){\makebox[0pt]{\scriptsize $r$}}
\put(11,4){\makebox[0pt]{\scriptsize $-1$}}
\put(16,3){\makebox[0pt]{\scriptsize $q$}}
\put(21,4){\makebox[0pt]{\scriptsize $-1$}}
\put(26,3){\makebox[0pt]{\scriptsize $q^{-1}$}}
\put(31,4){\makebox[0pt]{\scriptsize $q$}}
\put(36,3){\makebox[0pt]{\scriptsize $q^{-1}$}}
\put(41,4){\makebox[0pt]{\scriptsize $q$}}
\put(54,4){\makebox[0pt]{\scriptsize $q$}}
\put(58,3){\makebox[0pt]{\scriptsize $q^{-1}$}}
\put(63,4){\makebox[0pt]{\scriptsize $q$}}
\end{picture}
\end{align*}
\begin{align*}
\quad \Rightarrow
r_3(X)\colon ~ 
% p1 p2 p3 p4 p5 p6 p7 p8 p9 p10 p11 p12 p13 p14
% x-----x-----x-----x-----x---...---x-------x
\rule[-4\unitlength]{0pt}{5\unitlength}
\begin{picture}(63,4)(0,3)
\put(1,1){\circle{2}}
\put(2,1){\line(1,0){8}}
\put(11,1){\circle{2}}
\put(12,1){\line(1,0){8}}
\put(21,1){\circle{2}}
\put(22,1){\line(1,0){8}}
\put(31,1){\circle{2}}
\put(32,1){\line(1,0){8}}
\put(41,1){\circle{2}}
\put(42,1){\line(1,0){2}}
\put(47,1){\makebox[0pt]{$\ldots $}}
\put(50,1){\line(1,0){2}}
\put(53,1){\circle{2}}
\put(54,1){\line(1,0){8}}
\put(63,1){\circle{2}}
\put(1,4){\makebox[0pt]{\scriptsize $r^{-1}$}}
\put(6,3){\makebox[0pt]{\scriptsize $r$}}
\put(11,4){\makebox[0pt]{\scriptsize $q$}}
\put(16,3){\makebox[0pt]{\scriptsize $q^{-1}$}}
\put(21,4){\makebox[0pt]{\scriptsize $-1$}}
\put(26,3){\makebox[0pt]{\scriptsize $q$}}
\put(31,4){\makebox[0pt]{\scriptsize $-1$}}
\put(36,3){\makebox[0pt]{\scriptsize $q^{-1}$}}
\put(41,4){\makebox[0pt]{\scriptsize $q$}}
\put(54,4){\makebox[0pt]{\scriptsize $q$}}
\put(58,3){\makebox[0pt]{\scriptsize $q^{-1}$}}
\put(63,4){\makebox[0pt]{\scriptsize $q$}}
\end{picture}
\end{align*}
imply that $q^2r=1$ by $a^{r_2(X)}_{13}=-1$ and $a=2$. Then $\mathcal{D}\colon$
\begin{align*}
% p1 p2 p3 p4 p5 p6 p7 p8 p9 p10 p11 p12 p13 p14
% x-----x-----x-----x-----x---...---x-------x
\rule[-4\unitlength]{0pt}{5\unitlength}
\begin{picture}(63,4)(0,3)
\put(1,1){\circle{2}}
\put(2,1){\line(1,0){10}}
\put(13,1){\circle{2}}
\put(14,1){\line(1,0){10}}
\put(25,1){\circle{2}}
\put(26,1){\line(1,0){10}}
\put(37,1){\circle{2}}
\put(38,1){\line(1,0){10}}
\put(49,1){\circle{2}}
\put(50,1){\line(1,0){2}}
\put(55,1){\makebox[0pt]{$\ldots $}}
\put(58,1){\line(1,0){2}}
\put(61,1){\circle{2}}
\put(62,1){\line(1,0){10}}
\put(73,1){\circle{2}}
\put(1,4){\makebox[0pt]{\scriptsize $q^2$}}
\put(7,3){\makebox[0pt]{\scriptsize $q^{-2}$}}
\put(13,4){\makebox[0pt]{\scriptsize $-1$}}
\put(19,3){\makebox[0pt]{\scriptsize $q$}}
\put(25,4){\makebox[0pt]{\scriptsize $-1$}}
\put(31,3){\makebox[0pt]{\scriptsize $q^{-1}$}}
\put(37,4){\makebox[0pt]{\scriptsize $q$}}
\put(43,3){\makebox[0pt]{\scriptsize $q^{-1}$}}
\put(49,4){\makebox[0pt]{\scriptsize $q$}}
\put(60,4){\makebox[0pt]{\scriptsize $q$}}
\put(67,3){\makebox[0pt]{\scriptsize $q^{-1}$}}
\put(73,4){\makebox[0pt]{\scriptsize $q$}}
\end{picture}
\end{align*}
\noindent
is abbreviated by the simple chain
\begin{align*}
\tau_{r\cdots321}\
\rule[-4\unitlength]{0pt}{5\unitlength}
\begin{picture}(62,5)(0,3)
\put(25,3){\oval(50,6)}
\put(0,0){\makebox(50,6){\scriptsize $C(r-2,q^{-1};1,\ldots ,r-2)$}}
\put(50,3){\line(1,0){10}}
\put(61,3){\circle{2}}
\put(62,3){\line(1,0){10}}
\put(73,3){\circle{2}}
\put(55,4){\makebox[0pt]{\scriptsize $q$}}
\put(61,5){\makebox[0pt]{\scriptsize $-1$}}
\put(67,4){\makebox[0pt]{\scriptsize $q^{-2}$}}
\put(73,5){\makebox[0pt]{\scriptsize $q^2$}}
\end{picture}
\end{align*}
which is listed in the ninth row of Table \ref{tab.1}.

Applying the aforementioned method in cases $q_{33}=q_{i_1,i_1}=q_{i_2,i_2}=\ldots=q_{i_k,i_k}=-1$ for $i_1,i_2,\ldots,i_k\in \{4,5,\ldots,r\}$, $1\leq k<r-2$, and $q_{ii}\widetilde{q_{i,i+1}} = q_{jj}\widetilde{q_{j-1,j}} = 1$ for all $i \in \{1,4,\ldots,r-1\}\verb|\|\{i_1,i_2,\ldots,i_k\}$ and $j \in \{4,\ldots,r\}\verb|\|\{i_1,i_2,\ldots,i_k\}$, we obtain that all the generalized Dynkin diagrams $\mathcal{D}$ are listed in the tenth row of Table \ref{tab.1}.

%Similar reflections and conditions occur in cases $q_{33}=q_{i_1,i_1}=q_{i_2,i_2}=\ldots=q_{i_k,i_k}=-1$ for $i_1,i_2,\ldots,i_k\in \{4,5,\ldots,r\}$, $1\leq k<r-2$, and $q_{ii}\widetilde{q_{i,i+1}}-1 = q_{jj}\widetilde{q_{j-1,j}}-1 = 0$ for all $i \in \{1,4,\ldots,r-1\}\verb|\|\{i_1,i_2,\ldots,i_k\}$ and $j \in \{4,\ldots,r\}\verb|\|\{i_1,i_2,\ldots,i_k\}$. In these cases, all the generalized Dynkin diagrams $\mathcal{D}$ appear in row 10 of Table \ref{tab.1}.
%The similar reflections and conditions occur in cases $q_{33}= q_{kk}=-1$ for $k \in \{4,5,\ldots, r\}$ and $q_{ii}\widetilde{q_{i,i+1}}-1 = q_{jj}\widetilde{q_{j-1,j}}-1 = 0$ for all $i \in \{1,4,\ldots,r-1\}\verb|\|\{k\}$ and $j \in \{4,\ldots,r\}\verb|\|\{k\}$. %Please check Example \ref{example3}. It happens similarly for subcases with more vertices satisfying $q_{33}=q_{i_1,i_1}=q_{i_2,i_2}=\ldots=q_{i_k,i_k}=-1$ for $i_1,i_2,\ldots,i_k\in \{4,5,\ldots,r\}$ and $1<k<r-2$. In these cases, all the generalized Dynkin diagrams $\mathcal{D}$ appear in row 10 of Table \ref{tab.1}.

Subcase $d'_4$. Let $q_{11}=q_{33} = -1$, $q_{ii}\widetilde{q_{i,i+1}} = q_{jj}\widetilde{q_{j-1,j}} = 1$ for all $i \in \{4,\ldots,r-1\}$ and $j \in \{4,5,\ldots,r\}$. %It is known from Definition \ref{goodneighborhood} that $a \in \{1,2\}$. 
Set $r \colon=\widetilde{ q'_{12}}$, $q \colon= \widetilde{q_{23}}$ and $s \colon= \widetilde{q_{34}}$. Then $qr\neq 1$.
To avoid repetition, let $s\neq-1$. We have $r=-1$, $p\neq 2$, and $qs = 1$ by $a^{r_1(X)}_{23}=-1$ and $a^{r_3(X)}_{24}=0$. The reflections of $X$
\begin{align*}
r_2(X)\colon ~
%         x p1
%         |\p4
%         | \p6 p7 p8 p9 p10  p11 p12 p13
%      p2 |  x------x-----x-...-x-----x
%         | /
%         |/p5
%         x p3
\rule[-9\unitlength]{0pt}{12\unitlength}
\begin{picture}(58,18)(0,9)
\put(2,18){\circle{2}}
\put(2,2){\circle{2}}
\put(2,3){\line(0,1){14}}
\put(10,10){\circle{2}}
\put(9,10){\line(-1,1){7}}
\put(9,10){\line(-1,-1){7}}
\put(11,10){\line(1,0){10}}
\put(22,10){\circle{2}}
\put(23,10){\line(1,0){10}}
\put(34,10){\circle{2}}
\put(35,10){\line(1,0){2}}
\put(40,10){\makebox[0pt]{$\ldots $}}
\put(43,10){\line(1,0){2}}
\put(46,10){\circle{2}}
\put(47,10){\line(1,0){10}}
\put(58,10){\circle{2}}
\put(1,20){\makebox[0pt]{\scriptsize $-1$}}
\put(0,10){\makebox[0pt]{\scriptsize $-1$}}
\put(2,-1){\makebox[0pt]{\scriptsize $-1$}}
\put(6,15){\makebox[0pt]{\scriptsize $-q$}}
\put(7,4){\makebox[0pt]{\scriptsize $q^{-1}$}}
\put(10,13){\makebox[0pt]{\scriptsize $q$}}
\put(16,12){\makebox[0pt]{\scriptsize $q^{-1}$}}
\put(22,13){\makebox[0pt]{\scriptsize $q$}}
\put(28,12){\makebox[0pt]{\scriptsize $q^{-1}$}}
\put(34,13){\makebox[0pt]{\scriptsize $q$}}
\put(46,13){\makebox[0pt]{\scriptsize $q$}}
\put(52,12){\makebox[0pt]{\scriptsize $q^{-1}$}}
\put(58,13){\makebox[0pt]{\scriptsize $q$}}
\end{picture}
\quad \Leftarrow
X\colon~
% p1 p2 p3 p4 p5 p6 p7 p8 p9 p10 p11 p12 p13 p14
% x-----x-----x-----x-----x---...---x-------x
\rule[-4\unitlength]{0pt}{5\unitlength}
\begin{picture}(63,4)(0,3)
\put(1,1){\circle{2}}
\put(2,1){\line(1,0){8}}
\put(11,1){\circle{2}}
\put(12,1){\line(1,0){8}}
\put(21,1){\circle{2}}
\put(22,1){\line(1,0){8}}
\put(31,1){\circle{2}}
\put(32,1){\line(1,0){8}}
\put(41,1){\circle{2}}
\put(42,1){\line(1,0){2}}
\put(47,1){\makebox[0pt]{$\ldots $}}
\put(50,1){\line(1,0){2}}
\put(53,1){\circle{2}}
\put(54,1){\line(1,0){8}}
\put(63,1){\circle{2}}
\put(1,4){\makebox[0pt]{\scriptsize $-1$}}
\put(6,3){\makebox[0pt]{\scriptsize $-1$}}
\put(11,4){\makebox[0pt]{\scriptsize $-1$}}
\put(16,3){\makebox[0pt]{\scriptsize $q$}}
\put(21,4){\makebox[0pt]{\scriptsize $-1$}}
\put(26,3){\makebox[0pt]{\scriptsize $q^{-1}$}}
\put(31,4){\makebox[0pt]{\scriptsize $q$}}
\put(36,3){\makebox[0pt]{\scriptsize $q^{-1}$}}
\put(41,4){\makebox[0pt]{\scriptsize $q$}}
\put(54,4){\makebox[0pt]{\scriptsize $q$}}
\put(58,3){\makebox[0pt]{\scriptsize $q^{-1}$}}
\put(63,4){\makebox[0pt]{\scriptsize $q$}}
\end{picture}
\end{align*}
\begin{align*}
\quad \Rightarrow
r_3(X)\colon ~ 
% p1 p2 p3 p4 p5 p6 p7 p8 p9 p10 p11 p12 p13 p14
% x-----x-----x-----x-----x---...---x-------x
\rule[-4\unitlength]{0pt}{5\unitlength}
\begin{picture}(63,4)(0,3)
\put(1,1){\circle{2}}
\put(2,1){\line(1,0){8}}
\put(11,1){\circle{2}}
\put(12,1){\line(1,0){8}}
\put(21,1){\circle{2}}
\put(22,1){\line(1,0){8}}
\put(31,1){\circle{2}}
\put(32,1){\line(1,0){8}}
\put(41,1){\circle{2}}
\put(42,1){\line(1,0){2}}
\put(47,1){\makebox[0pt]{$\ldots $}}
\put(50,1){\line(1,0){2}}
\put(53,1){\circle{2}}
\put(54,1){\line(1,0){8}}
\put(63,1){\circle{2}}
\put(1,4){\makebox[0pt]{\scriptsize $-1$}}
\put(6,3){\makebox[0pt]{\scriptsize $-1$}}
\put(11,4){\makebox[0pt]{\scriptsize $q$}}
\put(16,3){\makebox[0pt]{\scriptsize $q^{-1}$}}
\put(21,4){\makebox[0pt]{\scriptsize $-1$}}
\put(26,3){\makebox[0pt]{\scriptsize $q$}}
\put(31,4){\makebox[0pt]{\scriptsize $-1$}}
\put(36,3){\makebox[0pt]{\scriptsize $q^{-1}$}}
\put(41,4){\makebox[0pt]{\scriptsize $q$}}
\put(54,4){\makebox[0pt]{\scriptsize $q$}}
\put(58,3){\makebox[0pt]{\scriptsize $q^{-1}$}}
\put(63,4){\makebox[0pt]{\scriptsize $q$}}
\end{picture}
\end{align*}
imply that $q^2=-1$ by $a^{r_2(X)}_{31}=-1$ and hence $a=2$. Then $\mathcal{D}\colon$
\begin{align*}
% p1 p2 p3 p4 p5 p6 p7 p8 p9 p10 p11 p12 p13 p14
% x-----x-----x-----x-----x---...---x-------x
\rule[-4\unitlength]{0pt}{5\unitlength}
\begin{picture}(63,4)(0,3)
\put(1,1){\circle{2}}
\put(2,1){\line(1,0){10}}
\put(13,1){\circle{2}}
\put(14,1){\line(1,0){10}}
\put(25,1){\circle{2}}
\put(26,1){\line(1,0){10}}
\put(37,1){\circle{2}}
\put(38,1){\line(1,0){10}}
\put(49,1){\circle{2}}
\put(50,1){\line(1,0){2}}
\put(55,1){\makebox[0pt]{$\ldots $}}
\put(58,1){\line(1,0){2}}
\put(61,1){\circle{2}}
\put(62,1){\line(1,0){10}}
\put(73,1){\circle{2}}
\put(1,4){\makebox[0pt]{\scriptsize $-1$}}
\put(7,3){\makebox[0pt]{\scriptsize $-1$}}
\put(13,4){\makebox[0pt]{\scriptsize $-1$}}
\put(19,3){\makebox[0pt]{\scriptsize $q$}}
\put(25,4){\makebox[0pt]{\scriptsize $-1$}}
\put(31,3){\makebox[0pt]{\scriptsize $q^{-1}$}}
\put(37,4){\makebox[0pt]{\scriptsize $q$}}
\put(43,3){\makebox[0pt]{\scriptsize $q^{-1}$}}
\put(49,4){\makebox[0pt]{\scriptsize $q$}}
\put(60,4){\makebox[0pt]{\scriptsize $q$}}
\put(67,3){\makebox[0pt]{\scriptsize $q^{-1}$}}
\put(73,4){\makebox[0pt]{\scriptsize $q$}}
\end{picture}
\end{align*}
\noindent
is abbreviated by the simple chain
\begin{align*}
\tau_{r\cdots321}\
\rule[-4\unitlength]{0pt}{5\unitlength}
\begin{picture}(62,5)(0,3)
\put(25,3){\oval(50,6)}
\put(0,0){\makebox(50,6){\scriptsize $C(r-2,q^{-1};1,\ldots ,r-2)$}}
\put(50,3){\line(1,0){10}}
\put(61,3){\circle{2}}
\put(62,3){\line(1,0){10}}
\put(73,3){\circle{2}}
\put(55,4){\makebox[0pt]{\scriptsize $q$}}
\put(61,5){\makebox[0pt]{\scriptsize $-1$}}
\put(67,4){\makebox[0pt]{\scriptsize $-1$}}
\put(73,5){\makebox[0pt]{\scriptsize $-1$}}
\end{picture}
\end{align*}
which is the case of $q^2=-1$ and $p\neq2$ in the ninth row of Table \ref{tab.1}.

By applying the aforementioned method in cases $q_{11}=q_{33}=q_{i_1,i_1}=q_{i_2,i_2}=\ldots=q_{i_k,i_k}=-1$ for $i_1,i_2,\ldots,i_k\in \{4,5,\ldots,r\}$, $1\leq k<r-2$, and $q_{ii}\widetilde{q_{i,i+1}} = q_{jj}\widetilde{q_{j-1,j}} = 1$ for all $i \in \{4,\ldots,r-1\}\verb|\|\{i_1,i_2,\ldots,i_k\}$ and $j \in \{4,\ldots,r\}\verb|\|\{i_1,i_2,\ldots,i_k\}$, we obtain that all associated generalized Dynkin diagrams $\mathcal{D}$ are the cases of $q^2=-1$ and $p\neq 2$ in the tenth row of Table \ref{tab.1}.
\end{proof}

Below we discuss the sporadic cases for ranks $5, 6, 7,$ and $8$. Theorem \ref{thm:goodnei} enables a direct case-by-case verification.
%The following discussions cover the sporadic cases of ranks 5, 6, 7, and 8. Theorem \ref{thm:goodnei} allows us to discuss case by case by hand. 

\begin{lemma}\label{l-5chainspe-a}
If $\gDd _{\chi ,E}$ is of the form
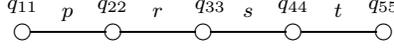
\setlength{\unitlength}{1mm}
\begin{align*}
\begin{picture}(50,4)(0,3)
\put(1,1){\circle{2}}
\put(2,1){\line(1,0){10}}
\put(13,1){\circle{2}}
\put(14,1){\line(1,0){10}}
\put(25,1){\circle{2}}
\put(26,1){\line(1,0){10}}
\put(37,1){\circle{2}}
\put(38,1){\line(1,0){10}}
\put(49,1){\circle{2}}
\put(1,4){\makebox[0pt]{\scriptsize $q_{11}$}}
\put(7,3){\makebox[0pt]{\scriptsize $q$}}
\put(13,4){\makebox[0pt]{\scriptsize $q_{22}$}}
\put(19,3){\makebox[0pt]{\scriptsize $r$}}
\put(26,4){\makebox[0pt]{\scriptsize $q_{33}$}}
\put(31,3){\makebox[0pt]{\scriptsize $s$}}
\put(37,4){\makebox[0pt]{\scriptsize $q_{44}$}}
\put(43,3){\makebox[0pt]{\scriptsize $t$}}
\put(49,4){\makebox[0pt]{\scriptsize $q_{55}$}}
\end{picture}
 % \Dchainfive{$q_{11}$}{$q$}{$q_{22}$}{$r$}{$q_{33}$}{$s$}{$q_{44}$}{$t$}{$q_{55}$}
\end{align*}
and has a good $A_5$ neighborhood, then there exists $i\in\{3,4\}$ satisfies $q_{ii}=-1$, which renders all subsequent cases infeasible.
\setlength{\unitlength}{1mm}
\begin{itemize}
    \item[$(a_1)$]\ 
    \Dchainfive{$q^{-1}$}{$q$}{$q^{-1}$}{$q$}{$-1$}{$-1$}{$-1$}{$-1$}{$-1$}
    ,\quad  $q\not=-1$,\ $p\not=2$
\vspace{-2mm}
    \item[$(a_2)$]\ 
\Dchainfive{$-1$}{$-1$}
{$-1$}{$-1$}{$-1$}{$q$}{$q^{-1}$}{$q$}{$q^{-1}$} 
,\quad  $q\not=-1$,\ $p\not=2$
\vspace{-2mm}
    \item[$(a_3)$]\ 
\Dchainfive{$-1$}{$-1$}
{$-1$}{$-1$}{$-1$}{$q$}{$q^{-1}$}{$q$}{$-1$} ,\quad  $q\not=-1$,\ $p\not=2$
\vspace{-2mm}
    \item[$(a_4)$]\ 
 \Dchainfive{$-1$}{$q$}{$q^{-1}$}{$q$}{$-1$}
{$-1$}{$-1$}{$-1$}{$-1$} ,\quad  $q\not=-1$,\ $p\not=2$
\vspace{-2mm}
\item[$(b_1)$]\ 
 \Dchainfive{$-1$}{$q^{-2}$}{$q^{2}$}{$q^{-2}$}{$q^{2}$}
{$q^{-2}$}{$-1$}{$q$}{$-1$} ,\quad $q\notin G'_3$,\ $p\not=3$
\vspace{-2mm}
 \item[$(b_2)$]\ 
 \Dchainfive{$-1$}{$q^{2}$}{$-1$}{$q^{-2}$}{$q^{2}$}
{$q^{-2}$}{$-1$}{$q$}{$-1$} ,\quad $q\notin G'_3$,\ $p\not=3$
\vspace{-2mm}
 \item[$(b_3)$]\ 
 \Dchainfive{$-1$}{$-1$}{$-1$}{$-1$}{$-1$}
{$-1$}{$-1$}{$q$}{$-1$}  ,\quad $q\not=-1$,\ $p\not=2$
\vspace{-2mm}
\item[$(c_1)$]\ 
 \Dchainfive{$q^{-1}$}{$q$}{$-1$}{$q^{-1}$}{$q$}
{$q^{-1}$}{$-1$}{$-1$}{$-1$}  
,\quad $q\not=-1$,\ $p\not=2$
\vspace{-2mm}
 \item[$(c_2)$]\ 
 \Dchainfive{$-1$}{$q^{-1}$}{$q$}{$q^{-1}$}{$q$}
{$q^{-1}$}{$-1$}{$-1$}{$-1$}   
,\quad $q\not=-1$,\ $p\not=2$
\vspace{-2mm}
 \item[$(c_3)$]\ 
 \Dchainfive{$-1$}{$q$}{$-1$}{$q^{-1}$}{$q$}
{$q^{-1}$}{$-1$}{$-1$}{$-1$} 
 ,\quad $q\not=-1$,\ $p\not=2$
 \vspace{-2mm}
 \item[$(c_4)$]\ 
 \Dchainfive{$-1$}{$-1$}{$-1$}{$-1$}{$-1$}
{$-1$}{$-1$}{$q^{-1}$}{$q$} 
 ,\quad $q\not=-1$,\ $p\not=2$
\end{itemize}
\end{lemma}
\begin{proof}
If $q_{33} \neq -1$ and $q_{44} \neq -1$, then $A^{r_3(X)} = A^{r_4(X)} = A_5$ by Lemma \ref{jslemma}, which contradicts Definition \ref{defA5}. 
In the following we analyze the cases $(a_1)-(c_3)$.
\begin{itemize}
 \item[{(i)}] If $X$ is of the type \Dchainfive{$q_{11}$}{$q$}{$q^{-1}$}{$q$}{$-1$}
{$-1$}{$-1$}{$-1$}{$-1$}\ \ then it satisfies the case $(A_{5_1})$ of Definition \ref{defA5}. Then the transformation
\begin{align*}
%\begin{picture}(50,4)(0,3)
%\put(1,1){\circle{2}}
%\put(2,1){\line(1,0){10}}
%\put(13,1){\circle{2}}
%\put(14,1){\line(1,0){10}}
%\put(25,1){\circle*{2}}
%\put(26,1){\line(1,0){10}}
%\put(37,1){\circle{2}}
%\put(38,1){\line(1,0){10}}
%\put(49,1){\circle{2}}
%\put(1,4){\makebox[0pt]{\scriptsize $q_{11}$}}
%\put(7,3){\makebox[0pt]{\scriptsize $q$}}
%\put(13,4){\makebox[0pt]{\scriptsize $q^{-1}$}}
%\put(19,3){\makebox[0pt]{\scriptsize $q$}}
%\put(25,4){\makebox[0pt]{\scriptsize $-1$}}
%\put(31,3){\makebox[0pt]{\scriptsize $-1$}}
%\put(37,4){\makebox[0pt]{\scriptsize $-1$}}
%\put(43,3){\makebox[0pt]{\scriptsize $-1$}}
%\put(49,4){\makebox[0pt]{\scriptsize $-1$}}
%\end{picture}
\Dchainfivec{$q_{11}$}{$q$}{$q^{-1}$}{$q$}{$-1$}{$-1$}{$-1$}{$-1$}{$-1$}
\qquad \Rightarrow \qquad
%\begin{picture}(38,11)(0,3)
%\put(1,1){\circle{2}}
%\put(2,1){\line(1,0){10}}
%\put(13,1){\circle{2}}
%\put(13,2){\line(2,3){6}}
%\put(14,1){\line(1,0){10}}
%\put(25,1){\circle{2}}
%\put(25,2){\line(-2,3){6}}
%\put(19,12){\circle{2}}
%\put(26,1){\line(1,0){10}}
%\put(37,1){\circle{2}}
%\put(1,4){\makebox[0pt]{\scriptsize $q_{11}$}}
%\put(7,3){\makebox[0pt]{\scriptsize $q$}}
%\put(12,4){\makebox[0pt]{\scriptsize $-1$}}
%\put(19,3){\makebox[0pt]{\scriptsize $-q$}}
%\put(26,4){\makebox[0pt]{\scriptsize $-1$}}
%\put(31,3){\makebox[0pt]{\scriptsize $-1$}}
%\put(37,4){\makebox[0pt]{\scriptsize $-1$}}
%\put(13.3,8){\makebox[0pt]{\scriptsize $q^{-1}$}}
%\put(23,8){\makebox[0pt][l]{\scriptsize $-1$}}
%\put(24,14){\makebox[0pt]{\scriptsize $-1$}}
%\end{picture}
\Dchainfivex{$q_{11}$}{$q$}{$-1$}{$-q$}{$-1$}{$-1$}{$-1$}{$q^{-1}$}{$-1$}{$-1$}
\end{align*}
implies $q_{11}=q^{-1}$ or $q_{11}=-1$. By $a_{41}^{r_2r_3(X)}=0$, one has $q\in G'_4$ and hence the further transformation
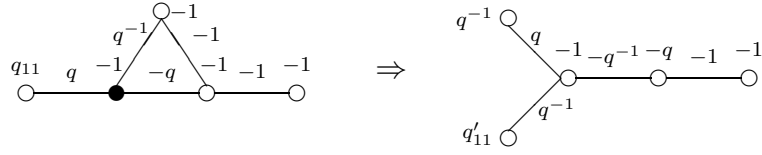
\begin{align*}
\begin{picture}(38,11)(0,3)
\put(1,1){\circle{2}}
\put(2,1){\line(1,0){10}}
\put(13,1){\circle*{2}}
\put(13,2){\line(2,3){6}}
\put(14,1){\line(1,0){10}}
\put(25,1){\circle{2}}
\put(25,2){\line(-2,3){6}}
\put(19,12){\circle{2}}
\put(26,1){\line(1,0){10}}
\put(37,1){\circle{2}}
\put(1,4){\makebox[0pt]{\scriptsize $q_{11}$}}
\put(7,3){\makebox[0pt]{\scriptsize $q$}}
\put(12,4){\makebox[0pt]{\scriptsize $-1$}}
\put(19,3){\makebox[0pt]{\scriptsize $-q$}}
\put(26,4){\makebox[0pt]{\scriptsize $-1$}}
\put(31,3){\makebox[0pt]{\scriptsize $-1$}}
\put(37,4){\makebox[0pt]{\scriptsize $-1$}}
\put(15,8){\makebox[0pt]{\scriptsize $q_{-1}$}}
\put(23,8){\makebox[0pt][l]{\scriptsize $-1$}}
\put(22,11){\makebox[0pt]{\scriptsize $-1$}}
\end{picture}
%\Dchainfivey{$q_{11}$}{$q$}{$-1$}{$-q$}{$-1$}{$-1$}{$-1$}{$q^{-1}$}{$-1$}{$-1$}
\qquad \Rightarrow \qquad
\begin{picture}(36,10)(0,9)
\put(36,9){\circle{2}}
\put(35,9){\line(-1,0){10}}
\put(24,9){\circle{2}}
\put(23,9){\line(-1,0){10}}
\put(12,9){\circle{2}}
\put(11,9){\line(-1,1){7}}
\put(11,9){\line(-1,-1){7}}
\put(4,17){\circle{2}}
\put(4,1){\circle{2}}
%\put(4,2){\line(0,1){14}}
\put(36,12){\makebox[0pt]{\scriptsize $-1$}}
\put(30,11){\makebox[0pt]{\scriptsize $-1$}}
\put(24,12){\makebox[0pt]{\scriptsize $-q$}}
\put(18,11){\makebox[0pt]{\scriptsize $-q^{-1}$}}
\put(12,12){\makebox[0pt]{\scriptsize $-1$}}
\put(7,14){\makebox[0pt][l]{\scriptsize $q$}}
\put(8,4){\makebox[0pt][l]{\scriptsize $q^{-1}$}}
\put(2,16){\makebox[0pt][r]{\scriptsize $q^{-1}$}}
\put(2,1){\makebox[0pt][r]{\scriptsize $q'_{11}$}}
\end{picture}
%\Dchainfivez{$-1$}{$-1$}{$-q$}{$-q^{-1}$}{$-1$}{$q$}{$q^{-1}$}{$q^{-1}$}{$q'_{11}$}
\end{align*}
where
\begin{align*}
  q'_{11}=&
  \begin{cases}
    -1& \quad \text{if  $q_{11}=q^{-1}$},\\
    q& \quad \text{if  $q_{11}=-1$}.
  \end{cases}
\end{align*}
yields $a_{45}^{r_2r_3(X)}=-2$, conflicting with the required relation $a_{45}^{r_2r_3(X)}=-1$. Consequently, the cases $(a_1)$ and $(a_4)$ are excluded.

 \item[(ii)]\ 
If $X$ is of the type \Dchainfive{$-1$}{$-1$}
{$-1$}{$-1$}{$-1$}{$q$}{$q^{-1}$}{$q$}{$q_{55}$} \ \ then the transformation 
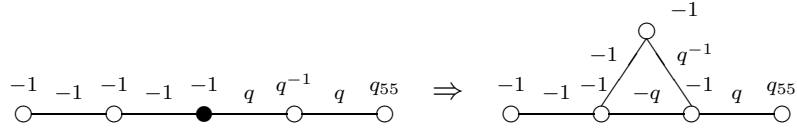
\begin{align*}
\begin{picture}(50,4)(0,3)
\put(1,1){\circle{2}}
\put(2,1){\line(1,0){10}}
\put(13,1){\circle{2}}
\put(14,1){\line(1,0){10}}
\put(25,1){\circle*{2}}
\put(26,1){\line(1,0){10}}
\put(37,1){\circle{2}}
\put(38,1){\line(1,0){10}}
\put(49,1){\circle{2}}
\put(1,4){\makebox[0pt]{\scriptsize $-1$}}
\put(7,3){\makebox[0pt]{\scriptsize $-1$}}
\put(13,4){\makebox[0pt]{\scriptsize $-1$}}
\put(19,3){\makebox[0pt]{\scriptsize $-1$}}
\put(25,4){\makebox[0pt]{\scriptsize $-1$}}
\put(31,3){\makebox[0pt]{\scriptsize $q$}}
\put(37,4){\makebox[0pt]{\scriptsize $q^{-1}$}}
\put(43,3){\makebox[0pt]{\scriptsize $q$}}
\put(49,4){\makebox[0pt]{\scriptsize $q_{55}$}}
\end{picture}
%\Dchainfivec{$-1$}{$-1$}{$-1$}{$-1$}{$-1$}{$q$}{$q^{-1}$}{$q$}{$q_{55}$}
\quad \Rightarrow \quad
\begin{picture}(38,11)(0,3)
\put(1,1){\circle{2}}
\put(2,1){\line(1,0){10}}
\put(13,1){\circle{2}}
\put(13,2){\line(2,3){6}}
\put(14,1){\line(1,0){10}}
\put(25,1){\circle{2}}
\put(25,2){\line(-2,3){6}}
\put(19,12){\circle{2}}
\put(26,1){\line(1,0){10}}
\put(37,1){\circle{2}}
\put(1,4){\makebox[0pt]{\scriptsize $-1$}}
\put(7,3){\makebox[0pt]{\scriptsize $-1$}}
\put(12,4){\makebox[0pt]{\scriptsize $-1$}}
\put(19,3){\makebox[0pt]{\scriptsize $-q$}}
\put(26,4){\makebox[0pt]{\scriptsize $-1$}}
\put(31,3){\makebox[0pt]{\scriptsize $q$}}
\put(37,4){\makebox[0pt]{\scriptsize $q_{55}$}}
\put(13.3,8){\makebox[0pt]{\scriptsize $-1$}}
\put(23,8){\makebox[0pt][l]{\scriptsize $q^{-1}$}}
\put(24,14){\makebox[0pt]{\scriptsize $-1$}}
\end{picture}
%\Dchainfivex{$-1$}{$-1$}{$-1$}{$-q$}{$-1$}{$q$}{$q_{55}$}{$-1$}{$q^{-1}$}{$-1$}
\end{align*}
implies $q_{55}=q^{-1}$ or $q_{55}=-1$. From $a_{41}^{r_2r_3(X)}=0$, we have $q=1$, which is a contradiction.
Hence the cases $(a_2)$ and $(a_3)$ are ruled out.

\item[(iii)] If $X$ is of the type \Dchainfive{$-1$}{$q^{-2}$}{$q^{2}$}{$q^{-2}$}{$q^{2}$}
{$q^{-2}$}{$-1$}{$q$}{$-1$} \ \ then it falls into case $(A_{5_2})$ of Definition \ref{defA5}. %The transformation 
%\begin{align*}
%\begin{picture}(50,4)(0,3)
%\put(1,1){\circle{2}}
%\put(2,1){\line(1,0){10}}
%\put(13,1){\circle{2}}
%\put(14,1){\line(1,0){10}}
%\put(25,1){\circle{2}}
%\put(26,1){\line(1,0){10}}
%\put(37,1){\circle*{2}}
%\put(38,1){\line(1,0){10}}
%\put(49,1){\circle{2}}
%\put(1,4){\makebox[0pt]{\scriptsize $-1$}}
%\put(7,3){\makebox[0pt]{\scriptsize $q^{-2}$}}
%\put(13,4){\makebox[0pt]{\scriptsize $q^2$}}
%\put(19,3){\makebox[0pt]{\scriptsize $q^{-2}$}}
%\put(25,4){\makebox[0pt]{\scriptsize $q^2$}}
%\put(31,3){\makebox[0pt]{\scriptsize $q^{-2}$}}
%\put(37,4){\makebox[0pt]{\scriptsize $-1$}}
%\put(43,3){\makebox[0pt]{\scriptsize $q$}}
%\put(49,4){\makebox[0pt]{\scriptsize $-1$}}
%\end{picture}
%\Dchainfived{$-1$}{$q^{-2}$}{$q^2$}{$q^{-2}$}{$q^2$}{$q^{-2}$}{$-1$}{$q$}{$-1$}
%~ \Rightarrow
%\Dchainfivew{$-1$}{$q^{-2}$}{$q^2$}{$q^{-2}$}{$-1$}{$q^2$}{$q^{-1}$}{$-1$}{$q^{-1}$}{$q$}
%\quad \Rightarrow 
%\tau_{12534}~ 
%\Dchainfivex{$-1$}{$q^{-2}$}{$-1$}{$q^2$}{$-1$}{$q^{-2}$}{$q^2$}{$q^{-3}$}{$q$}{$-1$}
%\end{align*}
%implies $a_{25}^{r_3r_4(X)}=-1$ by $q\notin G'_3$. By $a_{15}^{r_2r_3r_4(X)}=0$, we get $q\in G'_5$.
The following transformations
\begin{align*}
\begin{picture}(50,4)(0,3)
\put(1,1){\circle*{2}}
\put(2,1){\line(1,0){10}}
\put(13,1){\circle{2}}
\put(14,1){\line(1,0){10}}
\put(25,1){\circle{2}}
\put(26,1){\line(1,0){10}}
\put(37,1){\circle{2}}
\put(38,1){\line(1,0){10}}
\put(49,1){\circle{2}}
\put(1,4){\makebox[0pt]{\scriptsize $-1$}}
\put(7,3){\makebox[0pt]{\scriptsize $q^{-2}$}}
\put(13,4){\makebox[0pt]{\scriptsize $q^2$}}
\put(19,3){\makebox[0pt]{\scriptsize $q^{-2}$}}
\put(25,4){\makebox[0pt]{\scriptsize $q^2$}}
\put(31,3){\makebox[0pt]{\scriptsize $q^{-2}$}}
\put(37,4){\makebox[0pt]{\scriptsize $-1$}}
\put(43,3){\makebox[0pt]{\scriptsize $q$}}
\put(49,4){\makebox[0pt]{\scriptsize $-1$}}
\end{picture}
%\Dchainfivea{$-1$}{$q^{-2}$}{$q^2$}{$q^{-2}$}{$q^2$}{$q^{-2}$}{$-1$}{$q$}{$-1$}
\quad \Rightarrow \quad
\begin{picture}(50,4)(0,3)
\put(1,1){\circle{2}}
\put(2,1){\line(1,0){10}}
\put(13,1){\circle*{2}}
\put(14,1){\line(1,0){10}}
\put(25,1){\circle{2}}
\put(26,1){\line(1,0){10}}
\put(37,1){\circle{2}}
\put(38,1){\line(1,0){10}}
\put(49,1){\circle{2}}
\put(1,4){\makebox[0pt]{\scriptsize $-1$}}
\put(7,3){\makebox[0pt]{\scriptsize $q^2$}}
\put(13,4){\makebox[0pt]{\scriptsize $-1$}}
\put(19,3){\makebox[0pt]{\scriptsize $q^{-2}$}}
\put(25,4){\makebox[0pt]{\scriptsize $q^2$}}
\put(31,3){\makebox[0pt]{\scriptsize $q^{-2}$}}
\put(37,4){\makebox[0pt]{\scriptsize $-1$}}
\put(43,3){\makebox[0pt]{\scriptsize $q$}}
\put(49,4){\makebox[0pt]{\scriptsize $-1$}}
\end{picture}
%\Dchainfiveb{$-1$}{$q^2$}{$-1$}{$q^{-2}$}{$q^2$}{$q^{-2}$}{$-1$}{$q$}{$-1$}
\end{align*}
\begin{align*}
\Rightarrow \quad 
\begin{picture}(50,4)(0,3)
\put(1,1){\circle{2}}
\put(2,1){\line(1,0){10}}
\put(13,1){\circle{2}}
\put(14,1){\line(1,0){10}}
\put(25,1){\circle*{2}}
\put(26,1){\line(1,0){10}}
\put(37,1){\circle{2}}
\put(38,1){\line(1,0){10}}
\put(49,1){\circle{2}}
\put(1,4){\makebox[0pt]{\scriptsize $-1$}}
\put(7,3){\makebox[0pt]{\scriptsize $q^{-2}$}}
\put(13,4){\makebox[0pt]{\scriptsize $-1$}}
\put(19,3){\makebox[0pt]{\scriptsize $q^2$}}
\put(25,4){\makebox[0pt]{\scriptsize $-1$}}
\put(31,3){\makebox[0pt]{\scriptsize $q^{-2}$}}
\put(37,4){\makebox[0pt]{\scriptsize $-1$}}
\put(43,3){\makebox[0pt]{\scriptsize $q$}}
\put(49,4){\makebox[0pt]{\scriptsize $-1$}}
\end{picture}
%\Dchainfivec{$-1$}{$q^{-2}$}{$-1$}{$q^2$}{$-1$}{$q^{-2}$}{$-1$}{$q$}{$-1$}
\quad \Rightarrow \quad 
\begin{picture}(50,4)(0,3)
\put(1,1){\circle{2}}
\put(2,1){\line(1,0){10}}
\put(13,1){\circle{2}}
\put(14,1){\line(1,0){10}}
\put(25,1){\circle{2}}
\put(26,1){\line(1,0){10}}
\put(37,1){\circle{2}}
\put(38,1){\line(1,0){10}}
\put(49,1){\circle{2}}
\put(1,4){\makebox[0pt]{\scriptsize $-1$}}
\put(7,3){\makebox[0pt]{\scriptsize $q^{-2}$}}
\put(13,4){\makebox[0pt]{\scriptsize $q^2$}}
\put(19,3){\makebox[0pt]{\scriptsize $q^{-2}$}}
\put(26,4){\makebox[0pt]{\scriptsize $-1$}}
\put(31,3){\makebox[0pt]{\scriptsize $q^2$}}
\put(37,4){\makebox[0pt]{\scriptsize $q^{-2}$}}
\put(43,3){\makebox[0pt]{\scriptsize $q$}}
\put(49,4){\makebox[0pt]{\scriptsize $-1$}}
\end{picture}
%\Dchainfive{$-1$}{$q^{-2}$}{$q^2$}{$q^{-2}$}{$-1$}{$q^2$}{$q^{-2}$}{$q$}{$-1$}
\end{align*}
give $a_{45}^{r_3r_2r_1(X)}=-3$, which contradicts $a_{45}^{r_3r_2r_1(X)}\not=-3$. Hence the case $(b_1)$ is excluded.

 \item[(iv)] If $X$ is of the type \Dchainfive{$-1$}{$q^{2}$}{$-1$}{$q^{-2}$}{$q^{2}$}
{$q^{-2}$}{$-1$}{$q$}{$-1$}  \ \ then it satisfies the conditions for case $(A_{5_2})$ in Definition \ref{defA5}. Then we have 
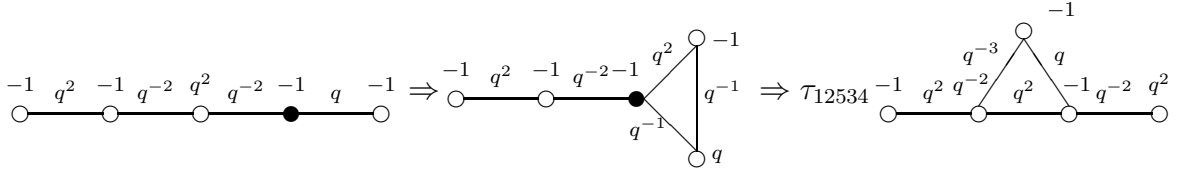
\begin{align*}
\begin{picture}(50,4)(0,3)
\put(1,1){\circle{2}}
\put(2,1){\line(1,0){10}}
\put(13,1){\circle{2}}
\put(14,1){\line(1,0){10}}
\put(25,1){\circle{2}}
\put(26,1){\line(1,0){10}}
\put(37,1){\circle*{2}}
\put(38,1){\line(1,0){10}}
\put(49,1){\circle{2}}
\put(1,4){\makebox[0pt]{\scriptsize $-1$}}
\put(7,3){\makebox[0pt]{\scriptsize $q^2$}}
\put(13,4){\makebox[0pt]{\scriptsize $-1$}}
\put(19,3){\makebox[0pt]{\scriptsize $q^{-2}$}}
\put(25,4){\makebox[0pt]{\scriptsize $q^{2}$}}
\put(31,3){\makebox[0pt]{\scriptsize $q^{-2}$}}
\put(37,4){\makebox[0pt]{\scriptsize $-1$}}
\put(43,3){\makebox[0pt]{\scriptsize $q$}}
\put(49,4){\makebox[0pt]{\scriptsize $-1$}}
\end{picture}
%\Dchainfived{$-1$}{$q^2$}{$-1$}{$q^{-2}$}{$q^2$}{$q^{-2}$}{$-1$}{$q$}{$-1$}
~ \Rightarrow 
\Dchainfivew{$-1$}{$q^2$}{$-1$}{$q^{-2}$}{$-1$}{$q^2$}{$q^{-1}$}{$-1$}{$q^{-1}$}{$q$} 
 \quad \Rightarrow 
\tau_{12534}~ 
\begin{picture}(38,11)(0,3)
\put(1,1){\circle{2}}
\put(2,1){\line(1,0){10}}
\put(13,1){\circle{2}}
\put(13,2){\line(2,3){6}}
\put(14,1){\line(1,0){10}}
\put(25,1){\circle{2}}
\put(25,2){\line(-2,3){6}}
\put(19,12){\circle{2}}
\put(26,1){\line(1,0){10}}
\put(37,1){\circle{2}}
\put(1,4){\makebox[0pt]{\scriptsize $-1$}}
\put(7,3){\makebox[0pt]{\scriptsize $q^2$}}
\put(12,4){\makebox[0pt]{\scriptsize $q^{-2}$}}
\put(19,3){\makebox[0pt]{\scriptsize $q^2$}}
\put(26,4){\makebox[0pt]{\scriptsize $-1$}}
\put(31,3){\makebox[0pt]{\scriptsize $q^{-2}$}}
\put(37,4){\makebox[0pt]{\scriptsize $q^2$}}
\put(13.3,8){\makebox[0pt]{\scriptsize $q^{-3}$}}
\put(23,8){\makebox[0pt][l]{\scriptsize $q$}}
\put(24,14){\makebox[0pt]{\scriptsize $-1$}}
\end{picture}
%\Dchainfivex{$-1$}{$q^2$}{$q^{-2}$}{$q^2$}{$-1$}{$q^{-2}$}{$q^2$}{$q^{-3}$}{$q$}{$-1$}
\end{align*}

Hence $q=1$ by $a_{15}^{r_2r_3r_4(X)}=0$, which is a contradiction. Hence the case $(b_2)$ is ruled out.
 
 \item[(v)] If $X$ is of the type \Dchainfive{$-1$}{$-1$}{$-1$}{$-1$}{$-1$}
{$-1$}{$-1$}{$q$}{$-1$} 
\ \  it satisfies the case $(A_{5_2})$ in Definition \ref{defA5}. Then we have the reflections of $X$ 
\begin{align*}
\begin{picture}(50,4)(0,3)
\put(1,1){\circle{2}}
\put(2,1){\line(1,0){10}}
\put(13,1){\circle{2}}
\put(14,1){\line(1,0){10}}
\put(25,1){\circle{2}}
\put(26,1){\line(1,0){10}}
\put(37,1){\circle*{2}}
\put(38,1){\line(1,0){10}}
\put(49,1){\circle{2}}
\put(1,4){\makebox[0pt]{\scriptsize $-1$}}
\put(7,3){\makebox[0pt]{\scriptsize $-1$}}
\put(13,4){\makebox[0pt]{\scriptsize $-1$}}
\put(19,3){\makebox[0pt]{\scriptsize $-1$}}
\put(25,4){\makebox[0pt]{\scriptsize $-1$}}
\put(31,3){\makebox[0pt]{\scriptsize $-1$}}
\put(37,4){\makebox[0pt]{\scriptsize $-1$}}
\put(43,3){\makebox[0pt]{\scriptsize $q$}}
\put(49,4){\makebox[0pt]{\scriptsize $-1$}}
\end{picture}
%\Dchainfived{$-1$}{$-1$}{$-1$}{$-1$}{$-1$}{$-1$}{$-1$}{$q$}{$-1$}
\quad
\Rightarrow \quad
\begin{picture}(36,10)(0,9)
\put(1,9){\circle{2}}
\put(2,9){\line(1,0){10}}
\put(13,9){\circle{2}}
\put(14,9){\line(1,0){10}}
\put(25,9){\circle*{2}}
\put(26,9){\line(1,1){7}}
\put(26,9){\line(1,-1){7}}
\put(33,17){\circle{2}}
\put(33,1){\circle{2}}
\put(33,2){\line(0,1){14}}
\put(1,12){\makebox[0pt]{\scriptsize $-1$}}
\put(7,11){\makebox[0pt]{\scriptsize $-1$}}
\put(13,12){\makebox[0pt]{\scriptsize $-1$}}
\put(19,11){\makebox[0pt]{\scriptsize $-1$}}
\put(23.5,12){\makebox[0pt]{\scriptsize $-1$}}
\put(30,14){\makebox[0pt][r]{\scriptsize $-1$}}
\put(29,4){\makebox[0pt][r]{\scriptsize $-q$}}
\put(35,16){\makebox[0pt][l]{\scriptsize $-1$}}
\put(34,9){\makebox[0pt][l]{\scriptsize $q^{-1}$}}
\put(35,1){\makebox[0pt][l]{\scriptsize $q$}}
\end{picture}
%\Dchainfivew{$-1$}{$-1$}{$-1$}{$-1$}{$-1$}{$-1$}{$-q$}{$-1$}{$q^{-1}$}{$q$}
\end{align*}
%and hence $q\in G'_4$ by $a_{53}^{r_4(X)}=-1$. Then the further transformation is
\begin{align*}
\Rightarrow \quad 
\tau_{12534} \quad 
\begin{picture}(38,11)(0,3)
\put(1,1){\circle{2}}
\put(2,1){\line(1,0){10}}
\put(13,1){\circle*{2}}
\put(13,2){\line(2,3){6}}
\put(14,1){\line(1,0){10}}
\put(25,1){\circle{2}}
\put(25,2){\line(-2,3){6}}
\put(19,12){\circle{2}}
\put(26,1){\line(1,0){10}}
\put(37,1){\circle{2}}
\put(1,4){\makebox[0pt]{\scriptsize $-1$}}
\put(7,3){\makebox[0pt]{\scriptsize $-1$}}
\put(12,4){\makebox[0pt]{\scriptsize $-1$}}
\put(19,3){\makebox[0pt]{\scriptsize $-1$}}
\put(26,4){\makebox[0pt]{\scriptsize $-1$}}
\put(31,3){\makebox[0pt]{\scriptsize $-1$}}
\put(37,4){\makebox[0pt]{\scriptsize $-1$}}
\put(15,8){\makebox[0pt]{\scriptsize $q$}}
\put(23,8){\makebox[0pt][l]{\scriptsize $q$}}
\put(22,11){\makebox[0pt]{\scriptsize $-1$}}
\end{picture}
%\Dchainfivey{$-1$}{$-1$}{$-1$}{$-1$}{$-1$}{$-1$}{$-1$}{$q$}{$q$}{$-1$}
\quad \Rightarrow \quad
\tau_{51234} \quad 
\Dchainfivev{$-1$}{$-1$}{$-1$}{$-1$}{$-1$}{$q^{-1}$}{$-1$}{$-1$}{$q^{-1}$}{$-1$}
\end{align*}
and obtain that $q\in G'_4$ by $a_{53}^{r_4(X)}=-1$ and $a_{15}^{r_2r_3r_4(X)}=-1$, conflicting with the requirement $a_{15}^{r_2r_3r_4(X)}=0$. Accordingly, the case $(b_3)$ is excluded.

  \item[(vi)] If $X$ is of the type \Dchainfive{$q^{-1}$}{$q$}{$-1$}{$q^{-1}$}{$q$}
{$q^{-1}$}{$-1$}{$-1$}{$-1$}   \ \ then $X$ falls under case $(A_{5_3})$ from Definition \ref{defA5}. On one side, the successive transformations 
\begin{align*}
\begin{picture}(50,4)(0,3)
\put(1,1){\circle{2}}
\put(2,1){\line(1,0){10}}
\put(13,1){\circle{2}}
\put(14,1){\line(1,0){10}}
\put(25,1){\circle{2}}
\put(26,1){\line(1,0){10}}
\put(37,1){\circle*{2}}
\put(38,1){\line(1,0){10}}
\put(49,1){\circle{2}}
\put(1,4){\makebox[0pt]{\scriptsize $q^{-1}$}}
\put(7,3){\makebox[0pt]{\scriptsize $q$}}
\put(13,4){\makebox[0pt]{\scriptsize $-1$}}
\put(19,3){\makebox[0pt]{\scriptsize $q^{-1}$}}
\put(25,4){\makebox[0pt]{\scriptsize $q$}}
\put(31,3){\makebox[0pt]{\scriptsize $q^{-1}$}}
\put(37,4){\makebox[0pt]{\scriptsize $-1$}}
\put(43,3){\makebox[0pt]{\scriptsize $-1$}}
\put(49,4){\makebox[0pt]{\scriptsize $-1$}}
\end{picture}
%\Dchainfived{$q^{-1}$}{$q$}{$-1$}{$q^{-1}$}{$q$}{$q^{-1}$}{$-1$}{$-1$}{$-1$}
~\Rightarrow
\Dchainfivew{$q^{-1}$}{$q$}{$-1$}{$q^{-1}$}{$-1$}{$q$}{$-q^{-1}$}{$-1$}{$-1$}{$-1$}
\quad \Rightarrow
\tau_{12534}~ 
\begin{picture}(38,11)(0,3)
\put(1,1){\circle{2}}
\put(2,1){\line(1,0){10}}
\put(13,1){\circle{2}}
\put(13,2){\line(2,3){6}}
\put(14,1){\line(1,0){10}}
\put(25,1){\circle{2}}
\put(25,2){\line(-2,3){6}}
\put(19,12){\circle{2}}
\put(26,1){\line(1,0){10}}
\put(37,1){\circle{2}}
\put(1,4){\makebox[0pt]{\scriptsize $q^{-1}$}}
\put(7,3){\makebox[0pt]{\scriptsize $q$}}
\put(12,4){\makebox[0pt]{\scriptsize $q^{-1}$}}
\put(19,3){\makebox[0pt]{\scriptsize $q$}}
\put(26,4){\makebox[0pt]{\scriptsize $-1$}}
\put(31,3){\makebox[0pt]{\scriptsize $q^{-1}$}}
\put(37,4){\makebox[0pt]{\scriptsize $q$}}
\put(13.3,8){\makebox[0pt]{\scriptsize $-q^{-2}$}}
\put(23,8){\makebox[0pt][l]{\scriptsize $-q$}}
\put(24,14){\makebox[0pt]{\scriptsize $-q^{-1}$}}
\end{picture}
%\Dchainfivex{$q^{-1}$}{$q$}{$q^{-1}$}{$q$}{$-1$}{$q^{-1}$}{$q$}{$-q^{-2}$}{$-q$}{$-q^{-1}$}
\end{align*}
together with the relation $a_{52}^{r_3r_4(X)}=-1$ force $q\in G'_3$.
 On the other side, applying the transformation
 \begin{align*}
\Dchainfiveu{$q^{-1}$}{$q$}{$-1$}{$q^{-1}$}{$-1$}{$q$}{$-q^{-1}$}{$-1$}{$-1$}{$-1$}
\qquad \Rightarrow \quad
\tau_{12354}\quad
\begin{picture}(50,4)(0,3)
\put(1,1){\circle{2}}
\put(2,1){\line(1,0){10}}
\put(13,1){\circle{2}}
\put(14,1){\line(1,0){10}}
\put(25,1){\circle{2}}
\put(26,1){\line(1,0){10}}
\put(37,1){\circle{2}}
\put(38,1){\line(1,0){10}}
\put(49,1){\circle{2}}
\put(1,4){\makebox[0pt]{\scriptsize $q^{-1}$}}
\put(7,3){\makebox[0pt]{\scriptsize $q$}}
\put(13,4){\makebox[0pt]{\scriptsize $-1$}}
\put(19,3){\makebox[0pt]{\scriptsize $q^{-1}$}}
\put(26,4){\makebox[0pt]{\scriptsize $q^{-1}$}}
\put(31,3){\makebox[0pt]{\scriptsize $-q$}}
\put(37,4){\makebox[0pt]{\scriptsize $-1$}}
\put(43,3){\makebox[0pt]{\scriptsize $-1$}}
\put(49,4){\makebox[0pt]{\scriptsize $-1$}}
\end{picture}
%\Dchainfive{$q^{-1}$}{$q$}{$-1$}{$q^{-1}$}{$-q^{-1}$}{$-q$}{$-1$}{$-1$}{$-1$}
\end{align*}
 and substituting $q\in G'_3$ yields $a_{32}^{r_5r_4(X)}=-2$, conflicting with the required identity $a_{32}^{r_5r_4(X)}=-1$. Consequently, case $(c_1)$ is ruled out.

 By parallel reasoning to the treatment of case $(c_1)$, cases $(c_2)$ and $(c_3)$ are likewise excluded.
 
\item[(vii)] If $X$ has the type \Dchainfive{$-1$}{$-1$}{$-1$}{$-1$}{$-1$}
{$-1$}{$-1$}{$q^{-1}$}{$q$},    \ \ then 
the successive reflections 
\begin{align*}
\begin{picture}(50,4)(0,3)
\put(1,1){\circle{2}}
\put(2,1){\line(1,0){10}}
\put(13,1){\circle{2}}
\put(14,1){\line(1,0){10}}
\put(25,1){\circle{2}}
\put(26,1){\line(1,0){10}}
\put(37,1){\circle*{2}}
\put(38,1){\line(1,0){10}}
\put(49,1){\circle{2}}
\put(1,4){\makebox[0pt]{\scriptsize $-1$}}
\put(7,3){\makebox[0pt]{\scriptsize $-1$}}
\put(13,4){\makebox[0pt]{\scriptsize $-1$}}
\put(19,3){\makebox[0pt]{\scriptsize $-1$}}
\put(25,4){\makebox[0pt]{\scriptsize $-1$}}
\put(31,3){\makebox[0pt]{\scriptsize $-1$}}
\put(37,4){\makebox[0pt]{\scriptsize $-1$}}
\put(43,3){\makebox[0pt]{\scriptsize $q^{-1}$}}
\put(49,4){\makebox[0pt]{\scriptsize $q$}}
\end{picture}
%\Dchainfived{$-1$}{$-1$}{$-1$}{$-1$}{$-1$}{$-1$}{$-1$}{$q^{-1}$}{$q$}
\quad \Rightarrow \quad
\begin{picture}(36,10)(0,9)
\put(1,9){\circle{2}}
\put(2,9){\line(1,0){10}}
\put(13,9){\circle{2}}
\put(14,9){\line(1,0){10}}
\put(25,9){\circle{2}}
\put(26,9){\line(1,1){7}}
\put(26,9){\line(1,-1){7}}
\put(33,17){\circle{2}}
\put(33,1){\circle*{2}}
\put(33,2){\line(0,1){14}}
\put(1,12){\makebox[0pt]{\scriptsize $-1$}}
\put(7,11){\makebox[0pt]{\scriptsize $-1$}}
\put(13,12){\makebox[0pt]{\scriptsize $-1$}}
\put(19,11){\makebox[0pt]{\scriptsize $-1$}}
\put(24,12){\makebox[0pt]{\scriptsize $-1$}}
\put(30,14){\makebox[0pt][r]{\scriptsize $-1$}}
\put(29,4){\makebox[0pt][r]{\scriptsize $-q^{-1}$}}
\put(35,16){\makebox[0pt][l]{\scriptsize $-1$}}
\put(34,9){\makebox[0pt][l]{\scriptsize $q$}}
\put(35,1){\makebox[0pt][l]{\scriptsize $-1$}}
\end{picture}
%\Dchainfiveu{$-1$}{$-1$}{$-1$}{$-1$}{$-1$}{$-1$}{$-q^{-1}$}{$-1$}{$q$}{$-1$}
\end{align*}
\begin{align*}
\qquad \Rightarrow \quad
\tau_{12354}\quad
\begin{picture}(50,4)(0,3)
\put(1,1){\circle{2}}
\put(2,1){\line(1,0){10}}
\put(13,1){\circle{2}}
\put(14,1){\line(1,0){10}}
\put(25,1){\circle{2}}
\put(26,1){\line(1,0){10}}
\put(37,1){\circle{2}}
\put(38,1){\line(1,0){10}}
\put(49,1){\circle{2}}
\put(1,4){\makebox[0pt]{\scriptsize $-1$}}
\put(7,3){\makebox[0pt]{\scriptsize $-1$}}
\put(13,4){\makebox[0pt]{\scriptsize $-1$}}
\put(19,3){\makebox[0pt]{\scriptsize $-1$}}
\put(26,4){\makebox[0pt]{\scriptsize $-q^{-1}$}}
\put(31,3){\makebox[0pt]{\scriptsize $-q$}}
\put(37,4){\makebox[0pt]{\scriptsize $-1$}}
\put(43,3){\makebox[0pt]{\scriptsize $q^{-1}$}}
\put(49,4){\makebox[0pt]{\scriptsize $q$}}
\end{picture}
%\Dchainfive{$-1$}{$-1$}{$-1$}{$-1$}{$-q^{-1}$}{$-q$}{$-1$}{$q^{-1}$}{$q$}
\end{align*}
together with the condition $ a_{32}^{r_5r_4(X)}=-1$, force $q=1$ and immediate a contradiction. Accordingly, the case $(c_4)$ is ruled out.
\end{itemize}
\end{proof}

Lemma \ref{l-5chainspe-a} streamlines the subsequent arguments considerably.

\begin{theorem}\label{theo.rank5}
Suppose ${\roots}^{[M]}$ is a finite set of roots of sporadic finite Cartan graphs of rank $5$. Then every admissable generalized Dynkin diagrams of $V$ occurs in rows $11$-$15$ of Table \ref{tab.1}.
\end{theorem}

\begin{proof}
Define $a\colon= -a_{24}^{r_3(X)}$, $b\colon=-a_{42}^{r_3(X)}$, $c\colon=-a_{35}^{r_4(X)}$, $d\colon=-a_{53}^{r_4(X)}$, and $e\colon=-a_{43}^{r_5(X)}$.
By Theorem \ref{thm:goodnei}, $A^X$ has a good $A_5$ neighborhood. Lemma \ref{l-5chainspe-a} guarantees $q_{ii}=-1$ for at least one index $i\in\{3,4\}$ across all admissible configurations. Since $A^{X}=A_{5}$, the identities $(2)_{q_{ii}}(q_{ii}\widetilde{q_{i,i+1}}-1)=(2)_{q_{jj}}(q_{jj}\widetilde{q_{j-1,j}}-1)=0$ hold for every $i\in \{1,2,3,4\}$ and $j\in \{2,3,4,5\}$. We divide the subsequent analysis into Steps $1$–$5$ according to the total count of indices $i\in \{1,2,3,4,5\}$ satisfying $q_{ii}=-1$.

Step $1$. Precisely one entry satisfying $q_{ii}=-1$. This case is further separated into two subcases.

Step $1.1$.
Set $q_{33}=-1$ and suppose
\[
q_{ii}\widetilde q_{i,i+1}=q_{jj}\widetilde q_{j-1,j}=1,\quad \forall\,i\in\{1,2,4\},\;j\in\{2,4,5\}.
\]
This yields $e=1$, and Definition \ref{defA5} forces $(a,b,c,d,e)=(1,1,0,0,1)$. Let $q:=q_{11}$, $r:=q_{44}$. The condition $a_{24}^{r_3(X)}=-1$ gives $qr\neq 1$. Successive reflections of $X$ read
\begin{align*}
X\colon~ 
\begin{picture}(50,4)(0,3)
\put(1,1){\circle{2}}
\put(2,1){\line(1,0){10}}
\put(13,1){\circle{2}}
\put(14,1){\line(1,0){10}}
\put(25,1){\circle*{2}}
\put(26,1){\line(1,0){10}}
\put(37,1){\circle{2}}
\put(38,1){\line(1,0){10}}
\put(49,1){\circle{2}}
\put(1,4){\makebox[0pt]{\scriptsize $q$}}
\put(7,3){\makebox[0pt]{\scriptsize $q^{-1}$}}
\put(13,4){\makebox[0pt]{\scriptsize $q$}}
\put(19,3){\makebox[0pt]{\scriptsize $q^{-1}$}}
\put(25,4){\makebox[0pt]{\scriptsize $-1$}}
\put(31,3){\makebox[0pt]{\scriptsize $r^{-1}$}}
\put(37,4){\makebox[0pt]{\scriptsize $r$}}
\put(43,3){\makebox[0pt]{\scriptsize $r^{-1}$}}
\put(49,4){\makebox[0pt]{\scriptsize $r$}}
\end{picture}
\quad \Rightarrow \quad
r_3(X)\colon~ 
\begin{picture}(38,11)(0,3)
\put(1,1){\circle{2}}
\put(2,1){\line(1,0){10}}
\put(13,1){\circle{2}}
\put(13,2){\line(2,3){6}}
\put(14,1){\line(1,0){10}}
\put(25,1){\circle{2}}
\put(25,2){\line(-2,3){6}}
\put(19,12){\circle{2}}
\put(26,1){\line(1,0){10}}
\put(37,1){\circle{2}}
\put(1,4){\makebox[0pt]{\scriptsize $q$}}
\put(7,3){\makebox[0pt]{\scriptsize $q^{-1}$}}
\put(12,4){\makebox[0pt]{\scriptsize $-1$}}
\put(19,3){\makebox[0pt]{\scriptsize $(qr)^{-1}$}}
\put(26,4){\makebox[0pt]{\scriptsize $-1$}}
\put(31,3){\makebox[0pt]{\scriptsize $r^{-1}$}}
\put(37,4){\makebox[0pt]{\scriptsize $r$}}
\put(13.3,8){\makebox[0pt]{\scriptsize $q$}}
\put(23,8){\makebox[0pt][l]{\scriptsize $r$}}
\put(24,14){\makebox[0pt]{\scriptsize $-1$}}
\end{picture}
\end{align*}
\begin{align*}
\Rightarrow \quad 
r_2r_3(X)\colon~ \tau_{32145}
\begin{picture}(38,11)(0,3)
\put(1,1){\circle{2}}
\put(2,1){\line(1,0){10}}
\put(13,1){\circle{2}}
\put(13,2){\line(2,3){6}}
\put(14,1){\line(1,0){10}}
\put(25,1){\circle{2}}
\put(25,2){\line(-2,3){6}}
\put(19,12){\circle{2}}
\put(26,1){\line(1,0){10}}
\put(37,1){\circle{2}}
\put(1,4){\makebox[0pt]{\scriptsize $q$}}
\put(7,3){\makebox[0pt]{\scriptsize $q^{-1}$}}
\put(12,4){\makebox[0pt]{\scriptsize $-1$}}
\put(19,3){\makebox[0pt]{\scriptsize $qr$}}
\put(26,4){\makebox[0pt]{\scriptsize $(qr)^{-1}$}}
\put(31,3){\makebox[0pt]{\scriptsize $r^{-1}$}}
\put(37,4){\makebox[0pt]{\scriptsize $r$}}
\put(13.3,8){\makebox[0pt]{\scriptsize $q$}}
\put(23,8){\makebox[0pt][l]{\scriptsize $q^{-2}r^{-1}$}}
\put(24,14){\makebox[0pt]{\scriptsize $-1$}}
\end{picture}
\end{align*}
From $a_{41}^{r_2r_3(X)}=0$ we obtain $q^2r=1$, while $a_{45}^{r_2r_3(X)}=-1$ yields either $qr=-1$ or $qr^2=1$. Combining these relations gives $q=r\in G_3'$ and $p\neq3$, hence $\mathcal{D}=\mathcal{D}_{11,1}^5$.

Step $1.2$.
Let $q_{44}=-1$ and
\[
q_{ii}\widetilde q_{i,i+1}=q_{jj}\widetilde q_{j-1,j}=1,\quad \forall\,i\in\{1,2,3\},\;j\in\{2,3,5\}.
\]
Then $(a,b,e)=(0,0,1)$, which together with Definition \ref{defA5} implies $(a,b,c,d,e)=(0,0,1,1,1)$. Set $q:=q_{11}$, $r:=q_{55}$. The entry $a_{35}^{r_4(X)}=-1$ enforces $qr\neq 1$.

On one side, iterated reflections
\begin{align*}
X\colon~
\begin{picture}(50,4)(0,3)
\put(1,1){\circle{2}}
\put(2,1){\line(1,0){10}}
\put(13,1){\circle{2}}
\put(14,1){\line(1,0){10}}
\put(25,1){\circle{2}}
\put(26,1){\line(1,0){10}}
\put(37,1){\circle*{2}}
\put(38,1){\line(1,0){10}}
\put(49,1){\circle{2}}
\put(1,4){\makebox[0pt]{\scriptsize $q$}}
\put(7,3){\makebox[0pt]{\scriptsize $q^{-1}$}}
\put(13,4){\makebox[0pt]{\scriptsize $q$}}
\put(19,3){\makebox[0pt]{\scriptsize $q^{-1}$}}
\put(25,4){\makebox[0pt]{\scriptsize $q$}}
\put(31,3){\makebox[0pt]{\scriptsize $q^{-1}$}}
\put(37,4){\makebox[0pt]{\scriptsize $-1$}}
\put(43,3){\makebox[0pt]{\scriptsize $r^{-1}$}}
\put(49,4){\makebox[0pt]{\scriptsize $r$}}
\end{picture}
\quad
\Rightarrow \quad
r_4(X)\colon~ 
\begin{picture}(36,10)(0,9)
\put(1,9){\circle{2}}
\put(2,9){\line(1,0){10}}
\put(13,9){\circle{2}}
\put(14,9){\line(1,0){10}}
\put(25,9){\circle{2}}
\put(26,9){\line(1,1){7}}
\put(26,9){\line(1,-1){7}}
\put(33,17){\circle{2}}
\put(33,1){\circle*{2}}
\put(33,2){\line(0,1){14}}
\put(1,12){\makebox[0pt]{\scriptsize $q$}}
\put(7,11){\makebox[0pt]{\scriptsize $q^{-1}$}}
\put(13,12){\makebox[0pt]{\scriptsize $q$}}
\put(19,11){\makebox[0pt]{\scriptsize $q^{-1}$}}
\put(24,12){\makebox[0pt]{\scriptsize $-1$}}
\put(30,14){\makebox[0pt][r]{\scriptsize $q$}}
\put(29,4){\makebox[0pt][r]{\scriptsize $(qr)^{-1}$}}
\put(35,16){\makebox[0pt][l]{\scriptsize $-1$}}
\put(34,9){\makebox[0pt][l]{\scriptsize $r$}}
\put(35,1){\makebox[0pt][l]{\scriptsize $-1$}}
\end{picture}
\end{align*}
\begin{align*}
 \Rightarrow \quad 
r_5r_4(X)\colon~ \tau_{12354}~ 
\begin{picture}(50,4)(0,3)
\put(1,1){\circle{2}}
\put(2,1){\line(1,0){10}}
\put(13,1){\circle{2}}
\put(14,1){\line(1,0){10}}
\put(25,1){\circle{2}}
\put(26,1){\line(1,0){10}}
\put(37,1){\circle{2}}
\put(38,1){\line(1,0){10}}
\put(49,1){\circle{2}}
\put(1,4){\makebox[0pt]{\scriptsize $q$}}
\put(7,3){\makebox[0pt]{\scriptsize $q^{-1}$}}
\put(13,4){\makebox[0pt]{\scriptsize $q$}}
\put(19,3){\makebox[0pt]{\scriptsize $q^{-1}$}}
\put(26,4){\makebox[0pt]{\scriptsize $(qr)^{-1}$}}
\put(31,3){\makebox[0pt]{\scriptsize $qr$}}
\put(37,4){\makebox[0pt]{\scriptsize $-1$}}
\put(43,3){\makebox[0pt]{\scriptsize $r^{-1}$}}
\put(49,4){\makebox[0pt]{\scriptsize $r$}}
\end{picture}
\end{align*}
together with $a_{32}^{r_5r_4(X)}=-1$ give $qr=-1$ or $q^2r=1$.

On the other side, further successive reflections
\begin{align*}
X\colon~ 
\begin{picture}(50,4)(0,3)
\put(1,1){\circle{2}}
\put(2,1){\line(1,0){10}}
\put(13,1){\circle{2}}
\put(14,1){\line(1,0){10}}
\put(25,1){\circle{2}}
\put(26,1){\line(1,0){10}}
\put(37,1){\circle*{2}}
\put(38,1){\line(1,0){10}}
\put(49,1){\circle{2}}
\put(1,4){\makebox[0pt]{\scriptsize $q$}}
\put(7,3){\makebox[0pt]{\scriptsize $q^{-1}$}}
\put(13,4){\makebox[0pt]{\scriptsize $q$}}
\put(19,3){\makebox[0pt]{\scriptsize $q^{-1}$}}
\put(25,4){\makebox[0pt]{\scriptsize $q$}}
\put(31,3){\makebox[0pt]{\scriptsize $q^{-1}$}}
\put(37,4){\makebox[0pt]{\scriptsize $-1$}}
\put(43,3){\makebox[0pt]{\scriptsize $r^{-1}$}}
\put(49,4){\makebox[0pt]{\scriptsize $r$}}
\end{picture}
\quad
\Rightarrow \quad
r_4(X)\colon~ 
\begin{picture}(36,10)(0,9)
\put(1,9){\circle{2}}
\put(2,9){\line(1,0){10}}
\put(13,9){\circle{2}}
\put(14,9){\line(1,0){10}}
\put(25,9){\circle*{2}}
\put(26,9){\line(1,1){7}}
\put(26,9){\line(1,-1){7}}
\put(33,17){\circle{2}}
\put(33,1){\circle{2}}
\put(33,2){\line(0,1){14}}
\put(1,12){\makebox[0pt]{\scriptsize $q$}}
\put(7,11){\makebox[0pt]{\scriptsize $q^{-1}$}}
\put(13,12){\makebox[0pt]{\scriptsize $q$}}
\put(19,11){\makebox[0pt]{\scriptsize $q^{-1}$}}
\put(23.5,12){\makebox[0pt]{\scriptsize $-1$}}
\put(30,14){\makebox[0pt][r]{\scriptsize $q$}}
\put(29,4){\makebox[0pt][r]{\scriptsize $(qr)^{-1}$}}
\put(35,16){\makebox[0pt][l]{\scriptsize $-1$}}
\put(34,9){\makebox[0pt][l]{\scriptsize $r$}}
\put(35,1){\makebox[0pt][l]{\scriptsize $-1$}}
\end{picture}
\end{align*}
\begin{align*}
 \Rightarrow \quad 
r_3r_4(X)\colon~ \tau_{12534}~
\begin{picture}(38,11)(0,3)
\put(1,1){\circle{2}}
\put(2,1){\line(1,0){10}}
\put(13,1){\circle{2}}
\put(13,2){\line(2,3){6}}
\put(14,1){\line(1,0){10}}
\put(25,1){\circle{2}}
\put(25,2){\line(-2,3){6}}
\put(19,12){\circle{2}}
\put(26,1){\line(1,0){10}}
\put(37,1){\circle{2}}
\put(1,4){\makebox[0pt]{\scriptsize $q$}}
\put(7,3){\makebox[0pt]{\scriptsize $q^{-1}$}}
\put(12,4){\makebox[0pt]{\scriptsize $-1$}}
\put(19,3){\makebox[0pt]{\scriptsize $q$}}
\put(26,4){\makebox[0pt]{\scriptsize $-1$}}
\put(31,3){\makebox[0pt]{\scriptsize $q^{-1}$}}
\put(37,4){\makebox[0pt]{\scriptsize $q$}}
\put(13.3,8){\makebox[0pt]{\scriptsize $q^{-2}r^{-1}$}}
\put(23,8){\makebox[0pt][l]{\scriptsize $qr$}}
\put(24,14){\makebox[0pt]{\scriptsize $(qr)^{-1}$}}
\end{picture}
\end{align*}
and $a_{52}^{r_3r_4(X)}=-1$ yield $qr=-1$ or $q^3r^2=1$.

Combining all constraints forces $qr=-1$ with $p\neq2$. Substituting into $a_{15}^{r_2r_3r_4(X)}=0$ gives $-q^{-2}=1$, hence $q\in G_4'$ and $\mathcal{D}=\mathcal{D}_{14,4}^5$, $p\neq2$.

Step $2$.
Assume exactly two entries satisfy $q_{ii}=-1$. We divide the analysis into seven subcases.

Step $2.1$.
Set $q_{44}=q_{55}=-1$ and suppose
\[
q_{ii}\widetilde q_{i,i+1}=q_{jj}\widetilde q_{j-1,j}=1,\quad \forall\,i\in\{1,2,3\},\;j\in\{2,3\}.
\]
This yields $(a,b)=(0,0)$. By Definition \ref{defA5}, only two possibilities occur:
\[
(a,b,c,d,e)=(0,0,1,1,1)\quad\text{or}\quad (a,b,c,d,e)=(0,0,1,1,2).
\]
Set $q:=q_{11}$ and $r^{-1}:=\widetilde q_{45}$.

When $(a,b,c,d,e)=(0,0,1,1,1)$, we have $A^{r_3(X)}\cong A_5$,
\[
A^{r_4(X)}=\begin{pmatrix}
2&-1&0&0&0\\
-1&2&-1&0&0\\
0&-1&2&-1&-1\\
0&0&-1&2&-1\\
0&0&-1&-1&2
\end{pmatrix},\qquad A^{r_5(X)}= A_5.
\]
Here $qr\neq 1$. From the reflection
\begin{align*}
X\colon~ 
\begin{picture}(50,4)(0,3)
\put(1,1){\circle{2}}
\put(2,1){\line(1,0){10}}
\put(13,1){\circle{2}}
\put(14,1){\line(1,0){10}}
\put(25,1){\circle{2}}
\put(26,1){\line(1,0){10}}
\put(37,1){\circle{2}}
\put(38,1){\line(1,0){10}}
\put(49,1){\circle*{2}}
\put(1,4){\makebox[0pt]{\scriptsize $q$}}
\put(7,3){\makebox[0pt]{\scriptsize $q^{-1}$}}
\put(13,4){\makebox[0pt]{\scriptsize $q$}}
\put(19,3){\makebox[0pt]{\scriptsize $q^{-1}$}}
\put(25,4){\makebox[0pt]{\scriptsize $q$}}
\put(31,3){\makebox[0pt]{\scriptsize $q^{-1}$}}
\put(37,4){\makebox[0pt]{\scriptsize $-1$}}
\put(43,3){\makebox[0pt]{\scriptsize $r^{-1}$}}
\put(49,4){\makebox[0pt]{\scriptsize $-1$}}
\end{picture}
\quad \Rightarrow \quad
r_5(X)\colon~ \begin{picture}(50,4)(0,3)
\put(1,1){\circle{2}}
\put(2,1){\line(1,0){10}}
\put(13,1){\circle{2}}
\put(14,1){\line(1,0){10}}
\put(25,1){\circle{2}}
\put(26,1){\line(1,0){10}}
\put(37,1){\circle{2}}
\put(38,1){\line(1,0){10}}
\put(49,1){\circle{2}}
\put(1,4){\makebox[0pt]{\scriptsize $q$}}
\put(7,3){\makebox[0pt]{\scriptsize $q^{-1}$}}
\put(13,4){\makebox[0pt]{\scriptsize $q$}}
\put(19,3){\makebox[0pt]{\scriptsize $q^{-1}$}}
\put(26,4){\makebox[0pt]{\scriptsize $q$}}
\put(31,3){\makebox[0pt]{\scriptsize $q^{-1}$}}
\put(37,4){\makebox[0pt]{\scriptsize $r^{-1}$}}
\put(43,3){\makebox[0pt]{\scriptsize $r$}}
\put(49,4){\makebox[0pt]{\scriptsize $-1$}}
\end{picture}
\end{align*}
the condition $a_{43}^{r_5(X)}=-1$ forces $r=-1$ and $p\neq 2$.

Next iterate reflections of $X$
\begin{align*}
X\colon~ 
\begin{picture}(50,4)(0,3)
\put(1,1){\circle{2}}
\put(2,1){\line(1,0){10}}
\put(13,1){\circle{2}}
\put(14,1){\line(1,0){10}}
\put(25,1){\circle{2}}
\put(26,1){\line(1,0){10}}
\put(37,1){\circle*{2}}
\put(38,1){\line(1,0){10}}
\put(49,1){\circle{2}}
\put(1,4){\makebox[0pt]{\scriptsize $q$}}
\put(7,3){\makebox[0pt]{\scriptsize $q^{-1}$}}
\put(13,4){\makebox[0pt]{\scriptsize $q$}}
\put(19,3){\makebox[0pt]{\scriptsize $q^{-1}$}}
\put(25,4){\makebox[0pt]{\scriptsize $q$}}
\put(31,3){\makebox[0pt]{\scriptsize $q^{-1}$}}
\put(37,4){\makebox[0pt]{\scriptsize $-1$}}
\put(43,3){\makebox[0pt]{\scriptsize $-1$}}
\put(49,4){\makebox[0pt]{\scriptsize $-1$}}
\end{picture}
\quad \Rightarrow \quad
r_4(X)\colon~ 
\begin{picture}(36,10)(0,9)
\put(1,9){\circle{2}}
\put(2,9){\line(1,0){10}}
\put(13,9){\circle{2}}
\put(14,9){\line(1,0){10}}
\put(25,9){\circle{2}}
\put(26,9){\line(1,1){7}}
\put(26,9){\line(1,-1){7}}
\put(33,17){\circle{2}}
\put(33,1){\circle{2}}
\put(33,2){\line(0,1){14}}
\put(1,12){\makebox[0pt]{\scriptsize $q$}}
\put(7,11){\makebox[0pt]{\scriptsize $q^{-1}$}}
\put(13,12){\makebox[0pt]{\scriptsize $q$}}
\put(19,11){\makebox[0pt]{\scriptsize $q^{-1}$}}
\put(24,12){\makebox[0pt]{\scriptsize $-1$}}
\put(30,14){\makebox[0pt][r]{\scriptsize $q$}}
\put(29,4){\makebox[0pt][r]{\scriptsize $-q^{-1}$}}
\put(35,16){\makebox[0pt][l]{\scriptsize $-1$}}
\put(34,9){\makebox[0pt][l]{\scriptsize $-1$}}
\put(35,1){\makebox[0pt][l]{\scriptsize $-1$}}
\end{picture}
\end{align*}
\begin{align*}
\Rightarrow \quad
r_5r_4(X)\colon~ \tau_{12354}~ 
\begin{picture}(50,4)(0,3)
\put(1,1){\circle{2}}
\put(2,1){\line(1,0){10}}
\put(13,1){\circle{2}}
\put(14,1){\line(1,0){10}}
\put(25,1){\circle{2}}
\put(26,1){\line(1,0){10}}
\put(37,1){\circle{2}}
\put(38,1){\line(1,0){10}}
\put(49,1){\circle{2}}
\put(1,4){\makebox[0pt]{\scriptsize $q$}}
\put(7,3){\makebox[0pt]{\scriptsize $q^{-1}$}}
\put(13,4){\makebox[0pt]{\scriptsize $q$}}
\put(19,3){\makebox[0pt]{\scriptsize $q^{-1}$}}
\put(26,4){\makebox[0pt]{\scriptsize $-q^{-1}$}}
\put(31,3){\makebox[0pt]{\scriptsize $-q$}}
\put(37,4){\makebox[0pt]{\scriptsize $-1$}}
\put(43,3){\makebox[0pt]{\scriptsize $-1$}}
\put(49,4){\makebox[0pt]{\scriptsize $-1$}}
\end{picture}
\end{align*}
together with $a_{32}^{r_5r_4(X)}=-1$ force $q^2=-1$. Hece $q\in G_4'$. A further reflection
\begin{align*}
r_4(X)\colon~ \Dchainfivew{$q$}{$q^{-1}$}{$q$}{$q^{-1}$}{$-1$}{$q$}{$-q^{-1}$}{$-1$}{$-1$}{$-1$}
\quad \Rightarrow \quad
r_3r_4(X)\colon~ \tau_{12534}~
\begin{picture}(38,11)(0,3)
\put(1,1){\circle{2}}
\put(2,1){\line(1,0){10}}
\put(13,1){\circle{2}}
\put(13,2){\line(2,3){6}}
\put(14,1){\line(1,0){10}}
\put(25,1){\circle{2}}
\put(25,2){\line(-2,3){6}}
\put(19,12){\circle{2}}
\put(26,1){\line(1,0){10}}
\put(37,1){\circle{2}}
\put(1,4){\makebox[0pt]{\scriptsize $q$}}
\put(7,3){\makebox[0pt]{\scriptsize $q^{-1}$}}
\put(12,4){\makebox[0pt]{\scriptsize $-1$}}
\put(19,3){\makebox[0pt]{\scriptsize $q$}}
\put(26,4){\makebox[0pt]{\scriptsize $-1$}}
\put(31,3){\makebox[0pt]{\scriptsize $q^{-1}$}}
\put(37,4){\makebox[0pt]{\scriptsize $q$}}
\put(13.3,8){\makebox[0pt]{\scriptsize $-q^{-2}$}}
\put(23,8){\makebox[0pt][l]{\scriptsize $-q$}}
\put(24,14){\makebox[0pt]{\scriptsize $-q^{-1}$}}
\end{picture}
\end{align*}
together with $a_{52}^{r_3r_4(X)}=-1$ yields $q^3=1$ and $p\neq3$, i.e.\ $q\in G_3'$. The inclusions $q\in G_4'\cap G_3'$ are contradictory, so this subcase is impossible.

When $(a,b,c,d,e)=(0,0,1,1,2)$, we have
\[
A^{r_3(X)}\cong A_5,\quad
A^{r_4(X)}=\begin{pmatrix}
2&-1&0&0&0\\
-1&2&-1&0&0\\
0&-1&2&-1&-1\\
0&0&-1&2&-1\\
0&0&-1&-1&2
\end{pmatrix},\quad
A^{r_5(X)}=\begin{pmatrix}
2&-1&0&0&0\\
-1&2&-1&0&0\\
0&-1&2&-1&0\\
0&0&-2&2&-1\\
0&0&0&-1&2
\end{pmatrix}.
\]
Then $qr\neq1$. From
\begin{align*}
X\colon~ 
\begin{picture}(50,4)(0,3)
\put(1,1){\circle{2}}
\put(2,1){\line(1,0){10}}
\put(13,1){\circle{2}}
\put(14,1){\line(1,0){10}}
\put(25,1){\circle{2}}
\put(26,1){\line(1,0){10}}
\put(37,1){\circle*{2}}
\put(38,1){\line(1,0){10}}
\put(49,1){\circle{2}}
\put(1,4){\makebox[0pt]{\scriptsize $q$}}
\put(7,3){\makebox[0pt]{\scriptsize $q^{-1}$}}
\put(13,4){\makebox[0pt]{\scriptsize $q$}}
\put(19,3){\makebox[0pt]{\scriptsize $q^{-1}$}}
\put(25,4){\makebox[0pt]{\scriptsize $q$}}
\put(31,3){\makebox[0pt]{\scriptsize $q^{-1}$}}
\put(37,4){\makebox[0pt]{\scriptsize $-1$}}
\put(43,3){\makebox[0pt]{\scriptsize $r^{-1}$}}
\put(49,4){\makebox[0pt]{\scriptsize $-1$}}
\end{picture}
\quad \Rightarrow \quad
r_4(X)\colon~ \Dchainfivet{$q$}{$q^{-1}$}{$q$}{$q^{-1}$}{$-1$}{$q$}{$(qr)^{-1}$}{$-1$}{$r$}{$r^{-1}$}
\end{align*}
the entry $a_{53}^{r_4(X)}=-1$ gives the alternative $r=-1$ or $qr^2=1$.

If $r=-1$, inspect the reflection
\begin{align*}
X\colon~ 
\begin{picture}(50,4)(0,3)
\put(1,1){\circle{2}}
\put(2,1){\line(1,0){10}}
\put(13,1){\circle{2}}
\put(14,1){\line(1,0){10}}
\put(25,1){\circle{2}}
\put(26,1){\line(1,0){10}}
\put(37,1){\circle{2}}
\put(38,1){\line(1,0){10}}
\put(49,1){\circle*{2}}
\put(1,4){\makebox[0pt]{\scriptsize $q$}}
\put(7,3){\makebox[0pt]{\scriptsize $q^{-1}$}}
\put(13,4){\makebox[0pt]{\scriptsize $q$}}
\put(19,3){\makebox[0pt]{\scriptsize $q^{-1}$}}
\put(25,4){\makebox[0pt]{\scriptsize $q$}}
\put(31,3){\makebox[0pt]{\scriptsize $q^{-1}$}}
\put(37,4){\makebox[0pt]{\scriptsize $-1$}}
\put(43,3){\makebox[0pt]{\scriptsize $-1$}}
\put(49,4){\makebox[0pt]{\scriptsize $-1$}}
\end{picture}
\quad \Rightarrow \quad
r_5(X)\colon~ 
\begin{picture}(50,4)(0,3)
\put(1,1){\circle{2}}
\put(2,1){\line(1,0){10}}
\put(13,1){\circle{2}}
\put(14,1){\line(1,0){10}}
\put(25,1){\circle{2}}
\put(26,1){\line(1,0){10}}
\put(37,1){\circle{2}}
\put(38,1){\line(1,0){10}}
\put(49,1){\circle{2}}
\put(1,4){\makebox[0pt]{\scriptsize $q$}}
\put(7,3){\makebox[0pt]{\scriptsize $q^{-1}$}}
\put(13,4){\makebox[0pt]{\scriptsize $q$}}
\put(19,3){\makebox[0pt]{\scriptsize $q^{-1}$}}
\put(26,4){\makebox[0pt]{\scriptsize $q$}}
\put(31,3){\makebox[0pt]{\scriptsize $q^{-1}$}}
\put(37,4){\makebox[0pt]{\scriptsize $-1$}}
\put(43,3){\makebox[0pt]{\scriptsize $-1$}}
\put(49,4){\makebox[0pt]{\scriptsize $-1$}}
\end{picture}
\end{align*}
which directly contradicts $e=2$.

If $qr^2=1$, Definition \ref{defA5} (condition $(A_{5_2})$) splits on $a_{25}^{r_3r_4(X)}\in\{0,-1\}$.
When $a_{25}^{r_3r_4(X)}=0$, we get $r^3=1$, so $q=r\in G_3'$ with $p\neq3$, and $\mathcal{D}=\mathcal{D}_{12,2}^5$.
When $a_{25}^{r_3r_4(X)}=-1$, substitute $q=r^{-2}$ into the reflection
\begin{align*}
r_4(X)\colon~ \Dchainfivew{$r^{-2}$}{$r^2$}{$r^{-2}$}{$r^2$}{$-1$}{$r^{-2}$}{$r$}{$-1$}{$r$}{$r^{-1}$}
\quad \Rightarrow \quad
r_3r_4(X)\colon~ \tau_{12534}~~ 
\begin{picture}(38,11)(0,3)
\put(1,1){\circle{2}}
\put(2,1){\line(1,0){10}}
\put(13,1){\circle{2}}
\put(13,2){\line(2,3){6}}
\put(14,1){\line(1,0){10}}
\put(25,1){\circle{2}}
\put(25,2){\line(-2,3){6}}
\put(19,12){\circle{2}}
\put(26,1){\line(1,0){10}}
\put(37,1){\circle{2}}
\put(1,4){\makebox[0pt]{\scriptsize $r^{-2}$}}
\put(7,3){\makebox[0pt]{\scriptsize $r^2$}}
\put(12,4){\makebox[0pt]{\scriptsize $-1$}}
\put(19,3){\makebox[0pt]{\scriptsize $r^{-2}$}}
\put(26,4){\makebox[0pt]{\scriptsize $-1$}}
\put(31,3){\makebox[0pt]{\scriptsize $r^2$}}
\put(37,4){\makebox[0pt]{\scriptsize $r^{-2}$}}
\put(13.3,8){\makebox[0pt]{\scriptsize $r^3$}}
\put(23,8){\makebox[0pt][l]{\scriptsize $r^{-1}$}}
\put(24,14){\makebox[0pt]{\scriptsize $-1$}}
\end{picture}
\end{align*}
and deduce $r\in G_5'$, hence $\mathcal{D}=\mathcal{D}_{15,6}^5$ with $p\neq5$.

Step $2.2$. Let $q_{11}=q_{33}=-1$ and $q_{ii}\widetilde{q_{i,i+1}}=q_{jj}\widetilde{q_{j-1,j}}=1$ for all $i\in \{2,4\}$ and $j\in \{2,4,5\}$. Then $(c,d,e)=(0,0,1)$. Hence $(a,b,c,d,e)=(1,1,0,0,1)$. Set $q\colon=q_{22}$ and $r\colon=q_{44}$. We obtain $qr\ne 1$ by $a^{r_3(X)}_{24}=-1$. The reflections of $X$
\begin{align*}
X\colon~
\begin{picture}(50,4)(0,3)
\put(1,1){\circle{2}}
\put(2,1){\line(1,0){10}}
\put(13,1){\circle{2}}
\put(14,1){\line(1,0){10}}
\put(25,1){\circle*{2}}
\put(26,1){\line(1,0){10}}
\put(37,1){\circle{2}}
\put(38,1){\line(1,0){10}}
\put(49,1){\circle{2}}
\put(1,4){\makebox[0pt]{\scriptsize $-1$}}
\put(7,3){\makebox[0pt]{\scriptsize $q^{-1}$}}
\put(13,4){\makebox[0pt]{\scriptsize $q$}}
\put(19,3){\makebox[0pt]{\scriptsize $q^{-1}$}}
\put(25,4){\makebox[0pt]{\scriptsize $-1$}}
\put(31,3){\makebox[0pt]{\scriptsize $r^{-1}$}}
\put(37,4){\makebox[0pt]{\scriptsize $r$}}
\put(43,3){\makebox[0pt]{\scriptsize $r^{-1}$}}
\put(49,4){\makebox[0pt]{\scriptsize $r$}}
\end{picture}
%\Dchainfivec{$-1$}{$q^{-1}$}{$q$}{$q^{-1}$}{$-1$}{$r^{-1}$}{$r$}{$r^{-1}$}{$r$}
\quad \Rightarrow \quad
r_3(X)\colon~ 
\begin{picture}(38,11)(0,3)
\put(1,1){\circle{2}}
\put(2,1){\line(1,0){10}}
\put(13,1){\circle{2}}
\put(13,2){\line(2,3){6}}
\put(14,1){\line(1,0){10}}
\put(25,1){\circle{2}}
\put(25,2){\line(-2,3){6}}
\put(19,12){\circle{2}}
\put(26,1){\line(1,0){10}}
\put(37,1){\circle{2}}
\put(1,4){\makebox[0pt]{\scriptsize $-1$}}
\put(7,3){\makebox[0pt]{\scriptsize $q^{-1}$}}
\put(12,4){\makebox[0pt]{\scriptsize $-1$}}
\put(19,3){\makebox[0pt]{\scriptsize $(qr)^{-1}$}}
\put(26,4){\makebox[0pt]{\scriptsize $-1$}}
\put(31,3){\makebox[0pt]{\scriptsize $r^{-1}$}}
\put(37,4){\makebox[0pt]{\scriptsize $r$}}
\put(13.3,8){\makebox[0pt]{\scriptsize $q$}}
\put(23,8){\makebox[0pt][l]{\scriptsize $r$}}
\put(24,14){\makebox[0pt]{\scriptsize $-1$}}
\end{picture}
%\Dchainfivex{$-1$}{$q^{-1}$}{$-1$}{$(qr)^{-1}$}{$-1$}{$r^{-1}$}{$r$}{$q$}{$r$}{$-1$}
\end{align*}
\begin{align*}
 \Rightarrow \quad 
r_2r_3(X)\colon~ 
\begin{picture}(38,11)(0,3)
\put(1,1){\circle{2}}
\put(2,1){\line(1,0){10}}
\put(13,1){\circle{2}}
\put(13,2){\line(2,3){6}}
\put(14,1){\line(1,0){10}}
\put(25,1){\circle{2}}
\put(25,2){\line(-2,3){6}}
\put(19,12){\circle{2}}
\put(26,1){\line(1,0){10}}
\put(37,1){\circle{2}}
\put(1,4){\makebox[0pt]{\scriptsize $q$}}
\put(7,3){\makebox[0pt]{\scriptsize $q^{-1}$}}
\put(12,4){\makebox[0pt]{\scriptsize $-1$}}
\put(19,3){\makebox[0pt]{\scriptsize $qr$}}
\put(26,4){\makebox[0pt]{\scriptsize $(qr)^{-1}$}}
\put(31,3){\makebox[0pt]{\scriptsize $r^{-1}$}}
\put(37,4){\makebox[0pt]{\scriptsize $r$}}
\put(13.3,8){\makebox[0pt]{\scriptsize $q$}}
\put(23,8){\makebox[0pt][l]{\scriptsize $q^{-2}r^{-1}$}}
\put(24,14){\makebox[0pt]{\scriptsize $q^{-1}$}}
\end{picture}
%\Dchainfivex{$q$}{$q^{-1}$}{$-1$}{$qr$}{$(qr)^{-1}$}{$r^{-1}$}{$r$}{$q$}{$q^{-2}r^{-1}$}{$q^{-1}$}
\end{align*}
imply $q^2r=1$ by $a_{41}^{r_2r_3(X)}=0$. And $a_{45}^{r_2r_3(X)}=-1$ gives $qr^2=1$ or $qr=-1$. Then $q=r\in G'_3$ and $p\ne 3$. Hence $\cD=\cD_{12,13}^5$. 

Step $2.3$. Let $q_{33}=q_{55}=-1$ and suppose $q_{ii}\widetilde{q_{i,i+1}}=q_{jj}\widetilde{q_{j-1,j}}=1$ for all $i\in \{1,2,4\}$ and $j\in \{2,4\}$. Then $(c,d)=(0,0)$. Hence $(a,b,c,d,e)=(1,1,0,0,1)$. 
This case is  symmetric to Step $2.2$. Hence $\cD=\tau_{54321}\cD_{12,13}^5$.

Step $2.4$. Let $q_{11}=q_{44}=-1$ and $q_{ii}\widetilde{q_{i,i+1}}=q_{jj}\widetilde{q_{j-1,j}}=1$ for all $i\in \{2,3\}$ and $j\in \{2,3,5\}$. Then $(a,b,e)=(0,0,1)$. By Definition \ref{defA5} one obtains $(a,b,c,d,e)=(0,0,1,1,1)$. Set $q\colon=q_{33}$ and $r\colon=q_{55}$. Lemma \ref{jslemma} gives $qr\ne 1$. On one hand, the reflections of $X$
\begin{align*}
X\colon~ 
\begin{picture}(50,4)(0,3)
\put(1,1){\circle{2}}
\put(2,1){\line(1,0){10}}
\put(13,1){\circle{2}}
\put(14,1){\line(1,0){10}}
\put(25,1){\circle{2}}
\put(26,1){\line(1,0){10}}
\put(37,1){\circle*{2}}
\put(38,1){\line(1,0){10}}
\put(49,1){\circle{2}}
\put(1,4){\makebox[0pt]{\scriptsize $-1$}}
\put(7,3){\makebox[0pt]{\scriptsize $q^{_1}$}}
\put(13,4){\makebox[0pt]{\scriptsize $q$}}
\put(19,3){\makebox[0pt]{\scriptsize $q^{-1}$}}
\put(25,4){\makebox[0pt]{\scriptsize $q$}}
\put(31,3){\makebox[0pt]{\scriptsize $q^{-1}$}}
\put(37,4){\makebox[0pt]{\scriptsize $-1$}}
\put(43,3){\makebox[0pt]{\scriptsize $r^{-1}$}}
\put(49,4){\makebox[0pt]{\scriptsize $r$}}
\end{picture}
%\Dchainfived{$-1$}{$q^{-1}$}{$q$}{$q^{-1}$}{$q$}{$q^{-1}$}{$-1$}{$r^{-1}$}{$r$}
\quad \Rightarrow \quad
r_4(X)\colon~ 
\begin{picture}(36,10)(0,9)
\put(1,9){\circle{2}}
\put(2,9){\line(1,0){10}}
\put(13,9){\circle{2}}
\put(14,9){\line(1,0){10}}
\put(25,9){\circle*{2}}
\put(26,9){\line(1,1){7}}
\put(26,9){\line(1,-1){7}}
\put(33,17){\circle{2}}
\put(33,1){\circle{2}}
\put(33,2){\line(0,1){14}}
\put(1,12){\makebox[0pt]{\scriptsize $-1$}}
\put(7,11){\makebox[0pt]{\scriptsize $q^{-1}$}}
\put(13,12){\makebox[0pt]{\scriptsize $q$}}
\put(19,11){\makebox[0pt]{\scriptsize $q^{-1}$}}
\put(23.5,12){\makebox[0pt]{\scriptsize $-1$}}
\put(30,14){\makebox[0pt][r]{\scriptsize $q$}}
\put(29,4){\makebox[0pt][r]{\scriptsize $(qr)^{-1}$}}
\put(35,16){\makebox[0pt][l]{\scriptsize $-1$}}
\put(34,9){\makebox[0pt][l]{\scriptsize $r$}}
\put(35,1){\makebox[0pt][l]{\scriptsize $-1$}}
\end{picture}
%\Dchainfivew{$-1$}{$q^{-1}$}{$q$}{$q^{-1}$}{$-1$}{$q$}{$(qr)^{-1}$}{$-1$}{$r$}{$-1$}
\end{align*}
\vspace{-3mm}
\begin{align}\label{eq-a7}
 \Rightarrow \quad 
r_3r_4(X)\colon~ \tau_{12534}~
\begin{picture}(38,11)(0,3)
\put(1,1){\circle{2}}
\put(2,1){\line(1,0){10}}
\put(13,1){\circle{2}}
\put(13,2){\line(2,3){6}}
\put(14,1){\line(1,0){10}}
\put(25,1){\circle{2}}
\put(25,2){\line(-2,3){6}}
\put(19,12){\circle{2}}
\put(26,1){\line(1,0){10}}
\put(37,1){\circle{2}}
\put(1,4){\makebox[0pt]{\scriptsize $-1$}}
\put(7,3){\makebox[0pt]{\scriptsize $q^{-1}$}}
\put(12,4){\makebox[0pt]{\scriptsize $-1$}}
\put(19,3){\makebox[0pt]{\scriptsize $q$}}
\put(26,4){\makebox[0pt]{\scriptsize $-1$}}
\put(31,3){\makebox[0pt]{\scriptsize $q^{-1}$}}
\put(37,4){\makebox[0pt]{\scriptsize $q$}}
\put(13.3,8){\makebox[0pt]{\scriptsize $q^{-2}r^{-1}$}}
\put(23,8){\makebox[0pt][l]{\scriptsize $qr$}}
\put(24,14){\makebox[0pt]{\scriptsize $(qr)^{-1}$}}
\end{picture}
%\Dchainfivex{$-1$}{$q^{-1}$}{$-1$}{$q$}{$-1$}{$q^{-1}$}{$q$}{$q^{-2}r^{-1}$}{$qr$}{$(qr)^{-1}$}
\end{align}
imply $qr=-1$ or $q^3r^2=1$ by $a_{52}^{r_3r_4(X)}=-1$. And we have $q^{2}=-1$ by $a_{15}^{r_2r_3r_4(X)}=0$. Hence $q\in G'_4$. On the other hand, the reflections of $X$
\begin{align*}
X\colon~ 
\begin{picture}(50,4)(0,3)
\put(1,1){\circle{2}}
\put(2,1){\line(1,0){10}}
\put(13,1){\circle{2}}
\put(14,1){\line(1,0){10}}
\put(25,1){\circle{2}}
\put(26,1){\line(1,0){10}}
\put(37,1){\circle*{2}}
\put(38,1){\line(1,0){10}}
\put(49,1){\circle{2}}
\put(1,4){\makebox[0pt]{\scriptsize $-1$}}
\put(7,3){\makebox[0pt]{\scriptsize $q^{-1}$}}
\put(13,4){\makebox[0pt]{\scriptsize $q$}}
\put(19,3){\makebox[0pt]{\scriptsize $q^{-1}$}}
\put(25,4){\makebox[0pt]{\scriptsize $q$}}
\put(31,3){\makebox[0pt]{\scriptsize $q^{-1}$}}
\put(37,4){\makebox[0pt]{\scriptsize $-1$}}
\put(43,3){\makebox[0pt]{\scriptsize $r^{-1}$}}
\put(49,4){\makebox[0pt]{\scriptsize $r$}}
\end{picture}
%\Dchainfived{$-1$}{$q^{-1}$}{$q$}{$q^{-1}$}{$q$}{$q^{-1}$}{$-1$}{$r^{-1}$}{$r$}
\quad \Rightarrow \quad
r_4(X)\colon~ 
\begin{picture}(36,10)(0,9)
\put(1,9){\circle{2}}
\put(2,9){\line(1,0){10}}
\put(13,9){\circle{2}}
\put(14,9){\line(1,0){10}}
\put(25,9){\circle{2}}
\put(26,9){\line(1,1){7}}
\put(26,9){\line(1,-1){7}}
\put(33,17){\circle{2}}
\put(33,1){\circle*{2}}
\put(33,2){\line(0,1){14}}
\put(1,12){\makebox[0pt]{\scriptsize $-1$}}
\put(7,11){\makebox[0pt]{\scriptsize $q^{-1}$}}
\put(13,12){\makebox[0pt]{\scriptsize $q$}}
\put(19,11){\makebox[0pt]{\scriptsize $q^{-1}$}}
\put(24,12){\makebox[0pt]{\scriptsize $-1$}}
\put(30,14){\makebox[0pt][r]{\scriptsize $q$}}
\put(29,4){\makebox[0pt][r]{\scriptsize $(qr)^{-1}$}}
\put(35,16){\makebox[0pt][l]{\scriptsize $-1$}}
\put(34,9){\makebox[0pt][l]{\scriptsize $r$}}
\put(35,1){\makebox[0pt][l]{\scriptsize $-1$}}
\end{picture}
%\Dchainfiveu{$-1$}{$q^{-1}$}{$q$}{$q^{-1}$}{$-1$}{$q$}{$(qr)^{-1}$}{$-1$}{$r$}{$-1$}
\end{align*}
\vspace{-4mm}
\begin{align*}
 \Rightarrow \quad 
r_5r_4(X)\colon~ \tau_{12354}~ 
\begin{picture}(50,4)(0,3)
\put(1,1){\circle{2}}
\put(2,1){\line(1,0){10}}
\put(13,1){\circle{2}}
\put(14,1){\line(1,0){10}}
\put(25,1){\circle{2}}
\put(26,1){\line(1,0){10}}
\put(37,1){\circle{2}}
\put(38,1){\line(1,0){10}}
\put(49,1){\circle{2}}
\put(1,4){\makebox[0pt]{\scriptsize $-1$}}
\put(7,3){\makebox[0pt]{\scriptsize $q^{-1}$}}
\put(13,4){\makebox[0pt]{\scriptsize $q$}}
\put(19,3){\makebox[0pt]{\scriptsize $q^{-1}$}}
\put(26,4){\makebox[0pt]{\scriptsize $(qr)^{-1}$}}
\put(31,3){\makebox[0pt]{\scriptsize $qr$}}
\put(37,4){\makebox[0pt]{\scriptsize $-1$}}
\put(43,3){\makebox[0pt]{\scriptsize $r^{-1}$}}
\put(49,4){\makebox[0pt]{\scriptsize $r$}}
\end{picture}
%\Dchainfive{$-1$}{$q^{-1}$}{$q$}{$q^{-1}$}{$(qr)^{-1}$}{$qr$}{$-1$}{$r^{-1}$}{$r$}
\end{align*}
imply $qr=-1$ or $q^2r=1$ by $a_{32}^{r_5r_4(X)}=-1$. Hence one has to have $qr=-1$ and $p\ne 2$. Then from the reflections of $X$
\begin{align*}
X\colon~ 
\begin{picture}(50,4)(0,3)
\put(1,1){\circle*{2}}
\put(2,1){\line(1,0){10}}
\put(13,1){\circle{2}}
\put(14,1){\line(1,0){10}}
\put(25,1){\circle{2}}
\put(26,1){\line(1,0){10}}
\put(37,1){\circle{2}}
\put(38,1){\line(1,0){10}}
\put(49,1){\circle{2}}
\put(1,4){\makebox[0pt]{\scriptsize $-1$}}
\put(7,3){\makebox[0pt]{\scriptsize $q^{-1}$}}
\put(13,4){\makebox[0pt]{\scriptsize $q$}}
\put(19,3){\makebox[0pt]{\scriptsize $q^{-1}$}}
\put(25,4){\makebox[0pt]{\scriptsize $q$}}
\put(31,3){\makebox[0pt]{\scriptsize $q^{-1}$}}
\put(37,4){\makebox[0pt]{\scriptsize $-1$}}
\put(43,3){\makebox[0pt]{\scriptsize $-q$}}
\put(49,4){\makebox[0pt]{\scriptsize $-q^{-1}$}}
\end{picture}
%\Dchainfivea{$-1$}{$q^{-1}$}{$q$}{$q^{-1}$}{$q$}{$q^{-1}$}{$-1$}{$-q$}{$-q^{-1}$}
\quad \Rightarrow \quad
r_1(X)\colon~ 
\begin{picture}(50,4)(0,3)
\put(1,1){\circle{2}}
\put(2,1){\line(1,0){10}}
\put(13,1){\circle*{2}}
\put(14,1){\line(1,0){10}}
\put(25,1){\circle{2}}
\put(26,1){\line(1,0){10}}
\put(37,1){\circle{2}}
\put(38,1){\line(1,0){10}}
\put(49,1){\circle{2}}
\put(1,4){\makebox[0pt]{\scriptsize $-1$}}
\put(7,3){\makebox[0pt]{\scriptsize $q$}}
\put(13,4){\makebox[0pt]{\scriptsize $-1$}}
\put(19,3){\makebox[0pt]{\scriptsize $q^{-1}$}}
\put(25,4){\makebox[0pt]{\scriptsize $q$}}
\put(31,3){\makebox[0pt]{\scriptsize $q^{-1}$}}
\put(37,4){\makebox[0pt]{\scriptsize $-1$}}
\put(43,3){\makebox[0pt]{\scriptsize $-q$}}
\put(49,4){\makebox[0pt]{\scriptsize $-q^{-1}$}}
\end{picture}
%\Dchainfiveb{$-1$}{$q$}{$-1$}{$q^{-1}$}{$q$}{$q^{-1}$}{$-1$}{$-q$}{$-q^{-1}$}
\end{align*}
\vspace{-3mm}
\begin{align*}
 \Rightarrow
r_2r_1(X)\colon~ 
\begin{picture}(50,4)(0,3)
\put(1,1){\circle{2}}
\put(2,1){\line(1,0){10}}
\put(13,1){\circle{2}}
\put(14,1){\line(1,0){10}}
\put(25,1){\circle*{2}}
\put(26,1){\line(1,0){10}}
\put(37,1){\circle{2}}
\put(38,1){\line(1,0){10}}
\put(49,1){\circle{2}}
\put(1,4){\makebox[0pt]{\scriptsize $q$}}
\put(7,3){\makebox[0pt]{\scriptsize $q^{-1}$}}
\put(13,4){\makebox[0pt]{\scriptsize $-1$}}
\put(19,3){\makebox[0pt]{\scriptsize $q$}}
\put(25,4){\makebox[0pt]{\scriptsize $-1$}}
\put(31,3){\makebox[0pt]{\scriptsize $q^{-1}$}}
\put(37,4){\makebox[0pt]{\scriptsize $-1$}}
\put(43,3){\makebox[0pt]{\scriptsize $-q$}}
\put(49,4){\makebox[0pt]{\scriptsize $-q^{-1}$}}
\end{picture}
%\Dchainfivec{$q$}{$q^{-1}$}{$-1$}{$q$}{${-1}$}{$q^{-1}$}{$-1$}{$-q$}{$-q^{-1}$}
%\end{align*}
%\begin{align*}
\quad \Rightarrow 
r_3r_2r_1(X)\colon~ 
\begin{picture}(50,4)(0,3)
\put(1,1){\circle{2}}
\put(2,1){\line(1,0){10}}
\put(13,1){\circle{2}}
\put(14,1){\line(1,0){10}}
\put(25,1){\circle{2}}
\put(26,1){\line(1,0){10}}
\put(37,1){\circle{2}}
\put(38,1){\line(1,0){10}}
\put(49,1){\circle{2}}
\put(1,4){\makebox[0pt]{\scriptsize $q$}}
\put(7,3){\makebox[0pt]{\scriptsize $q^{-1}$}}
\put(13,4){\makebox[0pt]{\scriptsize $q$}}
\put(19,3){\makebox[0pt]{\scriptsize $q^{-1}$}}
\put(26,4){\makebox[0pt]{\scriptsize $-1$}}
\put(31,3){\makebox[0pt]{\scriptsize $q$}}
\put(37,4){\makebox[0pt]{\scriptsize $q^{-1}$}}
\put(43,3){\makebox[0pt]{\scriptsize $-q$}}
\put(49,4){\makebox[0pt]{\scriptsize $-q^{-1}$}}
\end{picture}
%\Dchainfive{$q$}{$q^{-1}$}{$q$}{$q^{-1}$}{$-1$}{$q$}{$q^{-1}$}{$-q$}{$-q^{-1}$}
\end{align*}
 we obtain $a_{45}^{r_3r_2r_1(X)}=-3$, which contradicts $a_{45}^{r_3r_2r_1(X)}\ne -3$.  

Step $2.5$. Let $q_{22}=q_{33}=-1$ and $q_{ii}\widetilde{q_{i,i+1}}=q_{jj}\widetilde{q_{j-1,j}}=1$ for all $i\in \{1,4\}$ and $j\in \{4,5\}$. Hence $(c,d,e)=(0,0,1)$. Definition \ref{defA5} implies $(a,b,c,d,e)=(1,1,0,0,1)$. Set $q\colon=q_{11}$, $r^{-1}\colon=\widetilde{q_{23}}$ and $s\colon=q_{44}$. By $A^{r_2(X)}=A_5$, one obtains $qr=1$. And $a^{r_4(X)}_{35}=-1$ yields $rs\ne 1$. Then the reflection of $X$
\begin{align*}
X\colon~ 
\begin{picture}(50,4)(0,3)
\put(1,1){\circle{2}}
\put(2,1){\line(1,0){10}}
\put(13,1){\circle*{2}}
\put(14,1){\line(1,0){10}}
\put(25,1){\circle{2}}
\put(26,1){\line(1,0){10}}
\put(37,1){\circle{2}}
\put(38,1){\line(1,0){10}}
\put(49,1){\circle{2}}
\put(1,4){\makebox[0pt]{\scriptsize $r^{-1}$}}
\put(7,3){\makebox[0pt]{\scriptsize $r$}}
\put(13,4){\makebox[0pt]{\scriptsize $-1$}}
\put(19,3){\makebox[0pt]{\scriptsize $r^{-1}$}}
\put(25,4){\makebox[0pt]{\scriptsize $-1$}}
\put(31,3){\makebox[0pt]{\scriptsize $s^{-1}$}}
\put(37,4){\makebox[0pt]{\scriptsize $s$}}
\put(43,3){\makebox[0pt]{\scriptsize $s^{-1}$}}
\put(49,4){\makebox[0pt]{\scriptsize $s$}}
\end{picture}
%\Dchainfiveb{$r^{-1}$}{$r$}{$-1$}{$r^{-1}$}{$-1$}{$s^{-1}$}{$s$}{$s^{-1}$}{$s$}
\quad \Rightarrow \quad
r_2(X)\colon~ 
\begin{picture}(50,4)(0,3)
\put(1,1){\circle{2}}
\put(2,1){\line(1,0){10}}
\put(13,1){\circle{2}}
\put(14,1){\line(1,0){10}}
\put(25,1){\circle{2}}
\put(26,1){\line(1,0){10}}
\put(37,1){\circle{2}}
\put(38,1){\line(1,0){10}}
\put(49,1){\circle{2}}
\put(1,4){\makebox[0pt]{\scriptsize $-1$}}
\put(7,3){\makebox[0pt]{\scriptsize $r^{-1}$}}
\put(13,4){\makebox[0pt]{\scriptsize $-1$}}
\put(19,3){\makebox[0pt]{\scriptsize $r$}}
\put(26,4){\makebox[0pt]{\scriptsize $r^{-1}$}}
\put(31,3){\makebox[0pt]{\scriptsize $s^{-1}$}}
\put(37,4){\makebox[0pt]{\scriptsize $s$}}
\put(43,3){\makebox[0pt]{\scriptsize $s^{-1}$}}
\put(49,4){\makebox[0pt]{\scriptsize $s$}}
\end{picture}
%\Dchainfive{$-1$}{$r^{-1}$}{$-1$}{$r$}{$r^{-1}$}{$s^{-1}$}{$s$}{$s^{-1}$}{$s$}
\end{align*}
together with the condition $a_{34}^{r_2(X)}=-1$, force $r=-1$ and $p\ne 2$. Lemma \ref{l-5chainspe-a}(a$_2$) excludes this configuration. 

Step $2.6$. Let $q_{33}=q_{44}=-1$ and $q_{ii}q'_{i,i+1}=q_{jj}q'_{j-1,j}=1$ for all $i\in \{1,2\}$ and $j\in \{2,5\}$. Then $e=1$. From the definition of a good $A_5$ neighborhood, either $(a,b,c,d)=(0,0,1,1)$ or $(a,b,c,d)=(1,1,0,0)$. Set $q\colon=q_{11}$, $r^{-1}\colon=q'_{34}$ and $s\colon=q_{55}$. When $(a,b,c,d)=(0,0,1,1)$, we have $A^{r_3(X)}=A_5$ and $A^{r_4(X)}=\begin{pmatrix}2&-1&0&0&0\\-1&2&-1&0&0\\0&-1&2&-1&-1\\0&0&-1&2&-1\\0&0&-1&-1&2\end{pmatrix}$, which yields $qr=1$ and $rs\ne 1$.
The reflection of $X$
\begin{align*}
X\colon~
\begin{picture}(50,4)(0,3)
\put(1,1){\circle{2}}
\put(2,1){\line(1,0){10}}
\put(13,1){\circle{2}}
\put(14,1){\line(1,0){10}}
\put(25,1){\circle*{2}}
\put(26,1){\line(1,0){10}}
\put(37,1){\circle{2}}
\put(38,1){\line(1,0){10}}
\put(49,1){\circle{2}}
\put(1,4){\makebox[0pt]{\scriptsize $r^{-1}$}}
\put(7,3){\makebox[0pt]{\scriptsize $r$}}
\put(13,4){\makebox[0pt]{\scriptsize $r^{-1}$}}
\put(19,3){\makebox[0pt]{\scriptsize $r$}}
\put(25,4){\makebox[0pt]{\scriptsize $-1$}}
\put(31,3){\makebox[0pt]{\scriptsize $r^{-1}$}}
\put(37,4){\makebox[0pt]{\scriptsize $-1$}}
\put(43,3){\makebox[0pt]{\scriptsize $s^{-1}$}}
\put(49,4){\makebox[0pt]{\scriptsize $s$}}
\end{picture}
%\Dchainfivec{$r^{-1}$}{$r$}{$r^{-1}$}{$r$}{$-1$}{$r^{-1}$}{$-1$}{$s^{-1}$}{$s$}
\quad \Rightarrow \quad
r_3(X)\colon~
\begin{picture}(50,4)(0,3)
\put(1,1){\circle{2}}
\put(2,1){\line(1,0){10}}
\put(13,1){\circle{2}}
\put(14,1){\line(1,0){10}}
\put(25,1){\circle{2}}
\put(26,1){\line(1,0){10}}
\put(37,1){\circle{2}}
\put(38,1){\line(1,0){10}}
\put(49,1){\circle{2}}
\put(1,4){\makebox[0pt]{\scriptsize $r^{-1}$}}
\put(7,3){\makebox[0pt]{\scriptsize $r$}}
\put(13,4){\makebox[0pt]{\scriptsize $-1$}}
\put(19,3){\makebox[0pt]{\scriptsize $r^{-1}$}}
\put(26,4){\makebox[0pt]{\scriptsize $-1$}}
\put(31,3){\makebox[0pt]{\scriptsize $r$}}
\put(37,4){\makebox[0pt]{\scriptsize $r^{-1}$}}
\put(43,3){\makebox[0pt]{\scriptsize $s^{-1}$}}
\put(49,4){\makebox[0pt]{\scriptsize $s$}}
\end{picture}
%\Dchainfive{$r^{-1}$}{$r$}{$-1$}{$r^{-1}$}{$-1$}{$r$}{$r^{-1}$}{$s^{-1}$}{$s$}
\end{align*}
together with the identity $a_{45}^{r_3(X)}=-1$ implies $r=-1$ and $p\ne 2$.
Lemma \ref{l-5chainspe-a}$(c_4)$ excludes this configuration entirely. 

When $(a,b,c,d)=(1,1,0,0)$, we have $A^{r_4(X)}=A_5$ and
\[
A^{r_3(X)}=\begin{pmatrix}2&-1&0&0&0\\-1&2&-1&-1&0\\0&-1&2&-1&0\\0&-1&-1&2&-1\\0&0&0&-1&2\end{pmatrix},
\]
which gives $qr\neq 1$ and $rs=1$. The reflection transformation
\begin{align*}
X\colon~ 
\begin{picture}(50,4)(0,3)
\put(1,1){\circle{2}}
\put(2,1){\line(1,0){10}}
\put(13,1){\circle{2}}
\put(14,1){\line(1,0){10}}
\put(25,1){\circle{2}}
\put(26,1){\line(1,0){10}}
\put(37,1){\circle*{2}}
\put(38,1){\line(1,0){10}}
\put(49,1){\circle{2}}
\put(1,4){\makebox[0pt]{\scriptsize $q$}}
\put(7,3){\makebox[0pt]{\scriptsize $q^{-1}$}}
\put(13,4){\makebox[0pt]{\scriptsize $q$}}
\put(19,3){\makebox[0pt]{\scriptsize $q^{-1}$}}
\put(25,4){\makebox[0pt]{\scriptsize $-1$}}
\put(31,3){\makebox[0pt]{\scriptsize $r^{-1}$}}
\put(37,4){\makebox[0pt]{\scriptsize $-1$}}
\put(43,3){\makebox[0pt]{\scriptsize $r$}}
\put(49,4){\makebox[0pt]{\scriptsize $r^{-1}$}}
\end{picture}
\quad \Rightarrow \quad
r_4(X)\colon~ 
\begin{picture}(50,4)(0,3)
\put(1,1){\circle{2}}
\put(2,1){\line(1,0){10}}
\put(13,1){\circle{2}}
\put(14,1){\line(1,0){10}}
\put(25,1){\circle{2}}
\put(26,1){\line(1,0){10}}
\put(37,1){\circle{2}}
\put(38,1){\line(1,0){10}}
\put(49,1){\circle{2}}
\put(1,4){\makebox[0pt]{\scriptsize $q$}}
\put(7,3){\makebox[0pt]{\scriptsize $q^{-1}$}}
\put(13,4){\makebox[0pt]{\scriptsize $q$}}
\put(19,3){\makebox[0pt]{\scriptsize $q^{-1}$}}
\put(26,4){\makebox[0pt]{\scriptsize $r^{-1}$}}
\put(31,3){\makebox[0pt]{\scriptsize $r$}}
\put(37,4){\makebox[0pt]{\scriptsize $-1$}}
\put(43,3){\makebox[0pt]{\scriptsize $r^{-1}$}}
\put(49,4){\makebox[0pt]{\scriptsize $-1$}}
\end{picture}
\end{align*}
together with the condition $a_{32}^{r_4(X)}=-1$, forces $r=-1$ and $p\neq 2$. Lemma \ref{l-5chainspe-a}(a$_1$) rules out this configuration.

Step $2.7$. Suppose $q_{22}=q_{44}=-1$ and $q_{ii}\widetilde{q_{i,i+1}}=q_{jj}\widetilde{q_{j-1,j}}=1$ for all $i\in \{1,3\}$ and $j\in \{3,5\}$. These constraints force $(a,b,e)=(0,0,1)$, so that $(a,b,c,d,e)=(0,0,1,1,1)$. Define $s\colon=q_{11}$, $q\colon=q_{33}$ and $r\colon=q_{55}$. Then $rs\ne 1$ by $a^{r_4(X)}_{35}=-1$. And $A^{r_2(X)}=A_5$ yields $qr=1$.
Reproducing the computational procedure from $2.4$, we deduce $a_{45}^{r_3r_2r_1(X)}=-3$, which is a contradiction.

Step $3$. Suppose three diagonal entries equal $-1$. The subsequent analysis is divided into nine separate subcases, enumerated as Steps $3.1$–$3.9$.

Step $3.1$. Let $q_{11}=q_{22}=q_{33}=-1$, alongside the identities $q_{44}\widetilde q_{45}=1$ and $q_{jj}\widetilde q_{j-1,j}=1$ for $j\in\{4,5\}$. These relations yield $(c,d,e)=(0,0,1)$, and the definition of a good $A_5$-neighborhood forces $(a,b,c,d,e)=(1,1,0,0,1)$. Set $q^{-1}:=\widetilde q_{12}$, $r^{-1}:=\widetilde q_{23}$ and $s:=q_{44}$. The condition $a_{24}^{r_3(X)}=-1$ imposes $rs\neq 1$, while $A^{r_2(X)}\cong A_5$ gives $qr=1$.
The reflection transformation
\begin{align*}
X\colon~ 
\begin{picture}(50,4)(0,3)
\put(1,1){\circle{2}}
\put(2,1){\line(1,0){10}}
\put(13,1){\circle*{2}}
\put(14,1){\line(1,0){10}}
\put(25,1){\circle{2}}
\put(26,1){\line(1,0){10}}
\put(37,1){\circle{2}}
\put(38,1){\line(1,0){10}}
\put(49,1){\circle{2}}
\put(1,4){\makebox[0pt]{\scriptsize $-1$}}
\put(7,3){\makebox[0pt]{\scriptsize $r$}}
\put(13,4){\makebox[0pt]{\scriptsize $-1$}}
\put(19,3){\makebox[0pt]{\scriptsize $r^{-1}$}}
\put(25,4){\makebox[0pt]{\scriptsize $-1$}}
\put(31,3){\makebox[0pt]{\scriptsize $s^{-1}$}}
\put(37,4){\makebox[0pt]{\scriptsize $s$}}
\put(43,3){\makebox[0pt]{\scriptsize $s^{-1}$}}
\put(49,4){\makebox[0pt]{\scriptsize $s$}}
\end{picture}
\quad \Rightarrow \quad
r_2(X)\colon~ 
\begin{picture}(50,4)(0,3)
\put(1,1){\circle{2}}
\put(2,1){\line(1,0){10}}
\put(13,1){\circle{2}}
\put(14,1){\line(1,0){10}}
\put(25,1){\circle{2}}
\put(26,1){\line(1,0){10}}
\put(37,1){\circle{2}}
\put(38,1){\line(1,0){10}}
\put(49,1){\circle{2}}
\put(1,4){\makebox[0pt]{\scriptsize $r$}}
\put(7,3){\makebox[0pt]{\scriptsize $r^{-1}$}}
\put(13,4){\makebox[0pt]{\scriptsize $-1$}}
\put(19,3){\makebox[0pt]{\scriptsize $r$}}
\put(26,4){\makebox[0pt]{\scriptsize $r^{-1}$}}
\put(31,3){\makebox[0pt]{\scriptsize $s^{-1}$}}
\put(37,4){\makebox[0pt]{\scriptsize $s$}}
\put(43,3){\makebox[0pt]{\scriptsize $s^{-1}$}}
\put(49,4){\makebox[0pt]{\scriptsize $s$}}
\end{picture}
\end{align*}
together with $a_{34}^{r_2(X)}=-1$, forces $r=-1$ and $p\neq 2$. Lemma \ref{l-5chainspe-a}(a$_2$) excludes this configuration.

Step $3.2$. Suppose  $q_{33}=q_{44}=q_{55}=-1$, $q_{22}\widetilde{q_{12}}=1$ and $q_{ii}\widetilde{q_{i,i+1}}=1$ for $i\in \{1,2\}$. Definition \ref{defA5} restricts the admissible parameters to two possibilities: $(a,b,c,d)=(1,1,0,0)$ or $(a,b,c,d)=(0,0,1,1)$. By parallel reasoning to Step $2.6$, both assignments are excluded.

Step $3.3$. Let $q_{11}=q_{22}=q_{44}=-1$, $q_{33}\widetilde{q_{34}}=1$ and $q_{jj}\widetilde{q_{j-1,j}}=1$ for $j\in \{3,5\}$. These constraints give $(a,b,e)=(0,0,1)$, and the definition of a good $A_5$ neighborhood forces $(a,b,c,d,e)=(0,0,1,1,1)$.
Define $q^{-1}\colon=\widetilde{q_{12}}$, $r\colon=q_{33}$ and $s\colon=q_{55}$. The inequality $rs\ne 1$ holds, while the condition $A^{r_2(X)}=A_5$ imposes $qr=1$. On one hand, the successive reflections of $X$
\begin{align*}
X\colon~
\begin{picture}(50,4)(0,3)
\put(1,1){\circle{2}}
\put(2,1){\line(1,0){10}}
\put(13,1){\circle{2}}
\put(14,1){\line(1,0){10}}
\put(25,1){\circle{2}}
\put(26,1){\line(1,0){10}}
\put(37,1){\circle*{2}}
\put(38,1){\line(1,0){10}}
\put(49,1){\circle{2}}
\put(1,4){\makebox[0pt]{\scriptsize $-1$}}
\put(7,3){\makebox[0pt]{\scriptsize $r$}}
\put(13,4){\makebox[0pt]{\scriptsize $-1$}}
\put(19,3){\makebox[0pt]{\scriptsize $r^{-1}$}}
\put(25,4){\makebox[0pt]{\scriptsize $r$}}
\put(31,3){\makebox[0pt]{\scriptsize $r^{-1}$}}
\put(37,4){\makebox[0pt]{\scriptsize $-1$}}
\put(43,3){\makebox[0pt]{\scriptsize $s^{-1}$}}
\put(49,4){\makebox[0pt]{\scriptsize $s$}}
\end{picture}
%\Dchainfived{$-1$}{$r$}{$-1$}{$r^{-1}$}{$r$}{$r^{-1}$}{$-1$}{$s^{-1}$}{$s$}
\quad \Rightarrow \quad
r_4(X)\colon~ 
\begin{picture}(36,10)(0,9)
\put(1,9){\circle{2}}
\put(2,9){\line(1,0){10}}
\put(13,9){\circle{2}}
\put(14,9){\line(1,0){10}}
\put(25,9){\circle*{2}}
\put(26,9){\line(1,1){7}}
\put(26,9){\line(1,-1){7}}
\put(33,17){\circle{2}}
\put(33,1){\circle{2}}
\put(33,2){\line(0,1){14}}
\put(1,12){\makebox[0pt]{\scriptsize $-1$}}
\put(7,11){\makebox[0pt]{\scriptsize $r$}}
\put(13,12){\makebox[0pt]{\scriptsize $-1$}}
\put(19,11){\makebox[0pt]{\scriptsize $r^{-1}$}}
\put(23.5,12){\makebox[0pt]{\scriptsize $-1$}}
\put(30,14){\makebox[0pt][r]{\scriptsize $r$}}
\put(29,4){\makebox[0pt][r]{\scriptsize $(rs)^{-1}$}}
\put(35,16){\makebox[0pt][l]{\scriptsize $-1$}}
\put(34,9){\makebox[0pt][l]{\scriptsize $s$}}
\put(35,1){\makebox[0pt][l]{\scriptsize $-1$}}
\end{picture}
%\Dchainfivew{$-1$}{$r$}{$-1$}{$r^{-1}$}{$-1$}{$r$}{$(rs)^{-1}$}{$-1$}{$s$}{$-1$}
\end{align*}
\vspace{-3mm}
\begin{align*}
 \Rightarrow \quad 
r_3r_4(X)\colon~ \tau_{12534}~~
\begin{picture}(38,11)(0,3)
\put(1,1){\circle{2}}
\put(2,1){\line(1,0){10}}
\put(13,1){\circle{2}}
\put(13,2){\line(2,3){6}}
\put(14,1){\line(1,0){10}}
\put(25,1){\circle{2}}
\put(25,2){\line(-2,3){6}}
\put(19,12){\circle{2}}
\put(26,1){\line(1,0){10}}
\put(37,1){\circle{2}}
\put(1,4){\makebox[0pt]{\scriptsize $-1$}}
\put(7,3){\makebox[0pt]{\scriptsize $r$}}
\put(12,4){\makebox[0pt]{\scriptsize $r^{-1}$}}
\put(19,3){\makebox[0pt]{\scriptsize $r$}}
\put(26,4){\makebox[0pt]{\scriptsize $-1$}}
\put(31,3){\makebox[0pt]{\scriptsize $r^{-1}$}}
\put(37,4){\makebox[0pt]{\scriptsize $r$}}
\put(13.3,8){\makebox[0pt]{\scriptsize $r^{-2}s^{-1}$}}
\put(23,8){\makebox[0pt][l]{\scriptsize $rs$}}
\put(24,14){\makebox[0pt]{\scriptsize $(rs)^{-1}$}}
\end{picture}
%\Dchainfivex{$-1$}{$r$}{$r^{-1}$}{$r$}{$-1$}{$r^{-1}$}{$r$}{$r^{-2}s^{-1}$}{$rs$}{$(rs)^{-1}$}
\end{align*}
together with $a_{52}^{r_3r_4(X)}=-1$, yield the alternative constraints $rs=-1$ or $r^3s^2=1$. On the other hand, iterated reflections of $X$
\begin{align*}
X\colon~ 
\begin{picture}(50,4)(0,3)
\put(1,1){\circle{2}}
\put(2,1){\line(1,0){10}}
\put(13,1){\circle{2}}
\put(14,1){\line(1,0){10}}
\put(25,1){\circle{2}}
\put(26,1){\line(1,0){10}}
\put(37,1){\circle*{2}}
\put(38,1){\line(1,0){10}}
\put(49,1){\circle{2}}
\put(1,4){\makebox[0pt]{\scriptsize $-1$}}
\put(7,3){\makebox[0pt]{\scriptsize $r$}}
\put(13,4){\makebox[0pt]{\scriptsize $-1$}}
\put(19,3){\makebox[0pt]{\scriptsize $r^{-1}$}}
\put(25,4){\makebox[0pt]{\scriptsize $r$}}
\put(31,3){\makebox[0pt]{\scriptsize $r^{-1}$}}
\put(37,4){\makebox[0pt]{\scriptsize $-1$}}
\put(43,3){\makebox[0pt]{\scriptsize $s^{-1}$}}
\put(49,4){\makebox[0pt]{\scriptsize $s$}}
\end{picture}
%\Dchainfived{$-1$}{$r$}{$-1$}{$r^{-1}$}{$r$}{$r^{-1}$}{$-1$}{$s^{-1}$}{$s$}
\quad \Rightarrow \quad
r_4(X)\colon~ 
\begin{picture}(36,10)(0,9)
\put(1,9){\circle{2}}
\put(2,9){\line(1,0){10}}
\put(13,9){\circle{2}}
\put(14,9){\line(1,0){10}}
\put(25,9){\circle*{2}}
\put(26,9){\line(1,1){7}}
\put(26,9){\line(1,-1){7}}
\put(33,17){\circle{2}}
\put(33,1){\circle{2}}
\put(33,2){\line(0,1){14}}
\put(1,12){\makebox[0pt]{\scriptsize $-1$}}
\put(7,11){\makebox[0pt]{\scriptsize $r$}}
\put(13,12){\makebox[0pt]{\scriptsize $-1$}}
\put(19,11){\makebox[0pt]{\scriptsize $r^{-1}$}}
\put(23.5,12){\makebox[0pt]{\scriptsize $-1$}}
\put(30,14){\makebox[0pt][r]{\scriptsize $r$}}
\put(29,4){\makebox[0pt][r]{\scriptsize $(rs)^{-1}$}}
\put(35,16){\makebox[0pt][l]{\scriptsize $-1$}}
\put(34,9){\makebox[0pt][l]{\scriptsize $s$}}
\put(35,1){\makebox[0pt][l]{\scriptsize $-1$}}
\end{picture}
%\Dchainfivew{$-1$}{$r$}{$-1$}{$r^{-1}$}{$-1$}{$r$}{$(rs)^{-1}$}{$-1$}{$s$}{$-1$}
\end{align*}
\begin{align*}
 \Rightarrow \quad 
r_5r_4(X)\colon~ \tau_{12354}~~ 
\begin{picture}(50,4)(0,3)
\put(1,1){\circle{2}}
\put(2,1){\line(1,0){10}}
\put(13,1){\circle{2}}
\put(14,1){\line(1,0){10}}
\put(25,1){\circle{2}}
\put(26,1){\line(1,0){10}}
\put(37,1){\circle{2}}
\put(38,1){\line(1,0){10}}
\put(49,1){\circle{2}}
\put(1,4){\makebox[0pt]{\scriptsize $-1$}}
\put(7,3){\makebox[0pt]{\scriptsize $r$}}
\put(13,4){\makebox[0pt]{\scriptsize $-1$}}
\put(19,3){\makebox[0pt]{\scriptsize $r^{-1}$}}
\put(26,4){\makebox[0pt]{\scriptsize $(rs)^{-1}$}}
\put(31,3){\makebox[0pt]{\scriptsize $rs$}}
\put(37,4){\makebox[0pt]{\scriptsize $-1$}}
\put(43,3){\makebox[0pt]{\scriptsize $s^{-1}$}}
\put(49,4){\makebox[0pt]{\scriptsize $s$}}
\end{picture}
%\Dchainfive{$-1$}{$r$}{$-1$}{$r^{-1}$}{$(rs)^{-1}$}{$rs$}{$-1$}{$s^{-1}$}{$s$}
\end{align*}
and $a_{32}^{r_5r_4(X)}=-1$ produce either $rs=-1$ or $r^2s=1$. Then $rs=-1$ and $p\ne 2$. However $a_{15}^{r_2r_3r_4(X)}=0$ yields $rs=1$, a direct contradiction.

Step $3.4$. Let $q_{22}=q_{44}=q_{55}=-1$, alongside $q_{33}\widetilde{q_{23}}=1$ and $q_{ii}\widetilde{q_{i,i+1}}=1$ for $i\in \{1,3\}$. Then $(a,b)=(0,0)$. Definition \ref{defA5} restricts the parameter tuple to two alternatives: $(a,b,c,d,e)=(0,0,1,1,2)$ or $(a,b,c,d,e)=(0,0,1,1,1)$. Set $q\colon=q_{11}$, $r\colon=q_{33}$ and $s^{-1}\colon=\widetilde{q_{45}}$. 

Suppose first $(a,b,c,d,e)=(0,0,1,1,2)$. Then $A^{r_3(X)}=A_5$, $A^{r_4(X)}=\begin{pmatrix}2&-1&0&0&0\\-1&2&-1&0&0\\0&-1&2&-1&-1\\0&0&-1&2&-1\\0&0&-1&-1&2\end{pmatrix}$ and $A^{r_5(X)}=\begin{pmatrix}2&-1&0&0&0\\-1&2&-1&0&0\\0&-1&2&-1&0\\0&0&-2&2&-1\\0&0&0&-1&2\end{pmatrix}$, which gives  $qr=1$ and $rs\ne 1$. From the reflection transformation
\begin{align}\label{eq-9}
X\colon~ 
\begin{picture}(50,4)(0,3)
\put(1,1){\circle{2}}
\put(2,1){\line(1,0){10}}
\put(13,1){\circle{2}}
\put(14,1){\line(1,0){10}}
\put(25,1){\circle{2}}
\put(26,1){\line(1,0){10}}
\put(37,1){\circle*{2}}
\put(38,1){\line(1,0){10}}
\put(49,1){\circle{2}}
\put(1,4){\makebox[0pt]{\scriptsize $r^{-1}$}}
\put(7,3){\makebox[0pt]{\scriptsize $r$}}
\put(13,4){\makebox[0pt]{\scriptsize $-1$}}
\put(19,3){\makebox[0pt]{\scriptsize $r^{-1}$}}
\put(25,4){\makebox[0pt]{\scriptsize $r$}}
\put(31,3){\makebox[0pt]{\scriptsize $r^{-1}$}}
\put(37,4){\makebox[0pt]{\scriptsize $-1$}}
\put(43,3){\makebox[0pt]{\scriptsize $s^{-1}$}}
\put(49,4){\makebox[0pt]{\scriptsize $-1$}}
\end{picture}
%\Dchainfived{$r^{-1}$}{$r$}{$-1$}{$r^{-1}$}{$r$}{$r^{-1}$}{$-1$}{$s^{-1}$}{$-1$}
\quad \Rightarrow \quad
r_4(X)\colon~ 
\Dchainfivet{$r^{-1}$}{$r$}{$-1$}{$r^{-1}$}{$-1$}{$r$}{$(rs)^{-1}$}{$-1$}{$s$}{$s^{-1}$}
\end{align}
the condition $a_{53}^{r_4(X)}=-1$ imposes $s=-1$ or $rs^2=1$. Moreover, the reflection of $X$
\begin{align*}
X\colon~ 
\begin{picture}(50,4)(0,3)
\put(1,1){\circle{2}}
\put(2,1){\line(1,0){10}}
\put(13,1){\circle{2}}
\put(14,1){\line(1,0){10}}
\put(25,1){\circle{2}}
\put(26,1){\line(1,0){10}}
\put(37,1){\circle{2}}
\put(38,1){\line(1,0){10}}
\put(49,1){\circle*{2}}
\put(1,4){\makebox[0pt]{\scriptsize $r^{-1}$}}
\put(7,3){\makebox[0pt]{\scriptsize $r$}}
\put(13,4){\makebox[0pt]{\scriptsize $-1$}}
\put(19,3){\makebox[0pt]{\scriptsize $r^{-1}$}}
\put(25,4){\makebox[0pt]{\scriptsize $r$}}
\put(31,3){\makebox[0pt]{\scriptsize $r^{-1}$}}
\put(37,4){\makebox[0pt]{\scriptsize $-1$}}
\put(43,3){\makebox[0pt]{\scriptsize $s^{-1}$}}
\put(49,4){\makebox[0pt]{\scriptsize $-1$}}
\end{picture}
%\Dchainfivee{$r^{-1}$}{$r$}{$-1$}{$r^{-1}$}{$r$}{$r^{-1}$}{$-1$}{$s^{-1}$}{$-1$}
\quad \Rightarrow \quad
r_5(X)\colon~ 
\begin{picture}(50,4)(0,3)
\put(1,1){\circle{2}}
\put(2,1){\line(1,0){10}}
\put(13,1){\circle{2}}
\put(14,1){\line(1,0){10}}
\put(25,1){\circle{2}}
\put(26,1){\line(1,0){10}}
\put(37,1){\circle{2}}
\put(38,1){\line(1,0){10}}
\put(49,1){\circle{2}}
\put(1,4){\makebox[0pt]{\scriptsize $r^{-1}$}}
\put(7,3){\makebox[0pt]{\scriptsize $r$}}
\put(13,4){\makebox[0pt]{\scriptsize $-1$}}
\put(19,3){\makebox[0pt]{\scriptsize $r^{-1}$}}
\put(26,4){\makebox[0pt]{\scriptsize $r$}}
\put(31,3){\makebox[0pt]{\scriptsize $r^{-1}$}}
\put(37,4){\makebox[0pt]{\scriptsize $s^{-1}$}}
\put(43,3){\makebox[0pt]{\scriptsize $s$}}
\put(49,4){\makebox[0pt]{\scriptsize $-1$}}
\end{picture}
%\Dchainfive{$r^{-1}$}{$r$}{$-1$}{$r^{-1}$}{$r$}{$r^{-1}$}{$s^{-1}$}{$s$}{$-1$}
\end{align*}
together with $a_{43}^{r_5(X)}=-2$ excludes $s=-1$,  leaving the constraint $rs^2=1$. By Definition \ref{defA5} $(A_{5_2})$, the value $a_{25}^{r_3r_4(X)}$ equals  either $0$ or $-1$. When $a_{25}^{r_3r_4(X)}=0$, the reflection (\ref{eq-9}) of $X$ yields $r^2s=1$. Hence $r=s\in G'_3$ and $p\ne 3$. Iterated reflections
\begin{align*}
X\colon~
\begin{picture}(50,4)(0,3)
\put(1,1){\circle{2}}
\put(2,1){\line(1,0){10}}
\put(13,1){\circle*{2}}
\put(14,1){\line(1,0){10}}
\put(25,1){\circle{2}}
\put(26,1){\line(1,0){10}}
\put(37,1){\circle{2}}
\put(38,1){\line(1,0){10}}
\put(49,1){\circle{2}}
\put(1,4){\makebox[0pt]{\scriptsize $r^{-1}$}}
\put(7,3){\makebox[0pt]{\scriptsize $r$}}
\put(13,4){\makebox[0pt]{\scriptsize $-1$}}
\put(19,3){\makebox[0pt]{\scriptsize $r^{-1}$}}
\put(25,4){\makebox[0pt]{\scriptsize $r$}}
\put(31,3){\makebox[0pt]{\scriptsize $r^{-1}$}}
\put(37,4){\makebox[0pt]{\scriptsize $-1$}}
\put(43,3){\makebox[0pt]{\scriptsize $r^{-1}$}}
\put(49,4){\makebox[0pt]{\scriptsize $-1$}}
\end{picture}
%\Dchainfiveb{$r^{-1}$}{$r$}{$-1$}{$r^{-1}$}{$r$}{$r^{-1}$}{$-1$}{$r^{-1}$}{$-1$}
\quad \Rightarrow \quad
r_2(X)\colon~
\begin{picture}(50,4)(0,3)
\put(1,1){\circle{2}}
\put(2,1){\line(1,0){10}}
\put(13,1){\circle{2}}
\put(14,1){\line(1,0){10}}
\put(25,1){\circle*{2}}
\put(26,1){\line(1,0){10}}
\put(37,1){\circle{2}}
\put(38,1){\line(1,0){10}}
\put(49,1){\circle{2}}
\put(1,4){\makebox[0pt]{\scriptsize $-1$}}
\put(7,3){\makebox[0pt]{\scriptsize $r^{-1}$}}
\put(13,4){\makebox[0pt]{\scriptsize $-1$}}
\put(19,3){\makebox[0pt]{\scriptsize $r$}}
\put(25,4){\makebox[0pt]{\scriptsize $-1$}}
\put(31,3){\makebox[0pt]{\scriptsize $r^{-1}$}}
\put(37,4){\makebox[0pt]{\scriptsize $-1$}}
\put(43,3){\makebox[0pt]{\scriptsize $r^{-1}$}}
\put(49,4){\makebox[0pt]{\scriptsize $-1$}}
\end{picture}
%\Dchainfivec{$-1$}{$r^{-1}$}{$-1$}{$r$}{$-1$}{$r^{-1}$}{$-1$}{$r^{-1}$}{$-1$}
\end{align*}
\vspace{-3mm}
\begin{align*}
 \Rightarrow \quad 
r_3r_2(X)\colon~ 
\Dchainfive{$-1$}{$r^{-1}$}{$r$}{$r^{-1}$}{$-1$}{$r$}{$r^{-1}$}{$r^{-1}$}{$-1$}
\end{align*}
produce $(a_{45}^{r_3r_2r_1(X)},a_{45}^{r_3r_2(X)})=(-2,-2)$, which contradicts $(a_{45}^{r_3r_2r_1(X)},a_{45}^{r_3r_2(X)})\ne (-2,-2)$. 

If instead  $a_{25}^{r_3r_4(X)}=-1$, further reflection of $X$
\begin{align*}
r_4(X)\colon~ 
\Dchainfivew{$r^{-1}$}{$r$}{$-1$}{$r^{-1}$}{$-1$}{$r$}{$(rs)^{-1}$}{$-1$}{$s$}{$s^{-1}$}
\quad \Rightarrow \quad
r_3r_4(X)\colon~ \tau_{12534}~~ 
\begin{picture}(38,11)(0,3)
\put(1,1){\circle{2}}
\put(2,1){\line(1,0){10}}
\put(13,1){\circle{2}}
\put(13,2){\line(2,3){6}}
\put(14,1){\line(1,0){10}}
\put(25,1){\circle{2}}
\put(25,2){\line(-2,3){6}}
\put(19,12){\circle{2}}
\put(26,1){\line(1,0){10}}
\put(37,1){\circle{2}}
\put(1,4){\makebox[0pt]{\scriptsize $r^{-1}$}}
\put(7,3){\makebox[0pt]{\scriptsize $r$}}
\put(12,4){\makebox[0pt]{\scriptsize $r^{-1}$}}
\put(19,3){\makebox[0pt]{\scriptsize $r$}}
\put(26,4){\makebox[0pt]{\scriptsize $-1$}}
\put(31,3){\makebox[0pt]{\scriptsize $r^{-1}$}}
\put(37,4){\makebox[0pt]{\scriptsize $r$}}
\put(13.3,8){\makebox[0pt]{\scriptsize $r^{-2}s^{-1}$}}
\put(23,8){\makebox[0pt][l]{\scriptsize $rs$}}
\put(24,14){\makebox[0pt]{\scriptsize $-r^{-1}s^{-2}$}}
\end{picture}
%\Dchainfivex{$r^{-1}$}{$r$}{$r^{-1}$}{$r$}{$-1$}{$r^{-1}$}{$r$}{$r^{-2}s^{-1}$}{$rs$}{$-r^{-1}s^{-2}$}
\end{align*}
enforces either $r=-1$ or $r^3s=1$. The case $r=-1$ combined with $a_{15}^{r_2r_3r_4(X)}=0$ yields $s=-1$ and $p\ne 2$, a contradiction. Under $r^3s=1$ together with $rs^2=1$, we obtain $r\in G'_5$. By arguments parallel to Case $a_{2_4}$, we have $a_{45}^{r_3r_2r_1(X)}=-3$, conflicting with $a_{45}^{r_3r_2r_1(X)}\ne -3$.

Now take $(a,b,c,d,e)=(0,0,1,1,1)$. Then $A^{r_3(X)}=A_5$, $A^{r_4(X)}=\begin{pmatrix}2&-1&0&0&0\\-1&2&-1&0&0\\0&-1&2&-1&-1\\0&0&-1&2&-1\\0&0&-1&-1&2\end{pmatrix}$ and $A^{r_5(X)}=A_5$. Hence $qr=1$ and $rs\ne 1$. From
\begin{align*}
X\colon~ 
\Dchainfivee{$r^{-1}$}{$r$}{$-1$}{$r^{-1}$}{$r$}{$r^{-1}$}{$-1$}{$s^{-1}$}{$-1$}
\quad \Rightarrow \quad
r_5(X)\colon~ 
\Dchainfive{$r^{-1}$}{$r$}{$-1$}{$r^{-1}$}{$r$}{$r^{-1}$}{$s^{-1}$}{$s$}{$-1$}
\end{align*}
the condition $a_{43}^{r_5(X)}=-1$ forces $s=-1$ and $p\ne 2$, which is eliminated by Lemma \ref{l-5chainspe-a}$(c_1)$.

Step $3.5$. Let $q_{11}=q_{44}=q_{55}=-1$ and $q_{ii}\widetilde{q_{i,i+1}}=q_{jj}\widetilde{q_{j-1,j}}=1$ for all $i\in \{2,3\}$ and $j\in \{2,3\}$. These equalities give $(a,b)=(0,0)$, leaving two admissible assignments $(a,b,c,d,e)=(0,0,1,1,2)$ or $(a,b,c,d,e)=(0,0,1,1,1)$. Set $q\colon=q_{22}$ and $r^{-1}\colon=\widetilde{q_{45}}$.

For $(a,b,c,d,e)=(0,0,1,1,2)$, we have $A^{r_3(X)}=A_5$, $A^{r_4(X)}=\begin{pmatrix}2&-1&0&0&0\\-1&2&-1&0&0\\0&-1&2&-1&-1\\0&0&-1&2&-1\\0&0&-1&-1&2\end{pmatrix}$ and $A^{r_5(X)}=\begin{pmatrix}2&-1&0&0&0\\-1&2&-1&0&0\\0&-1&2&-1&0\\0&0&-2&2&-1\\0&0&0&-1&2\end{pmatrix}$. So $qr\ne 1$. The reflection
\begin{align}\label{eq-10}
X\colon~
\begin{picture}(50,4)(0,3)
\put(1,1){\circle{2}}
\put(2,1){\line(1,0){10}}
\put(13,1){\circle{2}}
\put(14,1){\line(1,0){10}}
\put(25,1){\circle{2}}
\put(26,1){\line(1,0){10}}
\put(37,1){\circle*{2}}
\put(38,1){\line(1,0){10}}
\put(49,1){\circle{2}}
\put(1,4){\makebox[0pt]{\scriptsize $-1$}}
\put(7,3){\makebox[0pt]{\scriptsize $q^{-1}$}}
\put(13,4){\makebox[0pt]{\scriptsize $q$}}
\put(19,3){\makebox[0pt]{\scriptsize $q^{-1}$}}
\put(25,4){\makebox[0pt]{\scriptsize $q$}}
\put(31,3){\makebox[0pt]{\scriptsize $q^{-1}$}}
\put(37,4){\makebox[0pt]{\scriptsize $-1$}}
\put(43,3){\makebox[0pt]{\scriptsize $r^{-1}$}}
\put(49,4){\makebox[0pt]{\scriptsize $-1$}}
\end{picture}
%\Dchainfived{$-1$}{$q^{-1}$}{$q$}{$q^{-1}$}{$q$}{$q^{-1}$}{$-1$}{$r^{-1}$}{$-1$}
\quad \Rightarrow \quad
r_4(X)\colon~
\Dchainfivet{$-1$}{$q^{-1}$}{$q$}{$q^{-1}$}{$-1$}{$q$}{$(qr)^{-1}$}{$-1$}{$r$}{$r^{-1}$}
\end{align}
gives $r=-1$ or $qr^2=1$ by $a_{53}^{r_4(X)}=-1$. A further reflection of $X$
\begin{align*}
X\colon~
\begin{picture}(50,4)(0,3)
\put(1,1){\circle{2}}
\put(2,1){\line(1,0){10}}
\put(13,1){\circle{2}}
\put(14,1){\line(1,0){10}}
\put(25,1){\circle{2}}
\put(26,1){\line(1,0){10}}
\put(37,1){\circle{2}}
\put(38,1){\line(1,0){10}}
\put(49,1){\circle*{2}}
\put(1,4){\makebox[0pt]{\scriptsize $-1$}}
\put(7,3){\makebox[0pt]{\scriptsize $q^{-1}$}}
\put(13,4){\makebox[0pt]{\scriptsize $q$}}
\put(19,3){\makebox[0pt]{\scriptsize $q^{-1}$}}
\put(25,4){\makebox[0pt]{\scriptsize $q$}}
\put(31,3){\makebox[0pt]{\scriptsize $q^{-1}$}}
\put(37,4){\makebox[0pt]{\scriptsize $-1$}}
\put(43,3){\makebox[0pt]{\scriptsize $r^{-1}$}}
\put(49,4){\makebox[0pt]{\scriptsize $-1$}}
\end{picture}
%\Dchainfivee{$-1$}{$q^{-1}$}{$q$}{$q^{-1}$}{$q$}{$q^{-1}$}{$-1$}{$r^{-1}$}{$-1$}
\quad \Rightarrow \quad
r_5(X)\colon~ 
\begin{picture}(50,4)(0,3)
\put(1,1){\circle{2}}
\put(2,1){\line(1,0){10}}
\put(13,1){\circle{2}}
\put(14,1){\line(1,0){10}}
\put(25,1){\circle{2}}
\put(26,1){\line(1,0){10}}
\put(37,1){\circle{2}}
\put(38,1){\line(1,0){10}}
\put(49,1){\circle{2}}
\put(1,4){\makebox[0pt]{\scriptsize $-1$}}
\put(7,3){\makebox[0pt]{\scriptsize $q^{-1}$}}
\put(13,4){\makebox[0pt]{\scriptsize $q$}}
\put(19,3){\makebox[0pt]{\scriptsize $q^{-1}$}}
\put(26,4){\makebox[0pt]{\scriptsize $q$}}
\put(31,3){\makebox[0pt]{\scriptsize $q^{-1}$}}
\put(37,4){\makebox[0pt]{\scriptsize $r^{-1}$}}
\put(43,3){\makebox[0pt]{\scriptsize $r$}}
\put(49,4){\makebox[0pt]{\scriptsize $-1$}}
\end{picture}
%\Dchainfive{$-1$}{$q^{-1}$}{$q$}{$q^{-1}$}{$q$}{$q^{-1}$}{$r^{-1}$}{$r$}{$-1$}
\end{align*}
rules out  $r= -1$, whence $qr^2=1$. By Definition \ref{defA5} $(A_{5_2})$, we obtain $a_{25}^{r_3r_4(X)}\in \{0,1\}$. 

When $a_{25}^{r_3r_4(X)}=0$, the reflection (\ref{eq-10}) shows $q^2r=1$. Then we deduce $r=q\in G'_3$ with $p\ne 3$ from $qr^2=1$. Hence $\cD=\cD_{13,3}^5$. 

When $a_{25}^{r_3r_4(X)}=-1$, the reflection 
\begin{align*}
r_4(X)\colon~ 
\Dchainfivew{$-1$}{$q^{-1}$}{$q$}{$q^{-1}$}{$-1$}{$q$}{$(qr)^{-1}$}{$-1$}{$r$}{$r^{-1}$}
\quad \Rightarrow \quad
r_3r_4(X)\colon~ \tau_{12534}~~
\begin{picture}(38,11)(0,3)
\put(1,1){\circle{2}}
\put(2,1){\line(1,0){10}}
\put(13,1){\circle{2}}
\put(13,2){\line(2,3){6}}
\put(14,1){\line(1,0){10}}
\put(25,1){\circle{2}}
\put(25,2){\line(-2,3){6}}
\put(19,12){\circle{2}}
\put(26,1){\line(1,0){10}}
\put(37,1){\circle{2}}
\put(1,4){\makebox[0pt]{\scriptsize $-1$}}
\put(7,3){\makebox[0pt]{\scriptsize $q^{-1}$}}
\put(12,4){\makebox[0pt]{\scriptsize $-1$}}
\put(19,3){\makebox[0pt]{\scriptsize $q$}}
\put(26,4){\makebox[0pt]{\scriptsize $-1$}}
\put(31,3){\makebox[0pt]{\scriptsize $q^{-1}$}}
\put(37,4){\makebox[0pt]{\scriptsize $q$}}
\put(13.3,8){\makebox[0pt]{\scriptsize $q^{-2}r^{-1}$}}
\put(23,8){\makebox[0pt][l]{\scriptsize $qr$}}
\put(24,14){\makebox[0pt]{\scriptsize $-q^{-1}r^{-2}$}}
\end{picture}
%\Dchainfivex{$-1$}{$q^{-1}$}{$-1$}{$q$}{$-1$}{$q^{-1}$}{$q$}{$q^{-2}r^{-1}$}{$qr$}{$-q^{-1}r^{-2}$}
\end{align*}
and $a_{15}^{r_2r_3r_4(X)}=0$ force $q^3r=1$. Together with $qr^2=1$ we obtain $r\in G'_5$. 
%Rewrite $X$ as
Then $\mathcal{D}$ is
\begin{align*} 
\Dchainfive{$-1$}{$r^2$}{$r^{-2}$}{$r^2$}{$r^{-2}$}{$r^2$}{$-1$}{$r^{-1}$}{$-1$}
\end{align*}
which is forbidden by Lemma \ref{l-5chainspe-a}$(b_1)$.

Finally take $(a,b,c,d,e)=(0,0,1,1,1)$. Then $A^{r_3(X)}=A_5$, $A^{r_4(X)}=\begin{pmatrix}2&-1&0&0&0\\-1&2&-1&0&0\\0&-1&2&-1&-1\\0&0&-1&2&-1\\0&0&-1&-1&2\end{pmatrix}$ and $A^{r_5(X)}=A_5$. Hence $qr\ne 1$. From 
\begin{align*}
X\colon~ 
\begin{picture}(50,4)(0,3)
\put(1,1){\circle{2}}
\put(2,1){\line(1,0){10}}
\put(13,1){\circle{2}}
\put(14,1){\line(1,0){10}}
\put(25,1){\circle{2}}
\put(26,1){\line(1,0){10}}
\put(37,1){\circle{2}}
\put(38,1){\line(1,0){10}}
\put(49,1){\circle*{2}}
\put(1,4){\makebox[0pt]{\scriptsize $-1$}}
\put(7,3){\makebox[0pt]{\scriptsize $q^{-1}$}}
\put(13,4){\makebox[0pt]{\scriptsize $q$}}
\put(19,3){\makebox[0pt]{\scriptsize $q^{-1}$}}
\put(25,4){\makebox[0pt]{\scriptsize $q$}}
\put(31,3){\makebox[0pt]{\scriptsize $q^{-1}$}}
\put(37,4){\makebox[0pt]{\scriptsize $-1$}}
\put(43,3){\makebox[0pt]{\scriptsize $r^{-1}$}}
\put(49,4){\makebox[0pt]{\scriptsize $-1$}}
\end{picture}
%\Dchainfivee{$-1$}{$q^{-1}$}{$q$}{$q^{-1}$}{$q$}{$q^{-1}$}{$-1$}{$r^{-1}$}{$-1$}
\quad \Rightarrow \quad
r_5(X)\colon~ 
\Dchainfive{$-1$}{$q^{-1}$}{$q$}{$q^{-1}$}{$q$}{$q^{-1}$}{$r^{-1}$}{$r$}{$-1$}
\end{align*}
the condition $a_{43}^{r_5(X)}=-1$ implies $r=-1$ and $p\ne 2$, which is excluded via Lemma \ref{l-5chainspe-a}$(c_2)$.

Step $3.6$. Let $q_{11}=q_{33}=q_{44}=-1$, $q_{22}\widetilde{q_{23}}=1$ and $q_{jj}\widetilde{q_{j-1,j}}=1$ for $j\in \{2,5\}$. Then $e=1$. Since $X$ has a good $A_5$ neighborhood,
we have either $(a,b,c,d,e)=(1,1,0,0,1)$ or $(a,b,c,d,e)=(0,0,1,1,1)$. Set $q\colon=q_{22}$, $r^{-1}\colon=\widetilde{q_{34}}$ and $s\colon=q_{55}$. If $(a,b,c,d,e)=(1,1,0,0,1)$, that is  $A^{r_3(X)}=\begin{pmatrix}2&-1&0&0&0\\-1&2&-1&-1&0\\0&-1&2&-1&0\\0&-1&-1&2&-1\\0&0&0&-1&2\end{pmatrix}$ and $A^{r_4(X)}=A^{r_5(X)}=A_5$. Then $qr\ne 1$ and $rs=1$. By $a_{32}^{r_4(X)}=-1$, we obtain $r=-1$ and $p\ne 2$.
By Lemma \ref{l-5chainspe-a} $(a_4)$, this configuration is ruled out.

When $(a,b,c,d,e)=(0,0,1,1,1)$, we have $A^{r_3(X)}=A_5$, $A^{r_4(X)}=\begin{pmatrix}2&-1&0&0&0\\-1&2&-1&0&0\\0&-1&2&-1&-1\\0&0&-1&2&-1\\0&0&-1&-1&2\end{pmatrix}$ and $A^{r_5(X)}=A_5$, which yield $qr=1$ and $rs\ne 1$.  The reflection of $X$
\begin{align*}
X\colon~ 
\begin{picture}(50,4)(0,3)
\put(1,1){\circle{2}}
\put(2,1){\line(1,0){10}}
\put(13,1){\circle{2}}
\put(14,1){\line(1,0){10}}
\put(25,1){\circle*{2}}
\put(26,1){\line(1,0){10}}
\put(37,1){\circle{2}}
\put(38,1){\line(1,0){10}}
\put(49,1){\circle{2}}
\put(1,4){\makebox[0pt]{\scriptsize $-1$}}
\put(7,3){\makebox[0pt]{\scriptsize $r$}}
\put(13,4){\makebox[0pt]{\scriptsize $r^{-1}$}}
\put(19,3){\makebox[0pt]{\scriptsize $r$}}
\put(25,4){\makebox[0pt]{\scriptsize $-1$}}
\put(31,3){\makebox[0pt]{\scriptsize $r^{-1}$}}
\put(37,4){\makebox[0pt]{\scriptsize $-1$}}
\put(43,3){\makebox[0pt]{\scriptsize $s^{-1}$}}
\put(49,4){\makebox[0pt]{\scriptsize $s$}}
\end{picture}
%\Dchainfivec{$-1$}{$r$}{$r^{-1}$}{$r$}{$-1$}{$r^{-1}$}{$-1$}{$s^{-1}$}{$s$}
\quad \Rightarrow \quad
r_3(X)\colon~
\begin{picture}(50,4)(0,3)
\put(1,1){\circle{2}}
\put(2,1){\line(1,0){10}}
\put(13,1){\circle{2}}
\put(14,1){\line(1,0){10}}
\put(25,1){\circle{2}}
\put(26,1){\line(1,0){10}}
\put(37,1){\circle{2}}
\put(38,1){\line(1,0){10}}
\put(49,1){\circle{2}}
\put(1,4){\makebox[0pt]{\scriptsize $-1$}}
\put(7,3){\makebox[0pt]{\scriptsize $r$}}
\put(13,4){\makebox[0pt]{\scriptsize $-1$}}
\put(19,3){\makebox[0pt]{\scriptsize $r^{-1}$}}
\put(26,4){\makebox[0pt]{\scriptsize $-1$}}
\put(31,3){\makebox[0pt]{\scriptsize $r$}}
\put(37,4){\makebox[0pt]{\scriptsize $r^{-1}$}}
\put(43,3){\makebox[0pt]{\scriptsize $s^{-1}$}}
\put(49,4){\makebox[0pt]{\scriptsize $s$}}
\end{picture}
%\Dchainfive{$-1$}{$r$}{$-1$}{$r^{-1}$}{$-1$}{$r$}{$r^{-1}$}{$s^{-1}$}{$s$}
\end{align*}
together with $a_{45}^{r_3(X)}=-1$, imply $r=-1$ and $p\ne 2$. This case is excluded by Lemma \ref{l-5chainspe-a}$(c_4)$.

Step $3.7$. Let $q_{22}=q_{33}=q_{55}=-1$, $q_{44}\widetilde{q_{34}}=1$ and suppose $q_{ii}\widetilde{q_{i,i+1}}=1$ for $i\in \{1,4\}$. Then $(a,b,c,d,e)=(1,1,0,0,1)$. Set $q\colon=q_{11}$, $r^{-1}\colon=\widetilde{q_{23}}$ and $s\colon=q_{44}$. By $a^{r_3(X)}_{24}=-1$ one obtains $rs\ne 1$. The reflection of $X$
\begin{align*}
X\colon~ 
\begin{picture}(50,4)(0,3)
\put(1,1){\circle{2}}
\put(2,1){\line(1,0){10}}
\put(13,1){\circle*{2}}
\put(14,1){\line(1,0){10}}
\put(25,1){\circle{2}}
\put(26,1){\line(1,0){10}}
\put(37,1){\circle{2}}
\put(38,1){\line(1,0){10}}
\put(49,1){\circle{2}}
\put(1,4){\makebox[0pt]{\scriptsize $q$}}
\put(7,3){\makebox[0pt]{\scriptsize $q^{-1}$}}
\put(13,4){\makebox[0pt]{\scriptsize $-1$}}
\put(19,3){\makebox[0pt]{\scriptsize $r^{-1}$}}
\put(25,4){\makebox[0pt]{\scriptsize $-1$}}
\put(31,3){\makebox[0pt]{\scriptsize $s^{-1}$}}
\put(37,4){\makebox[0pt]{\scriptsize $s$}}
\put(43,3){\makebox[0pt]{\scriptsize $s^{-1}$}}
\put(49,4){\makebox[0pt]{\scriptsize $-1$}}
\end{picture}
%\Dchainfiveb{$q$}{$q^{-1}$}{$-1$}{$r^{-1}$}{$-1$}{$s^{-1}$}{$s$}{$s^{-1}$}{$-1$}
\quad \Rightarrow \quad
r_2(X)\colon~ 
\begin{picture}(50,4)(0,3)
\put(1,1){\circle{2}}
\put(2,1){\line(1,0){10}}
\put(13,1){\circle{2}}
\put(14,1){\line(1,0){10}}
\put(25,1){\circle{2}}
\put(26,1){\line(1,0){10}}
\put(37,1){\circle{2}}
\put(38,1){\line(1,0){10}}
\put(49,1){\circle{2}}
\put(1,4){\makebox[0pt]{\scriptsize $-1$}}
\put(7,3){\makebox[0pt]{\scriptsize $q$}}
\put(13,4){\makebox[0pt]{\scriptsize $-1$}}
\put(19,3){\makebox[0pt]{\scriptsize $r$}}
\put(26,4){\makebox[0pt]{\scriptsize $r^{-1}$}}
\put(31,3){\makebox[0pt]{\scriptsize $s^{-1}$}}
\put(37,4){\makebox[0pt]{\scriptsize $s$}}
\put(43,3){\makebox[0pt]{\scriptsize $s^{-1}$}}
\put(49,4){\makebox[0pt]{\scriptsize $-1$}}
\end{picture}
%\Dchainfive{$-1$}{$q$}{$-1$}{$r$}{$r^{-1}$}{$s^{-1}$}{$s$}{$s^{-1}$}{$-1$}
\end{align*}
implies $qr=1$ by $a_{13}^{r_2(X)}=0$ and $r=-1$ by $a_{34}^{r_2(X)}=-1$. Then $q=r=-1$ and $p\ne 2$.  Invoking Lemma \ref{l-5chainspe-a} $(a_3)$, we conclude that this case is also impossible.

Step $3.8$. Let $q_{11}=q_{33}=q_{55}=-1$ and suppose $q_{ii}\widetilde{q_{i,i+1}}=q_{jj}\widetilde{q_{j-1,j}}=1$ for all $i\in \{2,4\}$ and $j\in \{2,4\}$. Then $(c,d)=(0,0)$. Hence $(a,b,c,d,e)=(1,1,0,0,1)$. 
Set $q\colon=q_{22}$ and $r\colon=q_{44}$. We obtain $qr\ne 1$ by $a^{r_3(X)}_{24}=-1$. The reflection of $X$
\begin{align*}
X\colon~ 
\begin{picture}(50,4)(0,3)
\put(1,1){\circle{2}}
\put(2,1){\line(1,0){10}}
\put(13,1){\circle{2}}
\put(14,1){\line(1,0){10}}
\put(25,1){\circle*{2}}
\put(26,1){\line(1,0){10}}
\put(37,1){\circle{2}}
\put(38,1){\line(1,0){10}}
\put(49,1){\circle{2}}
\put(1,4){\makebox[0pt]{\scriptsize $-1$}}
\put(7,3){\makebox[0pt]{\scriptsize $q^{-1}$}}
\put(13,4){\makebox[0pt]{\scriptsize $q$}}
\put(19,3){\makebox[0pt]{\scriptsize $q^{-1}$}}
\put(25,4){\makebox[0pt]{\scriptsize $-1$}}
\put(31,3){\makebox[0pt]{\scriptsize $r^{-1}$}}
\put(37,4){\makebox[0pt]{\scriptsize $r$}}
\put(43,3){\makebox[0pt]{\scriptsize $r^{-1}$}}
\put(49,4){\makebox[0pt]{\scriptsize $-1$}}
\end{picture}
%\Dchainfivec{$-1$}{$q^{-1}$}{$q$}{$q^{-1}$}{$-1$}{$r^{-1}$}{$r$}{$r^{-1}$}{$-1$}
\quad \Rightarrow \quad
r_3(X)\colon~ 
\begin{picture}(38,11)(0,3)
\put(1,1){\circle{2}}
\put(2,1){\line(1,0){10}}
\put(13,1){\circle{2}}
\put(13,2){\line(2,3){6}}
\put(14,1){\line(1,0){10}}
\put(25,1){\circle{2}}
\put(25,2){\line(-2,3){6}}
\put(19,12){\circle{2}}
\put(26,1){\line(1,0){10}}
\put(37,1){\circle{2}}
\put(1,4){\makebox[0pt]{\scriptsize $-1$}}
\put(7,3){\makebox[0pt]{\scriptsize $q^{-1}$}}
\put(12,4){\makebox[0pt]{\scriptsize $-1$}}
\put(19,3){\makebox[0pt]{\scriptsize $(qr)^{-1}$}}
\put(26,4){\makebox[0pt]{\scriptsize $-1$}}
\put(31,3){\makebox[0pt]{\scriptsize $r^{-1}$}}
\put(37,4){\makebox[0pt]{\scriptsize $-1$}}
\put(13.3,8){\makebox[0pt]{\scriptsize $q$}}
\put(23,8){\makebox[0pt][l]{\scriptsize $r$}}
\put(24,14){\makebox[0pt]{\scriptsize $-1$}}
\end{picture}
%\Dchainfivex{$-1$}{$q^{-1}$}{$-1$}{$(qr)^{-1}$}{$-1$}{$r^{-1}$}{$-1$}{$q$}{$r$}{$-1$}
\end{align*}
implies $q^2r=1$ by $a_{41}^{r_2r_3(X)}=0$. The condition $a_{45}^{r_2r_3(X)}=-1$ implies $qr=-1$ or $qr^2=1$. Then $q=r\in G'_3$ and $p\ne 3$. Hence we have $a_{32}^{r_4r_5(X)}=-2$ by the reflections of $X$
\begin{align*}
X\colon~ 
\begin{picture}(50,4)(0,3)
\put(1,1){\circle{2}}
\put(2,1){\line(1,0){10}}
\put(13,1){\circle{2}}
\put(14,1){\line(1,0){10}}
\put(25,1){\circle{2}}
\put(26,1){\line(1,0){10}}
\put(37,1){\circle{2}}
\put(38,1){\line(1,0){10}}
\put(49,1){\circle*{2}}
\put(1,4){\makebox[0pt]{\scriptsize $-1$}}
\put(7,3){\makebox[0pt]{\scriptsize $q^{-1}$}}
\put(13,4){\makebox[0pt]{\scriptsize $q$}}
\put(19,3){\makebox[0pt]{\scriptsize $q^{-1}$}}
\put(25,4){\makebox[0pt]{\scriptsize $-1$}}
\put(31,3){\makebox[0pt]{\scriptsize $r^{-1}$}}
\put(37,4){\makebox[0pt]{\scriptsize $r$}}
\put(43,3){\makebox[0pt]{\scriptsize $r^{-1}$}}
\put(49,4){\makebox[0pt]{\scriptsize $-1$}}
\end{picture}
%\Dchainfivee{$-1$}{$q^{-1}$}{$q$}{$q^{-1}$}{$-1$}{$r^{-1}$}{$r$}{$r^{-1}$}{$-1$}
\quad \Rightarrow \quad
r_5(X)\colon~ 
\begin{picture}(50,4)(0,3)
\put(1,1){\circle{2}}
\put(2,1){\line(1,0){10}}
\put(13,1){\circle{2}}
\put(14,1){\line(1,0){10}}
\put(25,1){\circle{2}}
\put(26,1){\line(1,0){10}}
\put(37,1){\circle*{2}}
\put(38,1){\line(1,0){10}}
\put(49,1){\circle{2}}
\put(1,4){\makebox[0pt]{\scriptsize $-1$}}
\put(7,3){\makebox[0pt]{\scriptsize $q^{-1}$}}
\put(13,4){\makebox[0pt]{\scriptsize $q$}}
\put(19,3){\makebox[0pt]{\scriptsize $q^{-1}$}}
\put(25,4){\makebox[0pt]{\scriptsize $-1$}}
\put(31,3){\makebox[0pt]{\scriptsize $r^{-1}$}}
\put(37,4){\makebox[0pt]{\scriptsize $-1$}}
\put(43,3){\makebox[0pt]{\scriptsize $r$}}
\put(49,4){\makebox[0pt]{\scriptsize $-1$}}
\end{picture}
%\Dchainfived{$-1$}{$q^{-1}$}{$q$}{$q^{-1}$}{$-1$}{$r^{-1}$}{$-1$}{$r$}{$-1$}
\end{align*}
\begin{align*}
 \Rightarrow \quad 
r_4r_5(X)\colon~
\begin{picture}(50,4)(0,3)
\put(1,1){\circle{2}}
\put(2,1){\line(1,0){10}}
\put(13,1){\circle{2}}
\put(14,1){\line(1,0){10}}
\put(25,1){\circle{2}}
\put(26,1){\line(1,0){10}}
\put(37,1){\circle{2}}
\put(38,1){\line(1,0){10}}
\put(49,1){\circle{2}}
\put(1,4){\makebox[0pt]{\scriptsize $-1$}}
\put(7,3){\makebox[0pt]{\scriptsize $q^{-1}$}}
\put(13,4){\makebox[0pt]{\scriptsize $q$}}
\put(19,3){\makebox[0pt]{\scriptsize $q^{-1}$}}
\put(26,4){\makebox[0pt]{\scriptsize $r^{-1}$}}
\put(31,3){\makebox[0pt]{\scriptsize $r$}}
\put(37,4){\makebox[0pt]{\scriptsize $-1$}}
\put(43,3){\makebox[0pt]{\scriptsize $r^{-1}$}}
\put(49,4){\makebox[0pt]{\scriptsize $r$}}
\end{picture}
%\Dchainfive{$-1$}{$q^{-1}$}{$q$}{$q^{-1}$}{$r^{-1}$}{$r$}{$-1$}{$r^{-1}$}{$r$}
\end{align*}
%imply $a_{32}^{r_4r_5(X)}=-2$. 
And the reflections of $X$
\begin{align*}
X\colon~ 
\begin{picture}(50,4)(0,3)
\put(1,1){\circle*{2}}
\put(2,1){\line(1,0){10}}
\put(13,1){\circle{2}}
\put(14,1){\line(1,0){10}}
\put(25,1){\circle{2}}
\put(26,1){\line(1,0){10}}
\put(37,1){\circle{2}}
\put(38,1){\line(1,0){10}}
\put(49,1){\circle{2}}
\put(1,4){\makebox[0pt]{\scriptsize $-1$}}
\put(7,3){\makebox[0pt]{\scriptsize $q^{-1}$}}
\put(13,4){\makebox[0pt]{\scriptsize $q$}}
\put(19,3){\makebox[0pt]{\scriptsize $q^{-1}$}}
\put(25,4){\makebox[0pt]{\scriptsize $-1$}}
\put(31,3){\makebox[0pt]{\scriptsize $r^{-1}$}}
\put(37,4){\makebox[0pt]{\scriptsize $r$}}
\put(43,3){\makebox[0pt]{\scriptsize $r^{-1}$}}
\put(49,4){\makebox[0pt]{\scriptsize $-1$}}
\end{picture}
%\Dchainfivea{$-1$}{$q^{-1}$}{$q$}{$q^{-1}$}{$-1$}{$r^{-1}$}{$r$}{$r^{-1}$}{$-1$}
\quad \Rightarrow  
r_1(X)\colon~
\begin{picture}(50,4)(0,3)
\put(1,1){\circle{2}}
\put(2,1){\line(1,0){10}}
\put(13,1){\circle*{2}}
\put(14,1){\line(1,0){10}}
\put(25,1){\circle{2}}
\put(26,1){\line(1,0){10}}
\put(37,1){\circle{2}}
\put(38,1){\line(1,0){10}}
\put(49,1){\circle{2}}
\put(1,4){\makebox[0pt]{\scriptsize $-1$}}
\put(7,3){\makebox[0pt]{\scriptsize $q$}}
\put(13,4){\makebox[0pt]{\scriptsize $-1$}}
\put(19,3){\makebox[0pt]{\scriptsize$q^{-1}$}}
\put(25,4){\makebox[0pt]{\scriptsize $-1$}}
\put(31,3){\makebox[0pt]{\scriptsize $r^{-1}$}}
\put(37,4){\makebox[0pt]{\scriptsize $r$}}
\put(43,3){\makebox[0pt]{\scriptsize $r^{-1}$}}
\put(49,4){\makebox[0pt]{\scriptsize $-1$}}
\end{picture}
%\Dchainfiveb{$-1$}{$q$}{$-1$}{$q^{-1}$}{$-1$}{$r^{-1}$}{$r$}{$r^{-1}$}{$-1$}
\end{align*}
\begin{align*}
 \Rightarrow \quad 
r_2r_1(X)\colon~ 
\begin{picture}(50,4)(0,3)
\put(1,1){\circle{2}}
\put(2,1){\line(1,0){10}}
\put(13,1){\circle{2}}
\put(14,1){\line(1,0){10}}
\put(25,1){\circle{2}}
\put(26,1){\line(1,0){10}}
\put(37,1){\circle{2}}
\put(38,1){\line(1,0){10}}
\put(49,1){\circle{2}}
\put(1,4){\makebox[0pt]{\scriptsize $q$}}
\put(7,3){\makebox[0pt]{\scriptsize $q^{-1}$}}
\put(13,4){\makebox[0pt]{\scriptsize $-1$}}
\put(19,3){\makebox[0pt]{\scriptsize $q$}}
\put(26,4){\makebox[0pt]{\scriptsize $q^{-1}$}}
\put(31,3){\makebox[0pt]{\scriptsize $r^{-1}$}}
\put(37,4){\makebox[0pt]{\scriptsize $r$}}
\put(43,3){\makebox[0pt]{\scriptsize $r^{-1}$}}
\put(49,4){\makebox[0pt]{\scriptsize $-1$}}
\end{picture}
%\Dchainfive{$q$}{$q^{-1}$}{$-1$}{$q$}{$q^{-1}$}{$r^{-1}$}{$r$}{$r^{-1}$}{$-1$}
\end{align*}
imply $a_{34}^{r_2r_1(X)}=-2$, which contradicts $(a_{32}^{r_4r_5(X)},a_{34}^{r_2r_1(X)})\ne (-2,-2)$.

Step $3.9$. Let $q_{22}=q_{33}=q_{44}=-1$ and suppose $q_{11}\widetilde{q_{12}}=q_{55}\widetilde{q_{45}}=1$. We have $e=1$. By Definition~\ref{defA5}, the only admissible tuples satisfy $(a,b,c,d,e)=(1,1,0,0,1)$ or $(a,b,c,d,e)=(0,0,1,1,1)$.
% we have either $(a,b,c,d,e)=(1,1,0,0,1)$ or $(a,b,c,d,e)=(0,0,1,1,1)$.  
Set $q\colon=q_{11}$, $r^{-1}\colon=\widetilde{q_{23}}$, $s^{-1}\colon=\widetilde{q_{34}}$ and $t\colon=q_{55}$. 

When $(a,b,c,d,e)=(1,1,0,0,1)$, we get \\ $A^{r_3(X)}=\begin{pmatrix}2&-1&0&0&0\\-1&2&-1&-1&0\\0&-1&2&-1&0\\0&-1&-1&2&-1\\0&0&0&-1&2\end{pmatrix}$ and $A^{r_4(X)}=A_5$. Thus $qr=st=1$ and $rs\ne 1$. The reflection of $X$
\begin{align*}
X\colon~ 
\begin{picture}(50,4)(0,3)
\put(1,1){\circle{2}}
\put(2,1){\line(1,0){10}}
\put(13,1){\circle*{2}}
\put(14,1){\line(1,0){10}}
\put(25,1){\circle{2}}
\put(26,1){\line(1,0){10}}
\put(37,1){\circle{2}}
\put(38,1){\line(1,0){10}}
\put(49,1){\circle{2}}
\put(1,4){\makebox[0pt]{\scriptsize $r^{-1}$}}
\put(7,3){\makebox[0pt]{\scriptsize $r$}}
\put(13,4){\makebox[0pt]{\scriptsize $-1$}}
\put(19,3){\makebox[0pt]{\scriptsize $r^{-1}$}}
\put(25,4){\makebox[0pt]{\scriptsize $-1$}}
\put(31,3){\makebox[0pt]{\scriptsize $s^{-1}$}}
\put(37,4){\makebox[0pt]{\scriptsize $-1$}}
\put(43,3){\makebox[0pt]{\scriptsize $s$}}
\put(49,4){\makebox[0pt]{\scriptsize $s^{-1}$}}
\end{picture}
%\Dchainfiveb{$r^{-1}$}{$r$}{$-1$}{$r^{-1}$}{$-1$}{$s^{-1}$}{$-1$}{$s$}{$s^{-1}$}
\quad \Rightarrow \quad
r_2(X)\colon~ 
\begin{picture}(50,4)(0,3)
\put(1,1){\circle{2}}
\put(2,1){\line(1,0){10}}
\put(13,1){\circle{2}}
\put(14,1){\line(1,0){10}}
\put(25,1){\circle{2}}
\put(26,1){\line(1,0){10}}
\put(37,1){\circle{2}}
\put(38,1){\line(1,0){10}}
\put(49,1){\circle{2}}
\put(1,4){\makebox[0pt]{\scriptsize $-1$}}
\put(7,3){\makebox[0pt]{\scriptsize $r^{-1}$}}
\put(13,4){\makebox[0pt]{\scriptsize $-1$}}
\put(19,3){\makebox[0pt]{\scriptsize $r$}}
\put(26,4){\makebox[0pt]{\scriptsize $r^{-1}$}}
\put(31,3){\makebox[0pt]{\scriptsize $s^{-1}$}}
\put(37,4){\makebox[0pt]{\scriptsize $-1$}}
\put(43,3){\makebox[0pt]{\scriptsize $s$}}
\put(49,4){\makebox[0pt]{\scriptsize $s^{-1}$}}
\end{picture}
%\Dchainfive{$-1$}{$r^{-1}$}{$-1$}{$r$}{$r^{-1}$}{$s^{-1}$}{$-1$}{$s$}{$s^{-1}$}
\end{align*}
implies $r=-1$ and $p\ne 2$ by $a_{34}^{r_2(X)}=-1$. And the reflection of $X$
\begin{align*}
X\colon~ 
\begin{picture}(50,4)(0,3)
\put(1,1){\circle{2}}
\put(2,1){\line(1,0){10}}
\put(13,1){\circle{2}}
\put(14,1){\line(1,0){10}}
\put(25,1){\circle{2}}
\put(26,1){\line(1,0){10}}
\put(37,1){\circle*{2}}
\put(38,1){\line(1,0){10}}
\put(49,1){\circle{2}}
\put(1,4){\makebox[0pt]{\scriptsize $-1$}}
\put(7,3){\makebox[0pt]{\scriptsize $-1$}}
\put(13,4){\makebox[0pt]{\scriptsize $-1$}}
\put(19,3){\makebox[0pt]{\scriptsize $-1$}}
\put(25,4){\makebox[0pt]{\scriptsize $-1$}}
\put(31,3){\makebox[0pt]{\scriptsize $s^{-1}$}}
\put(37,4){\makebox[0pt]{\scriptsize $-1$}}
\put(43,3){\makebox[0pt]{\scriptsize $s$}}
\put(49,4){\makebox[0pt]{\scriptsize $s^{-1}$}}
\end{picture}
%\Dchainfived{$-1$}{$-1$}{$-1$}{$-1$}{$-1$}{$s^{-1}$}{$-1$}{$s$}{$s^{-1}$}
\quad \Rightarrow \quad
r_4(X)\colon~ 
\begin{picture}(50,4)(0,3)
\put(1,1){\circle{2}}
\put(2,1){\line(1,0){10}}
\put(13,1){\circle{2}}
\put(14,1){\line(1,0){10}}
\put(25,1){\circle{2}}
\put(26,1){\line(1,0){10}}
\put(37,1){\circle{2}}
\put(38,1){\line(1,0){10}}
\put(49,1){\circle{2}}
\put(1,4){\makebox[0pt]{\scriptsize $-1$}}
\put(7,3){\makebox[0pt]{\scriptsize $-1$}}
\put(13,4){\makebox[0pt]{\scriptsize $-1$}}
\put(19,3){\makebox[0pt]{\scriptsize $-1$}}
\put(26,4){\makebox[0pt]{\scriptsize $s^{-1}$}}
\put(31,3){\makebox[0pt]{\scriptsize $s$}}
\put(37,4){\makebox[0pt]{\scriptsize $-1$}}
\put(43,3){\makebox[0pt]{\scriptsize $s^{-1}$}}
\put(49,4){\makebox[0pt]{\scriptsize $-1$}}
\end{picture}
%\Dchainfive{$-1$}{$-1$}{$-1$}{$-1$}{$s^{-1}$}{$s$}{$-1$}{$s^{-1}$}{$-1$}
\end{align*}
implies $s=-1$ and $p\ne 2$ by $a_{32}^{r_4(X)}=-1$. Hence this case is a contradiction to $rs\ne 1$.

When $(a,b,c,d,e)=(0,0,1,1,1)$, we obtain $A^{r_3(X)}=A_5$ and $A^{r_4(X)}=\begin{pmatrix}2&-1&0&0&0\\-1&2&-1&0&0\\0&-1&2&-1&-1\\0&0&-1&2&-1\\0&0&-1&-1&2\end{pmatrix}$. Hence $qr=rs=1$ and $st\ne 1$. The condition $a_{45}^{r_3(X)}=-1$ yields $s=-1$ and $p\ne 2$.
By Lemma \ref{l-5chainspe-a}$(c_4)$, this configuration cannot occur.

Step $4$. Suppose exactly four vertices are labeled $-1$. We proceed with five subcases as follows.

Step $4.1$. Let $q_{11}=q_{22}=q_{33}=q_{44}=-1$ and $q_{55}\widetilde{q_{45}}=1$. Then $e=1$. Hence, we have either $(a,b,c,d,e)=(1,1,0,0,1)$ or $(a,b,c,d,e)=(0,0,1,1,1)$. By the same reasoning as in Step $3.9$ shows that both cases are impossible.

Step $4.2$. Let $q_{22}=q_{33}=q_{44}=q_{55}=-1$ and $q_{11}\widetilde{q_{12}}=1$. Since $X$ has a good $A_5$ neighborhood, we have either $(a,b,c,d)=(1,1,0,0)$ or $(a,b,c,d)=(0,0,1,1)$. Set $q\colon=q_{11}$, $r^{-1}\colon=\widetilde{q_{23}}$, $s^{-1}\colon=\widetilde{q_{34}}$ and $t^{-1}\colon=\widetilde{q_{45}}$. If $(a,b,c,d)=(1,1,0,0)$.
Applying the same directional reflection as $(a,b,c,d,e)=(1,1,0,0,1)$ of Step $3.9$ for the same reason leads to a contradiction.

If $(a,b,c,d)=(0,0,1,1)$, that is $A^{r_3(X)}=A_5$ and $A^{r_4(X)}=\begin{pmatrix}2&-1&0&0&0\\-1&2&-1&0&0\\0&-1&2&-1&-1\\0&0&-1&2&-1\\0&0&-1&-1&2\end{pmatrix}$. Hence $qr=rs=1$ and $st\ne 1$. What's more, the reflection of $X$
\begin{align*}
X\colon~ 
\begin{picture}(50,4)(0,3)
\put(1,1){\circle{2}}
\put(2,1){\line(1,0){10}}
\put(13,1){\circle{2}}
\put(14,1){\line(1,0){10}}
\put(25,1){\circle*{2}}
\put(26,1){\line(1,0){10}}
\put(37,1){\circle{2}}
\put(38,1){\line(1,0){10}}
\put(49,1){\circle{2}}
\put(1,4){\makebox[0pt]{\scriptsize $s$}}
\put(7,3){\makebox[0pt]{\scriptsize $s^{-1}$}}
\put(13,4){\makebox[0pt]{\scriptsize $-1$}}
\put(19,3){\makebox[0pt]{\scriptsize $s$}}
\put(25,4){\makebox[0pt]{\scriptsize $-1$}}
\put(31,3){\makebox[0pt]{\scriptsize $s^{-1}$}}
\put(37,4){\makebox[0pt]{\scriptsize $-1$}}
\put(43,3){\makebox[0pt]{\scriptsize $t^{-1}$}}
\put(49,4){\makebox[0pt]{\scriptsize $-1$}}
\end{picture}
%\Dchainfivec{$s$}{$s^{-1}$}{$-1$}{$s$}{$-1$}{$s^{-1}$}{$-1$}{$t^{-1}$}{$-1$}
\quad \Rightarrow \quad
r_3(X)\colon~ 
\begin{picture}(50,4)(0,3)
\put(1,1){\circle{2}}
\put(2,1){\line(1,0){10}}
\put(13,1){\circle{2}}
\put(14,1){\line(1,0){10}}
\put(25,1){\circle{2}}
\put(26,1){\line(1,0){10}}
\put(37,1){\circle{2}}
\put(38,1){\line(1,0){10}}
\put(49,1){\circle{2}}
\put(1,4){\makebox[0pt]{\scriptsize $s$}}
\put(7,3){\makebox[0pt]{\scriptsize $s^{-1}$}}
\put(13,4){\makebox[0pt]{\scriptsize $s$}}
\put(19,3){\makebox[0pt]{\scriptsize $s^{-1}$}}
\put(26,4){\makebox[0pt]{\scriptsize $-1$}}
\put(31,3){\makebox[0pt]{\scriptsize $s$}}
\put(37,4){\makebox[0pt]{\scriptsize $s^{-1}$}}
\put(43,3){\makebox[0pt]{\scriptsize $t^{-1}$}}
\put(49,4){\makebox[0pt]{\scriptsize $-1$}}
\end{picture}
%\Dchainfive{$s$}{$s^{-1}$}{$s$}{$s^{-1}$}{$-1$}{$s$}{$s^{-1}$}{$t^{-1}$}{$-1$}
\end{align*}
implies $s=-1$ and $p\ne 2$ by $a_{45}^{r_3(X)}=-1$. By Lemma \ref{l-5chainspe-a} $(b_3)$, this configuration is ruled out.

Step $4.3$. Let $q_{11}=q_{22}=q_{33}=q_{55}=-1$ and $q_{44}\widetilde{q_{45}}=q_{44}\widetilde{q_{34}}=1$.  Then $(c,d)=(0,0)$. Hence in this case $(a,b,c,d,e)=(1,1,0,0,1)$. Set $q^{-1}\colon=\widetilde{q_{12}}$, $r^{-1}\colon=\widetilde{q_{23}}$ and $s\colon=q_{44}$. Then one obtains $rs\ne 1$ by $a^{r_3(X)}_{24}=-1$. The reflection of $X$
\begin{align*}
X\colon~ 
\begin{picture}(50,4)(0,3)
\put(1,1){\circle{2}}
\put(2,1){\line(1,0){10}}
\put(13,1){\circle*{2}}
\put(14,1){\line(1,0){10}}
\put(25,1){\circle{2}}
\put(26,1){\line(1,0){10}}
\put(37,1){\circle{2}}
\put(38,1){\line(1,0){10}}
\put(49,1){\circle{2}}
\put(1,4){\makebox[0pt]{\scriptsize $-1$}}
\put(7,3){\makebox[0pt]{\scriptsize $q^{-1}$}}
\put(13,4){\makebox[0pt]{\scriptsize $-1$}}
\put(19,3){\makebox[0pt]{\scriptsize $r^{-1}$}}
\put(25,4){\makebox[0pt]{\scriptsize $-1$}}
\put(31,3){\makebox[0pt]{\scriptsize $s^{-1}$}}
\put(37,4){\makebox[0pt]{\scriptsize $s$}}
\put(43,3){\makebox[0pt]{\scriptsize $s^{-1}$}}
\put(49,4){\makebox[0pt]{\scriptsize $-1$}}
\end{picture}
%\Dchainfiveb{$-1$}{$q^{-1}$}{$-1$}{$r^{-1}$}{$-1$}{$s^{-1}$}{$s$}{$s^{-1}$}{$-1$}
\quad \Rightarrow \quad
r_2(X)\colon~ 
\begin{picture}(50,4)(0,3)
\put(1,1){\circle{2}}
\put(2,1){\line(1,0){10}}
\put(13,1){\circle{2}}
\put(14,1){\line(1,0){10}}
\put(25,1){\circle{2}}
\put(26,1){\line(1,0){10}}
\put(37,1){\circle{2}}
\put(38,1){\line(1,0){10}}
\put(49,1){\circle{2}}
\put(1,4){\makebox[0pt]{\scriptsize $q^{-1}$}}
\put(7,3){\makebox[0pt]{\scriptsize $q$}}
\put(13,4){\makebox[0pt]{\scriptsize $-1$}}
\put(19,3){\makebox[0pt]{\scriptsize $r$}}
\put(26,4){\makebox[0pt]{\scriptsize $r^{-1}$}}
\put(31,3){\makebox[0pt]{\scriptsize $s^{-1}$}}
\put(37,4){\makebox[0pt]{\scriptsize $s$}}
\put(43,3){\makebox[0pt]{\scriptsize $s^{-1}$}}
\put(49,4){\makebox[0pt]{\scriptsize $-1$}}
\end{picture}
%\Dchainfive{$q^{-1}$}{$q$}{$-1$}{$r$}{$r^{-1}$}{$s^{-1}$}{$s$}{$s^{-1}$}{$-1$}
\end{align*}
implies $qr=1$ by $a_{13}^{r_2(X)}=0$ and $r=-1$ by $a_{34}^{r_2(X)}=-1$. Hence $q=r=-1$ and $p\ne 2$. By Lemma \ref{l-5chainspe-a} $(a_3)$, this configuration is excluded. 

Step $4.4$. Let $q_{11}=q_{33}=q_{44}=q_{55}=-1$ and $q_{22}\widetilde{q_{23}}=q_{22}\widetilde{q_{12}}=1$. From Definition \ref{defA5}, the tuple $(a, b, c, d)$ can only be $(1, 1, 0, 0)$ or $(0, 0, 1, 1)$.
%we have either $(a,b,c,d)=(1,1,0,0)$ or $(a,b,c,d)=(0,0,1,1)$.
Set $q\colon=q_{22}$, $r^{-1}\colon=\widetilde{q_{34}}$ and $s^{-1}\colon=\widetilde{q_{45}}$. If $(a,b,c,d)=(1,1,0,0)$, that is \\ $A^{r_3(X)}=\begin{pmatrix}2&-1&0&0&0\\-1&2&-1&-1&0\\0&-1&2&-1&0\\0&-1&-1&2&-1\\0&0&0&-1&2\end{pmatrix}$ and $A^{r_4(X)}=A_5$. Then $qr\ne 1$ and $rs=1$. The reflection of $X$
\begin{align*}
X\colon~ 
\begin{picture}(50,4)(0,3)
\put(1,1){\circle{2}}
\put(2,1){\line(1,0){10}}
\put(13,1){\circle{2}}
\put(14,1){\line(1,0){10}}
\put(25,1){\circle{2}}
\put(26,1){\line(1,0){10}}
\put(37,1){\circle*{2}}
\put(38,1){\line(1,0){10}}
\put(49,1){\circle{2}}
\put(1,4){\makebox[0pt]{\scriptsize $-1$}}
\put(7,3){\makebox[0pt]{\scriptsize $q^{-1}$}}
\put(13,4){\makebox[0pt]{\scriptsize $q$}}
\put(19,3){\makebox[0pt]{\scriptsize $q^{-1}$}}
\put(25,4){\makebox[0pt]{\scriptsize $-1$}}
\put(31,3){\makebox[0pt]{\scriptsize $r^{-1}$}}
\put(37,4){\makebox[0pt]{\scriptsize $-1$}}
\put(43,3){\makebox[0pt]{\scriptsize $r$}}
\put(49,4){\makebox[0pt]{\scriptsize $-1$}}
\end{picture}
%\Dchainfived{$-1$}{$q^{-1}$}{$q$}{$q^{-1}$}{$-1$}{$r^{-1}$}{$-1$}{$r$}{$-1$}
\quad \Rightarrow \quad
r_4(X)\colon~
\begin{picture}(50,4)(0,3)
\put(1,1){\circle{2}}
\put(2,1){\line(1,0){10}}
\put(13,1){\circle{2}}
\put(14,1){\line(1,0){10}}
\put(25,1){\circle{2}}
\put(26,1){\line(1,0){10}}
\put(37,1){\circle{2}}
\put(38,1){\line(1,0){10}}
\put(49,1){\circle{2}}
\put(1,4){\makebox[0pt]{\scriptsize $-1$}}
\put(7,3){\makebox[0pt]{\scriptsize $q^{-1}$}}
\put(13,4){\makebox[0pt]{\scriptsize $q$}}
\put(19,3){\makebox[0pt]{\scriptsize $q^{-1}$}}
\put(26,4){\makebox[0pt]{\scriptsize $r^{-1}$}}
\put(31,3){\makebox[0pt]{\scriptsize $r$}}
\put(37,4){\makebox[0pt]{\scriptsize $-1$}}
\put(43,3){\makebox[0pt]{\scriptsize $r^{-1}$}}
\put(49,4){\makebox[0pt]{\scriptsize $r$}}
\end{picture}
%\Dchainfive{$-1$}{$q^{-1}$}{$q$}{$q^{-1}$}{$r^{-1}$}{$r$}{$-1$}{$r^{-1}$}{$r$}
\end{align*}
implies $r=-1$ and $p\ne 2$ by $a_{32}^{r_4(X)}=-1$. By Lemma \ref{l-5chainspe-a} $(a_4)$, this configuration cannot occur.

If $(a,b,c,d)=(0,0,1,1)$, then $A^{r_4(X)}=\begin{pmatrix}2&-1&0&0&0\\-1&2&-1&0&0\\0&-1&2&-1&-1\\0&0&-1&2&-1\\0&0&-1&-1&2\end{pmatrix}$ and $A^{r_3(X)}=A_5$. In this case $qr=1$ and $rs\ne 1$. The reflection of $X$
\begin{align*}
X\colon~ 
\begin{picture}(50,4)(0,3)
\put(1,1){\circle{2}}
\put(2,1){\line(1,0){10}}
\put(13,1){\circle{2}}
\put(14,1){\line(1,0){10}}
\put(25,1){\circle*{2}}
\put(26,1){\line(1,0){10}}
\put(37,1){\circle{2}}
\put(38,1){\line(1,0){10}}
\put(49,1){\circle{2}}
\put(1,4){\makebox[0pt]{\scriptsize $-1$}}
\put(7,3){\makebox[0pt]{\scriptsize $r$}}
\put(13,4){\makebox[0pt]{\scriptsize $r^{-1}$}}
\put(19,3){\makebox[0pt]{\scriptsize $r$}}
\put(25,4){\makebox[0pt]{\scriptsize $-1$}}
\put(31,3){\makebox[0pt]{\scriptsize $r^{-1}$}}
\put(37,4){\makebox[0pt]{\scriptsize $-1$}}
\put(43,3){\makebox[0pt]{\scriptsize $s^{-1}$}}
\put(49,4){\makebox[0pt]{\scriptsize $-1$}}
\end{picture}
%\Dchainfivec{$-1$}{$r$}{$r^{-1}$}{$r$}{$-1$}{$r^{-1}$}{$-1$}{$s^{-1}$}{$-1$}
\quad \Rightarrow \quad
r_3(X)\colon~ 
\begin{picture}(50,4)(0,3)
\put(1,1){\circle{2}}
\put(2,1){\line(1,0){10}}
\put(13,1){\circle{2}}
\put(14,1){\line(1,0){10}}
\put(25,1){\circle{2}}
\put(26,1){\line(1,0){10}}
\put(37,1){\circle{2}}
\put(38,1){\line(1,0){10}}
\put(49,1){\circle{2}}
\put(1,4){\makebox[0pt]{\scriptsize $-1$}}
\put(7,3){\makebox[0pt]{\scriptsize $r$}}
\put(13,4){\makebox[0pt]{\scriptsize $-1$}}
\put(19,3){\makebox[0pt]{\scriptsize $r^{-1}$}}
\put(26,4){\makebox[0pt]{\scriptsize $-1$}}
\put(31,3){\makebox[0pt]{\scriptsize $r$}}
\put(37,4){\makebox[0pt]{\scriptsize $r^{-1}$}}
\put(43,3){\makebox[0pt]{\scriptsize $s^{-1}$}}
\put(49,4){\makebox[0pt]{\scriptsize $-1$}}
\end{picture}
%\Dchainfive{$-1$}{$r$}{$-1$}{$r^{-1}$}{$-1$}{$r$}{$r^{-1}$}{$s^{-1}$}{$-1$}
\end{align*}
implies $r=-1$ and $p\ne 2$ by $a_{45}^{r_3(X)}=-1$. Accordingly, this case is ruled out by Lemma \ref{l-5chainspe-a}$(b_3)$.

Step $4.5$. Let $q_{11}=q_{22}=q_{44}=q_{55}=-1$ and $q_{33}\widetilde{q_{34}}=q_{33}\widetilde{q_{23}}=1$. Then $(a,b)=(0,0)$. Thus, either $(a,b,c,d,e)=(0,0,1,1,2)$ or $(a,b,c,d,e)=(0,0,1,1,1)$. Set $q^{-1}\colon=\widetilde{q_{12}}$, $r\colon=q_{33}$ and $s^{-1}\colon=\widetilde{q_{45}}$. If $(a,b,c,d,e)=(0,0,1,1,2)$, that is $A^{r_3(X)}=A_5$, $A^{r_4(X)}=\begin{pmatrix}2&-1&0&0&0\\-1&2&-1&0&0\\0&-1&2&-1&-1\\0&0&-1&2&-1\\0&0&-1&-1&2\end{pmatrix}$ and $A^{r_5(X)}=\begin{pmatrix}2&-1&0&0&0\\-1&2&-1&0&0\\0&-1&2&-1&0\\0&0&-2&2&-1\\0&0&0&-1&2\end{pmatrix}$. Then $qr=1$ and $rs\ne 1$. On one hand, the reflection of $X$
\begin{align}\label{eq-11}
X\colon~
\begin{picture}(50,4)(0,3)
\put(1,1){\circle{2}}
\put(2,1){\line(1,0){10}}
\put(13,1){\circle*{2}}
\put(14,1){\line(1,0){10}}
\put(25,1){\circle{2}}
\put(26,1){\line(1,0){10}}
\put(37,1){\circle{2}}
\put(38,1){\line(1,0){10}}
\put(49,1){\circle{2}}
\put(1,4){\makebox[0pt]{\scriptsize $-1$}}
\put(7,3){\makebox[0pt]{\scriptsize $r$}}
\put(13,4){\makebox[0pt]{\scriptsize $-1$}}
\put(19,3){\makebox[0pt]{\scriptsize $r^{-1}$}}
\put(25,4){\makebox[0pt]{\scriptsize $r$}}
\put(31,3){\makebox[0pt]{\scriptsize $r^{-1}$}}
\put(37,4){\makebox[0pt]{\scriptsize $-1$}}
\put(43,3){\makebox[0pt]{\scriptsize $s^{-1}$}}
\put(49,4){\makebox[0pt]{\scriptsize $-1$}}
\end{picture}
%\Dchainfived{$-1$}{$r$}{$-1$}{$r^{-1}$}{$r$}{$r^{-1}$}{$-1$}{$s^{-1}$}{$-1$}
\quad \Rightarrow \quad
r_4(X)\colon~ 
\Dchainfivet{$-1$}{$r$}{$-1$}{$r^{-1}$}{$-1$}{$r$}{$(rs)^{-1}$}{$-1$}{$s$}{$s^{-1}$}
\end{align}
gives $s=-1$ or $rs^2=1$ by $a_{53}^{r_4(X)}=-1$. On the other hand, the reflection of $X$
\begin{align*}
X\colon~ 
\begin{picture}(50,4)(0,3)
\put(1,1){\circle{2}}
\put(2,1){\line(1,0){10}}
\put(13,1){\circle{2}}
\put(14,1){\line(1,0){10}}
\put(25,1){\circle{2}}
\put(26,1){\line(1,0){10}}
\put(37,1){\circle{2}}
\put(38,1){\line(1,0){10}}
\put(49,1){\circle*{2}}
\put(1,4){\makebox[0pt]{\scriptsize $-1$}}
\put(7,3){\makebox[0pt]{\scriptsize $r$}}
\put(13,4){\makebox[0pt]{\scriptsize $-1$}}
\put(19,3){\makebox[0pt]{\scriptsize $r^{-1}$}}
\put(25,4){\makebox[0pt]{\scriptsize $r$}}
\put(31,3){\makebox[0pt]{\scriptsize $r^{-1}$}}
\put(37,4){\makebox[0pt]{\scriptsize $-1$}}
\put(43,3){\makebox[0pt]{\scriptsize $s^{-1}$}}
\put(49,4){\makebox[0pt]{\scriptsize $-1$}}
\end{picture}
%\Dchainfivee{$-1$}{$r$}{$-1$}{$r^{-1}$}{$r$}{$r^{-1}$}{$-1$}{$s^{-1}$}{$-1$}
\quad \Rightarrow \quad
r_5(X)\colon~ 
\begin{picture}(50,4)(0,3)
\put(1,1){\circle{2}}
\put(2,1){\line(1,0){10}}
\put(13,1){\circle{2}}
\put(14,1){\line(1,0){10}}
\put(25,1){\circle{2}}
\put(26,1){\line(1,0){10}}
\put(37,1){\circle{2}}
\put(38,1){\line(1,0){10}}
\put(49,1){\circle{2}}
\put(1,4){\makebox[0pt]{\scriptsize $-1$}}
\put(7,3){\makebox[0pt]{\scriptsize $r$}}
\put(13,4){\makebox[0pt]{\scriptsize $-1$}}
\put(19,3){\makebox[0pt]{\scriptsize $r^{-1}$}}
\put(26,4){\makebox[0pt]{\scriptsize $r$}}
\put(31,3){\makebox[0pt]{\scriptsize $r^{-1}$}}
\put(37,4){\makebox[0pt]{\scriptsize $s^{-1}$}}
\put(43,3){\makebox[0pt]{\scriptsize $s$}}
\put(49,4){\makebox[0pt]{\scriptsize $-1$}}
\end{picture}
%\Dchainfive{$-1$}{$r$}{$-1$}{$r^{-1}$}{$r$}{$r^{-1}$}{$s^{-1}$}{$s$}{$-1$}
\end{align*}
implies $s\ne -1$. Hence we obtain $rs^2=1$. From Definition \ref{defA5}  $(A_{5_2})$, $a_{25}^{r_3r_4(X)}
\in \{0,1\}$. If $a_{25}^{r_3r_4(X)}=0$, the reflection (\ref{eq-11}) of $X$ implies $r^2s=1$. Then $r=s\in G'_3$ and $p\ne 3$. %Rewrite $X$ as
%\begin{align*}
%\begin{picture}(50,4)(0,3)
%\put(1,1){\circle{2}}
%\put(2,1){\line(1,0){10}}
%\put(13,1){\circle{2}}
%\put(14,1){\line(1,0){10}}
%\put(25,1){\circle{2}}
%\put(26,1){\line(1,0){10}}
%\put(37,1){\circle{2}}
%\put(38,1){\line(1,0){10}}
%\put(49,1){\circle{2}}
%\put(1,4){\makebox[0pt]{\scriptsize $-1$}}
%\put(7,3){\makebox[0pt]{\scriptsize $r$}}
%\put(13,4){\makebox[0pt]{\scriptsize $-1$}}
%\put(19,3){\makebox[0pt]{\scriptsize $r^{-1}$}}
%\put(26,4){\makebox[0pt]{\scriptsize $r$}}
%\put(31,3){\makebox[0pt]{\scriptsize $r^{-1}$}}
%\put(37,4){\makebox[0pt]{\scriptsize $-1$}}
%\put(43,3){\makebox[0pt]{\scriptsize $r^{-1}$}}
%\put(49,4){\makebox[0pt]{\scriptsize $-1$}}
%\end{picture}
    %\Dchainfive{$-1$}{$r$}{$-1$}{$r^{-1}$}{$r$}{$r^{-1}$}{$-1$}{$r^{-1}$}{$-1$}
%\end{align*}
Hence $\cD=\cD_{13,5}^5$. If $a_{25}^{r_3r_4(X)}=-1$, the reflection of $X$
\begin{align*}
r_4(X)\colon~
 \Dchainfivew{$-1$}{$r$}{$-1$}{$r^{-1}$}{$-1$}{$r$}{$(rs)^{-1}$}{$-1$}{$s$}{$s^{-1}$}
\quad \Rightarrow \quad
r_3r_4(X)\colon~ \tau_{12534}~~ 
\begin{picture}(38,11)(0,3)
\put(1,1){\circle{2}}
\put(2,1){\line(1,0){10}}
\put(13,1){\circle{2}}
\put(13,2){\line(2,3){6}}
\put(14,1){\line(1,0){10}}
\put(25,1){\circle{2}}
\put(25,2){\line(-2,3){6}}
\put(19,12){\circle{2}}
\put(26,1){\line(1,0){10}}
\put(37,1){\circle{2}}
\put(1,4){\makebox[0pt]{\scriptsize $-1$}}
\put(7,3){\makebox[0pt]{\scriptsize $r$}}
\put(12,4){\makebox[0pt]{\scriptsize $r^{-1}$}}
\put(19,3){\makebox[0pt]{\scriptsize $r$}}
\put(26,4){\makebox[0pt]{\scriptsize $-1$}}
\put(31,3){\makebox[0pt]{\scriptsize $r^{-1}$}}
\put(37,4){\makebox[0pt]{\scriptsize $r$}}
\put(13.3,8){\makebox[0pt]{\scriptsize $r^{-2}s^{-1}$}}
\put(23,8){\makebox[0pt][l]{\scriptsize $rs$}}
\put(24,14){\makebox[0pt]{\scriptsize $r^{-1}s^{-2}$}}
\end{picture}
%\Dchainfivex{$-1$}{$r$}{$r^{-1}$}{$r$}{$-1$}{$r^{-1}$}{$r$}{$r^{-2}s^{-1}$}{$rs$}{$-r^{-1}s^{-2}$}
\end{align*}
implies $r=-1$ or $r^3s=1$. If $r=-1$, we have $s=-1$ and $p\ne 2$ by $a_{15}^{r_2r_3r_4(X)}=0$, which is a contradiction. If $r^3s=1$, we obtain $s\in G'_5$ by $rs^2=1$.
Then $\mathcal{D}$ is
\begin{align*}
\begin{picture}(50,4)(0,3)
\put(1,1){\circle{2}}
\put(2,1){\line(1,0){10}}
\put(13,1){\circle{2}}
\put(14,1){\line(1,0){10}}
\put(25,1){\circle{2}}
\put(26,1){\line(1,0){10}}
\put(37,1){\circle{2}}
\put(38,1){\line(1,0){10}}
\put(49,1){\circle{2}}
\put(1,4){\makebox[0pt]{\scriptsize $-1$}}
\put(7,3){\makebox[0pt]{\scriptsize $s^{-2}$}}
\put(13,4){\makebox[0pt]{\scriptsize $-1$}}
\put(19,3){\makebox[0pt]{\scriptsize $s^2$}}
\put(26,4){\makebox[0pt]{\scriptsize $s^{-2}$}}
\put(31,3){\makebox[0pt]{\scriptsize $s^2$}}
\put(37,4){\makebox[0pt]{\scriptsize $-1$}}
\put(43,3){\makebox[0pt]{\scriptsize $s^{-1}$}}
\put(49,4){\makebox[0pt]{\scriptsize $-1$}}
\end{picture}
%\Dchainfive{$-1$}{$s^{-2}$}{$-1$}{$s^2$}{$s^{-2}$}{$s^2$}{$-1$}{$s^{-1}$}{$-1$}
\end{align*}
Accordingly, this case is excluded by Lemma \ref{l-5chainspe-a}$(b_2)$.

When $(a,b,c,d,e)=(0,0,1,1,1)$, we obtain $A^{r_3(X)}=A_5$, $A^{r_4(X)}=\begin{pmatrix}2&-1&0&0&0\\-1&2&-1&0&0\\0&-1&2&-1&-1\\0&0&-1&2&-1\\0&0&-1&-1&2\end{pmatrix}$ and $A^{r_5(X)}=A_5$. Then $qr=1$ and $rs\ne 1$. Since $a_{43}^{r_5(X)}=-1$, we obtain $s=-1$ and $p\ne 2$.  By Lemma \ref{l-5chainspe-a} $(c_3)$, this configuration cannot occur.

Step $5$. Suppose all five vertices are labelled $-1$. That is $q_{11}=q_{22}=q_{33}=q_{44}=q_{55}=-1$.  %By Definition \ref{defA5}, we have either $(a,b,c,d)=(1,1,0,0)$ or $(a,b,c,d)=(0,0,1,1)$. 
By the same argument as in Step $4.2$, this configuration is ruled out. 
\end{proof}

\begin{lemma}\label{lem:impossible6}
Suppose $\gDd _{\chi,E}$ admits the form
\begin{equation*}
	\Dchainsix{$q_{11}$}{$r$}{$q_{22}$}{$s$}{$q_{33}$}
{$t$}{$q_{44}$}{$u$}{$q_{55}$}{$v$}{$q_{66}$}
\end{equation*}
and has a good $A_6$ neighborhood. Then $q_{44}=-1$ and the configurations listed below cannot occur.
\begin{itemize}
	\item[$(a_1)$]\ \Dchainsix{$q$}{$q^{-1}$}{$q$}{$q^{-1}$}{$q$}{$q^{-1}$}{$-1$}{$-1$}{$-1$}{$-1$}{$-1$},\quad $q\ne -1, \ p\ne 2$
    \vspace{-2mm}
	\item[$(a_2)$]\ \Dchainsix{$-1$}{$q^{-1}$}{$q$}{$q^{-1}$}{$q$}{$q^{-1}$}{$-1$}{$-1$}{$-1$}{$-1$}{$-1$},\quad $q\ne -1, \ p\ne 2$
    \vspace{-2mm}
	\item[$(a_3)$]\ \Dchainsix{$q^{-1}$}{$q$}{$-1$}{$q^{-1}$}{$q$}{$q^{-1}$}{$-1$}{$-1$}{$-1$}{$-1$}{$-1$},\quad $q\ne -1, \ p\ne 2$
    \vspace{-2mm}
    \item[$(a_4)$]\ \Dchainsix{$-1$}{$q$}{$-1$}{$q^{-1}$}{$q$}{$q^{-1}$}{$-1$}{$-1$}{$-1$}{$-1$}{$-1$},\quad $q\ne -1, \ p\ne 2$
    \vspace{-2mm}
    \item[$(a_5)$]\ \Dchainsix{$q$}{$q^{-1}$}{$-1$}{$q$}{$-1$}{$q^{-1}$}{$-1$}{$-1$}{$-1$}{$-1$}{$-1$},\quad $q\ne -1, \ p\ne 2$
    \vspace{-2mm}
	\item[$(b_1)$]\ \Dchainsix{$-1$}{$-1$}{$-1$}{$-1$}{$-1$}{$-1$}{$-1$}{$q^{-1}$}{$-1$}{$-1$}{$-1$}, \quad $q\in G'_4, \ p\ne 2$
    \vspace{-2mm}
	\item[$(b_2)$]\ \Dchainsix{$-1$}{$-1$}{$-1$}{$-1$}{$-1$}{$-1$}{$-1$}{$q^{-1}$}{$q$}{$q^{-1}$}{$q$},\quad $q\ne -1, \ p\ne 2$
    \vspace{-2mm}
	\item[$(b_3)$]\ \Dchainsix{$-1$}{$-1$}{$-1$}{$-1$}{$-1$}{$-1$}{$-1$}{$q^{-1}$}{$q$}{$-1$}{$q$}, \quad $q\ne -1, \ p\ne 2$
    \vspace{-2mm}
	\item[$(b_4)$]\ \Dchainsix{$-1$}{$-1$}{$-1$}{$-1$}{$-1$}{$-1$}{$-1$}{$q^{-1}$}{$q$}{$q^{-1}$}{$-1$}, \quad $q\ne -1, \ p\ne 2$
    \vspace{-2mm}
	\item[$(b_5)$]\ \Dchainsix{$-1$}{$-1$}{$-1$}{$-1$}{$-1$}{$-1$}{$-1$}{$q^{-1}$}{$-1$}{$q$}{$-1$}, \quad $q\ne -1, \ p\ne 2$
    \vspace{-2mm}
	\item[$(b_6)$]\ \Dchainsix{$-1$}{$-1$}{$-1$}{$-1$}{$-1$}{$-1$}{$-1$}{$q$}{$-1$}{$q^{-1}$}{$q$}, \quad $q\ne -1, \ p\ne 2$
    \end{itemize}
\end{lemma}

\begin{proof}
If $q_{44}\neq-1$, then $A^{r_4(X)}=A_6$, contradicting Definition \ref{defA6}. We now split the remaining discussion into separate cases below.
\begin{itemize}
    \item[(i)] Let $X$ have the diagram \Dchainsix{$q$}{$q^{-1}$}{$-1$}{$q$}{$-1$}{$q^{-1}$}{$-1$}{$-1$}{$-1$}{$-1$}{$-1$}\, 
	with $q\ne -1$. From the reflection
    \begin{gather*}
	    	\Dchainsixd{$q$}{$q^{-1}$}{$-1$}{$q$}{$-1$}{$q^{-1}$}{$-1$}{$-1$}{$-1$}{$-1$}{$-1$}
		\quad \Rightarrow \quad  \DchainsixM{$q$}{$q^{-1}$}{$-1$}{$q$}{$-1$}{$-q^{-1}$}{$-1$}{$-1$}{$-1$}{$-1$}{$-1$}{$-1$}\\
    \end{gather*}
the condition $a^{r_5r_4}_{36}=0$ force $q=-1$ and $p\ne 2$, contradicting $q\ne -1$. Consequently, configuration $(a_1)$ is excluded.
    
    \item[(ii)] Suppose $X$ is given by \Dchainsix{$q^{-1}$}{$q$}{$-1$}{$q^{-1}$}{$q$}{$q^{-1}$}{$-1$}{$-1$}{$-1$}{$-1$}{$-1$} where $q\ne -1$. The reflection
 \begin{gather*}
    \Dchainsixd{$q^{-1}$}{$q$}{$-1$}{$q^{-1}$}{$q$}{$q^{-1}$}{$-1$}{$-1$}{$-1$}{$-1$}{$-1$} \quad \Rightarrow \quad  \DchainsixM{$q^{-1}$}{$q$}{$-1$}{$q^{-1}$}{$-1$}{$-q^{-1}$}{$-1$}{$-1$}{$-1$}{$q$}{$-1$}{$-1$}
    \end{gather*}
 together with $a^{r_5r_4}_{36}=0$ yields $q=1$, conflicting with $q\ne 1$. Thus case $(a_2)$ cannot hold.
 
An analogous argument, based on similar reflections and conditions, applies to cases $(a_3)$, $(a_4)$ and $(a_5)$, so these three configurations are ruled out.

    \item[(iii)] Let $X$ takes the form \Dchainsix{$-1$}{$-1$}{$-1$}{$-1$}{$-1$}{$-1$}{$-1$}{$q^{-1}$}{$-1$}{$-1$}{$-1$} where $q\in G'_4$. Using
    \begin{gather*}
    	\Dchainsixd{$-1$}{$-1$}{$-1$}{$-1$}{$-1$}{$-1$}{$-1$}{$q^{-1}$}{$-1$}{$-1$}{$-1$}\quad \Rightarrow \quad  \DchainsixMc{$-1$}{$-1$}{$-1$}{$-1$}{$-1$}{$-q^{-1}$}{$-1$}{$-1$}{$-1$}{$-1$}{$q$}{$-1$}
    \end{gather*}
the identity $a^{r_3r_4(X)}_{25}=0$ implies $q=1$, a contradiction to $q\ne 1$. Hence $(b_1)$ is eliminated.

Cases $(b_2)$-$(b_6)$ share identical reflection rules and restrictive conditions, so the same argument as for $(b_1)$  rules out all remaining $b$-type configurations. 
      \end{itemize}
\end{proof}

\begin{theorem}\label{theo.rank6}
Suppose ${\roots}^{[M]}$ is a finite set of roots of sporadic finite Cartan graphs of rank $6$. Then every admissable generalized Dynkin diagrams of $V$ occurs in rows $16$-$19$ of Table \ref{tab.1}.
\end{theorem}
\begin{proof}
By Theorem \ref{thm:goodnei},  either $C(M)$ is standard of type $E_6$ or there exists a point $X$ such that $A^X$ has a good $A_6$ neighborhood.

Case $e$. Suppose $C(M)$ is standard of
type $E_6$. Since $A^{X}=E_{6}$, we have $(2)_{q_{22}}(q_{22}\widetilde{q_{24}}-1)=(2)_{q_{ii}}(q_{ii}\widetilde{q_{i,i+1}}-1)=(2)_{q_{jj}}(q_{jj}\widetilde{q_{j-1,j}}-1)=0$ for all $i\in \{1,3,4,5\}$ and $j\in \{2,4,5,6\}$. Set $q\colon=\widetilde{q_{12}}$. If $q\neq-1$, then $q^2\neq1$. By Lemma \ref{jslemma}, we obtain $q_{ii}\neq-1$ for all $i\in \{1,2,3,4,5,6\}$ and $q_{22}\widetilde{q_{24}}=q_{ii}\widetilde{q_{i,i+1}}=q_{jj}\widetilde{q_{j-1,j}}=1$ for all $i\in \{1,3,4,5\}$ and $j\in \{2,4,5,6\}$. Otherwise, if any vertex is -1, then $q^2=1$ by $A^X=E_6$, which is a contradiction. Hence $\cD=\cD^6_{16,1}$. If $q=-1$, then $q_{ii}=-1$ for all $i \in \{1,2,3,4,5,6\}$ by $A^X=E_6$. Hence $\cD=\cD^6_{16',1}$.

Case $a$. Suppose that $X$ has a good $A_6$ neighborhood.
Then $A^{X}=A_{6}$ implies $(2)_{q_{ii}}(q_{ii}\widetilde{q_{i,i+1}}-1)=(2)_{q_{jj}}(q_{jj}\widetilde{q_{j-1,j}}-1)=0$, for all $i\in \{1,2,3,4,5\}$ and $j\in \{2,3,4,5,6\}$. 
From Lemma \ref{lem:impossible6}, we know $q_{44}=-1$ in all admissible configurations. We further classify all cases into six categories (Step $1$- Step $6$) according to the number of vertices with $q_{ii}=-1$.

Step $1$. Suppose exactly one vertex is labelled $-1$. By Lemma \ref{lem:impossible6}, we have $q_{44}=-1$, leaving only one case $\colon$ $q_{44}=-1$ and $q_{ii}\widetilde{q_{i,i+1}}=q_{jj}\widetilde{q_{j-1,j}}=1$ for all $i\in \{1,2,3,5\}$ and $j\in \{2,3,5,6\}$. It follows that $a=0$. Set $q\colon=q_{11}$ and $r\colon=q_{55}$. We get $qr\ne 1$ by $a_{35}^{r_4(X)}=-1$.
We next analyze the reflections of $X$
\vspace{-3mm}
\begin{gather*}
X\colon \quad 
\Dchainsixd{$q$}{$q^{-1}$}{$q$}{$q^{-1}$}{$q$}{$q^{-1}$}{$-1$}{$r^{-1}$}{$r$}{$r^{-1}$}{$r$}
\quad 
\Rightarrow \quad r_4(X)\colon \quad 
\DchainsixM{$q$}{$q^{-1}$}{$q$}{$q^{-1}$}{$-1$}{$(qr)^{-1}$}{$-1$}{$r^{-1}$}{$r$}{$q$}{$r$}{$-1$}
\end{gather*}
\vspace{-3mm}
\begin{align*}
	\quad \Rightarrow r_5r_4(X)\colon  \tau_{123564}~
\rule[-4\unitlength]{0pt}{5\unitlength}
\begin{picture}(50,16)(0,3)
\put(1,3){\circle{2}}
\put(2,3){\line(1,0){10}}
\put(13,3){\circle{2}}
\put(14,3){\line(1,0){10}}
\put(25,3){\circle{2}}
%\put(25,2){\line(2,3){6}}
\put(26,3){\line(1,0){10}}
\put(37,3){\circle{2}}
\put(38,3){\line(1,1){7}}
\put(46,10){\circle{2}}
\put(38,3){\line(1,-1){7}}
\put(46,-4){\circle{2}}
\put(1,6){\makebox[0pt]{\scriptsize $q$}}
\put(7,5){\makebox[0pt]{\scriptsize $q^{-1}$}}
\put(13,6){\makebox[0pt]{\scriptsize $q$}}
\put(18,5){\makebox[0pt]{\scriptsize $q^{-1}$}}
\put(25,6){\makebox[0pt]{\scriptsize $(qr)^{-1}$}}
\put(32,5){\makebox[0pt]{\scriptsize $qr$}}
\put(37,5){\makebox[0pt]{\scriptsize $-1$}}
\put(39,-3){\makebox[0pt]{\scriptsize $r^{-1}$}}
\put(49,-4){\makebox[0pt]{\scriptsize $r$}}
%\put(25,9){\makebox[0pt]{\scriptsize $q^{-2}r^{-1}$}}
\put(41,8){\makebox[0pt]{\scriptsize $r$}}
\put(49,10){\makebox[0pt]{\scriptsize $-1$}}
\end{picture}.
\end{align*}
From $a_{32}^{r_5r_4(X)}=-1$, we obtain the alternatives $rq=-1$ or $q^{2}r=1$. The condition $a_{36}^{r_5r_4(X)}=0$ further yields $qr^2=1$. Combining these relations gives $q=r\in G'_3$, so $\gDd=\gDd_{17,1}^6$ and $p\ne 3$.

Step $2$. Suppose exactly two vertices carry the label $-1$. Since $q_{44}=-1$, we divide the analysis into five mutually exclusive subcases.

Step $2.1$. Let $q_{11}=q_{44}=-1$ and $q_{ii}\widetilde{q_{i,i+1}}=q_{jj}\widetilde{q_{j-1,j}}=1$ for all $i\in \{2,3,5\}$ and $j\in \{2,3,5,6\}$. This yields $a=0$. Set $q\colon=q_{22}$ and $r\colon=q_{55}$. Then we have $qr \ne 1$ by $a_{35}^{r_4(X)}=-1$. The reflection of $X$
  \begin{align*}
	X\colon\ \Dchainsixd{$-1$}{$q^{-1}$}{$q$}{$q^{-1}$}{$q$}{$q^{-1}$}{$-1$}{$r^{-1}$}{$r$}{$r^{-1}$}{$r$}\quad\Rightarrow\quad r_4(X)\colon\ \DchainsixM{$-1$}{$q^{-1}$}{$q$}{$q^{-1}$}{$-1$}{$(qr)^{-1}$}{$-1$}{$r^{-1}$}{$r$}{$q$}{$r$}{$-1$}
  \end{align*}
implies $r=q^2$ by $a_{25}^{r_3r_4(X)}=0$ and $q=r^2$ by $a_{36}^{r_5r_4(X)}=0$. Then $r=q\in G'_3$. Hence $\gDd =\gDd_{18,19}^6$ and $p\ne 3$. 

Step $2.2$. Let $q_{22}=q_{44}=-1$ and $q_{ii}\widetilde{q_{i,i+1}}=q_{jj}\widetilde{q_{j-1,j}}=1$ for all $i\in \{1,3,5\}$ and $j\in \{3,5,6\}$. These relations force $a=0$. Set $q\colon=q_{11}, r\colon=q_{33}$ and $s\colon=q_{55}$. Then we obtain $qr=1$ and $rs\ne 1$. The reflection of $X$
  \begin{align*}
  	X\colon\ \Dchainsixd{$q$}{$q^{-1}$}{$-1$}{$q$}{$q^{-1}$}{$q$}{$-1$}{$s^{-1}$}{$s$}{$s^{-1}$}{$s$}\quad \Rightarrow \quad r_4(X)\colon\ \DchainsixM{$q$}{$q^{-1}$}{$-1$}{$q$}{$-1$}{$qs^{-1}$}{$-1$}{$s^{-1}$}{$s$}{$q^{-1}$}{$s$}{$-1$}
  \end{align*}
  implies $s=q^2$ by $a_{25}^{r_3r_4(X)}=0$ and $q=s^2$ by $a_{36}^{r_5r_4(X)}=0$. Hence $q\in G'_3$ and $p\ne 3$. Then the reflections of $X$
\begin{align*}
  X\colon\ \Dchainsixb{$q$}{$q^{-1}$}{$-1$}{$q$}{$q^{-1}$}{$q$}{$-1$}{$s^{-1}$}{$s$}{$s^{-1}$}{$s$}\quad 
  	\Rightarrow \quad r_2(X)\colon\ \Dchainsixc{$-1$}{$q$}{$-1$}{$q^{-1}$}{$-1$}{$q$}{$-1$}{$s^{-1}$}{$s$}{$s^{-1}$}{$s$}
\end{align*}
\begin{align*}
  \Rightarrow\quad r_3r_2(X)\colon\ \Dchainsix{$-1$}{$q$}{$q^{-1}$}{$q$}{$-1$}{$q^{-1}$}{$q$}{$s^{-1}$}{$s$}{$s^{-1}$}{$s$}
  \end{align*}
imply $(a_{45}^{r_3r_2r_1(X)},a_{45}^{r_3r_2(X)})=(-2,-2)$, which contradicts $(a_{45}^{r_3r_2r_1(X)},a_{45}^{r_3r_2(X)})\ne (-2,-2)$. 

Step $2.3$. Let $q_{33}=q_{44}=-1$ and $q_{ii}\widetilde{q_{i,i+1}}=q_{jj}\widetilde{q_{j-1,j}}=1$ for all $i\in \{1,2,5\}$ and $j\in \{2,5,6\}$. Hence $a=0$. Set $q\colon=q_{11}, \ r^{-1}\colon=\widetilde{q_{34}},\ s\colon=q_{55}$. By Lemma \ref{jslemma}, we have $qr=1$ and $rs\ne 1$. Then the reflection of $X$
\begin{align*}
  X\colon\ \Dchainsixc{$r^{-1}$}{$r$}{$r^{-1}$}{$r$}{$-1$}{$r^{-1}$}{$-1$}{$s^{-1}$}{$s$}{$s^{-1}$}{$s$}\quad 
  	\Rightarrow \quad r_3(X)\colon\ \Dchainsix{$r^{-1}$}{$r$}{$-1$}{$r^{-1}$}{$-1$}{$r$}{$r^{-1}$}{$s^{-1}$}{$s$}{$s^{-1}$}{$s$}
\end{align*}
implies $r=-1$ and $p\ne 2$ by $a_{45}^{r_3(X)}=-1$. By Lemma \ref{lem:impossible6} $(b_2)$, this configuration cannot occur.

Step $2.4$. Let $q_{44}=q_{55}=-1$ and $q_{ii}\widetilde{q_{i,i+1}}=q_{jj}\widetilde{q_{j-1,j}}=1$ for all $i\in \{1,2,3\}$ and $j\in \{2,3,6\}$. Set $q\colon=q_{11}, r^{-1}\colon=\widetilde{q_{45}}$ and $s\colon=q_{66}$. Since $X$ has a good $A_6$ neighborhood, we have $qr\ne 1$. And we obtain $r=-1$ and $p\ne 2$ by $a_{43}^{r_5(X)}=-1$. The reflection of X
  \begin{align*}
  	X\colon\ \Dchainsixd{$q$}{$q^{-1}$}{$q$}{$q^{-1}$}{$q$}{$q^{-1}$}{$-1$}{$-1$}{$-1$}{$s^{-1}$}{$s$}\quad \Rightarrow \quad r_4(X)\colon \ \DchainsixM{$q$}{$q^{-1}$}{$q$}{$q^{-1}$}{$-1$}{$-q^{-1}$}{$-1$}{$s^{-1}$}{$s$}{$q$}{$-1$}{$-1$}
  \end{align*}
implies $q^2=-1$ by $a_{25}^{r_3r_4(X)}=0$ and $qs=-1$ by $a_{36}^{r_5r_4(X)}=0$. Hence $q=s\in G'_4$ and $p\ne 2$. %Rewrite $X$ as
%\begin{align*}
%    \Dchainsix{$q$}{$-q$}{$q$}{$-q$}{$q$}{$-q$}{$-1$}{$-1$}{$-1$}{$-q$}{$q$}
%\end{align*}
%and we have 
 Then $\gDd=\gDd_{19,1}^6$ and $p\not=2$.

Step $2.5$. Let $q_{44}=q_{66}=-1$ and $q_{ii}\widetilde{q_{i,i+1}}=q_{jj}\widetilde{q_{j-1,j}}=1$ for all $i\in \{1,2,3,5\}$ and $j\in \{2,3,5\}$. Then $a=0$. Set $q\colon=q_{11}$ and $r\colon=q_{55}$. Then $qr \ne 1$ by $a^{r_4(X)}_{35}=-1$. Hence we obtain $r=q^{-2}$ by $a^{r_3r_4(X)}_{25}=0$ and $q= r^{-2}$ by $a^{r_5r_4(X)}_{36}=-1$. Thus $r=q\in G'_3$ and $p\ne 3$.
Then the reflections of $X$
\begin{align*}
  	X\colon\ \Dchainsixf{$q$}{$q^{-1}$}{$q$}{$q^{-1}$}{$q$}{$q^{-1}$}{$-1$}{$q^{-1}$}{$q$}{$q^{-1}$}{$-1$} \quad 
  	\Rightarrow \quad r_6(X)\colon\ 
  	\Dchainsixe{$q$}{$q^{-1}$}{$q$}{$q^{-1}$}{$q$}{$q^{-1}$}{$-1$}{$q^{-1}$}{$-1$}{$q$}{$-1$} 
  \end{align*}
\vspace{-4mm}
\begin{align*}
  \Rightarrow	r_5r_6(X)\colon \ \Dchainsix{$q$}{$q^{-1}$}{$q$}{$q^{-1}$}{$q$}{$q^{-1}$}{$q^{-1}$}{$q$}{$-1$}{$q^{-1}$}{$q$}
  \end{align*}
imply $q=-1$ and $p\ne 2$ by $a^{r_5r_6(X)}_{43}=-1$, which  contradicts $q\in G'_3$. 

Step $3$. Suppose exactly three vertices carry the label $-1$. Since $q_{44}=-1$ is fixed, we divide the analysis into eight mutually exclusive subcases.

Step $3.1$. Let $q_{11}=q_{22}=q_{44}=-1$ and $q_{ii}\widetilde{q_{i,i+1}}=q_{jj}\widetilde{q_{j-1,j}}=1$ for all $i\in \{3,5\}$ and $j\in \{3,5,6\}$. Then $a=0$. Set $q^{-1}\colon=\widetilde{q_{12}}, r\colon=q_{33}$ and $s\colon=q_{55}$. We obtain $rs \ne 1$ by $a_{35}^{r_4(X)}=-1$ and $qr=1$ by $a_{13}^{r_2(X)}=0$. Then the reflection of $X$
 \begin{align*}
    	X\colon\ \Dchainsixd{$-1$}{$r$}{$-1$}{$r^{-1}$}{$r$}{$r^{-1}$}{$-1$}{$s^{-1}$}{$s$}{$s^{-1}$}{$s$}\quad\Rightarrow\quad r_4(X)\colon\ \DchainsixM{$-1$}{$r$}{$-1$}{$r^{-1}$}{$-1$}{$(rs)^{-1}$}{$-1$}{$s^{-1}$}{$s$}{$r$}{$s$}{$-1$}
    \end{align*}
gives $s=r^{-2}$ by $a_{25}^{r_3r_4(X)}=0$ and $r=s^{-2}$ by $a_{36}^{r_5r_4(X)}=0$. Hence $r=s\in G'_3$ and $p\ne 3$. Then $\cD=\cD_{18,16}^6$. 

Step $3.2$. Let $q_{11}=q_{33}=q_{44}=-1$ and $q_{ii}\widetilde{q_{i,i+1}}=q_{jj}\widetilde{q_{j-1,j}}=1$ for all $i\in \{2,5\}$ and $j\in \{2,5,6\}$. hese conditions force $a=0$. Combining Definition \ref{defA6} and the conclusion of Step $2.3$, we exclude the present configuration. 

For the symmetric counterpart  $q_{33}=q_{44}=q_{66}=-1$ and $q_{ii}\widetilde{q_{i,i+1}}-1=q_{jj}\widetilde{q_{j-1,j}}-1=0$ for all $i\in \{1,2,5\}$ and $j\in \{2,5\}$, Lemma~\ref{l-7chainspe-a}$(b_4)$implies that this symmetric case cannot occur.

Step $3.3$. Let $q_{11}=q_{44}=q_{55}=-1$ and $q_{ii}\widetilde{q_{i,i+1}}=q_{jj}\widetilde{q_{j-1,j}}=1$ for all $i\in \{2,3\}$ and $j\in \{2,3,6\}$. Definition \ref{defA6} restricts $a \in \{0,1\}$. Set $q\colon=q_{22}, r^{-1}\colon=\widetilde{q_{45}}$ and $s\colon=q_{66}$. By $a_{35}^{r_4(X)}=-1$, we obtain $qr\ne 1$.  We further separate the remaining argument into two subcases.

Step $3.3.1$. Let $a=0$. The condition $a^{r_5(X)}_{46}=0$ enforces $rs=1$. Then the reflection of $X$
%\\ $A^{r_5(X)}=\begin{pmatrix}2&-1&0&0&0&0\\-1&2&-1&0&0&0\\0&-1&2&-1&0&0\\0&0&-1&2&-1&0\\0&0&0&-1&2&-1\\0&0&0&0&-1&2\end{pmatrix}$. Then $rs=1$ and the reflection of $X$
\begin{align*}
        X\colon\ \Dchainsixe{$-1$}{$q^{-1}$}{$q$}{$q^{-1}$}{$q$}{$q^{-1}$}{$-1$}{$r^{-1}$}{$-1$}{$r$}{$r^{-1}$}\quad
        \Rightarrow\quad r_5(X)\colon\ \Dchainsix{$-1$}{$q^{-1}$}{$q$}{$q^{-1}$}{$q$}{$q^{-1}$}{$r^{-1}$}{$r$}{$-1$}{$r^{-1}$}{$-1$}
    \end{align*}
implies $r=-1$ and $p\ne 2$ by $a_{43}^{r_5(X)}=-1$. By Lemma~\ref{lem:impossible6} $(a_2)$, this configuration cannot occur.

Step $3.3.2$. Suppose $a=1$. Here $a^{r_5(X)}_{46}=-1$, which yields $rs\ne 1$. 
%then \\ $A^{r_5(X)}=\begin{pmatrix}2&-1&0&0&0&0\\-1&2&-1&0&0&0\\0&-1&2&-1&0&0\\0&0&-1&2&-1&-1\\0&0&0&-1&2&-1\\0&0&0&-1&-1&2\end{pmatrix}$. Then $rs\ne 1$. 
We inspect the following reflection of $X$
\begin{align*}
X\colon\ \Dchainsixe{$-1$}{$q^{-1}$}{$q$}{$q^{-1}$}{$q$}{$q^{-1}$}{$-1$}{$r^{-1}$}{$-1$}{$s^{-1}$}{$s$}
\quad \Rightarrow \quad 
r_5(X)\colon\ \Dchainsixze{$-1$}{$q^{-1}$}{$q$}{$q^{-1}$}{$q$}{$q^{-1}$}{$r^{-1}$}{$(rs)^{-1}$}{$-1$}{$r$}{$s$}{$-1$}
\end{align*}.
The condition $a_{43}^{r_5(X)}=-1$ implies $r=-1$ and $p\ne 2$. Next we examine further successive reflections of $X$
    \begin{align*}
    	X\colon\ \Dchainsixa{$-1$}{$q^{-1}$}{$q$}{$q^{-1}$}{$q$}{$q^{-1}$}{$-1$}{$-1$}{$-1$}{$s^{-1}$}{$s$}\quad
    	\Rightarrow\quad r_1(X)\colon\ \Dchainsixb{$-1$}{$q$}{$-1$}{$q^{-1}$}{$q$}{$q^{-1}$}{$-1$}{$-1$}{$-1$}{$s^{-1}$}{$s$}
    \end{align*}
\vspace{-4mm}
    \begin{align*}
        \Rightarrow  r_2r_1(X)\colon\  \Dchainsixc{$q$}{$q^{-1}$}{$-1$}{$q$}{$-1$}{$q^{-1}$}{$-1$}{$-1$}{$-1$}{$s^{-1}$}{$s$}
    \end{align*}
\vspace{-4mm}
\begin{align*}
    	\ \Rightarrow\ r_3r_2r_1(X)\colon\ \Dchainsix{$q$}{$q^{-1}$}{$q$}{$q^{-1}$}{$-1$}{$q$}{$q^{-1}$}{$-1$}{$-1$}{$s^{-1}$}{$s$}
    \end{align*}.
Combining $a^{r_3r_2r_1(X)}_{45}=-1$ gives
$q=-1$ and $p\ne 2$, contradicting the earlier constraint $qr\ne 1$.

Step $3.4$. Let $q_{11}=q_{44}=q_{66}=-1$ and $q_{ii}\widetilde{q_{i,i+1}}=q_{jj}\widetilde{q_{j-1,j}}=1$ for all $i\in \{2,3,5\}$ and $j\in \{2,3,5\}$. Hence $a=0$. Set $q\colon=q_{22}$ and $r\colon=q_{55}$. By $a^{r_4(X)}_{35}=-1$, we obtain $qr \ne 1$. Then the reflections of $X$
\begin{align*}
    X\colon\ \Dchainsixf{$-1$}{$q^{-1}$}{$q$}{$q^{-1}$}{$q$}{$q^{-1}$}{$-1$}{$r^{-1}$}{$r$}{$r^{-1}$}{$-1$}\quad
    \Rightarrow\quad r_6(X)\colon\ \Dchainsixe{$-1$}{$q^{-1}$}{$q$}{$q^{-1}$}{$q$}{$q^{-1}$}{$-1$}{$r^{-1}$}{$-1$}{$r$}{$-1$}
\end{align*}
\vspace{-4mm}
\begin{align*}
     \Rightarrow\ r_5r_6(X)\colon\ \Dchainsix{$-1$}{$q^{-1}$}{$q$}{$q^{-1}$}{$q$}{$q^{-1}$}{$r^{-1}$}{$r$}{$-1$}{$r^{-1}$}{$r$}
    \end{align*}
imply $r=-1$ and $p\ne 2$ by $a_{43}^{r_5r_6(X)}=-1$. By Lemma~\ref{lem:impossible6} $(a_2)$, this case is ruled out.

Step $3.5$. Let $q_{22}=q_{33}=q_{44}=-1$ and $q_{ii}\widetilde{q_{i,i+1}}=q_{jj}\widetilde{q_{j-1,j}}=1$ for all $i\in \{1,5\}$ and $j\in \{5,6\}$. Hence $a=0$. From Definition \ref{defA6} and Step $2.3$, this case is excluded.

Consider the symmetric case that $q_{33}=q_{44}=q_{55}=-1$ and $q_{ii}\widetilde{q_{i,i+1}}=q_{jj}\widetilde{q_{j-1,j}}=1$ for all $i\in \{1,2\}$ and $j\in \{2,6\}$. Set $q\colon=q_{11}, r^{-1}\colon=\widetilde{q_{34}}, s^{-1}\colon=\widetilde{q_{45}}$ and $t\colon=q_{66}$. In this case we have $qr=1$ and $rs\ne 1$. The reflection of $X$
\begin{align*}
    	X\colon\ \Dchainsixc{$r^{-1}$}{$r$}{$r^{-1}$}{$r$}{$-1$}{$r^{-1}$}{$-1$}{$s^{-1}$}{$-1$}{$t^{-1}$}{$t$}\quad\Rightarrow\quad r_3(X)\colon\ \Dchainsix{$r^{-1}$}{$r$}{$-1$}{$r^{-1}$}{$-1$}{$r$}{$r^{-1}$}{$s^{-1}$}{$-1$}{$t^{-1}$}{$t$}
    \end{align*}
gives $r=-1$ and $p\ne 2$ by $a_{45}^{r_3(X)}=-1$. Hence $s\ne -1$ and the reflection of $X$
\begin{align*}
  	X\colon\ \Dchainsixd{$r^{-1}$}{$r$}{$r^{-1}$}{$r$}{$-1$}{$r^{-1}$}{$-1$}{$s^{-1}$}{$-1$}{$t^{-1}$}{$t$}\quad\Rightarrow\quad r_4(X)\colon\ \DchainsixM{$r^{-1}$}{$r$}{$r^{-1}$}{$r$}{$r^{-1}$}{$(rs)^{-1}$}{$s^{-1}$}{$t^{-1}$}{$t$}{$r$}{$s$}{$-1$}
    \end{align*}
implies $st=1$ by $a_{56}^{r_4(X)}=-1$. By Lemma~\ref{lem:impossible6} $(b_6)$,  this configuration cannot occur.

Step $3.6$. Let $q_{22}=q_{44}=q_{55}=-1$ and $q_{ii}\widetilde{q_{i,i+1}}=q_{jj}\widetilde{q_{j-1,j}}=1$ for all $i\in \{1,3\}$ and $j\in \{3,6\}$. By Definition \ref{defA6}, we have either $a=0$ or $a=1$. Set $q\colon=q_{11},\ r\colon=q_{33},\ s^{-1}\colon=\widetilde{q_{45}}$ and 
 $t\colon=q_{66}$. From $a_{13}^{r_2(X)}=0$ we get $qr=1$, while $a_{35}^{r_4(X)}=-1$ yields $rs\ne 1$. We split the remaining analysis into two subcases.

 Step $3.6.1$. Let $a=0$. Then  $a^{r_5(X)}_{46}=0$. 
 %\\ $A^{r_5(X)}=\begin{pmatrix}2&-1&0&0&0&0\\-1&2&-1&0&0&0\\0&-1&2&-1&0&0\\0&0&-1&2&-1&0\\0&0&0&-1&2&-1\\0&0&0&0&-1&2\end{pmatrix}$ 
Hence $st=1$. Then we have the reflection
\begin{align*}
        	X\colon\ \Dchainsixe{$r^{-1}$}{$r$}{$-1$}{$r^{-1}$}{$r$}{$r^{-1}$}{$-1$}{$s^{-1}$}{$-1$}{$s$}{$s^{-1}$}\quad
        	\Rightarrow\quad r_5(X)\colon\ \Dchainsix{$r^{-1}$}{$r$}{$-1$}{$r^{-1}$}{$r$}{$r^{-1}$}{$s^{-1}$}{$s$}{$-1$}{$s^{-1}$}{$-1$}
        \end{align*}
and $s=-1$ and $p\ne 2$ by $a_{43}^{r_5(X)}=-1$. By Lemma~\ref{lem:impossible6} $(a_3)$, 
this configuration cannot occur.

Step $3.6.2$. If $a=1$, then $a^{r_5(X)}_{46}=-1$. %\\ $A^{r_5(X)}=\begin{pmatrix}2&-1&0&0&0&0\\-1&2&-1&0&0&0\\0&-1&2&-1&0&0\\0&0&-1&2&-1&-1\\0&0&0&-1&2&-1\\0&0&0&-1&-1&2\end{pmatrix}$. 
Then we have $st \ne 1$. The reflection 
\begin{align*}
X\colon\ \Dchainsixe{$r^{-1}$}{$r$}{$-1$}{$r^{-1}$}{$r$}{$r^{-1}$}{$-1$}{$s^{-1}$}{$-1$}{$t^{-1}$}{$t$}
\quad \Rightarrow \quad 
r_5(X)\colon\ \Dchainsixze{$r^{-1}$}{$r$}{$-1$}{$r^{-1}$}{$r$}{$r^{-1}$}{$s^{-1}$}{$(st)^{-1}$}{$-1$}{$s$}{$t$}{$-1$}
\end{align*}
implies $s=-1$ and $p\ne 2$ by $a_{43}^{r_5(X)}=-1$. One the other hand, the reflections of $X$
        \begin{align*}
        	X\colon\ \Dchainsixb{$r^{-1}$}{$r$}{$-1$}{$r^{-1}$}{$r$}{$r^{-1}$}{$-1$}{$-1$}{$-1$}{$t^{-1}$}{$t$}\quad
        	\Rightarrow\quad r_2(X)\colon\ \Dchainsixc{$-1$}{$r^{-1}$}{$-1$}{$r$}{$-1$}{$r^{-1}$}{$-1$}{$-1$}{$-1$}{$t^{-1}$}{$t$}
        \end{align*}
    \vspace{-4mm}
        \begin{align*}
        	\Rightarrow\quad r_3r_2(X)\colon\ \Dchainsix{$-1$}{$r^{-1}$}{$r$}{$r^{-1}$}{$-1$}{$r$}{$r^{-1}$}{$-1$}{$-1$}{$t^{-1}$}{$t$}
        \end{align*}
give $r=-1$ and $p\ne 2$ by $a^{r_3r_2(X)}_{45}=-1$, which contradicts $rs\ne 1$.

  Step $3.7$. Let $q_{22}=q_{44}=q_{66}=-1$ and $q_{ii}\widetilde{q_{i,i+1}}=q_{jj}\widetilde{q_{j-1,j}}=1$ for all $i\in \{1,3,5\}$ and $j\in \{3,5\}$. In this case, we have $a=0$. Set $q\colon=q_{11},\ r\colon=q_{33}$ and $s\colon=q_{55}$. By Lemma \ref{jslemma}, we obtain $qr=1$ and $rs \ne 1$. Then the reflections of $X$
 \begin{align*}
     X\colon\ \Dchainsixf{$r^{-1}$}{$r$}{$-1$}{$r^{-1}$}{$r$}{$r^{-1}$}{$-1$}{$s^{-1}$}{$s$}{$s^{-1}$}{$-1$}\quad\Rightarrow\quad r_6(X)\colon\ \Dchainsixe{$r^{-1}$}{$r$}{$-1$}{$r^{-1}$}{$r$}{$r^{-1}$}{$-1$}{$s^{-1}$}{$-1$}{$s$}{$-1$}
    \end{align*}
    \vspace{-4mm}
\begin{align*}
     \quad\Rightarrow\quad r_5r_6(X)\colon\ \Dchainsix{$r^{-1}$}{$r$}{$-1$}{$r^{-1}$}{$r$}{$r^{-1}$}{$s^{-1}$}{$s$}{$-1$}{$s^{-1}$}{$s$}
    \end{align*}
imply $s=-1$ and $p\ne 2$ by $a_{43}^{r_5r_6(X)}=-1$. By Lemma~\ref{lem:impossible6} $(a_3)$, this configuration cannot occur. 

Step $3.8$. Let $q_{44}=q_{55}=q_{66}=-1$ and $q_{ii}\widetilde{q_{i,i+1}}=q_{jj}\widetilde{q_{j-1,j}}=1$ for all $i\in \{1,2,3\}$ and $j\in \{2,3\}$. Then $a \in \{0,1\}$. Set $q\colon=q_{11}, \ r^{-1}\colon=\widetilde{q_{45}},\ s^{-1}\colon=\widetilde{q_{56}}$. By $a_{35}^{r_4(X)}=-1$, we obtain $qr\ne 1$. We proceed with the proof in two separate steps. 

Step $3.8.1$. If $a=0$, then $a^{r_5(X)}_{46}=0$ implise %\\ $A^{r_5(X)}=\begin{pmatrix}2&-1&0&0&0&0\\-1&2&-1&0&0&0\\0&-1&2&-1&0&0\\0&0&-1&2&-1&0\\0&0&0&-1&2&-1\\0&0&0&0&-1&2\end{pmatrix}$.
$rs=1$. In addition, the reflection of $X$
\begin{align*}
X\colon\ \Dchainsixe{$q$}{$q^{-1}$}{$q$}{$q^{-1}$}{$q$}{$q^{-1}$}{$-1$}{$r^{-1}$}{$-1$}{$s^{-1}$}{$-1$}\quad \Rightarrow \quad 
r_5(X)\colon\ \Dchainsix{$q$}{$q^{-1}$}{$q$}{$q^{-1}$}{$q$}{$q^{-1}$}{$r^{-1}$}{$r$}{$-1$}{$s$}{$s^{-1}$}
\end{align*}
implies $r=-1$ and $p\ne 2$ by $a_{43}^{r_5(X)}=-1$. Hence $r=s=-1$. By Lemma \ref{lem:impossible6} $(a_1)$, this configuration cannot occur. 

Step $3.8.2$. If $a=1$, then $a^{r_5(X)}_{46}=-1$.
%\\ $A^{r_5(X)}=\begin{pmatrix}2&-1&0&0&0&0\\-1&2&-1&0&0&0\\0&-1&2&-1&0&0\\0&0&-1&2&-1&-1\\0&0&0&-1&2&-1\\0&0&0&-1&-1&2\end{pmatrix}$. 
Hence $rs\ne 1$. The reflection of $X$
\begin{align*}
X\colon\ \Dchainsixe{$q$}{$q^{-1}$}{$q$}{$q^{-1}$}{$q$}{$q^{-1}$}{$-1$}{$r^{-1}$}{$-1$}{$s^{-1}$}{$-1$}
\quad \Rightarrow \quad 
r_5(X)\colon\ \Dchainsixze{$q$}{$q^{-1}$}{$q$}{$q^{-1}$}{$q$}{$q^{-1}$}{$r^{-1}$}{$(rs)^{-1}$}{$s^{-1}$}{$r$}{$s$}{$-1$}
\end{align*}
implies $r=-1$ and $p\ne 2$ by $a_{46}^{r_5(X)}=-1$. Additionally, the reflection of $X$
\begin{align*}
X\colon\ \Dchainsixf{$q$}{$q^{-1}$}{$q$}{$q^{-1}$}{$q$}{$q^{-1}$}{$-1$}{$-1$}{$-1$}{$s^{-1}$}{$-1$}
\quad \Rightarrow \quad 
r_6(X)\colon\ \Dchainsix{$q$}{$q^{-1}$}{$q$}{$q^{-1}$}{$q$}{$q^{-1}$}{$-1$}{$-1$}{$s^{-1}$}{$s$}{$-1$}
\end{align*}
implies $s=-1$ and $p\ne 2$ by $a_{54}^{r_6(X)}=-1$, which contradicts $rs\ne 1$.

Step $4$. Assume exactly four vertices carry label $-1$. Combined with $q_{44}=-1$, we subdivide the analysis into eight subcases.

Step $4.1$. Let $q_{11}=q_{22}=q_{33}=q_{44}=-1$ and  $q_{55}\widetilde{q_{56}}=q_{jj}\widetilde{q_{j-1,j}}=1$  for $j\in \{5,6\}$. Then $a=0$. By Definition \ref{defA6} and Step $2.3$, this configuration cannot occur. 

Consider the symmetric case that $q_{33}=q_{44}=q_{55}=q_{66}=-1$, $q_{22}\widetilde{q_{21}}=1$ and $q_{ii}\widetilde{q_{i,i+1}}=1$ for $i\in \{1,2\}$. Set $q\colon=q_{11},\ r^{-1}\colon=\widetilde{q_{34}},\ s^{-1}\colon=\widetilde{q_{45}}$ and $t^{-1}\colon=\widetilde{q_{56}}$. It follows from Lemma \ref{jslemma} and Definition \ref{defA6} that $qr=1$ and $rs \ne 1$. The reflection of $X$
    \begin{align*}
    	X\colon\ \Dchainsixc{$r^{-1}$}{$r$}{$r^{-1}$}{$r$}{$-1$}{$r^{-1}$}{$-1$}{$s^{-1}$}{$-1$}{$t^{-1}$}{$-1$}\quad
    	\Rightarrow\quad r_3(X)\colon\ \Dchainsix{$r^{-1}$}{$r$}{$-1$}{$r^{-1}$}{$-1$}{$r$}{$r^{-1}$}{$s^{-1}$}{$-1$}{$t^{-1}$}{$-1$}  
    \end{align*}
gives $r=-1$ and $p\ne 2$ by $a_{45}^{r_3(X)}=-1$. Then the reflection of $X$
\begin{align*}
  	X\colon\ \Dchainsixf{$-1$}{$-1$}{$-1$}{$-1$}{$-1$}{$-1$}{$-1$}{$s^{-1}$}{$-1$}{$t^{-1}$}{$-1$} \quad
  	\Rightarrow \quad r_6(X)\colon\ 
  	\Dchainsix{$-1$}{$-1$}{$-1$}{$-1$}{$-1$}{$-1$}{$-1$}{$s^{-1}$}{$t^{-1}$}{$t$}{$-1$} 
  \end{align*}
implies $st=1$ or $t=-1$ by $a_{54}^{r_6(X)}=-1$. If $st=1$, this case is impossible by Lemma~\ref{lem:impossible6} $(b_5)$. If $t=-1$, then $s\ne -1$ and the reflection of $X$
\begin{align*}
    	X\colon\ \Dchainsixd{$-1$}{$-1$}{$-1$}{$-1$}{$-1$}{$-1$}{$-1$}{$s^{-1}$}{$-1$}{$-1$}{$-1$}\quad
    	\Rightarrow\quad r_4(X)\colon\ \DchainsixM{$-1$}{$-1$}{$-1$}{$-1$}{$-1$}{$-s^{-1}$}{$s^{-1}$}{$-1$}{$-1$}{$-1$}{$s$}{$-1$}  
    \end{align*}
implies $s^2=-1$ by $a_{53}^{r_4(X)}=-1$. Hence $s\in G'_4$. By Lemma~\ref{lem:impossible6} $(b_1)$, this configuration cannot occur.

Step $4.2$. Let $q_{11}=q_{22}=q_{44}=q_{55}=-1$,  $q_{33}\widetilde{q_{34}}=1$ and $q_{jj}\widetilde{q_{j-1,j}}=1$ for $j\in \{3,6\}$. Then $a \in \{0,1\}$. Set $q^{-1}\colon=\widetilde{q_{12}},\ r\colon=q_{33},\ s^{-1}\colon=\widetilde{q_{45}}$ and $t\colon=q_{66}$. We obtain $qr=1$ by $a_{13}^{r_2(X)}=0$ and $rs\ne 1$ by $a_{35}^{r_4(X)}=-1$.

Step $4.2.1$. If $a=0$. Then $a^{r_5(X)}_{46}=0$ yields $st=1$.
%that is \\ $A^{r_5(X)}=\begin{pmatrix}2&-1&0&0&0&0\\-1&2&-1&0&0&0\\0&-1&2&-1&0&0\\0&0&-1&2&-1&0\\0&0&0&-1&2&-1\\0&0&0&0&-1&2\end{pmatrix}$. Then $st=1$. 
Hence the reflection of $X$
\begin{align*}
    X\colon\ \Dchainsixe{$-1$}{$r$}{$-1$}{$r^{-1}$}{$r$}{$r^{-1}$}{$-1$}{$s^{-1}$}{$-1$}{$s$}{$s^{-1}$}\quad\Rightarrow\quad 
    r_5(X)\colon\ \Dchainsix{$-1$}{$r$}{$-1$}{$r^{-1}$}{$r$}{$r^{-1}$}{$s^{-1}$}{$s$}{$-1$}{$s^{-1}$}{$-1$}
\end{align*}
implies $s=-1$ and $p\ne 2$ by $a_{43}^{r_5(X)}=-1$. By Lemma \ref{lem:impossible6} $(a_4)$, this configuration cannot occur.

Step $4.2.2$. If $a=1$, that is $a^{r_5(X)}_{46}=-1$.
%then \\ $A^{r_5(X)}=\begin{pmatrix}2&-1&0&0&0&0\\-1&2&-1&0&0&0\\0&-1&2&-1&0&0\\0&0&-1&2&-1&-1\\0&0&0&-1&2&-1\\0&0&0&-1&-1&2\end{pmatrix}$. 
Hence we have $st\ne 1$. Following the same method as in the proof of Step $3.6.2$, we obtain a contradiction.

Step $4.3$. Let $q_{11}=q_{22}=q_{44}=q_{66}=-1$ and $q_{ii}\widetilde{q_{i,i+1}}=q_{jj}\widetilde{q_{j-1,j}}=1$ for all $i\in \{3,5\}$ and $j\in \{3,5\}$. Then $a=0$. Set $q^{-1}\colon=\widetilde{q_{12}},\ r\colon=q_{33}$ and $s\colon=q_{55}$. We obtain $qr=1$ by $a^{r_2(X)}_{13}=0$ and $rs \ne 1$ by $a^{r_4(X)}_{35}=-1$. 
%$X$ satisfying the definition of a good $A_6$ neighborhood implies that $qr=1$ and $rs \ne 1$. 
Then the reflections of $X$
\begin{align*}
  	X\colon\ \Dchainsixf{$-1$}{$r$}{$-1$}{$r^{-1}$}{$r$}{$r^{-1}$}{$-1$}{$s^{-1}$}{$s$}{$s^{-1}$}{$-1$} \quad
  	\Rightarrow \quad r_6(X)\colon\ 
  	\Dchainsixe{$-1$}{$r$}{$-1$}{$r^{-1}$}{$r$}{$r^{-1}$}{$-1$}{$s^{-1}$}{$-1$}{$s$}{$-1$} 
\end{align*}
\begin{align*}
 \quad\Rightarrow \quad r_5r_6(X)\colon\ 
  	\Dchainsix{$-1$}{$r$}{$-1$}{$r^{-1}$}{$r$}{$r^{-1}$}{$s^{-1}$}{$s$}{$-1$}{$s^{-1}$}{$s$} 
\end{align*}
imply $s=-1$ and $p\ne 2$ by $a_{43}^{r_5r_6(X)}=-1$. By Lemma \ref{lem:impossible6} $(a_4)$, this configuration cannot occur.  

Step $4.4$. Let $q_{11}=q_{33}=q_{44}=q_{55}=-1$,  $q_{22}\widetilde{q_{23}}=1$ and $q_{jj}\widetilde{q_{j-1,j}}=1$ for $j\in \{2,6\}$. Set $q\colon=q_{22},\ r^{-1}\colon=\widetilde{q_{34}},\ s^{-1}\colon=\widetilde{q_{45}}$ and $t\colon=q_{66}$. Then we have $qr=1$ by $a^{r_3(X)}_{24}=0$ and $rs\ne 1$ by $a^{r_4(X)}_{35}=-1$. The reflection of $X$
    \begin{align*}
    	X\colon\ \Dchainsixc{$-1$}{$r$}{$r^{-1}$}{$r$}{$-1$}{$r^{-1}$}{$-1$}{$s^{-1}$}{$-1$}{$t^{-1}$}{$t$}\quad
    	\Rightarrow\quad r_3(X)\colon\ \Dchainsix{$-1$}{$r$}{$-1$}{$r^{-1}$}{$-1$}{$r$}{$r^{-1}$}{$s^{-1}$}{$-1$}{$t^{-1}$}{$t$}  
    \end{align*}
implies $r=-1$ and $p\ne 2$ by $a_{45}^{r_3(X)}=-1$. Hence $s\ne -1$. In addition, the reflection of $X$
        \begin{align*}
        	X\colon\ \Dchainsixd{$-1$}{$-1$}{$-1$}{$-1$}{$-1$}{$-1$}{$-1$}{$s^{-1}$}{$-1$}{$t^{-1}$}{$t$}\quad\Rightarrow\quad r_4(X)\colon\ \DchainsixM{$-1$}{$-1$}{$-1$}{$-1$}{$-1$}{$-s^{-1}$}{$s^{-1}$}{$t^{-1}$}{$t$}{$-1$}{$s$}{$-1$}
        \end{align*}
implies $st=1$ by $a^{r_4(X)}_{56}=-1$. By Lemma \ref{lem:impossible6} $(b_6)$, this configuration cannot occur. 

Consider the symmetric case that $q_{22}=q_{33}=q_{44}=q_{66}=-1$,  $q_{55}\widetilde{q_{45}}=1$ and $q_{ii}\widetilde{q_{i,i+1}}=1$ for $i\in \{1,5\}$. In this case $a=0$. Set $q\colon=q_{11},\ r^{-1}\colon=\widetilde{q_{23}},\ s^{-1}\colon=\widetilde{q_{34}}$ and $t\colon=q_{55}$. By Definition \ref{defA6}, we have $qr=rs=1$ and $st\ne 1$. Hence the reflections of $X$
\begin{align*}
  	X\colon\ \Dchainsixf{$s$}{$s^{-1}$}{$-1$}{$s$}{$-1$}{$s^{-1}$}{$-1$}{$t^{-1}$}{$t$}{$t^{-1}$}{$-1$} \quad
  	\Rightarrow \quad r_6(X)\colon\ 
  	\Dchainsixe{$s$}{$s^{-1}$}{$-1$}{$s$}{$-1$}{$s^{-1}$}{$-1$}{$t^{-1}$}{$-1$}{$t$}{$-1$} 
\end{align*}
\begin{align*}
 \quad\Rightarrow \quad r_5r_6(X)\colon\ 
  	\Dchainsix{$s$}{$s^{-1}$}{$-1$}{$s$}{$-1$}{$s^{-1}$}{$t^{-1}$}{$t$}{$-1$}{$t^{-1}$}{$t$} 
\end{align*}
imply $t=-1$ and $p\ne 2$ by $a_{43}^{r_5r_6(X)}=-1$. By Lemma \ref{lem:impossible6} $(a_5)$, this configuration cannot occur.

Step $4.5$. Let $q_{11}=q_{33}=q_{44}=q_{66}=-1$ and $q_{ii}\widetilde{q_{i,i+1}}=q_{jj}\widetilde{q_{j-1,j}}=1$ for all $i\in \{2,5\}$ and $j\in \{2,5\}$. Then $a=0$. Set $q\colon=q_{22},\ r^{-1}\colon=\widetilde{q_{34}}$ and $s\colon=q_{55}$. We obtain $qr=1$ by $a^{r_3(X)}_{24}=0$ and $rs\ne 1$ by $a^{r_4(X)}_{35}=-1$. Then the reflection of $X$
    \begin{align*}
    	X\colon\ \Dchainsixc{$-1$}{$r$}{$r^{-1}$}{$r$}{$-1$}{$r^{-1}$}{$-1$}{$s^{-1}$}{$s$}{$s^{-1}$}{$-1$}\quad
    	\Rightarrow\quad r_3(X)\colon\ \Dchainsix{$-1$}{$r$}{$-1$}{$r^{-1}$}{$-1$}{$r$}{$r^{-1}$}{$s^{-1}$}{$s$}{$s^{-1}$}{$-1$}  
    \end{align*}
implies $r=-1$ and $p\ne 2$ by $a^{r_3(X)}_{45}=-1$. By Lemma~\ref{lem:impossible6} $(b_4)$, this configuration cannot occur.

Step $4.6$. Let $q_{11}=q_{44}=q_{55}=q_{66}=-1$ and $q_{ii}\widetilde{q_{i,i+1}}=q_{jj}\widetilde{q_{j-1,j}}=1$ for all $i\in \{2,3\}$ and $j\in \{2,3\}$. By the definition of good $A_6$ neighborhood, we have $a\in \{0,1\}$. Set $q\colon=q_{22},\ r^{-1}\colon=\widetilde{q_{45}}$ and $s^{-1}\colon=\widetilde{q_{56}}$. %The same argument as in Step $3.8$ shows that this step is impossible. 
By $a_{35}^{r_4(X)}=-1$, we have $qr\ne 1$. %Below we split in two steps.

Step $4.6.1$. If $a=0$, then $a^{r_5(X)}_{46}=0$.
%that is \\ $A^{r_5(X)}=\begin{pmatrix}2&-1&0&0&0&0\\-1&2&-1&0&0&0\\0&-1&2&-1&0&0\\0&0&-1&2&-1&0\\0&0&0&-1&2&-1\\0&0&0&0&-1&2\end{pmatrix}$. 
Hence $rs=1$. Then the reflection of $X$
 \begin{align*}
    X\colon\ \Dchainsixe{$-1$}{$q^{-1}$}{$q$}{$q^{-1}$}{$q$}{$q^{-1}$}{$-1$}{$r^{-1}$}{$-1$}{$r$}{$-1$}\quad\Rightarrow\quad 
    r_5(X)\colon\ \Dchainsix{$-1$}{$q^{-1}$}{$q$}{$q^{-1}$}{$q$}{$q^{-1}$}{$r^{-1}$}{$r$}{$-1$}{$r^{-1}$}{$r$}
\end{align*}
implies $r=-1$ and $p\ne 2$ by $a_{43}^{r_5(X)}=-1$. By Lemma \ref{lem:impossible6} $(a_2)$, this configuration cannot occur.

Step $4.6.2$. If $a=1$, then $a^{r_5(X)}_{46}=-1$.
%\\ $A^{r_5(X)}=\begin{pmatrix}2&-1&0&0&0&0\\-1&2&-1&0&0&0\\0&-1&2&-1&0&0\\0&0&-1&2&-1&-1\\0&0&0&-1&2&-1\\0&0&0&-1&-1&2\end{pmatrix}$. 
Hence $rs\ne 1$. The same argument as in Step $3.8.2$ shows that this step is a contradiction. 
 
Step $4.7$. Let $q_{22}=q_{33}=q_{44}=q_{55}=-1$ and  $q_{11}\widetilde{q_{12}}=q_{66}\widetilde{q_{56}}=1$. Hence $a \in \{0, 1\}$. Set $q\colon=q_{11},\ r^{-1}\colon=\widetilde{q_{23}},\ s^{-1}\colon=\widetilde{q_{34}},\ t^{-1}\colon=\widetilde{q_{45}}$ and $u\colon=q_{66}$. By Lemma \ref{jslemma}, we obtain $rs=qr=1$ and $st\ne 1$. Then $qt\ne 1$. Below we split in two steps.

Step $4.7.1$. If $a=0$, that is %$A^{r_5(X)}=\begin{pmatrix}2&-1&0&0&0&0\\-1&2&-1&0&0&0\\0&-1&2&-1&0&0\\0&0&-1&2&-1&0\\0&0&0&-1&2&-1\\0&0&0&0&-1&2\end{pmatrix}$.
$A^{r_5(X)}=A_6$. Hence $tu=1$. Then the reflection of $X$
 \begin{align*}
    X\colon\ \Dchainsixe{$q$}{$q^{-1}$}{$-1$}{$q$}{$-1$}{$q^{-1}$}{$-1$}{$t^{-1}$}{$-1$}{$t$}{$t^{-1}$}\quad\Rightarrow\quad 
    r_5(X)\colon\ \Dchainsix{$q$}{$q^{-1}$}{$-1$}{$q$}{$-1$}{$q^{-1}$}{$t^{-1}$}{$t$}{$-1$}{$t^{-1}$}{$-1$}
\end{align*}
implies $t=-1$ and $p\ne 2$ by $a_{43}^{r_5(X)}=-1$. Thus this case is impossible by Lemma \ref{lem:impossible6} $(a_5)$.

Step $4.7.2$. If $a=1$, then $a^{r_5(X)}_{46}=-1$. Hence $tu\ne 1$. The reflection of $X$
%then \\ $A^{r_5(X)}=\begin{pmatrix}2&-1&0&0&0&0\\-1&2&-1&0&0&0\\0&-1&2&-1&0&0\\0&0&-1&2&-1&-1\\0&0&0&-1&2&-1\\0&0&0&-1&-1&2\end{pmatrix}$. Hence $tu\ne 1$. The reflection of $X$
\begin{align*}
X\colon\ \Dchainsixe{$q$}{$q^{-1}$}{$-1$}{$q$}{$-1$}{$q^{-1}$}{$-1$}{$t^{-1}$}{$-1$}{$u^{-1}$}{$u$}
\quad \Rightarrow \quad 
r_5(X)\colon\ \Dchainsixze{$q$}{$q^{-1}$}{$-1$}{$q$}{$-1$}{$q^{-1}$}{$t^{-1}$}{$(tu)^{-1}$}{$-1$}{$t$}{$u$}{$-1$}
\end{align*}
implies $t=-1$ and $p\ne 2$ by $a_{43}^{r_5(X)}=-1$. On the other hand, the reflection of $X$
\begin{align*}
X\colon\ \Dchainsixc{$q$}{$q^{-1}$}{$-1$}{$q$}{$-1$}{$q^{-1}$}{$-1$}{$-1$}{$-1$}{$u^{-1}$}{$u$}
\quad \Rightarrow \quad 
r_3(X)\colon\ \Dchainsix{$q$}{$q^{-1}$}{$q$}{$q^{-1}$}{$-1$}{$q$}{$q^{-1}$}{$-1$}{$-1$}{$u^{-1}$}{$u$}
\end{align*}
implies $q=-1$ and $p\ne 2$ by $a_{45}^{r_3(X)}=-1$, which contradicts $qt\ne 1$.

Step $4.8$. Let $q_{22}=q_{44}=q_{55}=q_{66}=-1$, $q_{33}\widetilde{q_{32}}=1$ and $q_{ii}\widetilde{q_{i,i+1}}=1$ for $i\in \{1,3\}$. Hence we have either $a=0$ or $a=1$. Set $q\colon=q_{11},\ r\colon=q_{33},\ s^{-1}\colon=\widetilde{q_{45}}$ and $t^{-1}\colon=\widetilde{q_{56}}$. Then $qr=1$ follows from $a_{13}^{r_2(X)}=0$, while $rs\neq 1$ is enforced by $a_{35}^{r_4(X)}=-1$. The remaining reasoning proceeds analogously to Step $4.6$.

Step $5$. Assume five vertices bear label $-1$. In view of $q_{44}=-1$, we split the argument into three disjoint cases.

Step $5.1$. Let $q_{11}=q_{22}=q_{33}=q_{44}=q_{55}=-1$ and $q_{66}\widetilde{q_{56}}=1$. Set $q^{-1}\colon=\widetilde{q_{12}},\ r^{-1}\colon=\widetilde{q_{23}},\ s^{-1}\colon=\widetilde{q_{34}},\ t^{-1}\colon=\widetilde{q_{45}}$ and $u\colon=q_{66}$. We have $rs=1$ by $a_{24}^{r_3(X)}=0$, $qr=1$ by $a_{13}^{r_2(X)}=0$, and $st\ne 1$ by $a_{35}^{r_4(X)}=-1$. Then the reflection of $X$
\begin{align*}
X\colon\ \Dchainsixc{$-1$}{$s^{-1}$}{$-1$}{$s$}{$-1$}{$s^{-1}$}{$-1$}{$t^{-1}$}{$-1$}{$u^{-1}$}{$u$}
\quad \Rightarrow \quad 
r_3(X)\colon\ \Dchainsix{$-1$}{$s^{-1}$}{$s$}{$s^{-1}$}{$-1$}{$s$}{$s^{-1}$}{$t^{-1}$}{$-1$}{$u^{-1}$}{$u$}
\end{align*}
implies $s=-1$ and $p\ne 2$ by $a_{45}^{r_3(X)}=-1$. Hence $t\ne -1$. Then the reflection of $X$
        \begin{align*}
        	X\colon\ \Dchainsixd{$-1$}{$-1$}{$-1$}{$-1$}{$-1$}{$-1$}{$-1$}{$t^{-1}$}{$-1$}{$u^{-1}$}{$u$}\quad\Rightarrow\quad r_4(X)\colon\ \DchainsixM{$-1$}{$-1$}{$-1$}{$-1$}{$-1$}{$-t^{-1}$}{$t^{-1}$}{$u^{-1}$}{$u$}{$-1$}{$t$}{$-1$}
        \end{align*}
implies $tu=1$ by $a^{r_4(X)}_{56}=-1$. By Lemma~\ref{lem:impossible6} $(b_6)$, this configuration cannot occur.  

Consider the symmetric case that $q_{22}=q_{33}=q_{44}=q_{55}=q_{66}=-1$ and $q_{11}\widetilde{q_{12}}=1$. Set $q\colon=q_{11},\ r^{-1}\colon=\widetilde{q_{23}},\ s^{-1}\colon=\widetilde{q_{34}},\ t^{-1}\colon=\widetilde{q_{45}}$, and $u^{-1}\colon=\widetilde{q_{56}}$. Then $rs=1,\ qr=1$ and $st\ne 1$. Employing an argument analogous to the one above, we find from $a_{45}^{r_3(X)}=-1$ that $q=-1$ and $p\ne 2$, and from $a_{56}^{r_4(X)}=-1$ that $tu=1$.
%Rewrite $X$ as
%\begin{align*}
%    \Dchainsix{$-1$}{$-1$}{$-1$}{$-1$}{$-1$}{$-1$}{$-1$}{$t^{-1}$}{$-1$}{$t$}{$-1$}
%\end{align*}
By Lemma~\ref{lem:impossible6} $(b_5)$, this configuration cannot occur.

Step $5.2$. Let $q_{11}=q_{22}=q_{33}=q_{44}=q_{66}=-1$ and $q_{55}\widetilde{q_{56}}=q_{55}\widetilde{q_{45}}=1$. Then $a=0$. Set $q^{-1}\colon=\widetilde{q_{12}},\ r^{-1}\colon=\widetilde{q_{23}},\ s^{-1}\colon=\widetilde{q_{34}}$ and $t\colon=q_{55}$. By Definition \ref{defA6}, we obtain $qr=rs=1$ and $st\ne 1$. Then the reflection of $X$
\begin{align*}
    X\colon\ \Dchainsixc{$-1$}{$q^{-1}$}{$-1$}{$q$}{$-1$}{$q^{-1}$}{$-1$}{$t^{-1}$}{$t$}{$t^{-1}$}{$-1$}\quad
        \Rightarrow\quad r_3(X)\colon\ \Dchainsix{$-1$}{$q^{-1}$}{$q$}{$q^{-1}$}{$-1$}{$q$}{$q^{-1}$}{$t^{-1}$}{$t$}{$t^{-1}$}{$-1$} \end{align*}
implies $q=-1$ and $p\ne 2$ by $a^{r_3(X)}_{45}=-1$. By Lemma~\ref{lem:impossible6} $(b_4)$, this configuration cannot occur. 

Consider the symmetric case that $q_{11}=q_{33}=q_{44}=q_{55}=q_{66}=-1$, and $q_{22}\widetilde{q_{23}}=q_{22}\widetilde{q_{12}}=1$. Set $q\colon=q_{22},\ r^{-1}\colon=\widetilde{q_{34}},\ s^{-1}\colon=\widetilde{q_{45}}$ and $t^{-1}\colon=\widetilde{q_{56}}$. We obtain $qr=1$ by $a^{r_3(X)}_{24}=0$ and $rs\ne 1$ $a^{r_4(X)}_{35}=-1$. Using an argument similar to that in Step $5.1$, we deduce from $a^{r_3(X)}_{45}=-1$ that $q=-1$ and $p\ne 2$, and from $a^{r_4(X)}_{56}=-1$ that $st=1$.
%Then $\mathcal{D}$ is 
%\begin{align*}
   % \Dchainsix{$-1$}{$-1$}{$-1$}{$-1$}{$-1$}{$-1$}{$-1$}{$s^{-1}$}{$-1$}{$s$}{$-1$}.
%\end{align*}
By Lemma~\ref{lem:impossible6} $(b_5)$, this configuration cannot occur. 

Step $5.3$. Let $q_{11}=q_{22}=q_{44}=q_{55}=q_{66}=-1$ and $q_{33}\widetilde{q_{34}}=q_{33}\widetilde{q_{23}}=1$. By Definition \ref{defA6}, it follows that $a \in \{0,1\}$. Set $q^{-1}\colon=\widetilde{q_{12}},\ r\colon=q_{33},\ s^{-1}\colon=\widetilde{q_{45}}$ and $t^{-1}\colon=\widetilde{q_{56}}$. By $a_{13}^{r_2(X)}=0$, we obtain $qr=1$ and from $a_{35}^{r_4(X)}=-1$, we have $rs\ne 1$. The application of the same directional reflection as in Step $4.6$ demonstrates the impossibility of this case.

Step $6$. Suppose all vertices are labeled $-1$. That is $q_{11}=q_{22}=q_{33}=q_{44}=q_{55}=q_{66}=-1$. Set $q^{-1}\colon=\widetilde{q_{12}},\ r^{-1}\colon=\widetilde{q_{23}},\ s^{-1}\colon=\widetilde{q_{34}},\ t^{-1}\colon=\widetilde{q_{45}}$ and $u^{-1}\colon=\widetilde{q_{56}}$. We have $qr=1$ by $a_{13}^{r_2(X)}=0$, $rs=1$ by $a_{24}^{r_3(X)}=0$ and $st\ne 1$ by $a_{35}^{r_4(X)}=-1$. Then the reflection of $X$
     \begin{align*}
     X\colon\ \Dchainsixc{$-1$}{$s^{-1}$}{$-1$}{$s$}{$-1$}{$s^{-1}$}{$-1$}{$t^{-1}$}{$-1$}{$u^{-1}$}{$-1$}\quad
         	\Rightarrow\quad r_3(X)\colon\ \Dchainsix{$-1$}{$s^{-1}$}{$s$}{$s^{-1}$}{$-1$}{$s$}{$s^{-1}$}{$t^{-1}$}{$-1$}{$u^{-1}$}{$-1$}
    \end{align*}
implies $s=-1$ and $p\ne 2$ by $a^{r_3(X)}_{45}=-1$. Hence $t\ne -1$. Then the reflection of $X$
    \begin{align*}
        X\colon\ \Dchainsixd{$-1$}{$-1$}{$-1$}{$-1$}{$-1$}{$-1$}{$-1$}{$t^{-1}$}{$-1$}{$u^{-1}$}{$-1$}\quad\Rightarrow\quad r_4(X)\colon\ \DchainsixM{$-1$}{$-1$}{$-1$}{$-1$}{$-1$}{$-t^{-1}$}{$t^{-1}$}{$u^{-1}$}{$-1$}{$-1$}{$t$}{$-1$}
    \end{align*}
implies $tu=1$ by $a^{r_4(X)}_{56}=-1$. By Lemma~\ref{lem:impossible6} $(b_5)$, this configuration cannot occur.
\end{proof}

%\subsection{Sporadic finite Cartan graphs of rank 7}\label{classfy4}
%\begin{lemma}\label{l-7chainmid}
%If $\mathcal{D}_{\chi ,E}$ is of the form
%\setlength{\unitlength}{1mm}
%\begin{center}
%\Dchainseven{$q_{11}$}{$r$}{$q_{22}$}{$s$}{$q_{33}$}{$t$}{$q_{44}$}{$u$}{$q_{55}$}{$v$}{$q_{66}$}{$w$}{$q_{77}$}
%\end{center}
%and has a good $A_7$ neighborhood, then $q_{33}=-1$.
%\end{lemma}
%\begin{proof}
%If $q_{33}\neq-1$, then $A^{r_3(X)}=A_7$, which contradicts Definition \ref{defA7}.
%\end{proof}

\begin{lemma}\label{l-7chainspe-a}
Let $\mathcal{D}_{\chi ,E}$ be 
\setlength{\unitlength}{1mm}
\begin{center}
\Dchainseven{$q_{11}$}{$r$}{$q_{22}$}{$s$}{$q_{33}$}{$t$}{$q_{44}$}{$u$}{$q_{55}$}{$v$}{$q_{66}$}{$w$}{$q_{77}$}
\end{center}
and has a good $A_7$ neighborhood, then $q_{33}=-1$ and the configurations listed below cannot occur.
\begin{itemize}
    \item[$(a_1)$]
\setlength{\unitlength}{1mm}
\ \Dchainseven{$-1$}{$q^{-1}$}{$-1$}{$q$}{$-1$}{$-1$}{$-1$}{$-1$}{$-1$}{$-1$}{$-1$}{$-1$}{$-1$}
,\quad $q\ne -1, \ p\ne 2$
\vspace{-2mm}
   \item[$(a_2)$]
\setlength{\unitlength}{1mm}
\ \Dchainseven{$q$}{$q^{-1}$}{$-1$}{$q$}{$-1$}{$-1$}{$-1$}{$-1$}{$-1$}{$-1$}{$-1$}{$-1$}{$-1$}
,\quad $q\ne -1, \ p\ne 2$
\vspace{-2mm}
\item[$(b_1)$]
\setlength{\unitlength}{1mm}
\ \Dchainseven{$-1$}{$q^{-1}$}{$q$}{$q^{-1}$}{$-1$}{$-1$}{$-1$}{$-1$}{$-1$}{$-1$}{$-1$}{$-1$}{$-1$}
,\quad $q\ne -1, \ p\ne 2$
\vspace{-2mm}
 \item[$(b_2)$]
\setlength{\unitlength}{1mm}
\ \Dchainseven{$q$}{$q^{-1}$}{$q$}{$q^{-1}$}{$-1$}{$-1$}{$-1$}{$-1$}{$-1$}{$-1$}{$-1$}{$-1$}{$-1$}
,\quad $q\ne -1, \ p\ne 2$
\vspace{-2mm}
    \item[$(c_1)$]
\setlength{\unitlength}{1mm}
\ \Dchainseven{$-1$}{$-1$}{$-1$}{$-1$}{$-1$}{$q^{-1}$}{$q$}{$q^{-1}$}{$q$}{$q^{-1}$}{$q$}{$q^{-1}$}{$q$}
,\quad $q\ne -1, \ p\ne 2$
\vspace{-2mm}
\item[$(c_2)$]
\setlength{\unitlength}{1mm}
\ \Dchainseven{$-1$}{$-1$}{$-1$}{$-1$}{$-1$}{$q^{-1}$}{$q$}{$q^{-1}$}{$q$}{$q^{-1}$}{$q$}{$q^{-1}$}{$-1$}
,\quad $q\ne -1, \ p\ne 2$
\vspace{-2mm}
\item[$(c_3)$]
\setlength{\unitlength}{1mm}
\ \Dchainseven{$-1$}{$-1$}{$-1$}{$-1$}{$-1$}{$q^{-1}$}{$q$}{$q^{-1}$}{$q$}{$q^{-1}$}{$-1$}{$q$}{$q^{-1}$}
,\quad $q\ne -1, \ p\ne 2$
\vspace{-2mm}
\item[$(c_4)$]
\setlength{\unitlength}{1mm}
\ \Dchainseven{$-1$}{$-1$}{$-1$}{$-1$}{$-1$}{$q^{-1}$}{$q$}{$q^{-1}$}{$q$}{$q^{-1}$}{$-1$}{$q$}{$-1$}
,\quad $q\ne -1, \ p\ne 2$
\vspace{-2mm}
\item[$(c_5)$]
\setlength{\unitlength}{1mm}
\ \Dchainseven{$-1$}{$-1$}{$-1$}{$-1$}{$-1$}{$q^{-1}$}{$q$}{$q^{-1}$}{$-1$}{$q$}{$-1$}{$q^{-1}$}{$-1$}
,\quad $q\ne -1, \ p\ne 2$
\vspace{-2mm}
\item[$(c_6)$]
\setlength{\unitlength}{1mm}
\ \Dchainseven{$-1$}{$-1$}{$-1$}{$-1$}{$-1$}{$q^{-1}$}{$q$}{$q^{-1}$}{$-1$}{$q$}{$-1$}{$q^{-1}$}{$q$}
,\quad $q\ne -1, \ p\ne 2$
\vspace{-2mm}
\item[$(c_7)$]
\setlength{\unitlength}{1mm}
\ \Dchainseven{$-1$}{$-1$}{$-1$}{$-1$}{$-1$}{$q^{-1}$}{$q$}{$q^{-1}$}{$-1$}{$q$}{$q^{-1}$}{$q$}{$-1$}
,\quad $q\ne -1, \ p\ne 2$
\vspace{-2mm}
\item[$(c_8)$]
\setlength{\unitlength}{1mm}
\ \Dchainseven{$-1$}{$-1$}{$-1$}{$-1$}{$-1$}{$q^{-1}$}{$q$}{$q^{-1}$}{$-1$}{$q$}{$q^{-1}$}{$q$}{$q^{-1}$}
,\quad $q\ne -1, \ p\ne 2$
\end{itemize}
\end{lemma}
\begin{proof}
If $q_{33}\neq-1$, then $A^{r_3(X)}=A_7$, which contradicts Definition \ref{defA7}. We now proceed to the remaining cases.
\begin{itemize}
    \item[(i)] Suppose $X$ is of the type
\Dchainseven{$q_{11}$}{$q^{-1}$}{$-1$}{$q$}{$-1$}{$-1$}{$-1$}{$-1$}{$-1$}{$-1$}{$-1$}{$-1$}{$-1$}
where $q_{11}=-1$ or $q_{11}=q$. Since $X$ has a good $A_7$ neighborhood, the transformation
\begin{align*}
\setlength{\unitlength}{1mm}
\Dchainsevena{$q_{11}$}{$q^{-1}$}{$-1$}{$q$}{$-1$}{$-1$}{$-1$}{$-1$}{$-1$}{$-1$}{$-1$}{$-1$}{$-1$}
\quad \Rightarrow \quad
\Dchainseven{$q'_{11}$}{$q$}{$-1$}{$q^{-1}$}{$q$}{$-1$}{$-1$}{$-1$}{$-1$}{$-1$}{$-1$}{$-1$}{$-1$}
\end{align*}
implies $q=-1$ and $p\not=2$ by $a_{34}^{r_2(X)}=-1$, where
\begin{align*}
  q'_{11}=&
  \begin{cases}
    q^{-1}& \quad \text{if  $q_{11}=-1$},\\
    -1& \quad \text{if  $q_{11}=q$}.
  \end{cases}
\end{align*} 
which contradicts $q\not=-1$. Hence the cases $(a_1)$ and $(a_2)$ cannot occur.

 \item[(ii)] Suppose $X$ takes the type
\Dchainseven{$q^{11}$}{$q^{-1}$}{$q$}{$q^{-1}$}{$-1$}{$-1$}{$-1$}{$-1$}{$-1$}{$-1$}{$-1$}{$-1$}{$-1$}
where $q_{11}=-1$ or $q_{11}=q$. We split in two cases $q^2=-1$ and $q^2\not=-1$. If $q^{-2}=-1$, the transformations
\begin{align*}
\begin{tikzpicture} 
  \tikzstyle{every node}=[scale=0.7]
  \draw(0,0) circle (1mm);
  \draw(1,0) circle (1mm);
  \filldraw(2,0) circle (1mm);
  \draw(3,0) circle (1mm);
  \draw(4,0) circle (1mm);
  \draw(5,0) circle (1mm);
  \draw(6,0) circle (1mm);
  \draw(0.1,0) -- (0.9,0) 
  (1.1,0) -- (1.9,0) 
  (2.1,0) -- (2.9,0) 
  (3.1,0) -- (3.9,0) 
  (4.1,0) -- (4.9,0) 
  (5.1,0) -- (5.9,0);  
  \node at (0,0.3){$q_{11}$};
  \node at (0.55,0.25) {$q^{-1}$};
  \node at (1,0.3){$q$};
  \node at (1.55,0.25){$q^{-1}$};
  \node at (2,0.3){$-1$};
  \node at (2.5,0.2){$-1$};
  \node at (3,0.3){$-1$};
  \node at (3.5,0.2){$-1$};
  \node at (4,0.3){$-1$};
  \node at (4.5,0.2){$-1$};
  \node at (5,0.3){$-1$};
  \node at (5.5,0.2){$-1$};
  \node at (6,0.3){$-1$};
\end{tikzpicture}
\Rightarrow
\begin{tikzpicture} 
  \tikzstyle{every node}=[scale=0.7]
  \draw(0,0) circle (1mm);
  \filldraw(1,0) circle (1mm);
  \draw(2,0) circle (1mm);
  \draw(3,0) circle (1mm);
  \draw(4,0) circle (1mm);
  \draw(5,0) circle (1mm);
  \draw(1.5,0.95) circle (1mm);
  \draw(0.1,0) -- (0.9,0) 
  (1.1,0) -- (1.9,0) 
  (2.1,0) -- (2.9,0) 
  (3.1,0) -- (3.9,0) 
  (4.1,0) -- (4.9,0) 
  (1.05,0.05) -- (1.45,0.9)
  (1.55,0.9) -- (2,0.05);
  \node at (0,0.3){$q_{11}$};
  \node at (0.5,0.2) {$q^{-1}$};
  \node at (0.9,0.3){$-1$};
  \node at (1.78,0.95){$-1$};
  \node at (1.5,0.25) {$-q^{-1}$};
  \node at (2.05,0.3) {$-1$};
  \node at (2.55,0.25) {$-1$};
  \node at (1.1,0.6) {$q$};
  \node at (1.9,0.6) {$-1$};
  \node at (3,0.3) {$-1$};
  \node at (3.5,0.25) {$-1$};
  \node at (4,0.3) {$-1$};
  \node at (4.5,0.25) {$-1$};
  \node at (5,0.3) {$-1$};
\end{tikzpicture}
\end{align*}

\begin{align*}
\Rightarrow %\quad
\tau_{3124567}~
\begin{tikzpicture} 
  \tikzstyle{every node}=[scale=0.7]
  \draw(0,0) circle (1mm);
  \draw(1,0) circle (1mm);
  \draw(2,0) circle (1mm);
  \draw(3,0) circle (1mm);
  \draw(4,0) circle (1mm);
  \draw(5,0) circle (1mm);
  \draw(1,1) circle (1mm);
  \draw(0.1,0) -- (0.9,0) 
  (1.1,0) -- (1.9,0) 
  (2.1,0) -- (2.9,0) 
  (3.1,0) -- (3.9,0) 
  (4.1,0) -- (4.9,0) 
  (1,0.1) -- (1,0.9);
  \node at (0,0.3){$q$};
  \node at (0.55,0.3) {$q^{-1}$};
  \node at (1,0.2){$-1$};
  \node at (1.2,0.5){$q$};
  \node at (1.35,1){$q'_{11}$};
  \node at (1.45,0.25) {$-q$};
  \node at (2.05,0.35) {$-q^{-1}$};
  \node at (2.5,0.25) {$-1$};
  \node at (3,0.3) {$-1$};
  \node at (3.5,0.25) {$-1$};
  \node at (4,0.3) {$-1$};
  \node at (4.5,0.25) {$-1$};
  \node at (5,0.3) {$-1$};
\end{tikzpicture}
\end{align*}
imply that $q=1$ by $a_{45}^{r_2r_3(x)}=-1$,
where
\begin{align*}
  q'_{11}=&
  \begin{cases}
    q^{-1}& \quad \text{if  $q_{11}=-1$},\\
    -1& \quad \text{if  $q_{11}=q$}.
  \end{cases}
\end{align*}
which is a contradiction. If $q^{-2} \not=-1$, the transformations
\begin{align*}
\begin{tikzpicture} 
  \tikzstyle{every node}=[scale=0.7]
  \draw(0,0) circle (1mm);
  \draw(1,0) circle (1mm);
  \filldraw(2,0) circle (1mm);
  \draw(3,0) circle (1mm);
  \draw(4,0) circle (1mm);
  \draw(5,0) circle (1mm);
  \draw(6,0) circle (1mm);
  \draw(0.1,0) -- (0.9,0) 
  (1.1,0) -- (1.9,0) 
  (2.1,0) -- (2.9,0) 
  (3.1,0) -- (3.9,0) 
  (4.1,0) -- (4.9,0) 
  (5.1,0) -- (5.9,0);  
  \node at (0,0.3){$q_{11}$};
  \node at (0.55,0.25) {$q^{-1}$};
  \node at (1,0.3){$q$};
  \node at (1.55,0.25){$q^{-1}$};
  \node at (2,0.3){$-1$};
  \node at (2.5,0.2){$-1$};
  \node at (3,0.3){$-1$};
  \node at (3.5,0.2){$-1$};
  \node at (4,0.3){$-1$};
  \node at (4.5,0.2){$-1$};
  \node at (5,0.3){$-1$};
  \node at (5.5,0.2){$-1$};
  \node at (6,0.3){$-1$};
\end{tikzpicture}
\Rightarrow
\begin{tikzpicture} 
  \tikzstyle{every node}=[scale=0.7]
  \draw(0,0) circle (1mm);
  \filldraw(1,0) circle (1mm);
  \draw(2,0) circle (1mm);
  \draw(3,0) circle (1mm);
  \draw(4,0) circle (1mm);
  \draw(5,0) circle (1mm);
  \draw(1.5,0.95) circle (1mm);
  \draw(0.1,0) -- (0.9,0) 
  (1.1,0) -- (1.9,0) 
  (2.1,0) -- (2.9,0) 
  (3.1,0) -- (3.9,0) 
  (4.1,0) -- (4.9,0) 
  (1.05,0.05) -- (1.45,0.9)
  (1.55,0.9) -- (2,0.05);
  \node at (0,0.3){$q_{11}$};
  \node at (0.5,0.2) {$q^{-1}$};
  \node at (0.9,0.3){$-1$};
  \node at (1.78,0.95){$-1$};
  \node at (1.5,0.25) {$-q^{-1}$};
  \node at (2.05,0.3) {$-1$};
  \node at (2.55,0.25) {$-1$};
  \node at (1.1,0.6) {$q$};
  \node at (1.9,0.6) {$-1$};
  \node at (3,0.3) {$-1$};
  \node at (3.5,0.25) {$-1$};
  \node at (4,0.3) {$-1$};
  \node at (4.5,0.25) {$-1$};
  \node at (5,0.3) {$-1$};
\end{tikzpicture}
\end{align*}
\begin{align*}
\Rightarrow
\tau_{3214567}~
\begin{tikzpicture} 
  \tikzstyle{every node}=[scale=0.7]
  \draw(0,0) circle (1mm);
  \draw(1,0) circle (1mm);
  \draw(2,0) circle (1mm);
  \draw(3,0) circle (1mm);
  \draw(4,0) circle (1mm);
  \draw(5,0) circle (1mm);
  \draw(1.5,0.95) circle (1mm);
  \draw(0.1,0) -- (0.9,0) 
  (1.1,0) -- (1.9,0) 
  (2.1,0) -- (2.9,0) 
  (3.1,0) -- (3.9,0) 
  (4.1,0) -- (4.9,0) 
  (1.05,0.05) -- (1.45,0.9)
  (1.55,0.9) -- (2,0.05);
  \node at (0,0.3){$q$};
  \node at (0.5,0.2) {$q^{-1}$};
  \node at (0.9,0.3){$-1$};
  \node at (1.85,1){$q'_{11}$};
  \node at (1.5,0.2) {$-q$};
  \node at (2.2,0.3) {$-q^{-1}$};
  \node at (2.6,0.25) {$-1$};
  \node at (1.1,0.6) {$q$};
  \node at (2.1,0.6) {$-q^{-2}$};
  \node at (3.05,0.3) {$-1$};
  \node at (3.5,0.25) {$-1$};
  \node at (4,0.3) {$-1$};
  \node at (4.5,0.25) {$-1$};
  \node at (5,0.3) {$-1$};
\end{tikzpicture}
\end{align*}
imply that $q=1$ by $a_{45}^{r_2r_3(x)}=-1$,
where
\begin{align*}
  q'_{11}=&
  \begin{cases}
    q^{-1}& \quad \text{if  $q_{11}=-1$}\\
    -1& \quad \text{if  $q_{11}=q$}
  \end{cases}
\end{align*}
which is also a contradiction. Hence the cases $(b_1)$ and $(b_2)$ cannot occur.

    \item[(iii)] If $X$ is of the type
\Dchainseven{$-1$}{$-1$}{$-1$}{$-1$}{$-1$}{$q^{-1}$}{$q$}{$q^{-1}$}{$q$}{$q^{-1}$}{$q$}{$q^{-1}$}{$q$}, then the analysis splits depending on whether 
$q^2=-1$ or not. If $q^{-2}=-1$, then the transformations
\begin{align*}
\begin{tikzpicture} 
  \tikzstyle{every node}=[scale=0.7]
  \draw(0,0) circle (1mm);
  \draw(1,0) circle (1mm);
  \filldraw(2,0) circle (1mm);
  \draw(3,0) circle (1mm);
  \draw(4,0) circle (1mm);
  \draw(5,0) circle (1mm);
  \draw(6,0) circle (1mm);
  \draw(0.1,0) -- (0.9,0) 
  (1.1,0) -- (1.9,0)
  (2.1,0) -- (2.9,0) 
  (3.1,0) -- (3.9,0) 
  (4.1,0) -- (4.9,0) 
  (5.1,0) -- (5.9,0);  
  \node at (0,0.3){$-1$};
  \node at (0.5,0.2) {$-1$};
  \node at (1,0.3){$-1$};
  \node at (1.5,0.2){$-1$};
  \node at (2,0.3){$-1$};
  \node at (2.55,0.25){$q^{-1}$};
  \node at (3,0.3){$q$};
  \node at (3.5,0.25){$q^{-1}$};
  \node at (4,0.3){$q$};
  \node at (4.5,0.25){$q^{-1}$};
  \node at (5,0.3){$q$};
  \node at (5.5,0.25){$q^{-1}$};
  \node at (6,0.3){$q$};
\end{tikzpicture}
\Rightarrow
\begin{tikzpicture} 
  \tikzstyle{every node}=[scale=0.7]
  \draw(0,0) circle (1mm);
  \draw(1,0) circle (1mm);
  \filldraw(2,0) circle (1mm);
  \draw(3,0) circle (1mm);
  \draw(4,0) circle (1mm);
  \draw(5,0) circle (1mm);
  \draw(1.5,0.95) circle (1mm);
  \draw(0.1,0) -- (0.9,0) 
  (1.1,0) -- (1.9,0) 
  (2.1,0) -- (2.9,0) 
  (3.1,0) -- (3.9,0) 
  (4.1,0) -- (4.9,0) 
  (1.05,0.05) -- (1.45,0.9)
  (1.55,0.9) -- (2,0.05);
  \node at (0,0.3){$-1$};
  \node at (0.5,0.2) {$-1$};
  \node at (0.9,0.3){$-1$};
  \node at (1.78,0.95){$-1$};
  \node at (1.5,0.25) {$-q^{-1}$};
  \node at (2.05,0.3) {$-1$};
  \node at (2.55,0.25) {$q^{-1}$};
  \node at (1.05,0.6) {$-1$};
  \node at (1.9,0.6) {$q$};
  \node at (3,0.3) {$q$};
  \node at (3.5,0.25) {$q^{-1}$};
  \node at (4,0.3) {$q$};
  \node at (4.5,0.25) {$q^{-1}$};
  \node at (5,0.3) {$q$};
\end{tikzpicture}
\end{align*}
\begin{align*}
\Rightarrow
\begin{tikzpicture} 
  \tikzstyle{every node}=[scale=0.7]
  \draw(0,0) circle (1mm);
  \draw(1,0) circle (1mm);
  \draw(2,0) circle (1mm);
  \draw(3,0) circle (1mm);
  \draw(4,0) circle (1mm);
  \draw(5,0) circle (1mm);
  \draw(2,1) circle (1mm);
  \draw(0.1,0) -- (0.9,0) 
  (1.1,0) -- (1.9,0) 
  (2.1,0) -- (2.9,0) 
  (3.1,0) -- (3.9,0) 
  (4.1,0) -- (4.9,0) 
  (2,0.1) -- (2,0.9);
  \node at (0,0.3){$-1$};
  \node at (0.5,0.2) {$-1$};
  \node at (1,0.35){$-q^{-1}$};
  \node at (1.5,0.2) {$-q$};
  \node at (2.1, 0.2) {$-1$};
  \node at (2.6,0.2) {$q$};
  \node at (3,0.3) {$-1$};
  \node at (3.5,0.25) {$q^{-1}$};
  \node at (4,0.3) {$q$};
  \node at (4.5,0.25) {$q^{-1}$};
  \node at (5,0.3) {$q$};
  \node at (2.3,0.55) {$q^{-1}$};
  \node at (2.25,1) {$q$};
\end{tikzpicture}
\end{align*}
imply $q=1$ by $a_{21}^{r_4r_3(x)}=-1$, which is a contradiction. If $q^{-2} \not=-1$, then the transformations
\begin{align*}
\begin{tikzpicture} 
  \tikzstyle{every node}=[scale=0.7]
  \draw(0,0) circle (1mm);
  \draw(1,0) circle (1mm);
  \filldraw(2,0) circle (1mm);
  \draw(3,0) circle (1mm);
  \draw(4,0) circle (1mm);
  \draw(5,0) circle (1mm);
  \draw(6,0) circle (1mm);
  \draw(0.1,0) -- (0.9,0) 
  (1.1,0) -- (1.9,0)
  (2.1,0) -- (2.9,0) 
  (3.1,0) -- (3.9,0) 
  (4.1,0) -- (4.9,0) 
  (5.1,0) -- (5.9,0);  
  \node at (0,0.3){$-1$};
  \node at (0.5,0.2) {$-1$};
  \node at (1,0.3){$-1$};
  \node at (1.5,0.2){$-1$};
  \node at (2,0.3){$-1$};
  \node at (2.55,0.25){$q^{-1}$};
  \node at (3,0.3){$q$};
  \node at (3.5,0.25){$q^{-1}$};
  \node at (4,0.3){$q$};
  \node at (4.5,0.25){$q^{-1}$};
  \node at (5,0.3){$q$};
  \node at (5.5,0.25){$q^{-1}$};
  \node at (6,0.3){$q$};
\end{tikzpicture}
\Rightarrow
\begin{tikzpicture} 
  \tikzstyle{every node}=[scale=0.7]
  \draw(0,0) circle (1mm);
  \draw(1,0) circle (1mm);
  \filldraw(2,0) circle (1mm);
  \draw(3,0) circle (1mm);
  \draw(4,0) circle (1mm);
  \draw(5,0) circle (1mm);
  \draw(1.5,0.95) circle (1mm);
  \draw(0.1,0) -- (0.9,0) 
  (1.1,0) -- (1.9,0) 
  (2.1,0) -- (2.9,0) 
  (3.1,0) -- (3.9,0) 
  (4.1,0) -- (4.9,0) 
  (1.05,0.05) -- (1.45,0.9)
  (1.55,0.9) -- (2,0.05);
  \node at (0,0.3){$-1$};
  \node at (0.5,0.2) {$-1$};
  \node at (0.9,0.3){$-1$};
  \node at (1.78,0.95){$-1$};
  \node at (1.5,0.25) {$-q^{-1}$};
  \node at (2.05,0.3) {$-1$};
  \node at (2.55,0.25) {$q^{-1}$};
  \node at (1.05,0.6) {$-1$};
  \node at (1.9,0.6) {$q$};
  \node at (3,0.3) {$q$};
  \node at (3.5,0.25) {$q^{-1}$};
  \node at (4,0.3) {$q$};
  \node at (4.5,0.25) {$q^{-1}$};
  \node at (5,0.3) {$q$};
\end{tikzpicture}
\end{align*}
\begin{align*}
\Rightarrow
\begin{tikzpicture} 
  \tikzstyle{every node}=[scale=0.7]
  \draw(0,0) circle (1mm);
  \draw(1,0) circle (1mm);
  \draw(2,0) circle (1mm);
  \draw(3,0) circle (1mm);
  \draw(4,0) circle (1mm);
  \draw(1.5,0.95) circle (1mm);
  \draw(1.5,1.95) circle (1mm);
  \draw(0.1,0) -- (0.9,0) 
  (1.1,0) -- (1.9,0) 
  (2.1,0) -- (2.9,0) 
  (3.1,0) -- (3.9,0) 
  (1.05,0.05) -- (1.45,0.9)
  (1.55,0.9) -- (2,0.05)
  (1.5,1.05) -- (1.5,1.85);
  \node at (-0.07,0.3){$-1$};
  \node at (0.37,0.2) {$-1$};
  \node at (0.9,0.32){$-q^{-1}$};
  \node at (1.8,0.95){$-1$};
  \node at (1.5,0.24) {$-q^{-2}$};
  \node at (2.1,0.3) {$-1$};
  \node at (1.05,0.7) {$-q$};
  \node at (1.85,0.69) {$q$};
  \node at (2.6,0.25) {$q^{-1}$};
  \node at (3,0.3) {$q$};
  \node at (3.5,0.25) {$q^{-1}$};
  \node at (4,0.3) {$q$};
  \node at (1.8,1.5) {$q^{-1}$};
  \node at (1.8,1.95) {$q$};
\end{tikzpicture}
\end{align*}
imply $q=1$ by $a_{21}^{r_4r_3(x)}=-1$, which is a contradiction. Hence case $(c_1)$ cannot occur. The reflection rules and constraints for $(c_2)$–$(c_8)$ follow analogously to $(c_1)$, so the same reasoning eliminates all remaining cases $(c_2)$ through $(c_8)$.
\end{itemize}
\end{proof}

\begin{theorem}\label{theo.rank7}
Suppose ${\roots}^{[M]}$ is a finite set of roots of sporadic finite Cartan graphs of rank $7$. Then every admissable generalized Dynkin diagrams of $V$ occurs in rows $20$-$21$ of Table \ref{tab.1}.
\end{theorem}
\begin{proof}
By Theorem \ref{thm:goodnei}, $C(M)$ is either standard of type $E_7$, or there exists a point $X$ such that $A^X$ has a good $A_7$ neighborhood.

Case $e$. Suppose $C(M)$ is standard of
type $E_7$. Then $A^{X}=E_{7}$ yields $(2)_{q_{22}}(q_{22}\widetilde{q_{24}}-1)=(2)_{q_{ii}}(q_{ii}\widetilde{q_{i,i+1}}-1)=(2)_{q_{jj}}(q_{jj}\widetilde{q_{j-1,j}}-1)=0$ for all $i\in \{1,3,4,5,6\}$ and $j\in \{2,4,5,6,7\}$. Set $q\colon=\widetilde{q_{12}}$.
Since $A^X= E_7$, we have $\cD=\cD^7_{20,1}$ when $q\neq-1$, and $\cD=\cD^7_{20',1}$ when $q=-1$.

Case $a$. Suppose $X$ has a good $A_7$ neighborhood. Since $A^{X}=A_{7}$, we have $(2)_{q_{ii}}(q_{ii}\widetilde{q_{i,i+1}}-1)=(2)_{q_{jj}}(q_{jj}\widetilde{q_{j-1,j}}-1)=0$ for all $i\in \{1,2,3,4,5,6\}$ and $j\in \{2,3,4,5,6,7\}$. Additionally, we obtain $q_{33}=-1$ for all cases by Lemma \ref{l-7chainspe-a}. 
We categorize all possible configurations based on the number of vertices labeled $-1$, which yields the seven cases treated in Steps $1$–$7$.

Step $1$. Suppose exactly one vertex is labelled  $-1$, namely $q_{33}=-1$ and $q_{ii}q_{i,i+1}'=q_{jj}q_{j-1,j}'=1$ for all $i\in \{1,2,4,5,6\}$ and $j\in \{2,4,5,6,7\}$. Set $r\colon=q_{11}$ and $q\colon=q_{44}$. Then we obtain $rq \not=1$ by $a_{24}^{r_3(X)}=-1$. The reflections of $X$ 
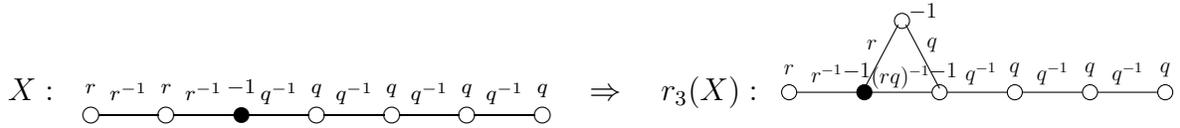
\begin{align*}  
  X\colon~\ \Dchainseveng{$r$}{$r^{-1}$}{$r$}{$r^{-1}$}{$-1$}{$q^{-1}$}{$q$}{$q^{-1}$}{$q$}{$q^{-1}$}{$q$}{$q^{-1}$}{$q$}
\quad \Rightarrow \quad
r_3(X)\colon~
\begin{tikzpicture} 
  \tikzstyle{every node}=[scale=0.7]
  \draw(0,0) circle (1mm);
  \filldraw(1,0) circle (1mm);
  \draw(2,0) circle (1mm);
  \draw(3,0) circle (1mm);
  \draw(4,0) circle (1mm);
  \draw(5,0) circle (1mm);
  \draw(1.5,0.95) circle (1mm);
  \draw(0.1,0) -- (0.9,0) 
  (1.1,0) -- (1.9,0) 
  (2.1,0) -- (2.9,0) 
  (3.1,0) -- (3.9,0) 
  (4.1,0) -- (4.9,0) 
  (1,0.05) -- (1.45,0.9)
  (1.55,0.9) -- (2,0.05);
  \node at (0,0.3){$r$};
  \node at (0.5,0.25) {$r^{-1}$};
  \node at (0.9,0.28){$-1$};
  \node at (1.78,1.08){$-1$};
  \node at (1.48,0.2) {$(rq)^{-1}$};
  \node at (2.05,0.28) {$-1$};
  \node at (2.55,0.25) {$q^{-1}$};
  \node at (1.1,0.63) {$r$};
  \node at (1.9,0.63) {$q$};
   \node at (3,0.3) {$q$};
  \node at (3.5,0.25) {$q^{-1}$};
  \node at (4,0.3) {$q$};
  \node at (4.5,0.25) {$q^{-1}$};
  \node at (5,0.3) {$q$};
\end{tikzpicture}
\end{align*}
\vspace{-4mm}
\begin{align*}
\Rightarrow \quad
r_2r_3(X)\colon~\tau_{3214567}
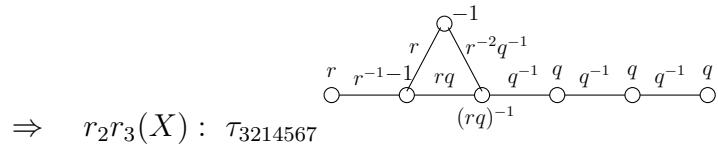
\begin{tikzpicture} 
  \tikzstyle{every node}=[scale=0.7]
  \draw(0,0) circle (1mm);
  \draw(1,0) circle (1mm);
  \draw(2,0) circle (1mm);
  \draw(3,0) circle (1mm);
  \draw(4,0) circle (1mm);
  \draw(5,0) circle (1mm);
  \draw(1.5,0.95) circle (1mm);
  \draw(0.1,0) -- (0.9,0) 
  (1.1,0) -- (1.9,0) 
  (2.1,0) -- (2.9,0) 
  (3.1,0) -- (3.9,0) 
  (4.1,0) -- (4.9,0) 
  (1,0.05) -- (1.45,0.9)
  (1.55,0.9) -- (2,0.05);
  \node at (0,0.3){$r$};
  \node at (0.5,0.25) {$r^{-1}$};
  \node at (0.9,0.28){$-1$};
    \node at (1.78,1.08){$-1$};
  \node at (1.5,0.2) {$rq$};
  \node at (2.05,-0.3) {$(rq)^{-1}$};
  \node at (2.55,0.25) {$q^{-1}$};
  \node at (1.1,0.63) {$r$};
  \node at (2.2,0.63) {$r^{-2}q^{-1}$};
  \node at (3,0.3) {$q$};
  \node at (3.5,0.25) {$q^{-1}$};
  \node at (4,0.3) {$q$};
  \node at (4.5,0.25) {$q^{-1}$};
  \node at (5,0.3) {$q$};
\end{tikzpicture}
\end{align*}
imply $rq^{2}=1$ or $rq=-1$ by $a_{45}^{r_2r_3(X)}=-1$.

Step $1.1$. If $rq^{2}=1$, then the reflection of $X$
\begin{align*}
r_3(X)\colon~
\begin{tikzpicture} 
  \tikzstyle{every node}=[scale=0.7]
  \draw(0,0) circle (1mm);
  \draw(1,0) circle (1mm);
  \filldraw(2,0) circle (1mm);
  \draw(3,0) circle (1mm);
  \draw(4,0) circle (1mm);
  \draw(5,0) circle (1mm);
  \draw(1.5,0.95) circle (1mm);
  \draw(0.1,0) -- (0.9,0) 
  (1.1,0) -- (1.9,0) 
  (2.1,0) -- (2.9,0) 
  (3.1,0) -- (3.9,0) 
  (4.1,0) -- (4.9,0) 
  (1,0.05) -- (1.45,0.9)
  (1.55,0.9) -- (2,0.05);
  \node at (0.1,0.35){$q^{-2}$};
  \node at (0.5,0.25) {$q^{2}$};
  \node at (0.9,0.28){$-1$};
  \node at (1.78,1.08){$-1$};
  \node at (1.5,0.2) {$q$};
  \node at (2.05,0.28) {$-1$};
  \node at (2.55,0.25) {$q^{-1}$};
  \node at (1.1,0.63) {$q^{-2}$};
  \node at (1.9,0.63) {$q$};
  \node at (3,0.3) {$q$};
  \node at (3.5,0.25) {$q^{-1}$};
  \node at (4,0.3) {$q$};
  \node at (4.5,0.25) {$q^{-1}$};
  \node at (5,0.3) {$q$};
\end{tikzpicture}
\quad \Rightarrow \quad
r_4r_3(X)\colon~
\begin{tikzpicture} 
  \tikzstyle{every node}=[scale=0.7]
  \draw(0,0) circle (1mm)
  (1,0) circle (1mm)
  (2,0) circle (1mm)
  (3,0) circle (1mm)
  (4,0) circle (1mm)
  (5,0) circle (1mm)
  (2,1) circle (1mm);
  \draw(0.1,0) -- (0.9,0) 
  (1.1,0) -- (1.9,0) 
  (2.1,0) -- (2.9,0) 
  (3.1,0) -- (3.9,0) 
  (4.1,0) -- (4.9,0) 
  (2,0.1) -- (2,0.9); 
  \node at (0.1,0.35){$q^{-2}$};
  \node at (0.5,0.25) {$q^{2}$};
  \node at (1,0.3){$q$};
  \node at (1.5,0.26){$q^{-1}$};
  \node at (2.2,0.2){$-1$};
  \node at (2.6,0.2){$q$};
  \node at (3,0.3){$-1$};
  \node at (3.5,0.25){$q^{-1}$};
  \node at (4,0.3){$q$};
  \node at (4.5,0.25){$q^{-1}$};
  \node at (5,0.3){$q$};
  \node at (2.3,0.5){$q^{-1}$};
  \node at (2.2,1){$q$};
\end{tikzpicture}
\end{align*}
implies that $q=-1$, $p\not=2$ or $q \in G'_{3}$, $p\not=3$ by  $a_{21}^{r_4r_3(X)}=-1$. Since $q=-1$ yields $r=1$, it is a contradiction. Rewrite $X$ as
\begin{align*}
    \Dchainseven{$q$}{$q^{-1}$}{$q$}{$q^{-1}$}{$-1$}{$q^{-1}$}{$q$}{$q^{-1}$}{$q$}{$q^{-1}$}{$q$}{$q^{-1}$}{$q$}
\end{align*}
where $q \in G'_{3}$, $p\not=3$. Hence $\mathcal{D}={\tau_{7654321}}\mathcal{D}^{7}_{21,1}$ and $p\ne 3$.

Step $1.2$. If $rq=-1$, then the reflection of $X$
\begin{align*}
r_3(X)\colon~
\begin{tikzpicture} 
  \tikzstyle{every node}=[scale=0.7]
  \draw(0,0) circle (1mm);
  \draw(1,0) circle (1mm);
  \filldraw(2,0) circle (1mm);
  \draw(3,0) circle (1mm);
  \draw(4,0) circle (1mm);
  \draw(5,0) circle (1mm);
  \draw(1.5,0.95) circle (1mm);
  \draw(0.1,0) -- (0.9,0) 
  (1.1,0) -- (1.9,0) 
  (2.1,0) -- (2.9,0) 
  (3.1,0) -- (3.9,0) 
  (4.1,0) -- (4.9,0) 
  (1,0.05) -- (1.45,0.9)
  (1.55,0.9) -- (2,0.05);
  \node at (0,0.35){$-q^{-1}$};
  \node at (0.45,0.25) {$-q$};
  \node at (0.9,0.28){$-1$};
  \node at (1.78,1.08){$-1$};
  \node at (1.5,0.2) {$-1$};
  \node at (2.05,0.28) {$-1$};
  \node at (2.55,0.25) {$q^{-1}$};
  \node at (1,0.63) {$-q^{-1}$};
  \node at (1.9,0.63) {$q$};
  \node at (3,0.3) {$q$};
  \node at (3.5,0.25) {$q^{-1}$};
  \node at (4,0.3) {$q$};
  \node at (4.5,0.25) {$q^{-1}$};
  \node at (5,0.3) {$q$};
\end{tikzpicture}
\quad \Rightarrow \quad
r_4r_3(X)\colon~
\begin{tikzpicture} 
  \tikzstyle{every node}=[scale=0.7]
  \draw(0,0) circle (1mm)
  (1,0) circle (1mm)
  (2,0) circle (1mm)
  (3,0) circle (1mm)
  (4,0) circle (1mm)
  (1.5,0.95) circle (1mm)
  (1.5,1.95) circle (1mm);
  \draw(0.1,0) -- (0.9,0) 
  (1.1,0) -- (1.9,0) 
  (2.1,0) -- (2.9,0) 
  (3.1,0) -- (3.9,0) 
  (2,0.1) -- (1.5,0.9)
  (1,0.1) -- (1.5,0.9)
  (1.5,1) -- (1.5,1.9);
  \node at (0,0.35){$-q^{-1}$};
  \node at (0.45,0.2) {$-q$};
  \node at (0.9,0.28){$-1$};
  \node at (1.8,1){$-1$};
  \node at (1.5,0.25) {$-q^{-1}$};
  \node at (2.05,0.28) {$-1$};
  \node at (2.55,0.25) {$q^{-1}$};
  \node at (1,0.6) {$-1$};
  \node at (1.95,0.6) {$q$};
  \node at (3,0.3) {$q$};
  \node at (3.5,0.25) {$q^{-1}$};
  \node at (4,0.3) {$q$};
  \node at (1.9,1.5) {$q^{-1}$};
  \node at (1.75,1.95) {$q$};
\end{tikzpicture}
\end{align*}
implies $q=-1$ and $p\ne 2$ by $a_{25}^{r_4r_3(X)}=0$. Rewrite $X$ as
\[
    \Dchainseven{$r$}{$r^{-1}$}{$r$}{$r^{-1}$}{$-1$}{$-1$}{$-1$}{$-1$}{$-1$}{$-1$}{$-1$}{$-1$}{$-1$}
\]
where $r\not=-1$ and $p\not=2$.
By Lemma~\ref{l-7chainspe-a} $(b_2)$, this configuration cannot occur.
  
Step $2$. Assume exactly two vertices carry label $-1$. Combined with $q_{33}=-1$, we divide the analysis into three subcases.

Step $2.1$. Let $q_{11}=q_{33}=-1$ and $q_{ii}\widetilde{q_{i,i+1}}=q_{jj}\widetilde{q_{j-1,j}}=1$ for all $i\in \{2,4,5,6\}$ and $j\in \{2,4,5,6,7\}$. Set $r\colon=q_{22}$ and $q\colon=q_{44}$. Then $rq \not=1$ by $a_{24}^{r_3(X)}=-1$. The reflections of $X$
\begin{align*}
\setlength{\unitlength}{1mm}
X\colon~
\Dchainsevenb{$-1$}{$r^{-1}$}{$r$}{$r^{-1}$}{$-1$}{$q^{-1}$}{$q$}{$q^{-1}$}{$q$}{$q^{-1}$}{$q$}{$q^{-1}$}{$q$}
\quad \Rightarrow \quad
r_1(X)\colon~
\Dchainsevena{$-1$}{$r$}{$-1$}{$r^{-1}$}{$-1$}{$q^{-1}$}{$q$}{$q^{-1}$}{$q$}{$q^{-1}$}{$q$}{$q^{-1}$}{$q$}
\end{align*}
\vspace{-3mm}
\begin{align*}
\quad \Rightarrow \quad
r_2r_1(X)\colon~
\Dchainseven{$r$}{$r^{-1}$}{$-1$}{$r$}{$r^{-1}$}{$q^{-1}$}{$q$}{$q^{-1}$}{$q$}{$q^{-1}$}{$q$}{$q^{-1}$}{$q$}
\end{align*}
imply that $r=-1$ and $p\ne 2$ by $a_{34}^{r_2r_1(X)}=-1$. By Lemma~\ref{l-7chainspe-a} $(c_1)$, this configuration cannot occur. 

Step $2.2$. Let $q_{22}=q_{33}=-1$ and $q_{ii}\widetilde{q_{i,i+1}}=q_{jj}\widetilde{q_{j-1,j}}=1$ for all $i\in \{1,4,5,6\}$ and $j\in \{4,5,6,7\}$. Set $r\colon=q_{11}$, $q\colon=\widetilde{q_{23}}$ and $s\colon=q_{44}$. Since $X$ has a good $A_7$ neighborhood, we have $r^{-1}q=1$ and $qs^{-1} \not=1$. The reflection of $X$
\begin{align*}    
\setlength{\unitlength}{1mm}
X\colon~
\Dchainsevena{$q$}{$q^{-1}$}{$-1$}{$q$}{$-1$}{$s^{-1}$}{$s$}{$s^{-1}$}{$s$}{$s^{-1}$}{$s$}{$s^{-1}$}{$s$}
\quad \Rightarrow \quad
r_2(X)\colon~
\Dchainseven{$-1$}{$q$}{$-1$}{$q^{-1}$}{$q$}{$s^{-1}$}{$s$}{$s^{-1}$}{$s$}{$s^{-1}$}{$s$}{$s^{-1}$}{$s$}
\end{align*}
implies $q=-1$ and $p\ne 2$ by $a_{34}^{r_2(X)}=-1$. By Lemma~\ref{l-7chainspe-a} $(c_1)$, this configuration cannot occur. 

Step $2.3$. Let $q_{33}=q_{kk}=-1$ for $k \in \{4,5,6,7\}$ and $q_{ii}q_{i,i+1}'=q_{jj}q_{j-1,j}'=1$ for all $i\in \{1,2,4,5,6\}\verb|\|\{k\}$ and $j\in \{2,4,5,6,7\}\verb|\|\{k\}$. Set $r\colon=q_{11}$, $q\colon=q'_{34}$ and $s\colon=q_{k-1,k}$. Then we have $r^{-1}q \not=1$ by $a^{r_3(X)}_{24}=-1$ and $qs=1$ by $A^{r_k(X)}=A_7$. Hence we obtain $q=-1$ and $p\neq 2$ by $a^{r_4\cdots r_k(X)}_{32}=-1$. By Lemma~\ref{l-7chainspe-a} $(b_2)$, these cases cannot occur. 

Step $3$. Assume exactly three vertices are labelled $-1$. Together with $q_{33}=-1$, we separate the discussion into seven subcases.
%we distinguish the following seven steps.

Step $3.1$. Let $q_{11}=q_{22}=q_{33}=-1$ and $q_{ii}\widetilde{q_{i,i+1}}=q_{jj}\widetilde{q_{j-1,j}}=1$ for all $i\in \{4,5,6\}$ and $j\in \{4,5,6,7\}$. Set $s\colon=\widetilde{q_{12}}$, $r\colon=\widetilde{q_{23}}$ and $q\colon=q_{44}$. We have $rq^{-1} \not=1$ by $a_{24}^{r_3(X)}=-1$ and $sr=1$ by $a_{13}^{r_2(X)}=0$.  Then the reflection of $X$
\begin{align*}    
\setlength{\unitlength}{1mm}
X\colon~
\Dchainsevena{$-1$}{$r^{-1}$}{$-1$}{$r$}{$-1$}{$q^{-1}$}{$q$}{$q^{-1}$}{$q$}{$q^{-1}$}{$q$}{$q^{-1}$}{$q$}
\quad \Rightarrow \quad
r_2(X)\colon~
\Dchainseven{$r^{-1}$}{$r$}{$-1$}{$r^{-1}$}{$r$}{$q^{-1}$}{$q$}{$q^{-1}$}{$q$}{$q^{-1}$}{$q$}{$q^{-1}$}{$q$}
\end{align*}
implies $r=-1$ and $p\ne 2$ by $a_{34}^{r_2(X)}=-1$. By Lemma~\ref{l-7chainspe-a} $(c_1)$, this configuration cannot occur.

Step $3.2$. Let $q_{11}=q_{33}=q_{kk}=-1$ for $k\in \{4,5,6,7\}$ and $q_{ii}\widetilde{q_{i,i+1}}=q_{jj}\widetilde{q_{j-1,j}}=1$ for all $i\in \{2,4,5,6\}\verb|\|\{k\}$ and $j\in \{2,4,5,6,7\}\verb|\|\{k\}$. Set $r\colon=q_{22}$, $s\colon=\widetilde{q_{34}}$, and $q\colon=\widetilde{q_{k-1,k}}$. By $a^{r_3(X)}_{24}=-1$, we have $rs\neq1$.
And for $k \neq 7$, $A^{r_k(X)}=A_7$ implies $\widetilde{q_{k,k+1}}=q^{-1}$. Then we obtain $s=-1$ and $p\ne 2$ by $a_{32}^{r_4(X)}=-1$. By Lemma~\ref{l-7chainspe-a} $(b_1)$, these cases cannot occur.

Step $3.3$. Let $q_{22}=q_{33}=q_{44}=-1$ and $q_{ii}\widetilde{q_{i,i+1}}=q_{jj}\widetilde{q_{j-1,j}}=1$ for all $i\in \{1,5,6\}$ and $j\in \{5,6,7\}$. Set $q\colon=q_{11}$, $s\colon=\widetilde{q_{23}}$, $t\colon=\widetilde{q_{34}}$ and $r\colon=q_{55}$. We have $r^{-1}s=1$ by $a^{r_2(X)}_{13}=0$, $st \not=1$ by $a^{r_3(X)}_{24}=0$, and $tq^{-1}=1$ by $a^{r_4(X)}_{35}=0$. Hence the reflection of $X$
\begin{align*}    
\setlength{\unitlength}{1mm}
X\colon~
\Dchainsevenf{$s$}{$s^{-1}$}{$-1$}{$s$}{$-1$}{$t$}{$-1$}{$t^{-1}$}{$t$}{$t^{-1}$}{$t$}{$t^{-1}$}{$t$}
\quad \Rightarrow \quad
r_4(X)\colon~
\Dchainseven{$s$}{$s^{-1}$}{$-1$}{$s$}{$t$}{$t^{-1}$}{$-1$}{$t$}{$-1$}{$t^{-1}$}{$t$}{$t^{-1}$}{$t$}
\end{align*}
implies $t=-1$ and $p\ne 2$ by $a_{32}^{r_4(X)}=-1$. By Lemma~\ref{l-7chainspe-a} $(a_2)$, this configuration cannot occur.

Step $3.4$. Let $q_{22}=q_{33}=q_{55}=-1$ and $q_{ii}\widetilde{q_{i,i+1}}=q_{jj}\widetilde{q_{j-1,j}}=1$ for all $i\in \{1,4,6\}$ and $j\in \{4,6,7\}$. Set $r\colon=q_{11}$, $s\colon=\widetilde{q_{23}}$, $q\colon=q_{44}$ and $t\colon=q_{66}$. We have $r^{-1}s=1$ by $a^{r_2(X)}_{13}=0$, $sq^{-1}\not=1$ by $a^{r_3(X)}_{24}=0$, and $qt=1$ by $a^{r_5(X)}_{46}=0$. Then the reflection of $X$
\begin{align*}    
\setlength{\unitlength}{1mm}
X\colon~
\Dchainsevena{$s$}{$s^{-1}$}{$-1$}{$s$}{$-1$}{$q^{-1}$}{$q$}{$q^{-1}$}{$-1$}{$q$}{$q^{-1}$}{$q$}{$q^{-1}$}
\quad \Rightarrow \quad
r_2(X)\colon~
\Dchainseven{$-1$}{$s$}{$-1$}{$s^{-1}$}{$s$}{$q^{-1}$}{$q$}{$q^{-1}$}{$-1$}{$q$}{$q^{-1}$}{$q$}{$q^{-1}$}
\end{align*}
implies $s=-1$ and $p\ne 2$ by $a_{34}^{r_2(X)}=-1$. By Lemma~\ref{l-7chainspe-a} $(c_8)$, this configuration cannot occur.

Step $3.5$. Let $q_{22}=q_{33}=q_{66}=-1$ and $q_{ii}\widetilde{q_{i,i+1}}=q_{jj}\widetilde{q_{j-1,j}}=1$ for all $i\in \{1,4,5\}$ and $j\in \{4,5,7\}$. Set $r\colon=q_{11}$, $s\colon=\widetilde{q_{23}}$, $q\colon=q_{44}$ and $t\colon=q_{77}$. From Definition \ref{defA7} and Lemma \ref{jslemma}, we have $r^{-1}s=qt=1$ and $sq^{-1} \not=1$. Furthermore, by $a_{34}^{r_2(X)}=-1$, we obtain $s=-1$ and $p\ne 2$. By Lemma~\ref{l-7chainspe-a} $(c_3)$, this configuration cannot occur. 

Step $3.6$. Let $q_{22}=q_{33}=q_{77}=-1$ and $q_{ii}\widetilde{q_{i,i+1}}=q_{jj}\widetilde{q_{j-1,j}}=1$ for all $i\in \{1,4,5,6\}$ and $j\in \{4,5,6\}$. Set $r\colon=q_{11}$, $s\colon=\widetilde{q_{23}}$ and $q\colon=q_{44}$. We obtain $r^{-1}s=1$ by $a^{r_2(X)}_{13}=0$ and $sq^{-1} \not=1$ by $a^{r_3(X)}_{24}=0$. As in the preceding step, we obtain $s=-1$ and $p\ne 2$ by $a_{34}^{r_2(X)}=-1$. By Lemma~\ref{l-7chainspe-a} $(c_2)$, this configuration cannot occur.

Step $3.7$. Let $q_{33}=q_{i_1,i_1}=q_{i_2,i_2}=-1$ for $i_1, i_2\in \{4,5,6,7\}$ and $q_{ii}\widetilde{q_{i,i+1}}=q_{jj}\widetilde{q_{j-1,j}}=1$ for all $i\in \{1,2,4,5,6\}\verb|\|\{i_1,i_2\}$ and $j\in \{2,4,5,6,7\}\verb|\|\{i_1,i_2\}$. Set $s\colon=\widetilde{q_{23}}$, $q\colon=\widetilde{q_{i_1-1,i_1}}$ and $t\colon=\widetilde{q_{i_2-1,i_2}}$. By Definition \ref{defA7}, we have $qt=1$ and $sq\neq1$. And $q_{77}=t$ if $i_2\neq 7$. Combining this with $a^{r_4\cdots r_{i_1}}_{32}=-1$, we deduce $s=-1$ and $p\neq 2$. In view of condition $(b_2)$ from Lemma~\ref{l-7chainspe-a}, such configurations cannot occur.

Step $4$. Assume exactly four vertices take label $-1$. Owing to $q_{33}=-1$, we partition the argument into ten subcases.

Step $4.1$. Let $q_{11}=q_{22}=q_{33}=q_{44}=-1$ and $q_{ii}\widetilde{q_{i,i+1}}=q_{jj}\widetilde{q_{j-1,j}}=1$ for all $i\in \{5,6\}$ and $j\in \{5,6,7\}$. Set $r\colon=\widetilde{q_{12}}$, $s\colon=\widetilde{q_{23}}$, $t\colon=\widetilde{q_{34}}$ and $q\colon=q_{55}$. Then we have $rs=1$ by $a_{13}^{r_2(X)}=0$, $tq^{-1}=1$ by $a_{35}^{r_4(X)}=0$ and $st \not=1$ by $a_{24}^{r_3(X)}=-1$. 
Hence, by $a_{32}^{r_4(X)}=-1$, we have $t=-1$ and $p\ne 2$. By Lemma~\ref{l-7chainspe-a} $(a_1)$, this configuration cannot occur. 

Step $4.2$. Let $q_{11}=q_{22}=q_{33}=q_{55}=-1$ and $q_{ii}\widetilde{q_{i,i+1}}=q_{jj}\widetilde{q_{j-1,j}}=1$ for all $i\in \{4,6\}$ and $j\in \{4,6,7\}$. Set $r\colon=\widetilde{q_{12}}$, $s\colon=\widetilde{q_{23}}$, $t\colon=q_{44}$ and $q\colon=q_{66}$. Since $X$ has a good $A_7$ neighborhood, we have $rs=tq=1$ and $st^{-1} \not=1$. 
Hence, it follows from $a_{34}^{r_2(X)}=-1$ that $s=-1$ and $p\ne 2$. By Lemma~\ref{l-7chainspe-a} $(c_8)$, this configuration is excluded.

%Consider the symmetric case that $q_{33}=q_{55}=q_{66}=q_{77}=-1$ and $q_{ii}\widetilde{q_{i,i+1}}-1=q_{jj}\widetilde{q_{j-1,j}}-1=0$ for all $i\in \{1,2,4\}$ and $j\in \{2,4\}$. Similarly to the proof in Step $2.4$, we ultimately obtain that this configuration is ruled out.
%the $\mathcal{D}$ in this case is
%\begin{align*}
%    \Dchainseven{$r$}{$r^{-1}$}{$r$}{$r^{-1}$}{$-1$}{$-1$}{$-1$}{$-1$}{$-1$}{$-1$}{$-1$}{$-1$}{$-1$}
%\end{align*} 
%By Lemma~\ref{l-7chainspe-a} $(b_2)$, this configuration is ruled out.

Step $4.3$. Let $q_{11}=q_{22}=q_{33}=q_{66}=-1$ and $q_{ii}\widetilde{q_{i,i+1}}=q_{jj}\widetilde{q_{j-1,j}}=1$ for all $i\in \{4,5\}$ and $j\in \{4,5,7\}$. Set $r\colon=\widetilde{q_{12}}$, $s\colon=\widetilde{q_{23}}$, $t\colon=q_{44}$ and $q\colon=q_{77}$. By Lemma \ref{jslemma} and Definition \ref{defA7}, we have $rs=1$, $tq=1$ and $st^{-1} \not=1$. 
Then the reflection of $X$
\begin{align*}    
\setlength{\unitlength}{1mm}
X\colon~
\Dchainsevena{$-1$}{$s^{-1}$}{$-1$}{$s$}{$-1$}{$t^{-1}$}{$t$}{$t^{-1}$}{$t$}{$t^{-1}$}{$-1$}{$t$}{$t^{-1}$}
\quad \Rightarrow \quad
r_2(X)\colon~
\Dchainseven{$s^{-1}$}{$s$}{$-1$}{$s^{-1}$}{$s$}{$t^{-1}$}{$t$}{$t^{-1}$}{$t$}{$t^{-1}$}{$-1$}{$t$}{$t^{-1}$}
\end{align*}
together with the condition $a_{34}^{r_2(X)}=-1$ forces $s=-1$ and $p\ne 2$. Invoking Lemma~\ref{l-7chainspe-a} $(c_3)$, we rule out this configuration. 

Step $4.4$. Let $q_{11}=q_{22}=q_{33}=q_{77}=-1$ and $q_{ii}\widetilde{q_{i,i+1}}=q_{jj}\widetilde{q_{j-1,j}}=1$ for all $i\in \{4,5,6\}$ and $j\in \{4,5,6\}$. As in Step $4.3$ above, Lemma~\ref{l-7chainspe-a}$(c_2)$.

Step $4.5$. Let $q_{11}=q_{33}=q_{i_1,i_1}=q_{i_2,i_2}=-1$ for $i_1,i_2 \in \{4,5,6,7\}$ and $q_{ii}\widetilde{q_{i,i+1}}=q_{jj}\widetilde{q_{j-1,j}}=1$ for all $i\in \{2,4,5,6\}\verb|\|\{i_1,i_2\}$ and $j\in \{2,4,5,6,7\}\verb|\|\{i_1,i_2\}$. Set $r\colon=q_{22}$, $s\colon=\widetilde{q_{34}}$ and $t\colon=\widetilde{q_{i_1,i_1+1}}$. Then $rs\neq 1$ and $st=1$ by Definition \ref{defA7}. And $\widetilde{q_{i_2,i_2+1}}=q$ if $i_2\neq 7$. Thus, the relation $a^{r_4\cdots r_{i_1}(X)}_{32}=-1$ forces $s=-1$ and $q\neq 2$. Lemma~\ref{l-7chainspe-a} $(b_1)$ teliminates all such cases.

 Step $4.6$. Let $q_{22}=q_{33}=q_{55}=q_{77}=-1$ and $q_{ii}\widetilde{q_{i,i+1}}=q_{jj}\widetilde{q_{j-1,j}}=1$ for all $i\in \{1,4,6\}$ and $j\in \{4,6\}$. Set $r\colon=q_{11}$, $s\colon=\widetilde{q_{23}}$, $t\colon=q_{44}$ and $q\colon=q_{66}$. Here we get $r^{-1}s=tq=1$ and $st^{-1} \not=1$. The remaining analysis follows the argument of Step $4.3$, and Lemma~\ref{l-7chainspe-a} $(c_7)$ rules out this case.

Step $4.7$. Let $q_{22}=q_{33}=q_{44}=q_{kk}=-1$ for $k\in \{5,6,7\}$ and $q_{ii}\widetilde{q_{i,i+1}}=q_{jj}\widetilde{q_{j-1,j}}=1$ for all $i\in \{1,5,6\}\verb|\|\{k\}$ and $j\in \{5,6,7\}\verb|\|\{k\}$. Set $r\colon=q_{11}$, $s\colon=\widetilde{q_{23}}$, $t\colon=\widetilde{q_{34}}$, and $u\colon=\widetilde{q_{45}}$. From Definition \ref{defA7}, we deduce $r^{-1}s=tu=1$ and $st \not=1$. When $k\neq7$, the condition $A^{r_k(X)}=A_7$ yields $\widetilde{q_{k,k+1}}=u$. Hence we obtain $t=-1$ and $p\ne 2$ by $a_{32}^{r_4(X)}=-1$. And Lemma~\ref{l-7chainspe-a} $(a_2)$, excludes all such configurations.  

Step $4.8$. Let $q_{22}=q_{33}=q_{55}=q_{66}=-1$ and $q_{ii}\widetilde{q_{i,i+1}}=q_{jj}\widetilde{q_{j-1,j}}=1$ for all $i\in \{1,4\}$ and $j\in \{4,7\}$. Set $r\colon=q_{11}$, $s\colon=\widetilde{q_{23}}$, $t\colon=q_{44}$ and $q\colon=q_{77}$. From Definition \ref{defA7}, we deduce $r^{-1}s=t^{-1}u=uq^{-1}=1$ and $st^{-1} \not=1$. The subsequent discussion is analogous to that in Step $4.3$. Lemma~\ref{l-7chainspe-a} $(c_6)$ rules out this case.

Step $4.9$. Let $q_{22}=q_{33}=q_{66}=q_{77}=-1$ and $q_{ii}\widetilde{q_{i,i+1}}=q_{jj}\widetilde{q_{j-1,j}}=1$ for all $i\in \{1,4,5\}$ and $j\in \{4,5\}$. Set $r\colon=q_{11}$, $s\colon=\widetilde{q_{23}}$, $t\colon=q_{44}$ and $q\colon=\widetilde{q_{67}}$. From Definition \ref{defA7}, we have $r^{-1}s=t^{-1}q=1$ and $st^{-1} \not=1$. 
Following the argument of Step $4.3$, the identity $a_{34}^{r_2(X)}=-1$ forces $s=-1$ and $p\ne 2$ and Lemma~\ref{l-7chainspe-a} $(c_4)$ excludes this case. 

Step $4.10$. Let $q_{33}=q_{i_1,i_1}=q_{i_2,i_2}=q_{i_3,i_3}=-1$ for $i_1,i_2,i_3\in \{4,5,6,7\}$ and $q_{ii}\widetilde{q_{i,i+1}}=q_{jj}\widetilde{q_{j-1,j}}=1$ for all $i\in \{1,2,4,5,6\}\verb|\|\{i_1,i_2,i_3\}$ and $j\in \{2,4,5,6,7\}\verb|\|\{i_1,i_2,i_3\}$. Set $r\colon=q_{11}$, $s\colon=\widetilde{q_{34}}$,$t\colon=\widetilde{q_{i_1,i_1+1}}$ and $q\colon=\widetilde{q_{i_2,i_2+1}}$. By Lemma \ref{jslemma}, we obtain that $st=tq=1$ and $r^{-1}s \not=1$. And $q_{77}=q$ if $i_3 \neq 7$. Then we obtain $s=-1$ and $p\ne 2$ from $a_{32}^{r_4\cdots r_{i_1}(X)}=-1$. Combined with Lemma~\ref{l-7chainspe-a}$(b_2)$, such a configuration cannot occur.

Step $5$. Suppose exactly five vertices are labelled $-1$. Since $q_{33}=-1$, we divide the argument into seven subsequent cases.

Step $5.1$. Let $q_{11}=q_{22}=q_{33}=q_{44}=q_{kk}=-1$, $q_{ii}\widetilde{q_{i,i+1}}=q_{jj}\widetilde{q_{j-1,j}}=1$ for $i\in \{5,6\}\verb|\|\{k\}$ and $j\in \{6,7\}\verb|\|\{k\}$. Set $r\colon=\widetilde{q_{12}}$, $s\colon=\widetilde{q_{23}}$, $t\colon=\widetilde{q_{34}}$, and $u\colon=\widetilde{q_{45}}$. Then $rs=1$ by $a_{13}^{r_2(X)}=0$, $st \not=1$ by $a_{24}^{r_3(X)}=-1$, and $tu=1$ by $a_{35}^{r_4(X)}=0$. And for $k\neq7$, $A^{r_k(X)}=A_7$ yields $\widetilde{q_{k,k+1}}=u^{-1}$. Hence $a_{32}^{r_4(X)}=-1$ forces $t=-1$ and $p\ne 2$. Combined with Lemma~\ref{l-7chainspe-a} $(a_1)$, these cases cannot occur.   

By Lemma~\ref{l-7chainspe-a} $(b_2)$ and arguments analogous to those above, the configuration satisfying $q_{33}=q_{44}=q_{55}=q_{66}=q_{77}=-1$, $q_{22}\widetilde{q_{12}}=1$, and $q_{ii}\widetilde{q_{i,i+1}}=1$ for $i\in \{2,4,5\}$ is likewise impossible.

Step $5.2$. Let $q_{11}=q_{22}=q_{33}=q_{55}=q_{66}=-1$, $q_{44}\widetilde{q_{45}}=1$ , together with $q_{jj}\widetilde{q_{j-1,j}}=1$ for $j\in \{4,7\}$. Set $r\colon=\widetilde{q_{12}}$, $s\colon=\widetilde{q_{23}}$, $t\colon=q_{44}$, $u\colon=\widetilde{q_{56}}$ and $q\colon=q_{77}$. By Definition \ref{defA7}, we have $rs=t^{-1}u=uq^{-1}=1$ and $st^{-1} \not=1$. Then the reflection of $X$
\begin{align*}    
\setlength{\unitlength}{1mm}
X\colon~
\Dchainsevena{$-1$}{$s^{-1}$}{$-1$}{$s$}{$-1$}{$t^{-1}$}{$t$}{$t^{-1}$}{$-1$}{$t$}{$-1$}{$t^{-1}$}{$t$}
\quad \Rightarrow \quad
r_2(X)\colon~
\Dchainseven{$s^{-1}$}{$s$}{$-1$}{$s^{-1}$}{$s$}{$t^{-1}$}{$t$}{$t^{-1}$}{$-1$}{$t$}{$-1$}{$t^{-1}$}{$t$}
\end{align*}
implies $s=-1$ and $p\ne 2$ by $a_{34}^{r_2(X)}=-1$. Combined with Lemma~\ref{l-7chainspe-a} $(c_6)$, such a configuration cannot occur.

A symmetric case arises when $q_{22}=q_{33}=q_{55}=q_{66}=q_{77}=-1$, $q_{44}\widetilde{q_{34}}=1$ and $q_{ii}\widetilde{q_{i,i+1}}=1$ for $i\in \{1,4\}$. By the same reasoning as before, this case is excluded by Lemma~\ref{l-7chainspe-a} $(c_5)$.

Step $5.3$. Let $q_{11}=q_{22}=q_{33}=q_{55}=q_{77}=-1$ and $q_{ii}\widetilde{q_{i,i+1}}=q_{jj}\widetilde{q_{j-1,j}}=1$ for all $i\in \{4,6\}$ and $j\in \{4,6\}$.  
We impose the same reflection constraints as in the prior step, whence Lemma~\ref{l-7chainspe-a}$(c_7)$ eliminates this configuration

Step $5.4$. Let $q_{11}=q_{22}=q_{33}=q_{66}=q_{77}=-1$ and $q_{ii}\widetilde{q_{i,i+1}}=q_{jj}\widetilde{q_{j-1,j}}=1$ for all $i\in \{4,5\}$ and $j\in \{4,5\}$. Carrying out computations parallel to those of Step $5.3$, Lemma~\ref{l-7chainspe-a}(c$_4$) rules out this case.

Step $5.5$. Let $q_{11}=q_{33}=q_{i_1,i_1}=q_{i_2,i_2}=q_{i_3,i_3}=-1$ for $i_1,i_2,i_3 \in \{4,5,6,7\}$ and  $q_{ii}\widetilde{q_{i+1,i}}=q_{jj}\widetilde{q_{j-1,j}}=1$ for all $i \in \{2,4,5,6\}\verb|\|\{i_1,i_2,i_3\}$ and $j\in \{2,4,5,6,7\}\verb|\|\{i_1,i_2,i_3\}$. 
Set $r\colon=q_{22}$, $s\colon=\widetilde{q_{34}}$, $t\colon=\widetilde{q_{45}}$, $u\colon=\widetilde{q_{56}}$ and $q\colon=q_{77}$. By Definition \ref{defA7}, we have $st=tu=uq^{-1}=1$ and $r^{-1}s \not=1$. Then the reflection of $X$
\begin{align*}    
\setlength{\unitlength}{1mm}
X\colon~
\Dchainsevenf{$-1$}{$r^{-1}$}{$r$}{$r^{-1}$}{$-1$}{$s$}{$-1$}{$s^{-1}$}{$-1$}{$s$}{$-1$}{$s^{-1}$}{$s$}
\quad \Rightarrow \quad
r_4(X)\colon~
\Dchainseven{$-1$}{$r^{-1}$}{$r$}{$r^{-1}$}{$s$}{$s^{-1}$}{$-1$}{$s$}{$s^{-1}$}{$s$}{$-1$}{$s^{-1}$}{$s$}
\end{align*}
together with the condition $a_{32}^{r_4(X)}=-1$,
forces $s=-1$ and $p\ne 2$. Lemma~\ref{l-7chainspe-a}$(b_1)$ excludes all such configurations.

Step $5.6$. Let $q_{22}=q_{33}=q_{44}=q_{55}=q_{66}=-1$ and $q_{11}\widetilde{q_{12}}=q_{77}\widetilde{q_{67}}=1$. Set $r\colon=q_{11}$, $s\colon=\widetilde{q_{23}}$, $t\colon=\widetilde{q_{34}}$, $u\colon=\widetilde{q_{45}}$, $v\colon=\widetilde{q_{56}}$ and $q\colon=q_{77}$. By Definition \ref{defA7}, we have $r^{-1}s=tu=uv=vq^{-1}=1$ and $st \not=1$. Then the reflection of $X$
\begin{align*}    
\setlength{\unitlength}{1mm}
X\colon~
\Dchainsevenf{$s$}{$s^{-1}$}{$-1$}{$s$}{$-1$}{$t$}{$-1$}{$t^{-1}$}{$-1$}{$t$}{$-1$}{$t^{-1}$}{$t$}
\quad \Rightarrow \quad
r_4(X)\colon~
\Dchainseven{$s$}{$s^{-1}$}{$-1$}{$s$}{$t$}{$t^{-1}$}{$-1$}{$t$}{$t^{-1}$}{$t$}{$-1$}{$t^{-1}$}{$t$}
\end{align*}
together with the condition $a_{32}^{r_4(X)}=-1$, forces $t=-1$ and $p\ne 2$. Lemma~\ref{l-7chainspe-a} $(a_2)$ rules out this case.

Step $5.7$. Let $q_{22}=q_{33}=q_{44}=q_{66}=q_{77}=-1$, $q_{55}\widetilde{q_{45}}=1$ and $q_{ii}\widetilde{q_{i,i+1}}=1$ for $i\in \{1,5\}$. Set $r\colon=q_{11}$, $s\colon=\widetilde{q_{23}}$, $t\colon=\widetilde{q_{34}}$, $u\colon=q_{55}$ and $q\colon=\widetilde{q_{67}}$. By Lemma\ref{jslemma}, we obtain $r^{-1}s=tu^{-1}=u^{-1}q=1$ and $st \not=1$. This case is excluded by reasoning parallel to Step $5.6$, using identical reflection conditions.

Step $6$. Assume exactly six vertices bear the label $-1$. In view of $q_{33}=-1$, we split the analysis into four subcases.

Step $6.1$. Let $q_{11}=q_{22}=q_{33}=q_{44}=q_{55}=q_{66}=-1$ and $q_{77}\widetilde{q_{67}}=1$. Set $r\colon=\widetilde{q_{12}}$, $s\colon=\widetilde{q_{23}}$, $t\colon=\widetilde{q_{34}}$, $u\colon=\widetilde{q_{45}}$, $v\colon=\widetilde{q_{56}}$ and $q\colon=q_{77}$. By Lemma \ref{jslemma} and Definition \ref{defA7}, we have $rs=tu=uv=vq^{-1}=1$ and $st \not=1$. Then the reflection of $X$
\begin{align*}    
\setlength{\unitlength}{1mm}
X\colon~
\Dchainsevenf{$-1$}{$s^{-1}$}{$-1$}{$s$}{$-1$}{$t$}{$-1$}{$t^{-1}$}{$-1$}{$t$}{$-1$}{$t^{-1}$}{$t$}
\quad \Rightarrow \quad
r_4(X)\colon~
\Dchainseven{$-1$}{$s^{-1}$}{$-1$}{$s$}{$t$}{$t^{-1}$}{$-1$}{$t$}{$t^{-1}$}{$t$}{$-1$}{$t^{-1}$}{$t$}
\end{align*}
implies $t=-1$ and $p\ne 2$ by $a_{32}^{r_4(X)}=-1$. Lemma~\ref{l-7chainspe-a} $(a_1)$ precludes such a case. 

Consider the symmetric case that $q_{22}=q_{33}=q_{44}=q_{55}=q_{66}=q_{77}=-1$ and $q_{11}\widetilde{q_{12}}=1$. Similarly to the proof above, this case is ruled out by Lemma~\ref{l-7chainspe-a}$(a_2)$.

Step $6.2$. Let $q_{11}=q_{22}=q_{33}=q_{44}=q_{55}=q_{77}=-1$ and $q_{66}\widetilde{q_{56}}=q_{66}\widetilde{q_{67}}=1$. Set $r\colon=\widetilde{q_{12}}$, $s\colon=\widetilde{q_{23}}$, $t\colon=\widetilde{q_{34}}$, $u\colon=\widetilde{q_{45}}$ and $q\colon=q_{66}$. Then $rs=tu=uq^{-1}=1$ and $st \not=1$. Hence the reflection of $X$
\begin{align*}    
\setlength{\unitlength}{1mm}
X\colon~
\Dchainsevenf{$-1$}{$s^{-1}$}{$-1$}{$s$}{$-1$}{$t$}{$-1$}{$t^{-1}$}{$-1$}{$t$}{$t^{-1}$}{$t$}{$-1$}
\quad \Rightarrow \quad
r_4(X)\colon~
\Dchainseven{$-1$}{$s^{-1}$}{$-1$}{$s$}{$t$}{$t^{-1}$}{$-1$}{$t$}{$t^{-1}$}{$t$}{$t^{-1}$}{$t$}{$-1$}
\end{align*}
implies $t=-1$ and $p\ne 2$ by $a_{32}^{r_4(X)}=-1$. Lemma~\ref{l-7chainspe-a}$(a_1)$ excludes this case.

Consider the symmetric case that $q_{11}=q_{33}=q_{44}=q_{55}=q_{66}=q_{77}=-1$ and $q_{22}\widetilde{q_{12}}=q_{22}\widetilde{q_{23}}=1$. By an argument analogous to the one above, Lemma~\ref{l-7chainspe-a}$(b_1)$ excludes this case.

Step $6.3$. Let $q_{11}=q_{22}=q_{33}=q_{44}=q_{66}=q_{77}=-1$ and $q_{55}\widetilde{q_{45}}=q_{55}\widetilde{q_{56}}=1$. Set $r\colon=\widetilde{q_{12}}$, $s\colon=\widetilde{q_{23}}$, $t\colon=\widetilde{q_{34}}$, $u\colon=q_{55}$ and $q\colon=\widetilde{q_{67}}$. Since $X$ has a good $A_7$ neighborhood, we have $rs=tu^{-1}=u^{-1}q=1$ and $st \not=1$. Hence $t=-1$ and $p\ne 2$ by $a_{32}^{r_4(X)}=-1$. The same reflection conditions and results from step $6.2$ shows that this configuration cannot occur.

Step $6.4$. Let $q_{11}=q_{22}=q_{33}=q_{55}=q_{66}=q_{77}=-1$ and $q_{44}\widetilde{q_{34}}=q_{44}\widetilde{q_{45}}=1$. Set $r\colon=\widetilde{q_{12}}$, $s\colon=\widetilde{q_{23}}$, $t\colon=q_{44}$, $u\colon=\widetilde{q_{56}}$ and $q\colon=\widetilde{q_{67}}$. By Definition \ref{defA7}, we obtain $rs=t^{-1}u=uq=1$ and $st^{-1} \not=1$.
Then the reflection of $X$
\begin{align*}    
\setlength{\unitlength}{1mm}
X\colon~
\Dchainsevena{$-1$}{$s^{-1}$}{$-1$}{$s$}{$-1$}{$t^{-1}$}{$t$}{$t^{-1}$}{$-1$}{$t$}{$-1$}{$t^{-1}$}{$-1$}
\quad \Rightarrow \quad
r_2(X)\colon~
\Dchainseven{$s^{-1}$}{$s$}{$-1$}{$s^{-1}$}{$s$}{$t^{-1}$}{$t$}{$t^{-1}$}{$-1$}{$t$}{$-1$}{$t^{-1}$}{$-1$}
\end{align*}
implies $s=-1$ and $p\ne 2$ by $a_{34}^{r_2(X)}=-1$. Lemma~\ref{l-7chainspe-a}$(c_5)$ eliminates this case.

Step $7$. Suppose all vertices are labeled $-1$. That is $q_{11}=q_{22}=q_{33}=q_{44}=q_{55}=q_{66}=q_{77}=-1$. Set $r\colon=\widetilde{q_{12}}$, $s\colon=\widetilde{q_{23}}$, $t\colon=\widetilde{q_{34}}$, $u\colon=\widetilde{q_{45}}$, $v\colon=\widetilde{q_{56}}$, and $q\colon=\widetilde{q_{67}}$. Then in this case, we have $rs=tu=uv=vq=1$ and $st \not=1$. Hence the reflection of $X$
\begin{align*}    
\setlength{\unitlength}{1mm}
X\colon~
\Dchainsevenf{$-1$}{$s^{-1}$}{$-1$}{$s$}{$-1$}{$t$}{$-1$}{$t^{-1}$}{$-1$}{$t$}{$-1$}{$t^{-1}$}{$-1$}
\quad \Rightarrow \quad
r_4(X)\colon~
\Dchainseven{$-1$}{$s^{-1}$}{$-1$}{$s$}{$t$}{$t^{-1}$}{$-1$}{$t$}{$t^{-1}$}{$t$}{$-1$}{$t^{-1}$}{$-1$}
\end{align*}
implies $t=-1$ and $p\ne 2$ by $a_{32}^{r_4(X)}=-1$. Lemma~\ref{l-7chainspe-a}$(a_1)$ eliminates this case.
\end{proof}

\begin{theorem}\label{theo.rank8}
Suppose $\roots^{[M]}$ is the finite root system of a sporadic finite Cartan graph of rank $8$. Then every admissible generalized Dynkin diagram of $V$ appears in the $22$nd row of the table \ref{tab.1}.
\end{theorem}
\begin{proof}
By Theorem \ref{thm:goodnei}, $C(M)$ is standard of type $E_8$. Then $A^{X}=E_{8}$ yields $(2)_{q_{22}}(q_{22}\widetilde{q_{24}}-1)=(2)_{q_{ii}}(q_{ii}\widetilde{q_{i,i+1}}-1)=(2)_{q_{jj}}(q_{jj}\widetilde{q_{j-1,j}}-1)=0$ for all $i\in \{1,3,4,5,6,7\}$ and $j\in \{2,4,5,6,7,8\}$. Set $q\colon=\widetilde{q_{12}}$. By $A^X=E_8$, one obtains $\cD=\cD^8_{22,1}$ if $q\neq-1$ and $\cD=\cD^8_{22',1}$ if $q=-1$.
\end{proof}

\begin{proof}[\textbf{Proof of Theorem~\ref{theoremking}}]
The forward implication is straightforward. Indeed, suppose the generalized Dynkin diagram $\mathcal{D}$ of $V$ lies in any row $s$ or $s'$ of Table~\ref{tab.1}. By Lemmas~\ref{lem:aij} and~\ref{jslemma}, $M$ admits all reflections, so we may attach a Cartan graph $\mathcal{C}(M)$ to $V$ via Theorem~\ref{theo.regualrcar}. This Cartan graph $\mathcal{C}(M)$ coincides with the one constructed from the Dynkin diagrams in row $s$ of Table~4 in~\cite{HeckCAS}, as recorded in the third column of Table~\ref{tab.2}. Furthermore, the arithmetic root systems associated to this Cartan graph are finite; see~\cite[Theorem~22]{HeckCAS}. Consequently, the Weyl groupoid $\mathcal{W}(\mathcal{C}(M))$ is finite, and Theorem~\ref{thm:rootofR_M} implies that $\mathcal{B}(V)$ possesses a finite set of roots $\Delta^{[M]}$.

Next we prove the converse direction. Assume that $\cB(V)$ has a finite set of roots ${\roots}^{[M]}$. By Theorem ~\ref{thm:rootofR_M}, $\cC(M)$ is a finite Cartan graph. According to Theorem \ref{rootsystems}, any such finite Cartan graph $\cC(M)$ is either non-sporadic or sporadic. If $\cC(M)$ is non-sporadic , Theorem \ref{Dydiagrams} implies that the generalized Dynkin diagram $\cD$ of $V$ belongs to rows $1$ to $10$ of Table \ref{tab.1}. If $\cC(M)$ is sporadic, it follows from Theorem \ref{theo.rank5}, \ref{theo.rank6}, \ref{theo.rank7}, and \ref{theo.rank8} that the generalized Dynkin diagram $\cD$ of $V$ lies in rows $11$ to $22$ of Table \ref{tab.1}. 
\end{proof}

\begin{remark}
If a generalized Dynkin diagram $\mathcal{D}$ appears in a row of Table \ref{tab.1}, then that row contains exactly the generalized Dynkin diagrams of all points of $\mathcal{C}(M)$. The corresponding exchange graphs for the sporadic case of $\mathcal{C}(M)$ are listed in Table \ref{tab.2}.
\end{remark}
%\begin{remark}
%Notice that the row of Table \ref{tab.1} containing $\mathcal{D}$ consists precisely of the generalized Dynkin diagrams of all points of $\mathcal{C}(M)$. That is, the row containing $\mathcal{D}$ in the Table \ref{tab.1} consists of all the generalized Dynkin diagrams that are Weyl equivalent to $\mathcal{D}$.
%\end{remark}

\section{Generalized Dynkin diagrams and exchanges graphs}
\input{table}

\section{Appendix}\label{app:rslist}
\begin{appendix}
\textnormal{For brevity, we adopt the product notation \(\prod_{i=1}^r i^{x_i}\) as a shorthand for the sum \(\sum_{i=1}^r x_i \alpha_{r + 1 - i}\).}

%\section{Rank $3$}
%\textbf{Rank $3$}
\hspace{12pt}\\
\baselineskip=10pt
\begin{tiny}
\noindent
\textnormal{ \textbf{Rank $3 \colon$}
Nr. $1$ with 6 positive roots (type $A_3$)$\colon$ $1$, $2$, $3$, $12$, $13$, $123$\\
Nr. $2$ with 7 positive roots (type $D'(3,1)$)$\colon$ $1$, $2$, $3$, $12$, $13$, $23$, $123$\\
Nr. $3$ with 8 positive roots (type $D'(3,2)$) $\colon$ $1$, $2$, $3$, $12$, $13$, $1^22$, $123$, $1^223$\\ 
Nr. $4$ with 9 positive roots (type $C_3$) $\colon$  $1$, $2$, $3$, $12$, $13$, $1^22$, $123$, $1^223$, $1^223^2$\\
Nr. $5$ with 9 positive roots (type $B_3$) $\colon$  $1$, $2$, $3$, $12$, $23$, $1^22$, $123$, $1^223$, $1^22^23$
}
\end{tiny}

%\subsection{Rank $4$}
\hspace{12pt}\\
\baselineskip=10pt
\begin{tiny}
\noindent
\textnormal{\textbf{Rank $4 \colon$}
Nr. $1$ with 10 positive roots (type $A_4$)$\colon$ $1$, $2$, $3$, $4$, $12$, $23$, $34$, $123$, $234$, $1234$\\
Nr. $2$ with 12 positive roots (type $D_4$)$\colon$ $1$, $2$, $3$, $4$, $13$, $23$, $34$, $123$, $134$, $234$, $1234$, $123^24$\\
Nr. $3$ with 13 positive roots (type $D'(4,1)$)$\colon$ $1$, $2$, $3$, $4$, $12$, $13$, $23$, $34$, $123$, $134$, $234$, $1234$, $123^24$\\
Nr. $4$ with 14 positive roots (type $D'(4,2)$)$\colon$ $1$, $2$, $3$, $4$, $12$, $13$, $23$, $34$, $123$, $134$, $234$, $123^2$, $1234$, $123^24$\\
Nr. $5$ with 15 positive roots (type $D'(4,3)$)$\colon$ $1$, $2$, $3$, $4$, $12$, $13$, $23$, $34$, $123$, $134$, $234$, $123^2$, $1234$, $123^24$, $123^24^2$\\
Nr. $6$ with 16 positive roots (type $C_4$)$\colon$ $1$, $2$, $3$, $4$, $12$, $13$, $34$, $12^2$, $123$, $134$, $12^23$, $1234$, $12^23^2$ $12^234$, $12^23^24$, $12^23^24^2$\\
Nr. $7$ with 16 positive roots (type $B_4$)$\colon$ $1$, $2$, $3$, $4$, $12$, $23$, $34$, $1^22$, $123$, $234$, $1^223$, $1234$, $1^22^23$, $1^2234$, $1^22^234$, $1^22^23^24$
}
\end{tiny}

%\section{Rank $5$}
\hspace{12pt}\\
\baselineskip=10pt
\begin{tiny}
\noindent
\textnormal{\textbf{Rank $5 \colon$}
Nr. $1$ with 15 positive roots (type $A_5$)$\colon$ $1$, $2$, $3$, $4$, $5$, $12$, $23$, $34$, $45$, $123$, $234$, $345$, $1234$, $2345$, $12345$\\
Nr. $2$ with 20 positive roots (type $D_5$)$\colon$ $1$, $2$, $3$, $4$, $5$, $13$, $23$, $34$, $45$, $123$, $134$, $234$, $345$, $1234$, $1345$, $2345$, $123^24$, $12345$, $123^245$, $123^24^25$\\
Nr. $3$ with 21 positive roots (type $D'(5,1)$)$\colon$ $1$, $2$, $3$, $4$, $5$, $12$, $13$, $23$, $34$, $45$, $123$, $134$, $234$, $345$, $1234$, $1345$, $2345$, $123^24$, $12345$, $123^245$, $123^24^25$\\
Nr. $4$ with 22 positive roots (type $D'(5,2)$)$\colon$ $1$, $2$, $3$, $4$, $5$, $12$, $13$, $23$, $34$, $45$, $123$, $134$, $234$, $345$, $123^2$, $1234$, $1345$, $2345$, $123^24$, $12345$, $123^245$, $123^24^25$\\
Nr. $5$ with 23 positive roots (type $D'(5,3)$)$\colon$ $1$, $2$, $3$, $4$, $5$, $12$, $13$, $23$, $34$, $45$, $123$, $134$, $234$, $345$, $123^2$, $1234$, $1345$, $2345$, $123^24$, $12345$, $123^24^2$, $123^245$, $123^24^25$\\
Nr. $6$ with 24 positive roots (type $D'(5,4)$)$\colon$ $1$, $2$, $3$, $4$, $5$, $12$, $13$, $23$, $34$, $45$, $123$, $134$, $234$, $345$, $123^2$, $1234$, $1345$, $2345$, $123^24$, $12345$, $123^24^2$, $123^245$, $123^24^25$, $123^24^25^2$\\
Nr. $7$ with 25 positive roots (type $C_5$)$\colon$ $1$, $2$, $3$, $4$, $5$, $12$, $23$, $34$, $45$, $12^2$, $123$, $234$, $345$, $12^23$, $1234$, $2345$, $12^23^2$, $12^234$, $12345$, $12^23^24$, $12^2345$, $12^23^24^2$, $12^23^245$, $12^23^24^25$, $12^23^24^25^2$\\
Nr. $8$ with 25 positive roots (type $B_5$)$\colon$ $1$, $2$, $3$, $4$, $5$, $12$, $23$, $34$, $45$, $1^22$, $123$, $234$, $345$, $1^223$, $1234$, $2345$, $1^22^23$, $1^2234$, $12345$, $1^22^234$, $1^22345$, $1^22^23^24$, $1^22^2345$, $1^22^23^245$, $1^22^23^24^25$
}
\end{tiny}

%\section{Rank $6$}
\hspace{12pt}\\
\baselineskip=10pt
\begin{tiny}
\noindent
\textnormal{\textbf{Rank $6 \colon$}
Nr. $1$ with 21 positive roots (type $A_6$)$\colon$ $1$, $2$, $3$, $4$, $5$, $6$, $12$, $23$, $34$, $45$, $56$, $123$, $234$, $345$, $456$, $1234$, $2345$, $3456$, $12345$, $23456$, $123456$\\
Nr. $2$ with 30 positive roots (type $D_6$)$\colon$ $1$, $2$, $3$, $4$, $5$, $6$, $13$, $23$, $34$, $45$, $56$, $123$, $134$, $234$, $345$, $456$, $1234$, $1345$, $2345$, $3456$, $123^24$, $12345$, $13456$, $23456$, $123^245$, $123456$, $123^24^25$, $123^2456$, $123^24^256$, $123^24^25^26$\\
Nr. $3$ with 31 positive roots (type $D'(6,1)$)$\colon$ $1$, $2$, $3$, $4$, $5$, $6$, $12$, $13$, $23$, $34$, $45$, $56$, $123$, $134$, $234$, $345$, $456$, $1234$, $1345$, $2345$, $3456$, $123^24$, $12345$, $13456$, $23456$, $123^245$, $123456$, $123^24^25$, $123^2456$, $123^24^256$, $123^24^25^26$\\
Nr. $4$ with 32 positive roots (type $D'(6,2)$)$\colon$ $1$, $2$, $3$, $4$, $5$, $6$, $12$, $13$, $23$, $34$, $45$, $56$, $123$, $134$, $234$, $345$, $456$, $123^2$, $1234$, $1345$, $2345$, $3456$, $123^24$, $12345$, $13456$, $23456$, $123^245$, $123456$, $123^24^25$, $123^2456$, $123^24^256$, $123^24^25^26$\\
Nr. $5$ with 33 positive roots (type $D'(6,3)$)$\colon$ $1$, $2$, $3$, $4$, $5$, $6$, $12$, $13$, $23$, $34$, $45$, $56$, $123$, $134$, $234$, $345$, $456$, $123^2$, $1234$, $1345$, $2345$, $3456$, $123^24$, $12345$, $13456$, $23456$, $123^24^2$, $123^245$, $123456$, $123^24^25$, $123^2456$, $123^24^256$, $123^24^25^26$\\
Nr. $6$ with 34 positive roots (type $D'(6,4)$)$\colon$ $1$, $2$, $3$, $4$, $5$, $6$, $12$, $13$, $23$, $34$, $45$, $56$, $123$, $134$, $234$, $345$, $456$, $123^2$, $1234$, $1345$, $2345$, $3456$, $123^24$, $12345$, $13456$, $23456$, $123^24^2$, $123^245$, $123456$, $123^24^25$, $123^2456$, $123^24^25^2$, $123^24^256$, $123^24^25^26$\\
Nr. $7$ with 35 positive roots (type $D'(6,5)$)$\colon$ $1$, $2$, $3$, $4$, $5$, $6$, $12$, $13$, $23$, $34$, $45$, $56$, $123$, $134$, $234$, $345$, $456$, $123^2$, $1234$, $1345$, $2345$, $3456$, $123^24$, $12345$, $13456$, $23456$, $123^24^2$, $123^245$, $123456$, $123^24^25$, $123^2456$, $123^24^25^2$, $123^24^256$, $123^24^25^26$, $123^24^25^26^2$\\
Nr. $8$ with 36 positive roots (type $C_6$)$\colon$ $1$, $2$, $3$, $4$, $5$, $6$, $12$, $23$, $34$, $45$, $56$, $12^2$, $123$, $234$, $345$, $456$, $12^23$, $1234$, $2345$, $3456$, $12^23^2$, $12^234$, $12345$, $23456$, $12^23^24$, $12^2345$, $123456$, $12^23^24^2$, $12^23^245$, $12^23456$, $12^23^24^25$, $12^23^2456$, $12^23^24^25^2$, $12^23^24^256$, $12^23^24^25^26$, $12^23^24^25^26^2$\\ 
Nr. $9$ with 36 positive roots (type $B_6$)$\colon$ $1$, $2$, $3$, $4$, $5$, $6$, $12$, $23$, $34$, $45$, $56$, $1^22$, $123$, $234$, $345$, $456$, $1^223$, $1234$, $2345$, $3456$, $1^22^23$, $1^2234$, $12345$, $23456$, $1^22^234$, $1^22345$, $123456$, $1^22^23^24$, $1^22^2345$, $1^223456$, $1^22^23^245$, $1^22^23456$, $1^22^23^24^25$, $1^22^23^2456$, $1^22^23^24^256$, $1^22^23^24^25^26$
}
\end{tiny}

%\section{Rank $7$}
\hspace{12pt}\\
\baselineskip=10pt
\begin{tiny}
\noindent
\textnormal{\textbf{Rank $7\colon$}
Nr. $1$ with 28 positive roots (type $A_7$)$\colon$ $1$, $2$, $3$, $4$, $5$, $6$, $7$, $12$, $23$, $34$, $45$, $56$, $67$, $123$, $234$, $345$, $456$, $567$, $1234$, $2345$, $3456$, $4567$, $12345$, $23456$, $34567$, $123456$, $234567$, $1234567$\\
Nr. $2$ with 42 positive roots (type $D_7$)$\colon$ $1$, $2$, $3$, $4$, $5$, $6$, $7$, $13$, $23$, $34$, $45$, $56$, $67$, $123$, $134$, $234$, $345$, $456$, $567$, $1234$, $1345$, $2345$, $3456$, $4567$, $123^24$, $12345$, $13456$, $23456$, $34567$, $123^245$, $123456$, $134567$, $234567$, $123^24^25$, $123^2456$, $1234567$, $123^24^256$, $123^24567$, $123^24^25^26$, $123^24^2567$, $123^24^25^267$, $123^24^25^26^27$\\
Nr. $3$ with 43 positive roots (type $D'(7,1)$)$\colon$ $1$, $2$, $3$, $4$, $5$, $6$, $7$, $12$, $13$, $23$, $34$, $45$, $56$, $67$, $123$, $134$, $234$, $345$, $456$, $567$, $1234$, $1345$, $2345$, $3456$, $4567$, $123^24$, $12345$, $13456$, $23456$, $34567$, $123^245$, $123456$, $134567$, $234567$, $123^24^25$, $123^2456$, $1234567$, $123^24^256$, $123^24567$, $123^24^25^26$, $123^24^2567$, $123^24^25^267$, $123^24^25^26^27$\\
Nr. $4$ with 44 positive roots (type $D'(7,2)$)$\colon$ $1$, $2$, $3$, $4$, $5$, $6$, $7$, $12$, $13$, $23$, $34$, $45$, $56$, $67$, $123$, $134$, $234$, $345$, $456$, $567$, $123^2$, $1234$, $1345$, $2345$, $3456$, $4567$, $123^24$, $12345$, $13456$, $23456$, $34567$, $123^245$, $123456$, $134567$, $234567$, $123^24^25$, $123^2456$, $1234567$, $123^24^256$, $123^24567$, $123^24^25^26$, $123^24^2567$, $123^24^25^267$, $123^24^25^26^27$\\
Nr. $5$ with 45 positive roots (type $D'(7,3)$)$\colon$ $1$, $2$, $3$, $4$, $5$, $6$, $7$, $12$, $13$, $23$, $34$, $45$, $56$, $67$, $123$, $134$, $234$, $345$, $456$, $567$, $123^2$, $1234$, $1345$, $2345$, $3456$, $4567$, $123^24$, $12345$, $13456$, $23456$, $34567$, $123^24^2$, $123^245$, $123456$, $134567$, $234567$, $123^24^25$, $123^2456$, $1234567$, $123^24^256$, $123^24567$, $123^24^25^26$, $123^24^2567$, $123^24^25^267$, $123^24^25^26^27$\\
Nr. $6$ with 46 positive roots (type $D'(7,4)$)$\colon$ $1$, $2$, $3$, $4$, $5$, $6$, $7$, $12$, $13$, $23$, $34$, $45$, $56$, $67$, $123$, $134$, $234$, $345$, $456$, $567$, $123^2$, $1234$, $1345$, $2345$, $3456$, $4567$, $123^24$, $12345$, $13456$, $23456$, $34567$, $123^24^2$, $123^245$, $123456$, $134567$, $234567$, $123^24^25$, $123^2456$, $1234567$, $123^24^25^2$, $123^24^256$, $123^24567$, $123^24^25^26$, $123^24^2567$, $123^24^25^267$, $123^24^25^26^27$\\
Nr. $7$ with 47 positive roots (type $D'(7,5)$)$\colon$ $1$, $2$, $3$, $4$, $5$, $6$, $7$, $12$, $13$, $23$, $34$, $45$, $56$, $67$, $123$, $134$, $234$, $345$, $456$, $567$, $123^2$, $1234$, $1345$, $2345$, $3456$, $4567$, $123^24$, $12345$, $13456$, $23456$, $34567$, $123^24^2$, $123^245$, $123456$, $134567$, $234567$, $123^24^25$, $123^2456$, $1234567$, $123^24^25^2$, $123^24^256$, $123^24567$, $123^24^25^26$, $123^24^2567$, $123^24^25^26^2$, $123^24^25^267$, $123^24^25^26^27$\\
Nr. $8$ with 48 positive roots (type $D'(7,6)$)$\colon$ $1$, $2$, $3$, $4$, $5$, $6$, $7$, $12$, $13$, $23$, $34$, $45$, $56$, $67$, $123$, $134$, $234$, $345$, $456$, $567$, $123^2$, $1234$, $1345$, $2345$, $3456$, $4567$, $123^24$, $12345$, $13456$, $23456$, $34567$, $123^24^2$, $123^245$, $123456$, $134567$, $234567$, $123^24^25$, $123^2456$, $1234567$, $123^24^25^2$, $123^24^256$, $123^24567$, $123^24^25^26$, $123^24^2567$, $123^24^25^26^2$, $123^24^25^267$, $123^24^25^26^27$, $123^24^25^26^27^2$\\
Nr. $9$ with 49 positive roots (type $C_7$)$\colon$ $1$, $2$, $3$, $4$, $5$, $6$, $7$, $12$, $23$, $34$, $45$, $56$, $67$, $12^2$, $123$, $234$, $345$, $456$, $567$, $12^23$, $1234$, $2345$, $3456$, $4567$, $12^23^2$, $12^234$, $12345$, $23456$, $34567$, $12^23^24$, $12^2345$, $123456$, $234567$, $12^23^24^2$, $12^23^245$, $12^23456$, $1234567$, $12^23^24^25$, $12^23^2456$, $12^234567$, $12^23^24^25^2$, $12^23^24^256$, $12^23^24567$, $12^23^24^25^26$, $12^23^24^2567$, $12^23^24^25^26^2$, $12^23^24^25^267$, $12^23^24^25^26^27$, $12^23^24^25^26^27^2$\\
Nr. $10$ with 49 positive roots (type $B_7$)$\colon$ $1$, $2$, $3$, $4$, $5$, $6$, $7$, $12$, $23$, $34$, $45$, $56$, $67$, $1^22$, $123$, $234$, $345$, $456$, $567$, $1^223$, $1234$, $2345$, $3456$, $4567$, $1^22^23$, $1^2234$, $12345$, $23456$, $34567$, $1^22^234$, $1^22345$, $123456$, $234567$, $1^22^23^24$, $1^22^2345$, $1^223456$, $1234567$, $1^22^23^245$, $1^22^23456$, $1^22^234567$, $1^22^23^24^25$, $1^22^23^2456$, $1^22^234567$, $1^22^23^24^256$, $1^22^23^24567$, $1^22^23^24^25^26$, $1^22^23^24^2567$, $1^22^23^24^25^267$, $1^22^23^24^25^26^27$
}
\end{tiny}

%\section{Rank $8$}
\hspace{12pt}\\
\baselineskip=10pt
\begin{tiny}
\noindent
\textnormal{\textbf{Rank $8\colon$ }
Nr. $1$ with 36 positive roots (type $A_8$)$\colon$ $1$, $2$, $3$, $4$, $5$, $6$, $7$, $8$, $12$, $23$, $34$, $45$, $56$, $67$, $78$, $123$, $234$, $345$, $456$, $567$, $678$, $1234$, $2345$, $3456$, $4567$, $5678$, $12345$, $23456$, $34567$, $45678$, $123456$, $234567$, $345678$, $1234567$, $2345678$, $12345678$\\
Nr. $2$ with 56 positive roots (type $D_8$)$\colon$ $1$, $2$, $3$, $4$, $5$, $6$, $7$, $8$, $13$, $23$, $34$, $45$, $56$, $67$, $78$, $123$, $134$, $234$, $345$, $456$, $567$, $678$, $1234$, $1345$, $2345$, $3456$, $4567$, $5678$, $123^24$, $12345$, $13456$, $23456$, $34567$, $45678$, $123^245$, $123456$, $134567$, $234567$, $345678$, $123^24^25$, $123^2456$, $1234567$, $1345678$, $2345678$, $123^24^256$, $123^24567$, $12345678$, $123^24^25^26$, $123^24^2567$, $123^245678$, $123^24^25^267$, $123^24^25678$, $123^24^25^26^27$, $123^24^25^2678$, $123^24^25^26^278$, $123^24^25^26^27^28$\\
Nr. $3$ with 57 positive roots (type $D'(8,1)$)$\colon$$1$, $2$, $3$, $4$, $5$, $6$, $7$, $8$, $12$, $13$, $23$, $34$, $45$, $56$, $67$, $78$, $123$, $134$, $234$, $345$, $456$, $567$, $678$, $1234$, $1345$, $2345$, $3456$, $4567$, $5678$, $123^24$, $12345$, $13456$, $23456$, $34567$, $45678$, $123^245$, $123456$, $134567$, $234567$, $345678$, $123^24^25$, $123^2456$, $1234567$, $1345678$, $2345678$, $123^24^256$, $123^24567$, $12345678$, $123^24^25^26$, $123^24^2567$, $123^245678$, $123^24^25^267$, $123^24^25678$, $123^24^25^26^27$, $123^24^25^2678$, $123^24^25^26^278$, $123^24^25^26^27^28$\\
Nr. $4$ with 58 positive roots (type $D'(8,2)$)$\colon$ $1$, $2$, $3$, $4$, $5$, $6$, $7$, $8$, $12$, $13$, $23$, $34$, $45$, $56$, $67$, $78$, $123$, $134$, $234$, $345$, $456$, $567$, $678$, $123^2$, $1234$, $1345$, $2345$, $3456$, $4567$, $5678$, $123^24$, $12345$, $13456$, $23456$, $34567$, $45678$, $123^245$, $123456$, $134567$, $234567$, $345678$, $123^24^25$, $123^2456$, $1234567$, $1345678$, $2345678$, $123^24^256$, $123^24567$, $12345678$, $123^24^25^26$, $123^24^2567$, $123^245678$, $123^24^25^267$, $123^24^25678$, $123^24^25^26^27$, $123^24^25^2678$, $123^24^25^26^278$, $123^24^25^26^27^28$\\
Nr. $5$ with 59 positive roots (type $D'(8,3)$)$\colon$ $1$, $2$, $3$, $4$, $5$, $6$, $7$, $8$, $12$, $13$, $23$, $34$, $45$, $56$, $67$, $78$, $123$, $134$, $234$, $345$, $456$, $567$, $678$, $123^2$, $1234$, $1345$, $2345$, $3456$, $4567$, $5678$, $123^24$, $12345$, $13456$, $23456$, $34567$, $45678$, $123^24^2$, $123^245$, $123456$, $134567$, $234567$, $345678$, $123^24^25$, $123^2456$, $1234567$, $1345678$, $2345678$, $123^24^256$, $123^24567$, $12345678$, $123^24^25^26$, $123^24^2567$, $123^245678$, $123^24^25^267$, $123^24^25678$, $123^24^25^26^27$, $123^24^25^2678$, $123^24^25^26^278$, $123^24^25^26^27^28$\\
Nr. $6$ with 60 positive roots (type $D'(8,4)$)$\colon$ $1$, $2$, $3$, $4$, $5$, $6$, $7$, $8$, $12$, $13$, $23$, $34$, $45$, $56$, $67$, $78$, $123$, $134$, $234$, $345$, $456$, $567$, $678$, $123^2$, $1234$, $1345$, $2345$, $3456$, $4567$, $5678$, $123^24$, $12345$, $13456$, $23456$, $34567$, $45678$, $123^24^2$, $123^245$, $123456$, $134567$, $234567$, $345678$, $123^24^25$, $123^2456$, $1234567$, $1345678$, $2345678$, $123^24^25^2$, $123^24^256$, $123^24567$, $12345678$, $123^24^25^26$, $123^24^2567$, $123^245678$, $123^24^25^267$, $123^24^25678$, $123^24^25^26^27$, $123^24^25^2678$, $123^24^25^26^278$, $123^24^25^26^27^28$\\
Nr. $7$ with 61 positive roots (type $D'(8,5)$)$\colon$ $1$, $2$, $3$, $4$, $5$, $6$, $7$, $8$, $12$, $13$, $23$, $34$, $45$, $56$, $67$, $78$, $123$, $134$, $234$, $345$, $456$, $567$, $678$, $123^2$, $1234$, $1345$, $2345$, $3456$, $4567$, $5678$, $123^24$, $12345$, $13456$, $23456$, $34567$, $45678$, $123^24^2$, $123^245$, $123456$, $134567$, $234567$, $345678$, $123^24^25$, $123^2456$, $1234567$, $1345678$, $2345678$, $123^24^25^2$, $123^24^256$, $123^24567$, $12345678$, $123^24^25^26$, $123^24^2567$, $123^245678$, $123^24^25^26^2$, $123^24^25^267$, $123^24^25678$, $123^24^25^26^27$, $123^24^25^2678$, $123^24^25^26^278$, $123^24^25^26^27^28$\\
Nr. $8$ with 62 positive roots (type $D'(8,6)$)$\colon$ $1$, $2$, $3$, $4$, $5$, $6$, $7$, $8$, $12$, $13$, $23$, $34$, $45$, $56$, $67$, $78$, $123$, $134$, $234$, $345$, $456$, $567$, $678$, $123^2$, $1234$, $1345$, $2345$, $3456$, $4567$, $5678$, $123^24$, $12345$, $13456$, $23456$, $34567$, $45678$, $123^24^2$, $123^245$, $123456$, $134567$, $234567$, $345678$, $123^24^25$, $123^2456$, $1234567$, $1345678$, $2345678$, $123^24^25^2$, $123^24^256$, $123^24567$, $12345678$, $123^24^25^26$, $123^24^2567$, $123^245678$, $123^24^25^26^2$, $123^24^25^267$, $123^24^25678$, $123^24^25^26^27$, $123^24^25^2678$, $123^24^25^26^27^2$, $123^24^25^26^278$, $123^24^25^26^27^28$\\
Nr. $9$ with 63 positive roots (type $D'(8,7)$)$\colon$ $1$, $2$, $3$, $4$, $5$, $6$, $7$, $8$, $12$, $13$, $23$, $34$, $45$, $56$, $67$, $78$, $123$, $134$, $234$, $345$, $456$, $567$, $678$, $123^2$, $1234$, $1345$, $2345$, $3456$, $4567$, $5678$, $123^24$, $12345$, $13456$, $23456$, $34567$, $45678$, $123^24^2$, $123^245$, $123456$, $134567$, $234567$, $345678$, $123^24^25$, $123^2456$, $1234567$, $1345678$, $2345678$, $123^24^25^2$, $123^24^256$, $123^24567$, $12345678$, $123^24^25^26$, $123^24^2567$, $123^245678$, $123^24^25^26^2$, $123^24^25^267$, $123^24^25678$, $123^24^25^26^27$, $123^24^25^2678$, $123^24^25^26^27^2$, $123^24^25^26^278$, $123^24^25^26^27^28$, $123^24^25^26^27^28^2$\\ 
Nr. $10$ with 64 positive roots (type $C_8$)$\colon$ $1$, $2$, $3$, $4$, $5$, $6$, $7$, $8$, $12$, $23$, $34$, $45$, $56$, $67$, $78$, $12^2$, $123$, $234$, $345$, $456$, $567$, $678$, $12^23$, $1234$, $2345$, $3456$, $4567$, $5678$, $12^23^2$, $12^234$, $12345$, $23456$, $34567$, $45678$, $12^23^24$, $12^2345$, $123456$, $234567$, $345678$, $12^23^24^2$, $12^23^245$, $12^23456$, $1234567$, $2345678$, $12^23^24^25$, $12^23^2456$, $12^234567$, $12345678$, $12^23^24^25^2$, $12^23^24^256$, $12^23^24567$, $12^2345678$, $12^23^24^25^26$, $12^23^24^2567$, $12^23^245678$, $12^23^24^25^26^2$, $12^23^24^25^267$, $12^23^24^25678$, $12^23^24^25^26^27$, $12^23^24^25^2678$, $12^23^24^25^26^27^2$, $12^23^24^25^26^278$, $12^23^24^25^26^27^28$, $12^23^24^25^26^27^28^2$\\ 
Nr. $11$ with 64 positive roots (type $B_8$)$\colon$ $1$, $2$, $3$, $4$, $5$, $6$, $7$, $8$, $12$, $23$, $34$, $45$, $56$, $67$, $78$, $1^22$, $123$, $234$, $345$, $456$, $567$, $678$, $1^223$, $1234$, $2345$, $3456$, $4567$, $5678$, $1^22^23$, $1^2234$, $12345$, $23456$, $34567$, $45678$, $1^22^234$, $1^22345$, $123456$, $234567$, $345678$, $1^22^23^24$, $1^22^2345$, $1^223456$, $1234567$, $2345678$, $1^22^23^245$, $1^22^23456$, $1^2234567$, $12345678$, $1^22^23^24^25$, $1^22^23^2456$, $1^22^234567$, $1^22345678$, $1^22^23^24^256$, $1^22^23^24567$, $1^22^2345678$, $1^22^23^24^25^26$, $1^22^23^24^2567$, $1^22^23^245678$, $1^22^23^24^25^267$, $1^22^23^24^25678$, $1^22^23^24^25^26^27$, $1^22^23^24^25^2678$, $1^22^23^24^25^26^278$, $1^22^23^24^25^26^27^28$
}
\end{tiny}

%\section{Rank $9$}
\hspace{12pt}\\
\baselineskip=10pt
\begin{tiny}
\noindent
\textnormal{\textbf{Rank $9\colon$}
Nr. $1$ with 45 positive roots (type $A_9$)$\colon$ $1$, $2$, $3$, $4$, $5$, $6$, $7$, $8$, $9$, $12$, $23$, $34$, $45$, $56$, $67$, $78$, $89$, $123$, $234$, $345$, $456$, $567$, $678$, $789$, $1234$, $2345$, $3456$, $4567$, $5678$, $6789$, $12345$, $23456$, $34567$, $45678$, $56789$, $123456$, $234567$, $345678$, $456789$, $1234567$, $2345678$, $3456789$, $12345678$, $23456789$, $123456789$\\
Nr. $2$ with 72 positive roots (type $D_9$)$\colon$ $1$, $2$, $3$, $4$, $5$, $6$, $7$, $8$, $9$, $13$, $23$, $34$, $45$, $56$, $67$, $78$, $89$, $123$, $134$, $234$, $345$, $456$, $567$, $678$, $789$, $1234$, $1345$, $2345$, $3456$, $4567$, $5678$, $6789$, $123^24$, $12345$, $13456$, $23456$, $34567$, $45678$, $56789$, $123^245$, $123456$, $134567$, $234567$, $345678$, $456789$, $123^24^25$, $123^2456$, $1234567$, $1345678$, $2345678$, $3456789$, $123^24^256$, $123^24567$, $12345678$, $13456789$, $23456789$, $123^24^2567$, $123^24^25^26$, $123^245678$, $123456789$, $123^24^25^267$, $123^24^25678$, $123^2456789$, $123^24^25^26^27$, $123^24^256789$, $123^24^25^2678$, $123^24^25^26^278$, $123^24^25^26789$, $123^24^25^26^27^28$, $123^24^25^26^2789$, $123^24^25^26^27^289$, $123^24^25^26^27^28^29$\\ 
Nr. $3$ with 73 positive roots (type $D'(9,1)$)$\colon$ $1$, $2$, $3$, $4$, $5$, $6$, $7$, $8$, $9$, $12$, $13$, $23$, $34$, $45$, $56$, $67$, $78$, $89$, $123$, $134$, $234$, $345$, $456$, $567$, $678$, $789$, $1234$, $1345$, $2345$, $3456$, $4567$, $5678$, $6789$, $123^24$, $12345$, $13456$, $23456$, $34567$, $45678$, $56789$, $123^245$, $123456$, $134567$, $234567$, $345678$, $456789$, $123^24^25$, $123^2456$, $1234567$, $1345678$, $2345678$, $3456789$, $123^24^256$, $123^24567$, $12345678$, $13456789$, $23456789$, $123^24^2567$, $123^24^25^26$, $123^245678$, $123456789$, $123^24^25^267$, $123^24^25678$, $123^2456789$, $123^24^25^26^27$, $123^24^256789$, $123^24^25^2678$, $123^24^25^26^278$, $123^24^25^26789$, $123^24^25^26^27^28$, $123^24^25^26^2789$, $123^24^25^26^27^289$, $123^24^25^26^27^28^29$\\ 
Nr. $4$ with 74 positive roots (type $D'(9,2)$)$\colon$ $1$, $2$, $3$, $4$, $5$, $6$, $7$, $8$, $9$, $12$, $13$, $23$, $34$, $45$, $56$, $67$, $78$, $89$, $123$, $134$, $234$, $345$, $456$, $567$, $678$, $789$, $123^2$ $1234$, $1345$, $2345$, $3456$, $4567$, $5678$, $6789$, $123^24$, $12345$, $13456$, $23456$, $34567$, $45678$, $56789$, $123^245$, $123456$, $134567$, $234567$, $345678$, $456789$, $123^24^25$, $123^2456$, $1234567$, $1345678$, $2345678$, $3456789$, $123^24^256$, $123^24567$, $12345678$, $13456789$, $23456789$, $123^24^2567$, $123^24^25^26$, $123^245678$, $123456789$, $123^24^25^267$, $123^24^25678$, $123^2456789$, $123^24^25^26^27$, $123^24^256789$, $123^24^25^2678$, $123^24^25^26^278$, $123^24^25^26789$, $123^24^25^26^27^28$, $123^24^25^26^2789$, $123^24^25^26^27^289$, $123^24^25^26^27^28^29$\\
Nr. $5$ with 75 positive roots (type $D'(9,3)$)$\colon$ $1$, $2$, $3$, $4$, $5$, $6$, $7$, $8$, $9$, $12$, $13$, $23$, $34$, $45$, $56$, $67$, $78$, $89$, $123$, $134$, $234$, $345$, $456$, $567$, $678$, $789$, $123^2$ $1234$, $1345$, $2345$, $3456$, $4567$, $5678$, $6789$, $123^24$, $12345$, $13456$, $23456$, $34567$, $45678$, $56789$, $123^24^2$, $123^245$, $123456$, $134567$, $234567$, $345678$, $456789$, $123^24^25$, $123^2456$, $1234567$, $1345678$, $2345678$, $3456789$, $123^24^256$, $123^24567$, $12345678$, $13456789$, $23456789$, $123^24^2567$, $123^24^25^26$, $123^245678$, $123456789$, $123^24^25^267$, $123^24^25678$, $123^2456789$, $123^24^25^26^27$, $123^24^256789$, $123^24^25^2678$, $123^24^25^26^278$, $123^24^25^26789$, $123^24^25^26^27^28$,\\ $123^24^25^26^2789$, $123^24^25^26^27^289$, $123^24^25^26^27^28^29$\\
Nr. $6$ with 76 positive roots (type $D'(9,4)$)$\colon$ $1$, $2$, $3$, $4$, $5$, $6$, $7$, $8$, $9$, $12$, $13$, $23$, $34$, $45$, $56$, $67$, $78$, $89$, $123$, $134$, $234$, $345$, $456$, $567$, $678$, $789$, $123^2$ $1234$, $1345$, $2345$, $3456$, $4567$, $5678$, $6789$, $123^24$, $12345$, $13456$, $23456$, $34567$, $45678$, $56789$, $123^24^2$, $123^245$, $123456$, $134567$, $234567$, $345678$, $456789$, $123^24^25$, $123^2456$, $1234567$, $1345678$, $2345678$, $3456789$, $123^24^25^2$, $123^24^256$, $123^24567$, $12345678$, $13456789$, $23456789$, $123^24^2567$, $123^24^25^26$, $123^245678$, $123456789$, $123^24^25^267$, $123^24^25678$, $123^2456789$, $123^24^25^26^27$, $123^24^256789$, $123^24^25^2678$, $123^24^25^26^278$, $123^24^25^26789$, $123^24^25^26^27^28$, $123^24^25^26^2789$, $123^24^25^26^27^289$, $123^24^25^26^27^28^29$\\ 
Nr. $7$ with 77 positive roots (type $D'(9,5)$)$\colon$ $1$, $2$, $3$, $4$, $5$, $6$, $7$, $8$, $9$, $12$, $13$, $23$, $34$, $45$, $56$, $67$, $78$, $89$, $123$, $134$, $234$, $345$, $456$, $567$, $678$, $789$, $123^2$ $1234$, $1345$, $2345$, $3456$, $4567$, $5678$, $6789$, $123^24$, $12345$, $13456$, $23456$, $34567$, $45678$, $56789$, $123^24^2$, $123^245$, $123456$, $134567$, $234567$, $345678$, $456789$, $123^24^25$, $123^2456$, $1234567$, $1345678$, $2345678$, $3456789$, $123^24^25^2$, $123^24^256$, $123^24567$, $12345678$, $13456789$, $23456789$, $123^24^2567$, $123^24^25^26$, $123^245678$, $123456789$, $123^24^25^26^2$, $123^24^25^267$, $123^24^25678$, $123^2456789$, $123^24^25^26^27$, $123^24^256789$, $123^24^25^2678$, $123^24^25^26^278$, $123^24^25^26789$, $123^24^25^26^27^28$, $123^24^25^26^2789$, $123^24^25^26^27^289$, $123^24^25^26^27^28^29$\\
Nr. $8$ with 78 positive roots (type $D'(9,6)$)$\colon$ $1$, $2$, $3$, $4$, $5$, $6$, $7$, $8$, $9$, $12$, $13$, $23$, $34$, $45$, $56$, $67$, $78$, $89$, $123$, $134$, $234$, $345$, $456$, $567$, $678$, $789$, $123^2$ $1234$, $1345$, $2345$, $3456$, $4567$, $5678$, $6789$, $123^24$, $12345$, $13456$, $23456$, $34567$, $45678$, $56789$, $123^24^2$, $123^245$, $123456$, $134567$, $234567$, $345678$, $456789$, $123^24^25$, $123^2456$, $1234567$, $1345678$, $2345678$, $3456789$, $123^24^25^2$, $123^24^256$, $123^24567$, $12345678$, $13456789$, $23456789$, $123^24^2567$, $123^24^25^26$, $123^245678$, $123456789$, $123^24^25^26^2$, $123^24^25^267$, $123^24^25678$, $123^2456789$, $123^24^25^26^27$, $123^24^256789$, $123^24^25^2678$, $123^24^25^26^27^2$, $123^24^25^26^278$, $123^24^25^26789$, $123^24^25^26^2789$, $123^24^25^26^27^28$, $123^24^25^26^27^289$, $123^24^25^26^27^28^29$\\ 
Nr. $9$ with 79 positive roots (type $D'(9,7)$)$\colon$ $1$, $2$, $3$, $4$, $5$, $6$, $7$, $8$, $9$, $12$, $13$, $23$, $34$, $45$, $56$, $67$, $78$, $89$, $123$, $134$, $234$, $345$, $456$, $567$, $678$, $789$, $123^2$ $1234$, $1345$, $2345$, $3456$, $4567$, $5678$, $6789$, $123^24$, $12345$, $13456$, $23456$, $34567$, $45678$, $56789$, $123^24^2$, $123^245$, $123456$, $134567$, $234567$, $345678$, $456789$, $123^24^25$, $123^2456$, $1234567$, $1345678$, $2345678$, $3456789$, $123^24^25^2$, $123^24^256$, $123^24567$, $12345678$, $13456789$, $23456789$, $123^24^2567$, $123^24^25^26$, $123^245678$, $123456789$, $123^24^25^26^2$, $123^24^25^267$, $123^24^25678$, $123^2456789$, $123^24^25^26^27$, $123^24^256789$, $123^24^25^2678$, $123^24^25^26^27^2$, $123^24^25^26^278$, $123^24^25^26789$, $123^24^25^26^27^28$, $123^24^25^26^2789$, $123^24^25^26^27^28^2$, $123^24^25^26^27^289$, $123^24^25^26^27^28^29$\\ 
Nr. $10$ with 80 positive roots (type $D'(9,8)$)$\colon$ $1$, $2$, $3$, $4$, $5$, $6$, $7$, $8$, $9$, $12$, $13$, $23$, $34$, $45$, $56$, $67$, $78$, $89$, $123$, $134$, $234$, $345$, $456$, $567$, $678$, $789$, $123^2$ $1234$, $1345$, $2345$, $3456$, $4567$, $5678$, $6789$, $123^24$, $12345$, $13456$, $23456$, $34567$, $45678$, $56789$, $123^24^2$, $123^245$, $123456$, $134567$, $234567$, $345678$, $456789$, $123^24^25$, $123^2456$, $1234567$, $1345678$, $2345678$, $3456789$, $123^24^25^2$, $123^24^256$, $123^24567$, $12345678$, $13456789$, $23456789$, $123^24^2567$, $123^24^25^26$, $123^245678$, $123456789$, $123^24^25^26^2$, $123^24^25^267$, $123^24^25678$, $123^2456789$, $123^24^25^26^27$, $123^24^256789$, $123^24^25^2678$, $123^24^25^26^27^2$, $123^24^25^26^278$, $123^24^25^26789$, $123^24^25^26^27^28$, $123^24^25^26^2789$, $123^24^25^26^27^28^2$, $123^24^25^26^27^289$, $123^24^25^26^27^28^29$, $123^24^25^26^27^28^29^2$\\
Nr. $11$ with 81 positive roots (type $C_9$)$\colon$ $1$, $2$, $3$, $4$, $5$, $6$, $7$, $8$, $9$, $12$, $23$, $34$, $45$, $56$, $67$, $78$, $89$, $12^2$, $123$, $234$, $345$, $456$, $567$, $678$, $789$, $12^23$, $1234$, $2345$, $3456$, $4567$, $5678$, $6789$, $12^23^2$, $12^234$, $12345$, $23456$, $34567$, $45678$, $56789$, $12^23^24$, $12^2345$, $123456$, $234567$, $345678$, $456789$, $12^23^24^2$, $12^23^245$, $12^23456$, $1234567$, $2345678$, $3456789$, $12^23^24^25$, $12^23^2456$, $12^234567$, $12345678$, $23456789$, $12^23^24^25^2$, $12^23^24^256$, $12^23^24567$, $12^2345678$, $123456789$, $12^23^24^25^26$, $12^23^24^2567$, $12^23^245678$, $12^23456789$, $12^23^24^25^26^2$, $12^23^24^25^267$, $12^23^24^25678$, $12^23^2456789$, $12^23^24^25^26^27$, $12^23^24^25^2678$, $12^23^24^256789$, $12^23^24^25^26^27^2$, $12^23^24^25^26^278$, $12^23^24^25^26789$, $12^23^24^25^26^27^28$, $12^23^24^25^26^2789$, $12^23^24^25^26^27^28^2$,\\$12^23^24^25^26^27^289$, $12^23^24^25^26^27^28^29$, $12^23^24^25^26^27^28^29^2$\\
Nr. $12$ with 81 positive roots (type $B_9$)$\colon$ $1$, $2$, $3$, $4$, $5$, $6$, $7$, $8$, $9$, $12$, $23$, $34$, $45$, $56$, $67$, $78$, $89$, $1^22$, $123$, $234$, $345$, $456$, $567$, $678$, $789$, $1^223$, $1234$, $2345$, $3456$, $4567$, $5678$, $6789$, $1^22^23$, $1^2234$, $12345$, $23456$, $34567$, $45678$, $56789$, $1^22^234$, $1^22345$, $123456$, $234567$, $345678$, $456789$, $1^22^23^24$, $1^22^2345$, $1^223456$, $1234567$, $2345678$, $3456789$, $1^22^23^245$, $1^22^23456$, $1^2234567$, $12345678$, $23456789$, $1^22^23^24^25$, $1^22^23^2456$, $1^22^234567$, $1^22345678$, $123456789$, $1^22^23^24^256$, $1^22^23^24567$, $1^22^2345678$, $1^223456789$, $1^22^23^24^25^26$, $1^22^23^24^2567$, $1^22^23^245678$, $1^22^23456789$, $1^22^23^24^25^267$, $1^22^23^24^25678$, $1^22^23^2456789$, $1^22^23^24^25^26^27$, $1^22^23^24^25^2678$, $1^22^23^24^256789$, $1^22^23^24^25^26^278$, $1^22^23^24^25^26789$, $1^22^23^24^25^26^27^28$,\\ $1^22^23^24^25^26^2789$, $1^22^23^24^25^26^27^289$, $1^22^23^24^25^26^27^28^29$
}
\end{tiny}

\end{appendix}

\section*{Dedication}
\begin{center}
\textit{Dedicated to Professor Zhixiang Wu on the occasion of his fifty-fifth birthday}
\end{center}
\vspace{3em}

\section*{Acknowledgment}
The authors thank N. Andruskiewitsch and H. Yamane for the invaluable suggestions to the references. Furthermore, the author J. Wang would like to thank her PhD advisor Prof. I. Heckenberger for the fruitful discussions during her PhD study in Germany, which have been crucial to the entire project. 
%参考文献
\clearpage
\bibliographystyle{plain}
\bibliography{ref}

\end{document}